\documentclass[11pt,a4paper,reqno]{amsart}
\pdfoutput=1
\usepackage{amsmath,amsthm,latexsym,amssymb,wasysym,mathrsfs}
\usepackage[margin=1.0in,tmargin=0.9in,bmargin=0.80in]{geometry}
\usepackage{hyperref}
\usepackage{bbm}
\usepackage{xargs}
\usepackage{enumerate}
\usepackage{enumitem}
\setlist[enumerate]{label={\upshape(\roman*)}}
\usepackage[textsize=footnotesize]{todonotes}
\usepackage[flushleft]{threeparttable}
\usepackage{multirow}
\usepackage{longtable}
\usepackage{times}
\usepackage{microtype}
\usepackage[numbers]{natbib}
\usepackage{xcolor}     
\usepackage{colortbl}
\usepackage{tabularx}
\usepackage{ltablex}
\usepackage[linesnumbered,ruled,vlined]{algorithm2e}
\usepackage{caption}
\SetKwInput{KwInput}{Input}               
\SetKwInput{KwOutput}{Output}
\SetKwInput{Kwreturn}{return} 
\keepXColumns

\definecolor{bred}{rgb}{0.8,0,0}
\hypersetup{colorlinks,linkcolor={blue},citecolor={bred},urlcolor={blue}}

\newtheorem{theorem}{Theorem}[section]
\newtheorem{proposition}[theorem]{Proposition}
\newtheorem{lemma}[theorem]{Lemma}
\newtheorem{corollary}[theorem]{Corollary}
\newtheorem{remark}[theorem]{Remark}
\newtheorem{definition}[theorem]{Definition}

\newtheorem{assumption}{Assumption}

\def\cF{\mathsf{F}}
\def\cHOLA{\mathsf{aHOLA}}
\def\cHOLLA{\mathsf{aHOLLA}}

\def\cK{\mathsf{K}}
\def\ca{\mathsf{a}}
\def\cb{\mathsf{b}}
\def\cD{\mathsf{D}}
\def\cL{\mathsf{L}}
\def\cR{\mathsf{R}}
\def\cOS{\mathsf{OS}}

\def\cc{\mathsf{c}}
\def\cC{\mathsf{C}}
\def\cM{\mathsf{M}}

\def\ce{\mathsf{e}}
\def\cI{\mathsf{I}}
\def\cT{\mathsf{T}}
\def\cS{\mathsf{S}}

\def\cKl{\overline{\mathsf{K}}}
\def\ccl{\overline{\mathsf{c}}}
\def\Lin{\mathsf{Lin}}
\def\Deltal{\overline{\Delta}}
\def\Xil{\overline{\Xi}}
\def\cSl{\overline{\mathsf{S}}}
\def\cMl{\overline{\mathsf{M}}}
\def\Jl{\overline{\mathfrak{J}}}
\def\Lin{\mathsf{Lin}}

\def\rmd{\mathrm{d}}
\newcommand{\E}{\mathbb{E}}
\newcommand{\N}{\mathbb{N}}
\newcommand{\R}{\mathbb{R}}

\newcommand{\lfrf}[1]{\left\lfloor #1 \right\rfloor}
\newcommand{\lcrc}[1]{\left\lceil #1 \right\rceil}

\allowdisplaybreaks

\begin{document}

\title[]{Non-asymptotic estimates for accelerated high order Langevin Monte Carlo algorithms}

\author[A. Neufeld]{Ariel Neufeld}
\author[Y. Zhang]{Ying Zhang}

\address{Division of Mathematical Sciences, Nanyang Technological University, 21 Nanyang Link, 637371 Singapore}
\email{ariel.neufeld@ntu.edu.sg}

\address{Financial Technology Thrust, Society Hub, The Hong Kong University of Science and Technology (Guangzhou), 
Guangzhou, China}
\email{yingzhang@hkust-gz.edu.cn}

\date{}
\thanks{
Financial support by the MOE AcRF Tier 2 Grant MOE-T2EP20222-0013 and the Guangzhou-HKUST(GZ) Joint Funding Programs (No.\ 2024A03J0630 and No.\ 2025A03J3322) is gratefully acknowledged.}

\begin{abstract}
In this paper, we propose two new algorithms, namely, aHOLA and aHOLLA, to sample from high-dimensional target distributions with possibly super-linearly growing potentials. We establish non-asymptotic convergence bounds for aHOLA in Wasserstein-1 and Wasserstein-2 distances with rates of convergence equal to $1+q/2$ and $1/2+q/4$, respectively, under a local H\"{o}lder condition with exponent $q\in(0,1]$ and a convexity at infinity condition on the potential of the target distribution. Similar results are obtained for aHOLLA under certain global continuity conditions and a dissipativity condition. Crucially, we achieve state-of-the-art rates of convergence of the proposed algorithms in the non-convex setting which are higher than those of the existing algorithms. Examples from high-dimensional sampling and logistic regression are presented, and numerical results support our main findings.
\end{abstract}
\maketitle

\section{Introduction}
We consider the problem of sampling from a high-dimensional target distribution 
\begin{equation}\label{eq:targdist}
\pi_{\beta}(\rmd \theta)\propto \exp(-\beta U(\theta))\,\rmd \theta,
\end{equation}
where $\theta\in\R^d$, $\beta>0$ is the so-called inverse temperature parameter, and $U:\R^d \to \R$ is some (smooth enough) function. The distribution $\pi_{\beta}$ in \eqref{eq:targdist} can be viewed as the invariant measure of the Langevin stochastic differential equation (SDE) given by
\begin{equation} \label{eq:sdeintro}
Z_0 := \theta_0, \quad \rmd Z_t=-h\left(Z_t\right) \rmd t+ \sqrt{2\beta^{-1}} \rmd B_t, \quad t \geq 0,
\end{equation}
where $\theta_0$ is an $\R^d$-valued random variable, $h:= \nabla U$,  and $(B_t)_{t \geq 0}$ is a $d$-dimensional Brownian motion. Thus, to sample from $\pi_{\beta}$, one may consider using algorithms that track \eqref{eq:sdeintro}. Widely used algorithms of this type include the unadjusted Langevin algorithm (ULA) (or the Langevin Monte Carlo (LMC) algorithm) and the stochastic gradient Langevin dynamics (SGLD) algorithm, which are the Euler discretization of \eqref{eq:sdeintro}. Theoretical guarantees for ULA and SGLD to sample approximately from $\pi_{\beta}$ have been well established in the literature under the conditions that the (stochastic) gradient of the potential of $\pi_{\beta}$ is globally Lipschitz continuous and is strongly convex, see \cite{convex,ppbdm,dalalyan,dk,DM16,durmus2019analysis}. Recent research focuses on the analysis of ULA and SGLD under weaker conditions so as to accommodate a variety of distributions. To relax the strong convexity condition of $h$, \textcolor{red}{\cite{nonconvex,berkeley,mou2022improved,neufeld2024robust,raginsky,sglddiscont,xu,sgldloc}} considered replacing it with certain (local) dissipativity or convexity at infinity condition, and obtain convergence results using techniques developed in \cite{eberleold,eberle2019quantitative}. To relax the global Lipschitz condition (in $\theta$) and replace it with a local Lipschitz (or H\"{o}lder) condition, certain techniques need to be applied to modify the algorithms (see, e.g., \cite{hutzenthaler2012,liu2013strong,mao2015truncated,eulerscheme} and references therein). This is due to the fact that the absolute moments of the aforementioned algorithms could diverge to infinity at a finite time point \cite{hutzenthaler2011}. In \cite{tula}, a taming technique proposed in \cite{hutzenthaler2012,eulerscheme} is applied to obtain the tamed ULA algorithm (TULA) and convergence results are obtained under a local Lipschitz condition. Several variants of ULA and SGLD have been developed by applying the taming techniques while their convergence results are established under relaxed conditions to accommodate $\pi_{\beta}$ with super-linearly growing potentials \cite{lim2021polygonal,lim2022langevin,lim2021nonasymptotic,lovas2023taming,mtula}. We also refer to \cite{akyildiz2020nonasymptotic,chau2022stochastic,cheng2018underdamped,dalalyan2020sampling,gao2022,liang2024non,mou2021high} for high-order Langevin algorithms, to \cite{chewi2024analysis,erdogdu2021convergence,mou2022improved,nguyen2022unadjusted,vempala2019rapid} for sampling guarantees under functional inequalities, and to \cite{orieux2012sampling,vono2022high} for inverse problem sampling.

The algorithms mentioned above are first-order methods which make use of the (stochastic) gradient of the potential of $\pi_{\beta}$. The state-of-the-art rates of convergence of these algorithms in Wasserstein-2 distances are shown to be 1 in the convex case while they are 1/2 in the non-convex case, which are obtained under certain conditions imposed on the Hessian of the potential. In \cite{dk}, an LMC algorithm with Ozaki discretization (LMCO) and its variant LMCO' are proposed, which are second-order methods making use of the Hessian of the potential. It is shown in \cite{dk} that second-order methods improve on the first-order methods in ill-conditioned cases. However, no improvements are made on the rates of convergences of these methods in Wasserstein distances. In \cite{hola}, a high order LMC algorithm (HOLA) is developed by applying a taming technique to an order 1.5 numerical scheme introduced in \cite{kloeden2013numerical}. HOLA makes use of the third derivative of the potential of $\pi_{\beta}$, and \cite{hola} shows that the rate of convergence of HOLA in Wasserstein-2 distance in the convex setting is 3/2, which is higher than that of the first and second-order methods.

In this paper, we mainly consider an algorithm that can be used to sample from $\pi_{\beta}$ with a super-linearly growing potential. To this end, we propose an accelerated HOLA algorithm (aHOLA), which is a variant of HOLA in \cite{hola} obtained using a new taming factor. For completeness, we also propose an accelerated high order linear LMC algorithm (aHOLLA) which is the counterpart of aHOLA in the linear case (i.e., in the case where the derivatives of the potential of $\pi_{\beta}$ are growing at most linearly). Crucially, we obtain non-asymptotic error bounds for aHOLA in Wasserstein distances under a local H\"{o}lder condition with exponent $q\in(0,1]$ and a convexity at infinity condition. Our results are applicable to various distributions including, e.g., the double-well potential distribution, which cannot be covered by the corresponding results in Wasserstein distances for HOLA in \cite{hola}. In addition, in the non-convex setting, we obtain state-of-the-art rates of convergence of aHOLA in Wasserstein-1 and Wasserstein-2 distances equal to $1+q/2$ and $1/2+q/4$, respectively, which improve the rates of convergence of the first and second-order algorithms in the existing literature. We achieve similar convergence results for aHOLLA under certain global H\"{o}lder and Lipschitz conditions and a dissipativity condition. To illustrate the applicability of our results, we use aHOLA and aHOLLA to sample from high-dimensional distributions including a multivariate standard Gaussian distribution, a multivariate Gaussian mixture distribution, and a double-well potential distribution. Numerical results confirm the theoretical improvement of our proposed algorithms in terms of the rate of convergence in Wasserstein distances. Moreover, we present a binary logistic regression example to show that aHOLLA can be used to sample from the associated posterior distribution.

The rest of the paper is organised as follows. Section \ref{sec:main_superlinear} presents the setting and assumptions together with the main results for aHOLA where the target distributions have super-linearly growing potentials. Section \ref{sec:main_linear} presents the setting and assumptions together with the main results for aHOLLA where the target distributions have at most linearly growing potentials. Section \ref{sec:literaturereview} discusses the related results in the literature in comparison with our work to highlight our contributions. Section \ref{sec:app} illustrates numerical results which support our main findings. Section \ref{sec:mtslproofs} contains the proofs of main results in Section \ref{sec:main_superlinear}. Finally, Appendix \ref{appen:aux} contains the proofs for the auxiliary results in Sections \ref{sec:main_superlinear}, \ref{sec:app}, and \ref{sec:mtslproofs}, while Appendices \ref{appen:constexp} and \ref{appen:constexplip} present tables summarising full expressions of all constants appearing in the statements for aHOLA and aHOLLA, respectively.

\subsection{Notation}
We conclude this section by introducing some notation. Let $(\Omega,\mathcal{F},\mathbb{P})$ be a probability space. We denote by $\E[Z]$  the expectation of a random variable $Z$. For $1\leq p<\infty$, $\mathscr{L}^p$ is used to denote the usual space of $p$-integrable real-valued random variables. Fix integers $d, m \geq 1$. For an $\R^d$-valued random variable $Z$, its law on $\mathcal{B}(\R^d)$, i.e. the Borel sigma-algebra of $\R^d$, is denoted by $\mathcal{L}(Z)$. For a positive real number $a$, we denote by $\lfrf{a}$ its integer part, and $\lcrc{a} := \min\{b\in \mathbb{Z}|b\geq a\} $. The Euclidean scalar product is denoted
by $\langle \cdot,\cdot\rangle$, with $|\cdot|$ standing for the corresponding norm (where the dimension of the space may vary depending on the context). Denote by $|M|_{\cF}$ and $M^{\cT}$ the Frobenius norm and the transpose of any given matrix $M \in \R^{d\times m}$, respectively. Let $f:\R^d \rightarrow \R$ and $g:\R^d \rightarrow \R^d$ be twice continuously differentiable functions. Denote by $\nabla f$, $\nabla^2 f$ and $\Delta f$ the gradient of $f$, the Hessian of $f$, and the Laplacian of $f$, respectively. Denote by $\vec{\Delta}g: \R^d \to \R^d$ the vector Laplacian of $g$, i.e., for all $\theta \in \R^d$,  $\vec{\Delta}g(\theta)$ is a vector in $\R^d$ whose $i$-th entry is $\sum_{j=1}^d \partial^2_{\theta^{(j)}} g^{(i)}(\theta)$. For any integer $q \geq 1$, let $\mathcal{P}(\R^q)$ denote the set of probability measures on $\mathcal{B}(\R^q)$. For $\mu\in\mathcal{P}(\R^d)$ and for a $\mu$-integrable function $f:\R^d\to\R$, the notation $\mu(f):=\int_{\R^d} f(\theta)\mu(\rmd \theta)$ is used. For $\mu,\nu\in\mathcal{P}(\R^d)$, let $\mathcal{C}(\mu,\nu)$ denote the set of probability measures $\zeta$ on $\mathcal{B}(\R^{2d})$ such that its respective marginals are $\mu,\nu$. For two Borel probability measures $\mu$ and $\nu$ defined on $\R^d$ with finite $p$-th moments, the Wasserstein distance of order $p \geq 1$ is defined as
\[
{W}_p(\mu,\nu):=
\left(\inf_{\zeta\in\mathcal{C}(\mu,\nu)}\int_{\R^d}\int_{\R^d}|\theta-\overline{\theta}|^p\zeta(\rmd \theta \rmd \overline{\theta})\right)^{1/p}.
\]

\section{Assumptions and main results for aHOLA}\label{sec:main_superlinear}
In this section, we introduce the newly designed aHOLA algorithm~\eqref{eq:ahola1}-\eqref{eq:ahola3} to sample from $\pi_{\beta}$ given in \eqref{eq:targdist}. The aHOLA algorithm is a modification of the HOLA algorithm proposed in \cite{hola}, which is motivated by the order 1.5 numerical scheme of SDE \eqref{eq:sdeintro} \cite[Chapter 10]{kloeden2013numerical}. We present the assumptions under which our main results are established, and provide convergence guarantees in Wasserstein distances for the aHOLA algorithm.

\subsection{Setting}
Let $U: \R^d \rightarrow \R$ be a three times continuously differentiable function satisfying $\int_{\R^d} e^{-\beta U(\theta)} \, \rmd \theta <\infty$ for any $\beta>0$ and denote by 
\[
h:=\nabla U, \quad H:=\nabla^2 U, \quad \Upsilon:=\vec{\Delta}(\nabla U)
\] 
its gradient, Hessian, and vector Laplacian, respectively. Furthermore, define, for any $\beta>0$,
\begin{equation}\label{eq:pibetaexp}
\pi_{\beta}(A) := \frac{\int_A e^{-\beta U(\theta)} \, \rmd \theta}{\int_{\R^d} e^{-\beta U(\theta)} \, \rmd \theta}, \quad A \in \mathcal{B}(\R^d).
\end{equation}

Let $\rho\in [2, \infty)\cap \N$ and $q \in (0,1]$. Then, for any $i,j \in \N$, $\lambda>0$, and for any function $f: \R^d \rightarrow \R^{i\times j}$, we denote by
\begin{equation}\label{eq:taming}
f_{\lambda}(\theta):= \frac{f(\theta)}{(1+\lambda^{3/2}|\theta|^{3(\rho+q-1)})^{1/3}}, \quad \theta \in \R^d.
\end{equation}
The accelerated high order Langevin Monte Carlo algorithm (aHOLA) for SDE \eqref{eq:sdeintro} is given by
\begin{equation}\label{eq:ahola1}
\theta_0^{\cHOLA}:=\theta_0 , \quad \theta_{n+1}^{\cHOLA} = \theta_n^{\cHOLA} +\lambda\phi^{\lambda}( \theta_n^{\cHOLA})+\sqrt{2\lambda\beta^{-1}}\psi^{\lambda}( \theta_n^{\cHOLA})\xi_{n+1}, \quad n \in \N_0,
\end{equation}
where $\lambda>0$ is the step size, $\beta>0$, $(\xi_n)_{n \in \N_0}$ are i.i.d.\ standard $d$-dimensional Gaussian random variables, and where for all $\theta \in \R^d$, 
\begin{equation}\label{eq:ahola2}
\phi^{\lambda}(\theta):= - h_{\lambda}( \theta) +(\lambda/2) \left(H_{\lambda}( \theta)h_{\lambda}( \theta)-\beta^{-1}\Upsilon_{\lambda}(\theta) \right),
\end{equation}
and 
\begin{equation}\label{eq:ahola3}
\psi^{\lambda}(\theta) :=\sqrt{\cI_d - \lambda H_{\lambda}(\theta)+(\lambda^2/3) (H_{\lambda}(\theta))^2}
\end{equation}
with $\cI_d$ being the $d \times d$ identity matrix. We refer to Algorithm~\ref{algo:ahola} for the implementation of aHOLA~ \eqref{eq:ahola1}-\eqref{eq:ahola3}.

\begin{remark}
We note that, as mentioned in \cite{dk,kloeden2013numerical}, instead of taking the matrix square root as in \eqref{eq:ahola3}, $\psi^{\lambda}(\theta_n^{\cHOLA})\xi_{n+1}$ in \eqref{eq:ahola1} can be computed by considering the transformation 
\begin{equation}\label{eq:ahola3a}
\left(\cI_d- (1/2) \lambda H_{\lambda}(\theta_n^{\cHOLA})\right)\widehat{\xi}_{n+1}+(\sqrt{3}/6)\lambda  H_{\lambda}(\theta_n^{\cHOLA})\xi'_{n+1} 
\end{equation}
with $(\widehat{\xi}_n)_{n\in \N_0}$ and $(\xi'_n)_{n\in \N_0}$ being two independent standard Gaussian vectors independent of $(\xi_n)_{n \in \N_0}$ and $\theta_0$ (see also Algorithm \ref{algo:ahola2ndv}). This is due to the fact that \eqref{eq:ahola3a} has the same distribution as \eqref{eq:ahola3}. 

Furthermore, we note that aHOLA \eqref{eq:ahola1}-\eqref{eq:ahola3} is developed based on the (tamed) order 1.5 scheme of SDE~\eqref{eq:sdeintro}, see \cite[Chapter 10]{kloeden2013numerical} and \cite{sabanis2019explicit}, which is a class of explicit numerical schemes obtained using the It\^{o}-Taylor expansion \cite{MR715753} and is given by
\begin{align}\label{eq:aholaoh}
\begin{split}
\theta_{n+1}^{\lambda} &= \theta_n^{\lambda} -\lambda h_{\lambda}(\theta_n^{\lambda})+(\lambda^2/2)\left(H_{\lambda}( \theta_n^{\lambda})h_{\lambda}( \theta_n^{\lambda})-\beta^{-1}\Upsilon_{\lambda}(\theta_n^{\lambda}) \right)\\
&\quad +\sqrt{2\lambda\beta^{-1}} \widetilde{\xi}_{n+1}-\lambda\sqrt{2\lambda\beta^{-1}} H_{\lambda}(\theta_n^{\lambda})\overline{\xi}_{n+1}, \quad n \in \N_0,
\end{split}
\end{align}
where $\theta_0^{\lambda} := \theta_0$, $(\widetilde{\xi}_{n})_{n\in\N_0}$ is a sequence of i.i.d.\ standard $d$-dimensional Gaussian random variables, and $(\overline{\xi}_n)_{n \in \N_0}$ is a sequence of i.i.d. $d$-dimensional Gaussian random variables with mean $0$ and covariance $(1/3)\cI_d$. More precisely, for any $n \in \N_0$, $\overline{\xi}_{n+1} = \int_{n}^{n+1}\int_{n}^s\,\rmd B_r^{\lambda}\,\rmd s$, where $B^{\lambda}_t :=B_{\lambda t}/\sqrt{\lambda}$, $t \geq 0$. One observes that the law of aHOLA \eqref{eq:ahola1}-\eqref{eq:ahola3} coincides with that of the algorithm \eqref{eq:aholaoh} at grid points, i.e., for each $n \in \N_0$, $\mathcal{L}(\theta_n^{\cHOLA}) = \mathcal{L}(\theta_n^{\lambda} ) $. It is assumed throughout the paper that the $\R^d$-valued random variable $\theta_0$ (the initial condition) is independent of $(\xi_{n})_{n\in\N_0}$, $(\widetilde{\xi}_{n})_{n\in\N_0}$, and $(\overline{\xi}_n)_{n \in \N_0}$.
 
Crucially, the (tamed) order 1.5 scheme converges in the strong $\mathscr{L}^2$-sense with the rate of convergence equal to 1.5 \cite{kloeden2013numerical,sabanis2019explicit}, which is higher than that for the corresponding Euler and Milstein schemes. Hence, we are interested in extending such a result to an infinite time horizon in Wasserstein distances under relaxed conditions. To this end, we propose the aHOLA algorithm \eqref{eq:ahola1}-\eqref{eq:ahola3}, which can be viewed as a variant of the HOLA algorithm proposed in \cite{hola}. The key difference between the two algorithms is that aHOLA adopts a newly designed taming factor defined in \eqref{eq:taming}. This allows us to derive uniform in time moment estimates in Lemma~\ref{lem:2ndpthmmt}, and to establish the convergence result in Wasserstein-1 distance with a desired rate of convergence as shown in Theorem~\ref{thm:mainw1} which cannot be achieved using HOLA~\cite{hola}.
\end{remark}

\begin{algorithm}[!ht]
\caption{aHOLA Algorithm \eqref{eq:ahola1}-\eqref{eq:ahola3}} \label{algo:ahola}
\KwInput{$\beta,\delta>0$, $d\in\N$, measurable function $U:\R^d\to \R$, initialization $\theta_0\in\R^d$.}
\KwOutput{Estimator $\theta_n^{\cHOLA}$.}
Set $\rho, q, L, K_h, K_H$ to be the constants given by Assumption \ref{asm:ALL}.\\
Set $a, b, r, \overline{r}$ to be the constants given by Assumption \ref{asm:AC}.\\
Set $C_0, C_1, C_2$ to be the constants given in Theorem \ref{thm:mainw1}.\\
Set $\lambda_{\max}:=\eqref{eq:stepsizemax}$.\\
Fix $\lambda \in \left(0, \min \left\{(\delta/(2C_2))^{2/(2+q)}, \lambda_{\max}\right\}  \right)$.\\
Fix $n>\max\left\{\delta^{-2/(2+q)}(2C_2)^{2/(2+q)}/C_0, 1/(\lambda_{\max}C_0)\right\}\ln(2C_1(\E[|\theta_0|^{16(\rho+1)}]+1)/\delta)$.\\
Set $\phi^{\lambda}:=\eqref{eq:ahola2}, \psi^{\lambda}:=\eqref{eq:ahola3}$.\\
Initialize $\theta_0^{\cHOLA} \leftarrow \theta_0$.\\
\For{$\mathfrak{n}=0,\cdots, n-1$} {
    Draw $\xi_{\mathfrak{n}+1}\sim \mathcal{N}(0,\cI_d)$.\\
    Set $\theta_{\mathfrak{n}+1}^{\cHOLA} = \theta_{\mathfrak{n}}^{\cHOLA} +\lambda\phi^{\lambda}( \theta_{\mathfrak{n}}^{\cHOLA})+\sqrt{2\lambda\beta^{-1}}\psi^{\lambda}( \theta_{\mathfrak{n}}^{\cHOLA})\xi_{{\mathfrak{n}}+1}$.
}
\Return $\theta_n^{\cHOLA}$.
\end{algorithm}

\begin{algorithm}[!ht]
\caption{aHOLA Algorithm \eqref{eq:ahola1}-\eqref{eq:ahola3} --- A Variant Using \eqref{eq:ahola3a}} \label{algo:ahola2ndv}
\KwInput{$\beta,\delta>0$, $d\in\N$, measurable function $U:\R^d\to \R$, initialization $\theta_0\in\R^d$.}
\KwOutput{Estimator $\theta_n^{\cHOLA}$.}
Set $\rho, q, L, K_h, K_H$ to be the constants given by Assumption \ref{asm:ALL}.\\
Set $a, b, r, \overline{r}$ to be the constants given by Assumption \ref{asm:AC}.\\
Set $C_0, C_1, C_2$ to be the constants given in Theorem \ref{thm:mainw1}.\\
Set $\lambda_{\max}:=\eqref{eq:stepsizemax}$.\\
Fix $\lambda \in \left(0, \min \left\{(\delta/(2C_2))^{2/(2+q)}, \lambda_{\max}\right\}  \right)$.\\
Fix $n>\max\left\{\delta^{-2/(2+q)}(2C_2)^{2/(2+q)}/C_0, 1/(\lambda_{\max}C_0)\right\}\ln(2C_1(\E[|\theta_0|^{16(\rho+1)}]+1)/\delta)$.\\
Set $\phi^{\lambda}:=\eqref{eq:ahola2}$.\\
Initialize $\theta_0^{\cHOLA} \leftarrow \theta_0$.\\
\For{$\mathfrak{n}=0,\cdots, n-1$} {
    Draw independent standard Gaussian vectors $\widehat{\xi}_{\mathfrak{n}+1}$ and $\xi'_{\mathfrak{n}+1}$.\\
    Set $\theta_{\mathfrak{n}+1}^{\cHOLA} = \theta_{\mathfrak{n}}^{\cHOLA} +\lambda\phi^{\lambda}( \theta_{\mathfrak{n}}^{\cHOLA}) +\sqrt{2\lambda\beta^{-1}}\left(\cI_d- (1/2) \lambda H_{\lambda}(\theta_\mathfrak{n}^{\cHOLA})\right)\widehat{\xi}_{\mathfrak{n}+1}  
    \allowbreak   \hphantom{\theta_{\mathfrak{n}+1}^{\cHOLA} ==}+(\sqrt{3}/6)\lambda  H_{\lambda}(\theta_\mathfrak{n}^{\cHOLA})\xi'_{\mathfrak{n}+1}.$
}
\Return $\theta_n^{\cHOLA}$.
\end{algorithm}

\subsection{Assumptions} Let $U: \R^d \rightarrow \R$ be three times continuously differentiable, and let $\rho\in [2, \infty)\cap \N$ and $q \in (0,1]$ be fixed. Moreover, we impose the following assumptions.

We first impose a condition on the initial value $\theta_0$.
\begin{assumption}\label{asm:AI}
The initial condition $\theta_0$ is independent of $(\xi_{n})_{n\in\N_0}$ and has a finite $16(\rho+1)$-th moment, i.e., $\E[|\theta_0|^{16(\rho+1)}]<\infty$. 
\end{assumption}

Then, we impose a local H\"{o}lder condition on the third derivative of $U$.
\begin{assumption}\label{asm:ALL}
There exists $L>0$ such that, for all $i = 1, \dots, d$, $\theta, \overline{\theta} \in \R^d$,
\[
|\nabla^2 h^{(i)}(\theta) - \nabla^2 h^{(i)}(\overline{\theta})| \leq L(1+|\theta|+|\overline{\theta}|)^{\rho-2}|\theta-\overline{\theta}|^q.
\]
In addition, there exist $K_h, K_H>0$ such that, for all $\theta \in \R^d$,
\[
|h(\theta)| \leq K_h(1+|\theta|^{\rho+q}), \quad |H(\theta)|\leq K_H(1+|\theta|^{\rho+q-1}).
\]
\end{assumption}

By Assumption \ref{asm:ALL}, we obtain local Lipschitz (or H\"{o}lder) conditions and growth conditions on the first, second, and third derivatives of $U$ as presented in the following remark. The proof is postponed to Appendix \ref{rmk:growthcproof}.
\begin{remark}\label{rmk:growthc} By Assumption \ref{asm:ALL}, for any $\theta, \overline{\theta} \in \R^d$, $i = 1, \dots, d$, we obtain the following estimates:
\begin{align}
|\nabla^2 h^{(i)}(\theta)|&\leq \cK_0(1+|\theta|)^{\rho+q-2}, \label{eq:growthc1}\\
|\nabla  h^{(i)}(\theta) - \nabla  h^{(i)}(\overline{\theta})| &\leq \cK_0(1+|\theta|+|\overline{\theta}|)^{\rho+q-2}|\theta-\overline{\theta}|, \label{eq:growthc2}\\
|\nabla h^{(i)}(\theta)|&\leq \cK_1(1+|\theta|)^{\rho+q-1},\nonumber \\
|  h^{(i)}(\theta) -  h^{(i)}(\overline{\theta})| &\leq \cK_1(1+|\theta|+|\overline{\theta}|)^{\rho+q-1}|\theta-\overline{\theta}|,\nonumber \\
| h^{(i)}(\theta)| & \leq \cK_2(1+|\theta|)^{\rho+q},\nonumber \\
|\Upsilon(\theta) -  \Upsilon(\overline{\theta})|& \leq d^{3/2}L(1+|\theta|+|\overline{\theta}|)^{\rho-2}|\theta-\overline{\theta}|^q,\label{eq:growthc3}\\
|\Upsilon(\theta)|&\leq \cK_{3,d}(1+|\theta|)^{\rho+q-2},\nonumber 
\end{align}
where $\cK_0 := 2^{1-q}\max\{L, |\nabla^2 h^{(1)}(0)|,\dots, |\nabla^2 h^{(d)}(0)|\}$, $\cK_1:=\max\{\cK_0,|\nabla  h^{(1)}(0)|,\dots, |\nabla  h^{(d)}(0)|\}$, $\cK_2:=\max\{\cK_1,|h^{(1)}(0)|,\dots, |h^{(d)}(0)|\}$, and $\cK_{3,d}:=\max\{d^{3/2}L, \Upsilon(0)\}$.
\end{remark}
\begin{remark} \label{rmk:dwasmall} One may notice that in Assumption \ref{asm:ALL}, we assume separately growth conditions of $h$ and $H$, which could have been deduced directly by using the polynomial  H\"{o}lder condition of $\nabla^2 h^{(i)}$, $i = 1,\dots, d$, as shown in Remark \ref{rmk:growthc}. The reason is that the stepsize restriction $\lambda_{\max}$ given in \eqref{eq:stepsizemax} is reciprocally related to the growth constants of $h$ and $H$ (i.e., $K_h$ and $K_H$, respectively), thus, separately imposing growth conditions allows us to optimize $\lambda_{\max}$. For example, consider the double well potential $U(\theta) = (1/4)|\theta|^4-(1/2)|\theta|^2$, $\theta \in \R^d$. In this case, it can be shown that, for any $i = 1, \dots, d$, $\nabla^2 h^{(i)}(\theta) =2\theta \ce_i^{\cT}+2\theta^{(i)}\cI_d+2\ce_i\theta^{\cT}$, where $\ce_i$ denotes the standard basis vector in $\R^d$ with its $i$-th element being $1$. We have that, for any $\theta, \overline{\theta} \in \R^d$,
\[
|\nabla^2 h^{(i)}(\theta)-\nabla^2 h^{(i)}(\overline{\theta})|\leq 6|\theta-\overline{\theta}|,
\]
which implies, by following the same arguments as in the proof of Remark \ref{rmk:growthc},  that
\[
|h(\theta)| \leq 24\sqrt{d}(1+|\theta|^3), \quad |H(\theta)|\leq 12\sqrt{d}(1+|\theta|^2).
\]
However, since $h(\theta)=\theta(|\theta|^2-1) $ and $H(\theta)=\cI_d(|\theta|^2-1)+2\theta\theta^{\cT}$, we have that
\[
|h(\theta)| \leq 2(1+|\theta|^3), \quad |H(\theta)|\leq 3(1+|\theta|^2).
\]
From the above calculations, we observe that the growth constants of $h$ and $H$ obtained by using the expressions is much smaller than those deduced using the polynomial H\"{o}lder condition. 
\end{remark}

Next, we impose a convexity at infinity condition on $U$. 
\begin{assumption}\label{asm:AC}
There exist constants $a,b>0$, and $\overline{r} \in [0,r)$ with $r:= \rho+q-1$ such that, for all $\theta, \overline{\theta} \in \R^d$,
\[
\langle \theta - \overline{\theta}, h(\theta)-h(\overline{\theta}) \rangle \geq a|\theta-\overline{\theta}|^2(|\theta|^r+|\overline{\theta}|^r) - b|\theta-\overline{\theta}|^2(|\theta|^{\overline{r}}+|\overline{\theta}|^{\overline{r}}).
\]
\end{assumption}

Under Assumptions \ref{asm:ALL} and \ref{asm:AC}, we obtain a dissipativity condition on $h$. The explicit statement is provided below and its proof is postponed to Appendix \ref{rmk:dissipativitycproof}.
\begin{remark}\label{rmk:dissipativityc} By Assumption \ref{asm:AC}, for any $\theta \in \R^d$, we obtain
\[
\langle \theta, h(\theta) \rangle \geq \ca_\cD|\theta|^{r+2}-\cb_\cD,
\]
where $\ca_\cD:=a/2$ and $\cb_\cD:=(a/2+b)\cR_\cD^{\overline{r}+2}+|h(0)|^2/(2a)$ with $\cR_\cD:=\max\{(4b/a)^{1/(r-\overline{r})}, 2^{1/r}\}$. Moreover, for any $\theta \in \R^d$, we have that
\begin{equation}\label{eq:superlineardisp}
\langle \theta, h(\theta) \rangle \geq \ca_\cD|\theta|^2-\overline{\cb}_\cD,
\end{equation}
where $\overline{\cb}_\cD:=\ca_\cD+\cb_\cD$.
\end{remark}

Furthermore, we obtain a one-sided Lipschitz condition on $h$ as presented below. The proof is postponed to Appendix \ref{rmk:oslcproof}. 
\begin{remark}\label{rmk:oslc} By Assumptions \ref{asm:ALL} and \ref{asm:AC}, for any $\theta, \overline{\theta} \in \R^d$, we obtain
\[
\langle \theta - \overline{\theta}, h(\theta)-h(\overline{\theta}) \rangle \geq  -\cL_{\cOS}| \theta - \overline{\theta}|^2,
\]
where $\cL_{\cOS}:=\sqrt{d}\cK_1(1+2\cR_{\cOS})^{\rho+q-1}>0$ with $\cR_{\cOS}: = (b/a)^{1/(r-\overline{r})}$.
\end{remark}

\begin{remark} We comment on our Assumptions~\ref{asm:AI}-\ref{asm:AC}:
\begin{itemize}
\item Assumption~\ref{asm:AI} imposes conditions on the initial value $\theta_0$ to make sense of the convergence upper bounds in Theorem~\ref{thm:mainw1} and Corollary~\ref{crl:mainw2}. In particular, the requirement that $\theta_0$ has a finite $16(\rho+1)$-th moment is due to Lemma~\ref{lem:oserroralg} and \eqref{eq:multidmvt}, which is the minimal moment estimate required on $\theta_0$ to establish the aforementioned convergence results using our current proof technique. We note that this condition can be easily satisfied by, e.g., constants or Gaussian random variables, which are typically used for algorithm initialization.
\item Assumption~\ref{asm:ALL} imposes a local H\"{o}lder condition on the third derivative of $U$, which allows super-linear growths of derivatives of $U$ (as shown in Remark~\ref{rmk:growthc}) so as to accommodate a wide range of target distributions including, e.g., the double-well potential distribution. In addition, we impose growth conditions of $h$ and $H$ to avoid unnecessarily small stepsize restriction as discussed in Remark~\ref{rmk:dwasmall}.
\item Assumptions~\ref{asm:AC} imposes a convexity at infinity condition on $U$, which is the same as \cite[Assumption~3]{mtula} and one may refer to \cite[Remark~2.4]{mtula} for a detailed discussion. This condition implies a dissipativity condition of $h$ as shown in Remark~\ref{rmk:dissipativityc}, which (together with Assumption~\ref{asm:ALL}) ensures the existence of a unique solution to SDE~\eqref{eq:sdeintro}, and which is crucial to obtain moment estimates of the aHOLA algorithm. Moreover, it implies a one-sided Lipschitz condition as shown in Remark~\ref{rmk:oslc}, which is key in establishing convergence results in Wasserstein distances.
\end{itemize}

Note that, while our assumptions accommodate a wide range of distributions, especially those with highly non-linear potentials, they exclude discrete distributions and some heavy-tailed distributions including, e.g., the Cauchy distributions.
\end{remark}
\subsection{Main results} Denote by
\begin{equation}\label{eq:stepsizemax}
\lambda_{\max}:=\min\{  1,\ca_\cD^{-1},(19\ca_\cD / 240K_h\max\{K_H,K_h\})^2,(\ca_\cD/120K_h^2K_H^2)^{2/3}, \ca_\cD/(480K_h^2K_H)\}.
\end{equation}

Under Assumptions \ref{asm:AI}, \ref{asm:ALL}, and \ref{asm:AC}, we obtain the following non-asymptotic error bound in Wasserstein-1 distance between the law of aHOLA \eqref{eq:ahola1}-\eqref{eq:ahola3} and $\pi_\beta$.

\begin{theorem}\label{thm:mainw1} Let Assumptions \ref{asm:AI}, \ref{asm:ALL}, and \ref{asm:AC} hold. Then, for any $\beta>0$, there exist positive constants $C_0, C_1,C_2$ such that, for any $n \in \N_0$, $0<\lambda\leq \lambda_{\max}$,
\[
W_1(\mathcal{L}(\theta_n^{\cHOLA}),\pi_{\beta}) \leq C_1 e^{-C_0 \lambda n}(\E[|\theta_0|^{16(\rho+1)}]+1) +C_2\lambda^{1+q/2},
\]
where $C_0, C_1,C_2$ are given explicitly in \eqref{eq:mainw1const} (see also Appendix \ref{appen:constexp}). Furthermore, for any $\beta>0$, $\delta>0$, if we choose
\begin{align*}
\lambda &\leq \min \left\{(\delta/(2C_2))^{2/(2+q)}, \lambda_{\max}\right\},\\
n &\geq \max\left\{\delta^{-2/(2+q)}(2C_2)^{2/(2+q)}/C_0, 1/(\lambda_{\max}C_0)\right\}\ln(2C_1(\E[|\theta_0|^{16(\rho+1)}]+1)/\delta),
\end{align*}
then, we have $W_1(\mathcal{L}(\theta_n^{\cHOLA}),\pi_{\beta}) \leq \delta$.
\end{theorem} 

Moreover, we can also obtain a non-asymptotic result in Wasserstein-2 distance between the law of aHOLA \eqref{eq:ahola1}-\eqref{eq:ahola3} and $\pi_\beta$ as presented below.

\begin{corollary}\label{crl:mainw2} Let Assumptions \ref{asm:AI}, \ref{asm:ALL}, and \ref{asm:AC} hold. Then, for any $\beta>0$, there exist positive constants $C_3, C_4,C_5$ such that, for any $n \in \N_0$, $0<\lambda\leq \lambda_{\max}$,
\[
W_2(\mathcal{L}(\theta_n^{\cHOLA}),\pi_{\beta}) \leq C_4 e^{-C_3\lambda n}(\E[|\theta_0|^{16(\rho+1)}]+1)^{1/2} +C_5\lambda^{1/2+q/4}
\]
where $C_3, C_4,C_5$ are given explicitly in \eqref{eq:mainw2const} (see also Appendix \ref{appen:constexp}). Furthermore, for any $\beta>0$, $\delta>0$, if we choose
\begin{align*}
\lambda &\leq \min \left\{(\delta/(2C_5))^{4/(2+q)}, \lambda_{\max}\right\},\\
n &\geq \max\left\{\delta^{-4/(2+q)}(2C_5)^{4/(2+q)}/C_3, 1/(\lambda_{\max}C_3)\right\}\ln(2C_4(\E[|\theta_0|^{16(\rho+1)}]+1)^{1/2}/\delta),
\end{align*}
then, we have $W_2(\mathcal{L}(\theta_n^{\cHOLA}),\pi_{\beta}) \leq \delta$.
\end{corollary}
The proofs of Theorem \ref{thm:mainw1} and Corollary \ref{crl:mainw2} are deferred in Section \ref{sec:mtslproofs}.
\section{Assumptions and main results for aHOLLA}\label{sec:main_linear}
In this section, we present the accelerated high order linear Langevin Monte Carlo algorithm (aHOLLA), which is the counterpart of the aHOLA algorithm in the linear setting where the derivatives of $U$ are growing at most linearly. We provide assumptions under which convergence results of aHOLLA are established, followed by the formal statement of our main results.
\subsection{Setting}
The setting in this case is similar to that described in Section \ref{sec:main_superlinear} except that there is no need to use tamed coefficients as in aHOLA \eqref{eq:ahola1}-\eqref{eq:ahola3}. More precisely, with the assumptions and notation defined up to \eqref{eq:pibetaexp}, the aHOLLA algorithm is given by 
\begin{equation}\label{eq:aholalip1}
\Theta_0^{\cHOLLA}:=\theta_0 , \quad \Theta_{n+1}^{\cHOLLA} = \Theta_n^{\cHOLLA} +\lambda\phi^{\lambda}_{\Lin} ( \Theta_n^{\cHOLLA})+\sqrt{2\lambda\beta^{-1}}\psi^{\lambda}_{\Lin}( \Theta_n^{\cHOLLA})\xi_{n+1}, \quad n \in \N_0,
\end{equation}
where for all $\theta \in \R^d$, 
\begin{equation}\label{eq:aholalip2}
\phi^{\lambda}_{\Lin}(\theta):= - h( \theta) +(\lambda/2) \left(H ( \theta)h ( \theta)-\beta^{-1}\Upsilon (\theta) \right),
\end{equation}
and 
\begin{equation}\label{eq:aholalip3}
\psi^{\lambda}_{\Lin}(\theta) :=\sqrt{\cI_d - \lambda H (\theta)+(\lambda^2/3) (H (\theta))^2}.
\end{equation}
\begin{remark}
The corresponding order 1.5 scheme of aHOLLA \eqref{eq:aholalip1}-\eqref{eq:aholalip3} in the linear case is given by
\begin{align}\label{eq:aholaohlip}
\begin{split}
\Theta_0^{\lambda} := \theta_0, \quad \Theta_{n+1}^{\lambda} &= \Theta_n^{\lambda} -\lambda h (\Theta_n^{\lambda})+(\lambda^2/2)\left(H ( \Theta_n^{\lambda})h ( \Theta_n^{\lambda})-\beta^{-1}\Upsilon (\Theta_n^{\lambda}) \right)\\
&\quad +\sqrt{2\lambda\beta^{-1}} \widetilde{\xi}_{n+1}-\lambda\sqrt{2\lambda\beta^{-1}} H (\Theta_n^{\lambda})\overline{\xi}_{n+1}, \quad n \in \N_0.
\end{split}
\end{align}
Then, it holds that $\mathcal{L}(\Theta_n^{\cHOLLA}) = \mathcal{L}(\Theta_n^{\lambda} ) $, for each $n \in \N_0$.
\end{remark}

We provide below the assumptions and main results for the aHOLLA algorithm \eqref{eq:aholalip1}-\eqref{eq:aholalip3}.
\subsection{Assumptions} Let $U: \R^d \rightarrow \R$ be three times continuously differentiable and let $q \in (0,1]$ be fixed. The following assumptions are stated, which can be viewed as counterparts to those stated in Section \ref{sec:main_superlinear}.

We first impose assumptions on the initial condition $\theta_0$.
\begin{assumption}\label{asm:AIlip}
The initial condition $\theta_0$ is independent of $(\xi_{n})_{n\in\N_0}$ and has a finite fourth moment, i.e., $\E[|\theta_0|^4]<\infty$. 
\end{assumption}

Then, we impose conditions on the first, second, and third derivatives of $U$.
\begin{assumption}\label{asm:ALLlip}
There exists $\overline{L}_1>0$ such that, for all $i = 1, \dots, d$, $\theta, \overline{\theta} \in \R^d$,
\[
|\nabla^2 h^{(i)}(\theta) - \nabla^2 h^{(i)}(\overline{\theta})| \leq \overline{L}_1|\theta-\overline{\theta}|^q.
\]
In addition, there exist $\overline{L}_2,\overline{L}_3>0$ such that, for all $\theta, \overline{\theta}  \in \R^d$,
\begin{align*}
|H(\theta) - H(\overline{\theta})| &\leq \overline{L}_2|\theta-\overline{\theta}|,\\
|h(\theta) - h(\overline{\theta})| &\leq \overline{L}_3|\theta-\overline{\theta}|.
\end{align*}
\end{assumption}

\begin{remark}\label{rmk:growthclip} By Assumption \ref{asm:ALLlip}, for any $\theta, \overline{\theta} \in \R^d$, $i = 1, \dots, d$, we obtain the following estimates:
\begin{alignat*}{3}
|\nabla^2 h^{(i)}(\theta)|
&\leq \cKl_0(1+|\theta|^q), 
&\qquad  
|H(\theta) \overline{\theta}| &\leq \overline{L}_3 | \overline{\theta}| ,
&\qquad  
|h(\theta)|  & \leq \cKl_1(1+|\theta|),\\
|\Upsilon(\theta) -  \Upsilon(\overline{\theta})|& \leq d^{3/2}\overline{L}_1|\theta-\overline{\theta}|^q,  
& 
|\Upsilon(\theta)|&\leq d\overline{L}_2,  & &
\end{alignat*}
where $\cKl_0 := \max\{\overline{L}_1, |\nabla^2 h^{(1)}(0)|,\dots, |\nabla^2 h^{(d)}(0)|\}$, $\cKl_1:=\max\{\overline{L}_3,|h(0)|\}$.
\end{remark}

Finally, we impose a dissipativity condition on $h$.
\begin{assumption}\label{asm:ADlip}
There exist constants $\overline{a},\overline{b}>0$ such that, for all $\theta \in \R^d$,
\[
\langle \theta, h(\theta) \rangle \geq \overline{a}|\theta|^2-\overline{b}.
\]
\end{assumption}
\subsection{Main results} Denote by
\begin{equation}\label{eq:stepsizemaxlip}
\overline{\lambda}_{\max}:=\min\left\{1,1/\overline{a},\overline{a}/(16\overline{L}_3\cKl_1),\overline{a}/(16\cKl_1^2), \overline{a}^{1/3}/(4\overline{L}_3^2\cKl_1^2)^{1/3}, \overline{a}^{1/2}/(16\overline{L}_3\cKl_1^2)^{1/3}  \right\}.
\end{equation}
Then, under Assumptions \ref{asm:AIlip}, \ref{asm:ALLlip}, and \ref{asm:ADlip}, we deduce the following non-asymptotic convergence results in Wasserstein distances for aHOLLA \eqref{eq:aholalip1}-\eqref{eq:aholalip3}.
\begin{theorem}\label{thm:mainw1lip} Let Assumptions \ref{asm:AIlip}, \ref{asm:ALLlip}, and \ref{asm:ADlip} hold. Then, for any $\beta>0$, there exist positive constants $C_{\Lin,0}, C_{\Lin,1},C_{\Lin,2}$ such that, for any $n \in \N_0$, $0<\lambda\leq \overline{\lambda}_{\max}$,
\[
W_1(\mathcal{L}(\Theta_n^{\cHOLLA}),\pi_{\beta}) \leq C_{\Lin,1} e^{-C_{\Lin,0} \lambda n}(\E[|\theta_0|^4]+1) +C_{\Lin,2}\lambda^{1+q/2},
\]
where $C_{\Lin,0}, C_{\Lin,1},C_{\Lin,2}$ are given explicitly in Appendix \ref{appen:constexplip}. Furthermore, for any $\beta>0$, $\delta>0$, if we choose
\begin{align*}
\lambda &\leq  \min \left\{(\delta/2C_{\Lin,2})^{2/(2+q)}, \overline{\lambda}_{\max}\right\}, \\
n &\geq \max\left\{(2C_{\Lin,2}/\delta)^{2/(2+q)}/C_{\Lin,0}, 1/(\overline{\lambda}_{\max}C_{\Lin,0})\right\}\ln(2C_{\Lin,1}(\E[|\theta_0|^4]+1)/\delta),
\end{align*}
then, we have $W_1(\mathcal{L}(\Theta_n^{\cHOLLA}),\pi_{\beta}) \leq \delta$. 
\end{theorem} 

\begin{corollary}\label{crl:mainw2lip} Let Assumptions \ref{asm:AIlip}, \ref{asm:ALLlip}, and \ref{asm:ADlip} hold. Then, for any $\beta>0$, there exist positive constants $C_{\Lin,3}, C_{\Lin,4},C_{\Lin,5}$ such that, for any $n \in \N_0$, $0<\lambda\leq \overline{\lambda}_{\max}$,
\[
W_2(\mathcal{L}(\Theta_n^{\cHOLLA}),\pi_{\beta}) \leq C_{\Lin,4} e^{-C_{\Lin,3}\lambda n}(\E[|\theta_0|^4]+1)^{1/2} +C_{\Lin,5}\lambda^{1/2+q/4}
\]
where $C_{\Lin,3}, C_{\Lin,4},C_{\Lin,5}$ are given explicitly in Appendix \ref{appen:constexplip}. Furthermore, for any $\beta>0$, $\delta>0$, if we choose
\begin{align*}
\lambda &\leq  \min \left\{(\delta/2C_{\Lin,5})^{4/(2+q)}, \overline{\lambda}_{\max}\right\}, \\
n &\geq \max\left\{(2C_{\Lin,5}/\delta)^{4/(2+q)}/C_{\Lin,3}, 1/(\overline{\lambda}_{\max}C_{\Lin,3})\right\}\ln(2C_{\Lin,4}(\E[|\theta_0|^4]+1)^{1/2}/\delta),
\end{align*}
then, we have $W_2(\mathcal{L}(\Theta_n^{\cHOLLA}),\pi_{\beta}) \leq \delta$.
\end{corollary}
The proofs of Theorem \ref{thm:mainw1lip} and Corollary \ref{crl:mainw2lip} follow the same ideas as in the proofs of Theorem~\ref{thm:mainw1} and Corollary~\ref{crl:mainw2} explained in Section~\ref{sec:mtslproofs}, and hence are deferred in Appendix \ref{appen:mtlinearproofs}.
\section{Related work and comparison}\label{sec:literaturereview}
In this section, we discuss related works in the literature and compare them with our results in Theorem~\ref{thm:mainw1}, Corollary~\ref{crl:mainw2}, Theorem \ref{thm:mainw1lip}, and Corollary~\ref{crl:mainw2lip}. More precisely, we classify these results based on two types of assumptions: (i) the curvature condition and (ii) the smoothness condition, and then discuss the corresponding convergence guarantees.

Many works \cite{convex,ppbdm,cheng2018underdamped,dalalyan,dk,dalalyan2020sampling,DM16,durmus2019analysis,mou2021high,vempala2019rapid} that investigate the sampling behavior of the Langevin-dynamics based algorithms impose a global Lipschitz condition and a strong convexity condition on $h$, i.e., there exist $\widetilde{m},\widetilde{L}>0$ such that for all $\theta, \overline{\theta}\in\R^d$,
\[
\begin{cases}|h(\theta )- h(\overline{\theta})|\leq \widetilde{L}|\theta - \overline{\theta}|, & \text{(Global Lipschitz continuity)}\\
\langle h(\theta )- h(\overline{\theta}), \theta - \overline{\theta}\rangle \geq \widetilde{m}|\theta - \overline{\theta}|^2. & \text{(Strong convexity)}
\end{cases}\]
Under these conditions, \cite{dalalyan} shows that the LMC algorithm requires $\widetilde{O}(d^3\delta^{-2})$ steps to reach a $\delta>0$ precision level in total variation distance, while \cite{convex,ppbdm,dk,DM16,durmus2019analysis} prove that the ($\delta$-)mixing time\footnote{The minimum number of steps required for an algorithm to reach within $\delta>0$ of the target distribution in certain distance.} scales as $\widetilde{O}(d\delta^{-2})$ in Wasserstein-2 distance. Moreover, \cite{vempala2019rapid} obtains the convergence result in Kullback-Leibler (KL) divergence with mixing time $\widetilde{O}(d\delta^{-1})$ under the additional assumption that the target distribution $\pi_{\beta}$ of the form \eqref{eq:targdist} satisfies a log-Sobolev inequality. For the underdamped Markov chain Monte Carlo algorithm (known also as the second-order Langevin algorithm), \cite{cheng2018underdamped,dalalyan2020sampling} establish convergence results in Wasserstein-2 distance with mixing time $\widetilde{O}(d^{1/2}\delta^{-1})$. Furthermore, \cite{mou2021high} proposes a third-order Langevin algorithm and provides a convergence result in Wasserstein-2 distance for the algorithm to sample from the distribution whose potential $U$ is of a ridge-separable form \cite[Eq.\ (1)]{mou2021high}. It is shown in \cite[Theorem 1]{mou2021high} that the mixing time is $\widetilde{O}(d^{1/4}\delta^{-1/2})$. 

However, the global Lipschitz condition and the strong convexity condition are restrictive in the sense that they cannot accommodate key applications including, e.g., the problem of sampling from the double-well potential distribution and from the posterior distribution associated with the Bridge linear regression \cite{fu1998penalized}. To address this issue, recent works focus on the relaxation of the aforementioned assumptions:
\begin{enumerate}
\item\label{item:comparisonsmth} To relax the strong convexity condition, one direction is to replace it with a convexity at infinity condition. \cite{berkeley} shows that, under the condition that $U$ is strongly convex outside a ball of radius~$c_{\cR}$, the LMC algorithm converges in Wasserstein-1 distance with mixing time $\widetilde{O}(de^{O(c_{\cR}^2)}\delta^{-2})$. 

Another direction considers to replace the strong convexity condition with a dissipativity condition \cite{nonconvex,mou2022improved,mtula,neufeld2024robust,raginsky,sglddiscont,xu,sgldloc}. In particular, \cite{raginsky,xu} establish convergence results in Wasserstein-2 distance for the SGLD algorithm, requiring respectively $\widetilde{O}(d(c_{\mathsf{spe}}\delta^4)^{-1})$ and $\widetilde{O}(d(c_{\mathsf{spe}}\delta)^{-1})$ steps together with a large enough minibatch size to reach a $\delta>0$ precision level, where $c_{\mathsf{spe}}>0$ denotes the spectral gap which scales as $O(e^{-\widetilde{O}(d)})$. Then, \cite{nonconvex}, and subsequently \cite{neufeld2024robust,sglddiscont,sgldloc}, improve substantially the results in \cite{raginsky,xu} by removing the requirement on the minibatch sizes, and show that the SGLD algorithm reaches $\delta$ precision level within $\widetilde{O}(e^{O(d)}\delta^{-2})$ and $\widetilde{O}(e^{O(d)}\delta^{-4})$ steps in Wasserstein-1 and Wasserstein-2 distances, respectively. These results are further improved in \cite{mtula} to $\widetilde{O}(e^{O(d)}\delta^{-1})$ and $\widetilde{O}(e^{O(d)}\delta^{-2})$, under a Hessian Lipschitz condition on $U$. Moreover, for the LMC algorithm with drift $h/2$, by assuming additionally a log-Sobolev inequality on $\pi_{\beta}$, \cite{mou2022improved} obtains convergence results in Wasserstein-1 and Wasserstein-2 distances with mixing times $\widetilde{O}(d^{3/2}(c_{\mathsf{LSI}}^{3/2}\delta)^{-1})$ and $\widetilde{O}(d(c_{\mathsf{LSI}}^{5/2}\delta)^{-1})$, respectively, where $c_{\mathsf{LSI}}>0$ denotes the log-Sobolev constant.

 In addition, a recent result \cite{chewi2024analysis} establishes a convergence result for the LMC algorithm in R\'{e}nyi divergence under the conditions that $\pi_{\beta}$ has a Lipschitz log-gradient and satisfies a log-Sobolev inequality with constant $c_{\mathsf{LSI}}>0$. It shows that the LMC algorithm requires $\widetilde{O}(dc_{\mathsf{LSI}}^2\delta^{-1})$ steps to reach a $\delta$ precision level.

\item \label{item:comparisoncur}To relax the global Lipschitz condition, one line of research focuses on sampling from distributions with (a mixture of) weakly smooth potentials \cite{erdogdu2021convergence,nguyen2022unadjusted}. Under the degenerate convexity at infinity condition, a $c^\mathsf{WS}_{1}$-dissipativity condition with $c^\mathsf{WS}_{1}\in[1,2]$, and a $c^\mathsf{WS}_{2}$-H\"{o}lder condition with $c^\mathsf{WS}_{2}\in[c^\mathsf{WS}_{3},1], c^\mathsf{WS}_{3}\in(0,c^\mathsf{WS}_{1}/2]$, \cite{erdogdu2021convergence} establishes a convergence result for the LMC algorithm in Wasserstein-$c^\mathsf{WS}_{1}$ distance with a mixing time $\widetilde{O}(d^{(3c^\mathsf{WS}_{1}+(1+c^\mathsf{WS}_{4})c^\mathsf{WS}_{2})/(c^\mathsf{WS}_{1}c^\mathsf{WS}_{2})} \delta^{-2c^\mathsf{WS}_{1}/c^\mathsf{WS}_{2}})$ with $c^\mathsf{WS}_{4}\geq0$. Then, \cite{nguyen2022unadjusted} considers to relax the aforementioned conditions and replaces them with the following weaker ones: a Poincar\'{e} inequality with constant $c_{\mathsf{PI}}>0$, a $\overline{c}^\mathsf{WS}_{2}$-mixture weakly smooth condition with $0<\overline{c}^\mathsf{WS}_{2}=\overline{c}^\mathsf{WS}_{2,1}<\dots<\overline{c}^\mathsf{WS}_{2,N}\leq 1$, and a $\overline{c}^\mathsf{WS}_{1}$-dissipativity condition with $\overline{c}^\mathsf{WS}_{1}\geq \overline{c}^\mathsf{WS}_{2,N}$. Under these conditions, \cite{nguyen2022unadjusted} obtains Wasserstein-$\overline{c}^\mathsf{WS}_{1}$ convergence result with a mixing time $\widetilde{O}(d^{(2+2(1+1/\overline{c}^\mathsf{WS}_{2})/\overline{c}^\mathsf{WS}_{1})\vee (2+4/\overline{c}^\mathsf{WS}_{2}+3\overline{c}^\mathsf{WS}_{2,N}/p)} c_{\mathsf{PI}}^{-(1+1/\overline{c}^\mathsf{WS}_{2})}	\delta^{-2\overline{c}^\mathsf{WS}_{1}(1+2/\overline{c}^\mathsf{WS}_{2})})$ with $p>1$. 

Another line of research focuses on sampling from distributions with super-linearly growing log-gradients \cite{lim2021polygonal,lim2022langevin,lim2021nonasymptotic,lovas2023taming,mtula,hola}. \cite{lim2021nonasymptotic,lovas2023taming} propose the TUSLA algorithm which is a variant of the SGLD algorithm, and establish convergence results in Wasserstein-1 and Wasserstein-2 distances with mixing times $\widetilde{O}(e^{O(d)}\delta^{-2})$ and $\widetilde{O}(e^{O(d)}\delta^{-4})$, respectively. These results are achieved under a polynomial Lipschitz condition imposed on the gradient of $U$. \cite{lim2021polygonal,lim2022langevin} propose another variant of the SGLD algorithm, called TH$\varepsilon$O POULA, which possesses superior empirical performance, and obtain similar convergence results to those in \cite{lim2021nonasymptotic,lovas2023taming}. In addition, \cite{mtula} proposes the mTULA algorithm, which can be viewed as a variant of the LMC algorithm. Theoretical guarantees are established in Wasserstein-1 and Wasserstein-2 distances with improved mixing times $\widetilde{O}(e^{O(d)}\delta^{-1})$ and $\widetilde{O}(e^{O(d)}\delta^{-2})$, respectively, but under a polynomial Lipschitz condition on the Hessian of $U$. \cite{hola} proposes the HOLA algorithm that utilizes the third derivative of~$U$. Under the conditions that the third derivative of $U$ is polynomially $q$-H\"{o}lder continuous, $q\in(0,1]$, and that $U$ is strongly convex, \cite{hola} shows that the HOLA algorithm converges in Wasserstein-2 distance within $\widetilde{O}(e^{O(d)}\delta^{-2/(2+q)})$ steps.
\end{enumerate}

In Table \ref{tab:comparison}, we compare our assumptions and results to those in the aforementioned works in \ref{item:comparisonsmth} and~\ref{item:comparisoncur}. In particular, our Theorem~\ref{thm:mainw1}, Corollary~\ref{crl:mainw2}, Theorem \ref{thm:mainw1lip}, and Corollary~\ref{crl:mainw2lip} provide convergence results for aHOLA \eqref{eq:ahola1}-\eqref{eq:ahola3} and aHOLLA \eqref{eq:aholalip1}-\eqref{eq:aholalip3} in Wasserstein-1 and Wasserstein-2 distances with mixing times $\widetilde{O}(e^{O(d)}\delta^{-2/(2+q)})$ and $\widetilde{O}(e^{O(d)}\delta^{-4/(2+q)})$. The results are obtained under a polynomial $q$-H\"{o}lder condition on the third derivative of $U$ (Assumption \ref{asm:ALL}) and a convexity at infinity condition (Assumption \ref{asm:AC}). The latter condition is weaker than the strong convexity condition imposed in \cite{hola}, hence allowing the application of our results to a broader range of distributions that cannot be covered by \cite[Theorems 1 and 3]{hola} including, e.g., the double-well distribution. In addition, our results improve the rates of convergence obtained in \cite{mtula}, from $\widetilde{O}(e^{O(d)}\delta^{-1})$ and $\widetilde{O}(e^{O(d)}\delta^{-2})$ to $\widetilde{O}(e^{O(d)}\delta^{-2/(2+q)})$ and $\widetilde{O}(e^{O(d)}\delta^{-4/(2+q)})$. These improvements are achieved due to the assumptions imposed on the high order derivatives of $U$.

{
\begin{table}[h]
\footnotesize 
\begin{tabularx}{\textwidth}
{ 
	>{\hsize=0.47\hsize\linewidth=\hsize}X|
	>{\hsize=2.24\hsize\linewidth=\hsize}c|
	>{\hsize=0.98\hsize\linewidth=\hsize}X|
	>{\hsize=0.96\hsize\linewidth=\hsize}X|
	>{\hsize=1.1\hsize\linewidth=\hsize}X|
	>{\hsize=0.25\hsize\linewidth=\hsize}X
} 
\hline
\hline
{WORK} 	 		& CONVERGENCE RATE 						& SMOOTHNESS 				& CURVATURE 		& OTHERS 		& DIST.\\
\hline
 \cite{berkeley}		& $\widetilde{O}(de^{O(c_{\cR}^2)}\delta^{-2})$		&Lipschitz $h$				&Convexity at $\infty$	&---			&W-1	\\
\hline
 \cite{raginsky}		& $\widetilde{O}(d(c_{\mathsf{spe}}\delta^4)^{-1})$	&Lipschitz $h$				&---				&Dissipativity	&W-2	\\
\hline	
 \cite{xu}			&  $\widetilde{O}(d(c_{\mathsf{spe}}\delta)^{-1})$	&Lipschitz $h$				&---				&Dissipativity	&W-2	\\
\hline	
\multirow{ 2}{3em}{\cite{nonconvex,neufeld2024robust,sglddiscont,sgldloc}	}	
				& $\widetilde{O}(e^{O(d)}\delta^{-2})$			&\multirow{ 2}{*}{Lipschitz $h$}	&\multirow{ 2}{*}{---}	&\multirow{ 2}{*}{Dissipativity}	&W-1	\\
	
 				& $\widetilde{O}(e^{O(d)}\delta^{-4})$			&						&				&			&W-2	\\
\hline	
\multirow{ 2}{3em}{\cite{mtula}	}	
				& $\widetilde{O}(e^{O(d)}\delta^{-1})$			&\multirow{ 2}{8em}{Lipschitz $h,H$}	&\multirow{ 2}{*}{---}	&	\multirow{ 2}{*}{Dissipativity}	&W-1	\\
	
 				& $\widetilde{O}(e^{O(d)}\delta^{-2})$			&						&				&			&W-2	\\
\hline	
\cite{hola}			& $\widetilde{O}(e^{O(d)}\delta^{-2/3})$		&Lipschitz $h, H, \nabla^2 h^{(i)}$	&Strong convexity	&---	&W-2	\\
\hline	
\multirow{ 2}{3em}{\cite{mou2022improved}	}	
				& $\widetilde{O}(d^{3/2}(c_{\mathsf{LSI}}^{3/2}\delta)^{-1})$	&\multirow{ 2}{8em}{Lipschitz $h,H$}		&\multirow{ 2}{*}{---}	&\multirow{ 2}{8em}{Dissipativity, LSI, LMC with drift $h/2$}	&W-1	\\
	
 				& $\widetilde{O}(d(c_{\mathsf{LSI}}^{5/2}\delta)^{-1})$&						&				&			&W-2	\\
\hline	
\cite{chewi2024analysis}	& $\widetilde{O}(dc_{\mathsf{LSI}}^2\delta^{-1})$	&	Lipschitz $h$				&---				&LSI 			&R\'{e}nyi	\\
\hline	
\multirow{ 2}{3em}{\textbf{Thm~\ref{thm:mainw1lip}, Cor~\ref{crl:mainw2lip}}	}	
				& $\widetilde{O}(e^{O(d)}\delta^{-2/(2+q)})$ 		&\multirow{ 2}{8em}{Lipschitz $h, H$, $q$-H\"{o}lder $\nabla^2 h^{(i)}$	}	&\multirow{ 2}{*}{---}	&\multirow{ 2}{8em}{Dissipativity}	&W-1	\\
 				&$\widetilde{O}(e^{O(d)}\delta^{-4/(2+q)})$		&						&				&			&W-2	\\
\hline	
\cite{erdogdu2021convergence}	
				& $\widetilde{O}\left(\frac{d^{(3c^\mathsf{WS}_{1}+(1+c^\mathsf{WS}_{4})c^\mathsf{WS}_{2})/(c^\mathsf{WS}_{1}c^\mathsf{WS}_{2})} }{\delta^{2c^\mathsf{WS}_{1}/c^\mathsf{WS}_{2}}}\right)$				&$c^\mathsf{WS}_{2}$-H\"{o}lder $h$				&Degenerate convexity at $\infty$	&$c^\mathsf{WS}_{1}$-dissipativity	&W-$c^\mathsf{WS}_{1}$	\\
\hline	
\cite{nguyen2022unadjusted}	
				& $\widetilde{O}\left(\frac{d^{\left(2+\frac{2(1+1/\overline{c}^\mathsf{WS}_{2})}{\overline{c}^\mathsf{WS}_{1}}\right)\vee \left(2+\frac{4}{\overline{c}^\mathsf{WS}_{2}}+\frac{3\overline{c}^\mathsf{WS}_{2,N}}{p}\right)}}{ c_{\mathsf{PI}}^{(1+1/\overline{c}^\mathsf{WS}_{2})}\delta^{2\overline{c}^\mathsf{WS}_{1}(1+2/\overline{c}^\mathsf{WS}_{2})}	}\right)$						&$\overline{c}^\mathsf{WS}_{2}$-mixture weakly smooth $h$	&---	&$\overline{c}^\mathsf{WS}_{1}$-dissipativity, Poincar\'{e} inequality	&W-$\overline{c}^\mathsf{WS}_{1}$	\\
\hline	
\multirow{ 2}{3em}{\cite{lim2021polygonal,lim2022langevin,lim2021nonasymptotic,lovas2023taming}	}	
				& $\widetilde{O}(e^{O(d)}\delta^{-2})$			&\multirow{ 2}{5em}{Polynomial Lipschitz $h$}			&\multirow{ 2}{*}{Convexity at $\infty$}			&	\multirow{ 2}{*}{---}	&W-1	\\
	
 				& $\widetilde{O}(e^{O(d)}\delta^{-4})$			&						&				&			&W-2	\\
\hline	
\multirow{ 2}{3em}{\cite{mtula}	}	
				& $\widetilde{O}(e^{O(d)}\delta^{-1})$			&\multirow{ 2}{8em}{Polynomial Lipschitz $h,H$}	&\multirow{ 2}{7em}{Convexity at $\infty$}	&	\multirow{ 2}{*}{---}	&W-1	\\
	
 				& $\widetilde{O}(e^{O(d)}\delta^{-2})$			&						&				&			&W-2	\\
\hline	
\cite{hola}			& $\widetilde{O}(e^{O(d)}\delta^{-2/(2+q)})$		&Polynomial $q$-H\"{o}lder 	$\nabla^2 h^{(i)}$		&Strong convexity	&---	&W-2	\\
\hline	
\multirow{ 2}{3em}{\textbf{Thm~\ref{thm:mainw1}, Cor~\ref{crl:mainw2}}	}	
				& $\widetilde{O}(e^{O(d)}\delta^{-2/(2+q)})$ 		&	\multirow{ 2}{7em}{Polynomial $q$-H\"{o}lder $ \nabla^2 h^{(i)}$	}		&\multirow{ 2}{*}{Convexity at $\infty$}	&\multirow{ 2}{8em}{---}	&W-1	\\
 				&$\widetilde{O}(e^{O(d)}\delta^{-4/(2+q)})$		&						&				&			&W-2	\\
\hline	
\hline
\caption{ Comparison between our main results, i.e., Theorem~\ref{thm:mainw1}, Corollary~\ref{crl:mainw2}, Theorem \ref{thm:mainw1lip}, Corollary~\ref{crl:mainw2lip} and results in other works. $h, H, \nabla^2 h^{(i)}$ denotes the first, second, third derivative of $U$, respectively. }
\label{tab:comparison}
\end{tabularx}
\end{table}
}

\section{Numerical experiments}\label{sec:app}
In this section, we illustrate the applicability of our results in Theorem \ref{thm:mainw1}, Corollary \ref{crl:mainw2}, Theorem \ref{thm:mainw1lip}, and Corollary \ref{crl:mainw2lip}. We show that aHOLA \eqref{eq:ahola1}-\eqref{eq:ahola3} and aHOLLA \eqref{eq:aholalip1}-\eqref{eq:aholalip3} can be used to solve the problems of high-dimensional sampling and binary logistic regression. We then compare the performance of our algorithms with existing algorithms including LMC~\cite{dk}, LMCO'~\cite{dk}, mTULA~\cite{mtula}, and HOLA~\cite{hola}.\footnote{The python code is available at \url{https://github.com/tracyyingzhang/aHOLA}.}

\subsection{High-dimensional sampling} 
We use aHOLA \eqref{eq:ahola1}-\eqref{eq:ahola3} and aHOLLA \eqref{eq:aholalip1}-\eqref{eq:aholalip3} to draw samples from various distributions. More precisely, we consider the following high-dimensional target distributions:
\begin{enumerate}
\item a multivariate standard Gaussian distribution with its potential given by
\begin{equation}\label{eq:gaudistr}
U(\theta):=\frac{1}{2}|\theta|^2, \quad \theta \in \R^d,
\end{equation}\label{itm:gaudistr}
\item a multivariate Gaussian mixture distribution with its potential given by
\begin{equation}\label{eq:gaumixtdistr}
U(\theta):=\frac{1}{2}|\theta - \hat{a}|^2 -\log(1+\exp(-2\langle \theta, \hat{a}\rangle)), \quad \theta \in \R^d,
\end{equation}
where $\hat{a}\in \R^d$ is a given vector with $|\hat{a}|>1$, and\label{itm:gaumixtdistr}
\item a double-well potential distribution with its potential given by
\begin{equation}\label{eq:dwdistr}
U(\theta):=\frac{1}{4}|\theta|^4-\frac{1}{2}|\theta|^2, \quad \theta \in \R^d.
\end{equation}\label{itm:dwdistr}
\end{enumerate}

\begin{figure}
\centering
\includegraphics[width=0.85\textwidth]{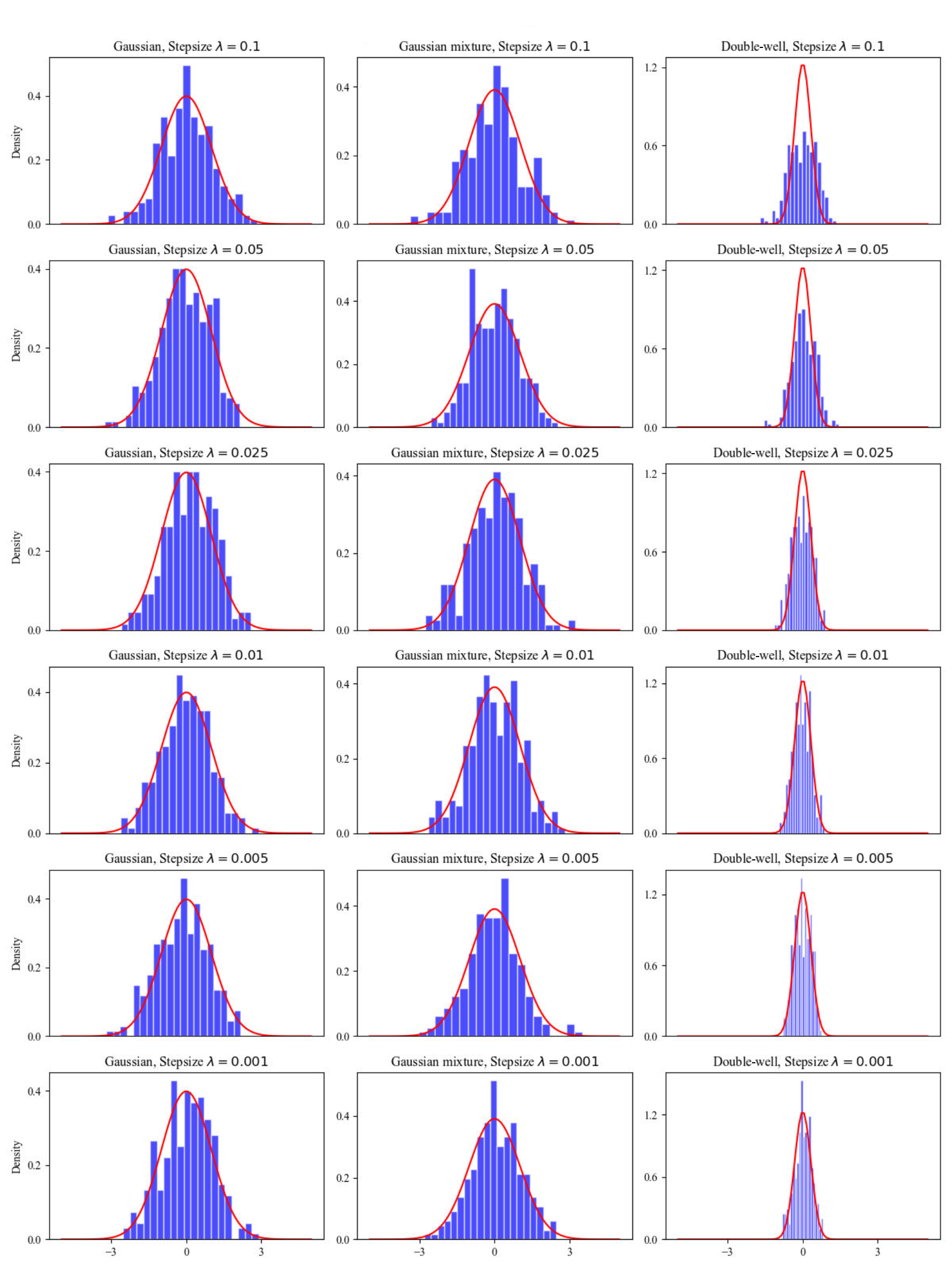}
\caption{Normalised histograms of the first components of the samples drawn using aHOLA and aHOLLA.}\label{fig:histogram}
\end{figure}

We note that the first two distributions in \ref{itm:gaudistr} and \ref{itm:gaumixtdistr} have potentials \eqref{eq:gaudistr} and \eqref{eq:gaumixtdistr} whose gradients are growing at most linearly, while the last distribution in \ref{itm:dwdistr} has potential \eqref{eq:dwdistr} whose gradient is growing super-linearly. We show in the following proposition that the examples \ref{itm:gaudistr} and \ref{itm:gaumixtdistr} satisfy Assumptions~\ref{asm:ALLlip} and \ref{asm:ADlip}, and \ref{itm:dwdistr} satisfies Assumptions \ref{asm:ALL} and \ref{asm:AC}.
\begin{proposition}\label{prop:example} The functions $U: \R^d \rightarrow \R$ given in \eqref{eq:gaudistr} and \eqref{eq:gaumixtdistr} satisfy Assumptions~\ref{asm:ALLlip} and \ref{asm:ADlip} while that given in \eqref{eq:dwdistr} satisfies Assumptions~\ref{asm:ALL} and \ref{asm:AC}.
\end{proposition}
\begin{proof}See Appendix \ref{prop:exampleproof}.
\end{proof}

For the numerical experiments, we set $d=100$. The initial value $\theta_0$ is set to be the null vector in $\R^d$ so as to satisfy Assumptions \ref{asm:AI} and \ref{asm:AIlip}. This, together with Proposition \ref{prop:example}, ensures that the proposed algorithms aHOLA \eqref{eq:ahola1}-\eqref{eq:ahola3} and aHOLLA \eqref{eq:aholalip1}-\eqref{eq:aholalip3} can sample approximately from the target distributions due to our main results. In addition, we set $\beta=1$ and consider different stepsizes $\lambda = \{0.001, 0.005, 0.01, 0.025, 0.05, 0.1\}$. The number of iterations $n$ is chosen such that $\lambda n = 400$ is fixed. For each distribution and stepsize mentioned above, we run 250 independent aHOLA or aHOLLA Markov chains and collect outputs from the last iterations of these Markov chains.

To illustrate the performance of the proposed algorithms, in Figure \ref{fig:histogram}, we plot the normalised histograms of the first components of the 250 samples generated using the method described above, and compare the histograms obtained numerically with their corresponding marginal density functions which are red lines superimposed on the plots. We note that the marginal density functions for the first component of a multivariate standard Gaussian distribution, a multivariate Gaussian mixture distribution, and a doubel-well potential distribution are given by
\begin{align}\label{eq:1ddensity}
\begin{split}
&\theta^{(1)}\ni \frac{1}{\sqrt{2\pi}}e^{-(\theta^{(1)})^2/2},\\
&\theta^{(1)}\ni \frac{1}{2\sqrt{2\pi}}\left(e^{-(\theta^{(1)}-\hat{a}^{(1)})^2/2}+e^{-(\theta^{(1)}+\hat{a}^{(1)})^2/2}\right),\\
&\theta^{(1)}\ni \frac{\Gamma(d/2)}{\sqrt{\pi}\Gamma((d-1)/2)}\frac{\int_0^\infty x^{(d-3)/2}\exp\left(-(x+(\theta^{(1)})^2)^2/4+(x+(\theta^{(1)})^2)/2\right)\, \rmd x}{\int_0^\infty x^{d/2-1}\exp\left(-x^2/4+x/2\right)\, \rmd x},
\end{split}
\end{align}
respectively, where we denote by $\theta^{(1)}, \hat{a}^{(1)}\in \R$ the first component of $\theta, \hat{a} \in \R^d$, respectively, and where $\Gamma$ denotes the gamma function. In addition, for the multivariate Gaussian mixture distribution, we set $|\hat{a}|=2$ in the experiments with all its components being equal.

In Figure \ref{fig:histogram}, we observe that the histograms obtained using samples generated by aHOLA \eqref{eq:ahola1}-\eqref{eq:ahola3} and aHOLLA \eqref{eq:aholalip1}-\eqref{eq:aholalip3} are close to their corresponding theoretical marginal density functions. This illustrates that the proposed algorithms can generate samples approximately from given target distributions with appropriately chosen stepsizes. The numerical results support our main findings, i.e., Theorem \ref{thm:mainw1}, Corollary \ref{crl:mainw2}, Theorem \ref{thm:mainw1lip}, and Corollary \ref{crl:mainw2lip}.

Moreover, we compare the sampling performance of aHOLLA and aHOLA with that of existing algorithms including LMC \cite[Eq.\ (2)]{dk},  LMCO' \cite[Eq.\ (17)]{dk}, mTULA \cite[Eqs.\ (6)-(7)]{mtula}, and HOLA~\cite[Eqs.\ (2)-(3)]{hola}. More precisely, we use aHOLLA, LMC, and LMCO' to sample from the standard normal distribution and the Gaussian mixture distribution given in \ref{itm:gaudistr} and \ref{itm:gaumixtdistr}, and use aHOLA, mTULA, and HOLA to sample from the double-well potential distribution given in \ref{itm:dwdistr}. Denote by $\pi_{\beta}^{\mathsf{sn}}$, $\pi_{\beta}^{\mathsf{gm}}$, and $\pi_{\beta}^{\mathsf{dw}}$ the marginal distributions of \ref{itm:gaudistr}, \ref{itm:gaumixtdistr} and \ref{itm:dwdistr}, respectively, with their densities given in \eqref{eq:1ddensity}. We compute the Wasserstein-1 distances between $\pi_{\beta}^{\mathsf{sn}}$, $\pi_{\beta}^{\mathsf{gm}}$, $\pi_{\beta}^{\mathsf{dw}}$ and their corresponding distributions generated by the aforementioned numerical algorithms using 200 samples (collected from 200 independent Markov chains in each setting). For the numerical experiments, we set $\beta=1,d=100$, $\theta_0=(2,2,\dots,2)\in \R^{100}$, $\lambda = \{0.025,0.05,0.075,0.1,0.25\}$, $n = 1000/\lambda$. 

In Table \ref{tab:samplingfirstbelowthreshold}, we present the first time (measured in number of iterations) the aforementioned Wasserstein distances fall below a precision level of 0.1, while we present in Table \ref{tab:samplingbestacc} the best accuracy achieved by each algorithm in Wasserstein-1 distance. We observe from Table \ref{tab:samplingfirstbelowthreshold} that, in most cases, the algorithms incorporating third order derivatives, i.e., aHOLLA, HOLA, and aHOLA, converge much faster than those utilizing only first and second order derivatives, i.e., LMC, LMCO', and mTULA. Moreover, as illustrated in Table \ref{tab:samplingbestacc}, the approximated distributions generated by aHOLLA, HOLA, and aHOLA are much closer to the target distributions in Wasserstein-1 distance compared to LMC, LMCO', and mTULA, and the distances become smaller when $\lambda$ decreases. These support our theoretical findings in Theorem \ref{thm:mainw1} and Theorem \ref{thm:mainw1lip} that aHOLLA and aHOLA converge with a higher rate of convergence in Wasserstein-1 distance. Furthermore, in Table \ref{tab:samplingrt}, we report the running time required to complete the task in each setting, while, in Table \ref{tab:samplingpic}, we present the per-iteration cost for each algorithm. We highlight that the computation time required by aHOLLA and aHOLA to complete each task is comparable to that of LMC, as it is not computationally more expensive to compute the Hessian-vector product and the vector Laplacian in \eqref{eq:ahola2}, \eqref{eq:ahola3a}, and \eqref{eq:aholalip2} than computing the gradient. The per-iteration cost for both aHOLLA and aHOLA scales as $\mathcal{O}(d)$, which is the same as that of LMC, LMCO', and mTULA. We note that the per-iteration cost for HOLA is $\mathcal{O}(d^2)$, which is due to the design of its taming factor that involves the computation of the full Hessian matrix.

\begin{table}[ht]
\footnotesize 
\begin{tabularx}{\textwidth}
{ 
	>{\hsize=1\hsize\linewidth=\hsize}c|
	>{\hsize=1\hsize\linewidth=\hsize}c
	>{\hsize=1\hsize\linewidth=\hsize}c
	>{\hsize=1\hsize\linewidth=\hsize}c
	>{\hsize=1\hsize\linewidth=\hsize}c
	>{\hsize=1\hsize\linewidth=\hsize}c
} 
\hline
\hline
			&\multicolumn{5}{c}{Standard normal distribution}\\
 	 		& $\lambda = 0.25$ 	& $\lambda = 0.1$		& $\lambda = 0.075$	& $\lambda = 0.05$	& $\lambda = 0.025$				\\
\hline
LMC			&8				&25				&32				&56				&78		\\
LMCO'			&10				&26				&27				&47				&66		\\
\rowcolor{lightgray}%
aHOLLA		&11				&26				&40				&45				&151		\\
\hline
			&\multicolumn{5}{c}{Gaussian-mixture distribution}\\	
  	 		& $\lambda = 0.25$ 	& $\lambda = 0.1$		& $\lambda = 0.075$	& $\lambda = 0.05$	& $\lambda = 0.025$				\\
\hline
LMC			&14				&42				&103				&72				&207		\\
LMCO'			&28				&40				&67				&104				&122		\\
\rowcolor{lightgray}%
aHOLLA		&14				&43				&77				&146				&104		\\
\hline
			&\multicolumn{5}{c}{Double-well potential distribution}\\	
 	 		& $\lambda = 0.25$ 	& $\lambda = 0.1$		& $\lambda = 0.075$	& $\lambda = 0.05$	& $\lambda = 0.025$				\\
\hline
mTULA		&NA				& NA				&NA				&NA					&42		\\
HOLA			&NA				& NA				&NA				&4365					&28	\\
\rowcolor{lightgray}%
aHOLA		&NA				& NA				&NA				&43					&25		\\
\hline	
\hline
\caption{{ (High-dimensional sampling)} Number of iterations required for an algorithm to reach within 0.1 precision of the target distribution in Wasserstein-1 distance. ``NA'' implies the algorithm never reaches the 0.1 precision level within $1000/\lambda$ iterations.}
\label{tab:samplingfirstbelowthreshold}
\end{tabularx}
\end{table}

\begin{table}[ht]
\footnotesize 
\begin{tabularx}{\textwidth}
{ 
	>{\hsize=1\hsize\linewidth=\hsize}c|
	>{\hsize=1\hsize\linewidth=\hsize}c
	>{\hsize=1\hsize\linewidth=\hsize}c
	>{\hsize=1\hsize\linewidth=\hsize}c
	>{\hsize=1\hsize\linewidth=\hsize}c
	>{\hsize=1\hsize\linewidth=\hsize}c
} 
\hline
\hline
			&\multicolumn{5}{c}{Standard normal distribution}\\
 	 		& $\lambda = 0.25$ 	& $\lambda = 0.1$		& $\lambda = 0.075$	& $\lambda = 0.05$	& $\lambda = 0.025$				\\
\hline
LMC			&0.0529			&0.0449			&0.0496			&0.0444			&0.0434		\\
LMCO'			&0.0533			&0.0534			&0.0419			&0.0387			&0.0438	\\
\rowcolor{lightgray}%
aHOLLA		&0.0484			&0.0403			&0.0491			&0.0459			&0.0382	\\
\hline
			&\multicolumn{5}{c}{Gaussian-mixture distribution}\\	
  	 		& $\lambda = 0.25$ 	& $\lambda = 0.1$		& $\lambda = 0.075$	& $\lambda = 0.05$	& $\lambda = 0.025$				\\
\hline
LMC			&0.051			&0.0454			&0.0445			&0.0406			&0.0432		\\
LMCO'			&0.051			&0.0471			&0.0465			&0.0429			&0.0436		\\
\rowcolor{lightgray}%
aHOLLA		&0.0457			&0.0419			&0.0457			&0.0421			&0.0397	\\
\hline
			&\multicolumn{5}{c}{Double-well potential distribution}\\	
 	 		& $\lambda = 0.25$ 	& $\lambda = 0.1$		& $\lambda = 0.075$	& $\lambda = 0.05$	& $\lambda = 0.025$				\\
\hline
mTULA		&0.3169			&0.1747			&0.1353			&0.1039				&0.05		\\
HOLA			&1.2993			&0.381			&0.2513			&0.0964				&0.0224	\\
\rowcolor{lightgray}%
aHOLA		&0.2049			&0.1337			&0.1115			&0.0795				&0.0383	\\
\hline	
\hline
\caption{(High-dimensional sampling) Best accuracy achieved among $1000/\lambda$ iterations, measured in Wasserstein-1 distance between the approximated distribution generated by a numerical algorithm and the target distribution.}
\label{tab:samplingbestacc}
\end{tabularx}
\end{table}

\begin{table}[ht]
\footnotesize 
\begin{tabularx}{\textwidth}
{ 
	>{\hsize=1\hsize\linewidth=\hsize}c|
	>{\hsize=1\hsize\linewidth=\hsize}c
	>{\hsize=1\hsize\linewidth=\hsize}c
	>{\hsize=1\hsize\linewidth=\hsize}c
	>{\hsize=1\hsize\linewidth=\hsize}c
	>{\hsize=1\hsize\linewidth=\hsize}c
} 
\hline
\hline
			&\multicolumn{5}{c}{Standard normal distribution}\\
 	 		& $\lambda = 0.25$ 	& $\lambda = 0.1$		& $\lambda = 0.075$	& $\lambda = 0.05$	& $\lambda = 0.025$				\\
\hline
LMC			&6.38				&14.33			&19.01			&28.52			&56.88		\\
LMCO'			&5.91				&14.61			&19.39			&29.05			&58.04	\\
\rowcolor{lightgray}%
aHOLLA		&6.55				&16.23			&21.63			&32.32			&64.53	\\
\hline
			&\multicolumn{5}{c}{Gaussian-mixture distribution}\\	
  	 		& $\lambda = 0.25$ 	& $\lambda = 0.1$		& $\lambda = 0.075$	& $\lambda = 0.05$	& $\lambda = 0.025$				\\
\hline
LMC			&25.38			&63.08			&84.34			&123.43			&246.84		\\
LMCO'			&25.55			&63.21			&84.23			&126.50			&252.15		\\
\rowcolor{lightgray}%
aHOLLA		&29.33			&71.36			&96.87			&142.64			&285.16	\\
\hline
			&\multicolumn{5}{c}{Double-well potential distribution}\\	
 	 		& $\lambda = 0.25$ 	& $\lambda = 0.1$		& $\lambda = 0.075$	& $\lambda = 0.05$	& $\lambda = 0.025$				\\
\hline
mTULA		&18.27			&46.58			&60.77			&92.99				&185.56		\\
HOLA			&24.00			&60.55			&79.43			&120.92				&241.55	\\
\rowcolor{lightgray}%
aHOLA		&21.70			&54.43			&72.54			&108.84				&217.81	\\
\hline	
\hline
\caption{(High-dimensional sampling) Running time (measured in seconds) for an algorithm to complete $200$ runs with each run consisting of $1000/\lambda$ iterations.}
\label{tab:samplingrt}
\end{tabularx}
\end{table}

\begin{table}[h]
\footnotesize 
\begin{tabularx}{\textwidth}
{ 
	>{\hsize=1\hsize\linewidth=\hsize}c
	>{\hsize=1\hsize\linewidth=\hsize}c
	>{\hsize=1\hsize\linewidth=\hsize}c
	>{\hsize=1\hsize\linewidth=\hsize}c
	>{\hsize=1\hsize\linewidth=\hsize}c
	>{\hsize=1\hsize\linewidth=\hsize}c
} 
\hline
\hline
LMC 	 		& LMCO'			& aHOLLA			& mTULA			&HOLA		&aHOLA				\\
\hline
$\mathcal{O}(d)$	& $\mathcal{O}(d)$	& $\mathcal{O}(d)$	&$\mathcal{O}(d)$		&$\mathcal{O}(d^2)$	&$\mathcal{O}(d)$\\
\hline	
\hline
\caption{(High-dimensional sampling) Per-iteration cost for each algorithm.}
\label{tab:samplingpic}
\end{tabularx}
\end{table}

\begin{table}
\footnotesize 
\begin{tabularx}{\textwidth}
{ 
	>{\hsize=1\hsize\linewidth=\hsize}c|
	>{\hsize=1\hsize\linewidth=\hsize}c
	>{\hsize=1\hsize\linewidth=\hsize}c
	>{\hsize=1\hsize\linewidth=\hsize}c
	>{\hsize=1\hsize\linewidth=\hsize}c
	>{\hsize=1\hsize\linewidth=\hsize}c
} 
\hline
\hline
$d = 2$		& $N = 100$				& $N = 200$				& $N = 300$				&$N = 400$		&$N = 500$	\\
\hline
LMC			&0.1261    				&0.0668  				&0.0361   				&0.0310		&0.0194	\\
LMCO'			&0.1579				&0.0668				&0.0439  				&0.0316		&0.0232	\\
\rowcolor{lightgray}%
aHOLLA		&0.1144				&0.0942				&0.0494				&0.0327		&0.0224	\\
\hline
$d = 5$ 	 	& $N = 100$				& $N = 200$				& $N = 300$				&$N = 400$		&$N = 500$	\\
\hline
LMC			& 0.3834 				&0.2140				&0.1355  				&0.1416		&0.0938	\\
LMCO'			& 0.3896				&0.2214				&0.1563   				&0.1148		&0.0881	\\
\rowcolor{lightgray}%
aHOLLA		&0.4465				&0.2121				&0.1294				&0.1142		&0.0810	\\
\hline	
$d = 10$		& $N = 100$				& $N = 200$				& $N = 300$				&$N = 400$		&$N = 500$	\\
\hline
LMC			&1.0737				&0.5184  				&0.3786  				&0.2641		&0.2413	\\
LMCO'			&1.0605				&0.5249				&0.4279				&0.2533		&0.2227	\\
\rowcolor{lightgray}%
aHOLLA		&1.1439				&0.6004				&0.3787   				&0.2976		&0.2305		\\
\hline	
\hline
\caption{(Binary logistic regression) MSE between the true value and the Bayesian posterior mean computed using $200$ samples generated by an algorithm.}
\label{tab:lrmse}
\end{tabularx}
\end{table}

\begin{table}[h]
\footnotesize 
\begin{tabularx}{\textwidth}
{ 
	>{\hsize=1\hsize\linewidth=\hsize}c|
	>{\hsize=1\hsize\linewidth=\hsize}c
	>{\hsize=1\hsize\linewidth=\hsize}c
	>{\hsize=1\hsize\linewidth=\hsize}c
	>{\hsize=1\hsize\linewidth=\hsize}c
	>{\hsize=1\hsize\linewidth=\hsize}c
} 
\hline
\hline
$d = 2$		& $N = 100$				& $N = 200$				& $N = 300$				&$N = 400$		&$N = 500$	\\
\hline
LMC			&19.55    				&22.85  				&26.72 				&30.30		&33.80	\\
LMCO'			&52.28				&62.82				&72.69 				&81.51		&91.09	\\
\rowcolor{lightgray}%
aHOLLA		&66.50				&79.88				&93.01				&106.05		&118.14	\\
\hline
$d = 5$ 	 	& $N = 100$				& $N = 200$				& $N = 300$				&$N = 400$		&$N = 500$	\\
\hline
LMC			&19.93   				&24.88				&29.32  				&33.70		&37.97	\\
LMCO'			&55.38				&69.47				&81.39				&93.34		&104.81	\\
\rowcolor{lightgray}%
aHOLLA		&69.98				&88.84				&104.65 				&120.30		&136.19	\\
\hline	
$d = 10$		& $N = 100$				& $N = 200$				& $N = 300$				&$N = 400$		&$N = 500$	\\
\hline
LMC			&21.89				&27.76 				&33.80 				&39.72		&45.28  	\\
LMCO'			&60.19				&76.62				&92.72				&108.89		&124.09	\\
\rowcolor{lightgray}%
aHOLLA		&76.02				&97.85				&119.59   				&139.62		&159.87	\\
\hline	
\hline
\caption{(Binary logistic regression) Running time (measured in seconds) for an algorithm to complete $50$ runs with each run consisting of $100000$ iterations.}
\label{tab:lrrt}
\end{tabularx}
\end{table}

\begin{table}[h]
\footnotesize 
\begin{tabularx}{\textwidth}
{ 
	>{\hsize=1\hsize\linewidth=\hsize}c
	>{\hsize=1\hsize\linewidth=\hsize}c
	>{\hsize=1\hsize\linewidth=\hsize}c
} 
\hline
\hline
LMC 	 		& LMCO'			& aHOLLA			\\
\hline
$\mathcal{O}(Nd)$	& $\mathcal{O}(Nd)$	& $\mathcal{O}(Nd)$	\\
\hline	
\hline
\caption{(Binary logistic regression) Per-iteration cost for each algorithm. $N$ denotes the number of data points used for each iteration.}
\label{tab:lrpic}
\end{tabularx}
\end{table}

\subsection{Binary logistic regression}
We consider the problem of logistic regression \cite{dalalyan} where we obtain a set of data points $\{x_i\}_{i=1}^N = \{(z_i, y_i)\}_{i=1}^N$ with $z_i \in \R^d$ denoting the feature vector and $y_i\in \{0,1\}$ denoting the binary label. Our goal is to estimate the conditional distribution of $Y_1$ given $Z_1$, which is equivalent to the estimation of the regression function $f(z) = \mathbb{P}(Y_1=1|Z_1 = z)$, $z\in\R^d$. In the logistic regression model, $f(z)$ is approximated by $f_{\theta}(z) = e^{\langle \theta, z\rangle}/(1+e^{\langle \theta, z\rangle})$. To estimate the parameter $\theta$, we adopt a Bayesian approach by sampling from a posterior distribution $\pi(\theta) = \pi_0(\theta)\Pi_{i=1}^Np(y_i|z_i, \theta)$, where $\pi_0$  is a standard normal prior and $p(y_i|z_i, \theta)$ is the likelihood function given by $p(y_i|z_i, \theta) = (1+e^{-\langle \theta, z_i\rangle})^{-y_i}(1+e^{\langle \theta, z_i\rangle})^{y_i-1}$. Thus, the posterior takes the form
\begin{equation}\label{eq:lrtargetdistribution}
\pi(\rmd \theta) \propto \exp\left\{-|\theta|^2/2 - \sum_{i=1}^N (1-y_i)\langle \theta, z_i\rangle -\sum_{i=1}^N \log(1+e^{-\langle \theta, z_i\rangle})\right\}\,\rmd \theta,
\end{equation}
which satisfies the form of \eqref{eq:targdist} with $\beta = 1$ and
\begin{equation}\label{eq:lrU}
U(\theta) = |\theta|^2/2 +\sum_{i=1}^N (1-y_i)\langle \theta, z_i\rangle +\sum_{i=1}^N \log(1+e^{-\langle \theta, z_i\rangle}),
\end{equation}
for all $\theta \in \R^d$. We aim to use aHOLLA~\eqref{eq:aholalip1}-\eqref{eq:aholalip3} to solve the logistic regression example described above, and we show, in the following result, that our theoretical findings can be applied.
\begin{proposition}\label{prop:lrexample} The function $U: \R^d \rightarrow \R$ given in \eqref{eq:lrU} satisfies Assumptions~\ref{asm:ALLlip} and \ref{asm:ADlip}.
\end{proposition}
\begin{proof}See Appendix \ref{prop:lrexampleproof}.
\end{proof}

For the numerical experiments, we set $\theta_0 = (2,2,\dots,2)\in\R^d$ to ensure that Assumption~\ref{asm:AIlip} is satisfied. Hence, Theorem \ref{thm:mainw1lip} and Corollary \ref{crl:mainw2lip} can be used to provide theoretical guarantees for aHOLLA to sample from $\pi$ defined in \eqref{eq:lrtargetdistribution}. Furthermore, we set $\beta = 1$, $d = \{2, 5, 10\}$ and $N =\{100,200,300,400,500\}$. For each pair of $d$ and $N$, and for each $i = 1,\dots,N$, we draw $z_i\in\R^d$ from a Rademacher distribution where each coordinate of $z_i$ takes the value of $1$ or $-1$ with probability $1/2$, and the corresponding $y_i$ given $z_i$ is drawn from a Bernoulli distribution with parameter $f_{\theta^*}(z_i)$, where $\theta^*$ is the true value set to be $(1,1,\dots,1)\in\R^d$. We generate samples from $\pi$ using LMC \cite[Eq.\ (2)]{dk},  LMCO' \cite[Eq.\ (17)]{dk}, and aHOLLA~\eqref{eq:aholalip1}-\eqref{eq:aholalip3}. For each algorithm, we set the stepsize $\lambda = 0.01$ and the number of iterations $n = 100000$, then we run 50 independent Markov chains and compute the sample mean using the last 200 samples for each chain.

To compare the performance of LMC, LMCO', and aHOLLA, we follow the approach described in \cite{dalalyan}. We compute the mean squared error (MSE) between the sample mean and $\theta^*$ averaged over 50 chains, and report the results in Table \ref{tab:lrmse}. We observe that the error increases when $d$ increases while it decreases when $N$ gets large. All three algorithms produce accurate and comparable results, indicating that they can be used to efficiently sample from the posterior distribution \eqref{eq:lrtargetdistribution}. Moreover, we present the running time and per-iteration cost for each algorithm in Tables \ref{tab:lrrt} and \ref{tab:lrpic}, respectively. These results show that it takes less time for LMC to complete the experiment in each setting. 

\subsection{Conclusion}
In this section, we illustrate the applicability of aHOLLA and aHOLA via two examples, namely, the high-dimensional sampling and the binary logistic regression. In the high-dimensional sampling example, aHOLLA and aHOLA exhibit superior empirical performance compared to the other algorithms in terms of the approximation accuracy measured in Wasserstein-1 distance. Crucially, the significant improvement in the accuracy is achieved by introducing only negligible (or small) amount of additional computation time. Thus, we recommend using our proposed algorithms for the problem of sampling from high-dimensional target distributions, especially those with highly non-linear potentials. Moreover, in the logistic regression example, we show that aHOLLA can be used to sample from a given posterior distribution, but it would take more time to complete the tasks compared to LMC. 

\section{Proof of main results for aHOLA}\label{sec:mtslproofs}
In this section, we provide the proofs for Theorem \ref{thm:mainw1} and Corollary \ref{crl:mainw2}. We first introduce auxiliary processes which we use throughout the convergence analysis. Then, we provide moment estimates for the newly introduced processes, which are followed by the detailed proofs for the main results. We postpone the proofs for the results presented in this section to Appendix \ref{sec:mtslproofsapd}.

\subsection{Auxiliary processes} Fix $\beta>0$. Consider the Langevin SDE $(Z_t)_{t \geq 0}$ given by
\begin{equation} \label{eq:sde}
Z_0 := \theta_0, \quad \rmd Z_t=-h\left(Z_t\right) \rmd t+ \sqrt{2\beta^{-1}} \rmd B_t,
\end{equation}
where $(B_t)_{t \geq 0}$ is a $d$-dimensional Brownian motion on $(\Omega,\mathcal{F},P)$ with its completed natural filtration denoted by $(\mathcal{F}_t)_{t\geq 0}$. Moreover, we assume that $(\mathcal{F}_t)_{t\geq 0}$ is independent of $\sigma(\theta_0)$. Under Assumptions~\ref{asm:ALL} and \ref{asm:AC}, and by Remarks \ref{rmk:dissipativityc} and \ref{rmk:oslc}, we note that the Langevin SDE \eqref{eq:sde} admits a unique solution, which is adapted to $\mathcal{F}_t \vee \sigma(\theta_0)$, $t\geq 0$, due to \cite[Theorem 1]{krylovsolmontone}. Its $2p$-th moment estimate with $p \in \N$ is provided in Lemma \ref{lem:sde2pmmt}, which can be used to further deduce the $2p$-th moment estimate of $\pi_{\beta}$ \eqref{eq:pibetaexp}. 

For each $\lambda>0$, recall $ B^{\lambda}_t :=B_{\lambda t}/\sqrt{\lambda}, t \geq 0$. Denote by $(\mathcal{F}^{\lambda}_t)_{t \geq 0}$ with $\mathcal{F}^{\lambda}_t := \mathcal{F}_{\lambda t}$, $t \geq 0$, its completed natural filtration, which is independent of $\sigma(\theta_0)$. Moreover, we denote by $Z^{\lambda}_t := Z_{\lambda t}$, $t \geq 0$, the time-changed version of Langevin SDE \eqref{eq:sde}, which is given by
\begin{equation} \label{eq:tcsde}
Z^{\lambda}_0:=\theta_0, \quad \rmd Z^{\lambda}_t = -\lambda h(Z^{\lambda}_t)\, \rmd t +\sqrt{2\lambda\beta^{-1}}\,\rmd B^{\lambda}_t.
\end{equation}

Furthermore, we denote by $(\widetilde{\theta}^{\lambda}_t)_{t \geq 0}$ the continuous-time interpolation of aHOLA \eqref{eq:ahola1}-\eqref{eq:ahola3} given by 
\begin{equation}\label{eq:aholaproc}
\widetilde{\theta}^{\lambda}_0 := \theta_0, \quad \rmd \widetilde{\theta}^{\lambda}_t= \lambda \phi^{\lambda}(\widetilde{\theta}^{\lambda}_{\lfrf{t}})\, \rmd t
+ \sqrt{2\lambda\beta^{-1}}\psi^{\lambda}(\widetilde{\theta}^{\lambda}_{\lfrf{t}}) \,\rmd B^{\lambda}_t,
\end{equation}
where $\phi^{\lambda}$ and $\psi^{\lambda}$ are defined in \eqref{eq:ahola2} and \eqref{eq:ahola3}, respectively.

\begin{remark} 
Similarly, denote by $(\overline{\theta}^{\lambda}_t)_{t \geq 0 }$ the continuous-time interpolation of the order 1.5 scheme \eqref{eq:aholaoh} given by
\begin{align}\label{eq:aholahoproc}
\begin{split}
 \rmd \overline{\theta}^{\lambda}_t&=  -\lambda h_{\lambda}(\overline{\theta}^{\lambda}_{\lfrf{t}})\,\rmd t+\lambda^2\int_{\lfrf{t}}^t\left(H_{\lambda}( \overline{\theta}^{\lambda}_{\lfrf{s}})h_{\lambda}( \overline{\theta}^{\lambda}_{\lfrf{s}})-\beta^{-1}\Upsilon_{\lambda}(\overline{\theta}^{\lambda}_{\lfrf{s}}) \right)\,\rmd s\,\rmd t\\
&\quad-\lambda\sqrt{2\lambda\beta^{-1}} \int_{{\lfrf{t}}}^t H_{\lambda}(\overline{\theta}_{\lfrf{s}}^{\lambda})\,\rmd B_s^{\lambda}\,\rmd t  +\sqrt{2\lambda\beta^{-1}} \, \rmd B^{\lambda}_t
\end{split}
\end{align}
with $\overline{\theta}^{\lambda}_0 := \theta_0$. We note that $\mathcal{L}(\widetilde{\theta}^{\lambda}_n)=\mathcal{L}(\theta_n^{\cHOLA})=\mathcal{L}(\theta_n^{\lambda})=\mathcal{L}(\overline{\theta}^{\lambda}_n)$, for each $n\in\N_0$.
\end{remark}

Finally, for any $s \geq 0$, consider the continuous-time process $(\zeta^{s,v, \lambda}_t)_{t \geq s}$ defined by
\begin{equation}\label{eq:auxproc}
\zeta^{s,v, \lambda}_s := v \in \R^d, \quad \rmd\zeta^{s,v, \lambda}_t = -\lambda h(\zeta^{s,v, \lambda}_t)\, \rmd t +\sqrt{2\lambda\beta^{-1}}\,\rmd B^{\lambda}_t.
\end{equation}
\begin{definition}\label{def:auxzeta} Fix $\lambda >0$. Define $T\equiv T(\lambda) : = \lfrf{1/\lambda}$. Then, for any $n \in \N_0$ and $t \geq nT$, define 
\[
\overline{\zeta}^{\lambda, n}_t := \zeta^{nT,\overline{\theta}^{\lambda}_{nT}, \lambda}_t.
\]
\end{definition}

\subsection{Moment estimates}\label{sec:me} We first introduce the following Lyapunov functions: for each $p\in [2, \infty)\cap {\N}$, define $V_p(\theta) := (1+|\theta|^2)^{p/2}$, for all $\theta \in \R^d$, as well as $\mathrm{v}_p(w) := (1+w^2)^{p/2}$, for all $w \geq 0$. We observe that $V_p$ is twice continuously differentiable and satisfies:
\begin{equation}\label{eq:contrassumption}
\sup_{\theta \in \R^d}|\nabla V_p(\theta)|/V_p(\theta) <\infty, \quad \lim_{|\theta|\to\infty} \nabla V_p(\theta)/V_p(\theta)=0.
\end{equation}
Furthermore, we denote by $\mathcal{P}_{V_p}(\R^d)$ the set of probability measures $\mu \in \mathcal{P}(\R^d)$ which satisfies $\int_{\R^d} V_p(\theta)\, \mu(\rmd \theta) <\infty$.

Next, we establish moment estimates for $(\widetilde{\theta}^{\lambda}_t)_{t\geq 0}$ given in \eqref{eq:aholaproc}. The results with explicit constants are provided below. We note that for any $p\in [2, \infty)\cap {\N}$ and $t\geq 0$, we have that $\E[|\widetilde{\theta}^{\lambda}_t|^{2p}] = \E[| \overline{\theta}^{\lambda}_t|^{2p}]$.
\begin{lemma}\label{lem:2ndpthmmt} Let Assumptions \ref{asm:AI}, \ref{asm:ALL}, and \ref{asm:AC} hold.  Then, we obtain the following estimates:
\begin{enumerate}[leftmargin=*]
\item For any $0<\lambda\leq \lambda_{\max}$, $n \in \N_0$, and $t \in (n, n+1]$,
\[
\E\left[ |\widetilde{\theta}^{\lambda}_t|^2  \right]  \leq \left(1 -\lambda(t-n)\ca_\cD\kappa \right)\left(1 -\lambda \ca_\cD\kappa \right)^n\E\left[ |\theta_0|^2\right]+  \cc_0\left(1+1/(\ca_\cD\kappa)\right),
\]
where the constants $\cc_0, \kappa$ are given explicitly in \eqref{eq:2ndmmtexpconst}. In particular, the above inequality implies $\sup_{t\geq 0}\E\left[|\widetilde{\theta}^{\lambda}_t|^2\right]  \leq  \E\left[|\theta_0|^2\right] +\cc_0(1+1/(\ca_\cD\kappa))<\infty$.
\label{lem:2ndpthmmti}
\item  For any $p\in [2, \infty)\cap {\N}$, $0<\lambda\leq \lambda_{\max}$, $n \in \N_0$, and $t \in (n, n+1]$,
\[
\E\left[ |\widetilde{\theta}^{\lambda}_t|^{2p} \right]  \leq \left(1 -\lambda(t-n)\ca_\cD\kappa \right)\left(1 -\lambda \ca_\cD\kappa \right)^n\E\left[|\theta_0|^{2p}\right]  + \cc_p\left(1+1/(\ca_\cD\kappa)\right),
\]
where $\kappa$ is given explicitly in \eqref{eq:2ndmmtexpconst} and $\cc_p$ is given in \eqref{eq:2pthmmtexpconst}. In particular, the above estimate implies $\sup_{t\geq 0}\E\left[|\widetilde{\theta}^{\lambda}_t|^{2p} \right]  \leq  \E\left[|\theta_0|^{2p}\right]+\cc_p\left(1+1/(\ca_\cD\kappa)\right)<\infty$. \label{lem:2ndpthmmtii}
\end{enumerate}
\end{lemma}
\begin{proof} See Appendix \ref{lem:2ndpthmmtproof}.
\end{proof}

We provide below a drift condition for $V_p$ (defined in the beginning of Section \ref{sec:me}), which is key to obtain the moment estimates of $(\bar{\zeta}^{\lambda, n}_t)_{t \geq nT}$ defined in Definition \ref{def:auxzeta}.
\begin{lemma}
\label{lem:driftcon} Let Assumptions \ref{asm:AI}, \ref{asm:ALL}, and \ref{asm:AC} hold. Then, for any $p\in [2, \infty)\cap {\N}$, $\theta \in \R^d$, we obtain
\[
\Delta V_p(\theta)/\beta - \langle \nabla V_p(\theta), h(\theta) \rangle \leq -\cc_{V,1}(p) V_p(\theta) +\cc_{V,2}(p),
\]
where $\cc_{V,1}(p) := \ca_\cD p/4$, $\cc_{V,2}(p) := (3/4)\ca_\cD p\mathrm{v}_p(\cM_V(p))$ with 
\[
\cM_V(p) := (1/3+4\overline{\cb}_\cD/(3\ca_\cD)+4d/(3\ca_\cD\beta)+4(p-2)/(3\ca_\cD\beta))^{1/2}.
\]
\end{lemma}
\begin{proof} See \cite[Lemma 3.5]{nonconvex}.
\end{proof}

By applying Lemma \ref{lem:2ndpthmmt} and \ref{lem:driftcon}, we obtain the second and the fourth moment estimates for $(\bar{\zeta}^{\lambda, n}_t)_{t \geq nT}$.
\begin{lemma}\label{lem:zetaprocme} Let Assumptions \ref{asm:AI}, \ref{asm:ALL}, and \ref{asm:AC} hold. Then, for any $p \in \N$, $0<\lambda\leq \lambda_{\max}$, $n \in \N_0$, and $t\geq nT$, we obtain
\[
\E[V_{2p}(\bar{\zeta}^{\lambda, n}_t)] \leq 2^{p-1}e^{-\lambda\ca_\cD \min\{ \kappa, 1 /2\} t}\E[|\theta_0|^{2p}] + 2^{p-1}\left( \cc_p\left(1+1/(\ca_\cD\kappa)\right)+1\right)+3\mathrm{v}_{2p}(\cM_V(2p)),
\]
where $\kappa$ and $\cc_p$ are given in \eqref{eq:2ndmmtexpconst} and \eqref{eq:2pthmmtexpconst} (see also Lemma \ref{lem:2ndpthmmt}) and $\cM_V(2p)$ is given in Lemma \ref{lem:driftcon}. 
\end{lemma}
\begin{proof} See \cite[Corollary 4.6]{mtula}.
\end{proof}

\subsection{Proof of main results} In this section, we present key results used to obtain Theorem \ref{thm:mainw1}. We consider establishing a non-asymptotic error bound in Wasserstein-1 distance between the law of $\overline{\theta}^{\lambda}_t$ given in \eqref{eq:aholahoproc} and $\pi_\beta$, i.e., $W_1(\mathcal{L}(\overline{\theta}^{\lambda}_t),\pi_\beta)$, for any $n \in \N_0$ and $t \in (nT, (n+1)T]$. This and the fact that $ W_1(\mathcal{L}(\theta_n^{\cHOLA}),\pi_{\beta})=W_1(\mathcal{L}(\overline{\theta}^{\lambda}_t),\pi_\beta) $ holds at any grid point yields the desired result. To this end, we split $W_1(\mathcal{L}(\overline{\theta}^{\lambda}_t),\pi_\beta)$ by using the law of $\overline{\zeta}^{\lambda, n}_t $ defined in Definition \ref{def:auxzeta} and $Z^{\lambda}_t$ given in \eqref{eq:tcsde} as follows:
\begin{equation}\label{eq:mtsplit}
W_1(\mathcal{L}(\overline{\theta}^{\lambda}_t),\pi_\beta) \leq W_1(\mathcal{L}(\overline{\theta}^{\lambda}_t),\mathcal{L}(\overline{\zeta}^{\lambda, n}_t))+W_1(\mathcal{L}(\overline{\zeta}^{\lambda, n}_t),\mathcal{L}(Z^{\lambda}_t))+W_1(\mathcal{L}(Z^{\lambda}_t),\pi_\beta).
\end{equation}

In the following lemma, we provide a non-asymptotic estimate for $W_2(\mathcal{L}(\overline{\theta}^{\lambda}_t),\mathcal{L}(\overline{\zeta}^{\lambda, n}_t))$, which can be used to upper bound the first term on the RHS of \eqref{eq:mtsplit}.
\begin{lemma}\label{lem:w1converp1} Let Assumptions \ref{asm:AI}, \ref{asm:ALL}, and \ref{asm:AC} hold. Then, for any $0<\lambda\leq \lambda_{\max}$, $n \in \N_0$, and $t \in (nT, (n+1)T]$, we obtain
\[
W_2(\mathcal{L}(\overline{\theta}^{\lambda}_t),\mathcal{L}(\overline{\zeta}^{\lambda, n}_t)) 
\leq \lambda^{1+q/2}\left(e^{- \ca_\cD \min\{ \kappa, 1 /2\} n/2}\cC_0\E[|\theta_0|^{16(\rho+1)}] +\cC_1 \right)^{1/2},
\]
where $\kappa, \cC_0, \cC_1$ are given explicitly in \eqref{eq:2ndmmtexpconst} and \eqref{eq:w1converp1const}.
\end{lemma}
\begin{proof} See Appendix \ref{lem:w1converp1proof}.
\end{proof}

For the last two terms on the RHS of \eqref{eq:mtsplit}, we observe that they can be viewed as Wasserstein-1 distances between distributions of Langevin processes starting from different initial points. Therefore, to obtain their upper bounds, we introduce a semi-metric which allows us to establish a contraction result for the Langevin SDE \eqref{eq:sde} under our assumptions. 

We consider the following semi-metric: for any $p\in [2, \infty)\cap {\N}$, $ \mu,\nu \in \mathcal{P}_{V_p}(\R^d)$, let
\begin{equation}\label{eq:semimetricw1p}
w_{1,p}(\mu,\nu):=\inf_{\zeta\in\mathcal{C}(\mu,\nu)}\int_{\mathbb{R}^d}\int_{\mathbb{R}^d} [1\wedge |\theta-\theta'|](1+V_p(\theta)+V_p(\theta'))\zeta(\rmd\theta, \rmd\theta').
\end{equation}
Then, we provide a result which states the contraction property of the Langevin SDE \eqref{eq:sde} in $w_{1,2}$.
\begin{proposition}\label{prop:contractionw12} Let Assumptions \ref{asm:AI}, \ref{asm:ALL}, and \ref{asm:AC} hold.  Moreover, let $\theta_0' \in \mathscr{L}^2$, and let $(Z_t')_{t \geq 0}$ be the solution of SDE \eqref{eq:sde} whose starting point $Z'_0 := \theta'_0$ is assumed to be independent of $\mathcal{F}_{\infty} := \sigma(\bigcup_{t \geq 0} \mathcal{F}_t)$. Then, we obtain
\begin{equation}\label{eq:w12contraction}
w_{1,2}(\mathcal{L}(Z_t),\mathcal{L}(Z'_t)) \leq \hat{\cc} e^{-\dot{\cc} t} w_{1,2}(\mathcal{L}(\theta_0),\mathcal{L}(\theta_0')),
\end{equation}
where the explicit expressions for $\dot{\cc}, \hat{\cc}$ are given below.

The contraction constant $\dot{\cc}$ is given by:
\begin{equation*}
\dot{\cc}:=\min\{\bar{\phi}, \cc_{V,1}(2), 4\cc_{V,2}(2) \epsilon \cc_{V,1}(2)\}/2,
\end{equation*}
where $\cc_{V,1}(2) :=  \ca_\cD/2$, $\cc_{V,2}(2) := 3 \ca_\cD \mathrm{v}_2(\cM_V(2))/2$ with $\cM_V(2):= (1/3+4\overline{\cb}_\cD/(3\ca_\cD)+4d/(3\ca_\cD\beta) )^{1/2}$, the constant $\bar{\phi} $ is given by
\begin{equation*}
\bar{\phi} := \left(\sqrt{8\pi/(\beta \cL_{\cOS})} \dot{\cc}_0  \exp \left( \left(\dot{\cc}_0 \sqrt{\beta \cL_{\cOS}/8} + \sqrt{8/(\beta \cL_{\cOS})} \right)^2 \right) \right)^{-1},
\end{equation*}
and $\epsilon >0$ is chosen such that
\begin{equation*}
\epsilon  \leq 1 \wedge    \left(4 \cc_{V,2}(2) \sqrt{2 \beta\pi/  \cL_{\cOS} }\int_0^{\dot{\cc}_1}\exp  \left( \left(s \sqrt{\beta \cL_{\cOS}/8}+\sqrt{8/(\beta \cL_{\cOS})}\right)^2 \right) \,\rmd s \right)^{-1}
\end{equation*}
with $\dot{\cc}_0 := 2(4\cc_{V,2}(2)(1+\cc_{V,1}(2))/\cc_{V,1}(2)-1)^{1/2}$ and $\dot{\cc}_1:=2(2 \cc_{V,2}(2)/\cc_{V,1}(2)-1)^{1/2}$.

Moreover, the constant $\hat{\cc}$ is given by:
\begin{equation*}
\hat{\cc}: =2(1+ \dot{\cc}_0)\exp(\beta \cL_{\cOS} \dot{\cc}_0^2/8+2\dot{\cc}_0)/\epsilon.
\end{equation*}
\end{proposition}
\begin{proof} We note that \cite[Assumption 2.1]{eberle2019quantitative} holds with $\kappa = \cL_{\cOS}$ due to Remark \ref{rmk:oslc}, \cite[Assumption 2.2]{eberle2019quantitative} holds with $V=V_2$ due to Remark \ref{lem:driftcon}, and \cite[Assumptions 2.4 and 2.5]{eberle2019quantitative} hold due to \eqref{eq:contrassumption}. Therefore, we can obtain \eqref{eq:w12contraction} following the same arguments as in the proof of \cite[Proposition~3.14]{nonconvex} based on \cite[Theorem 2.2, Corollary 2.3]{eberle2019quantitative}. In addition, $\dot{\cc}, \hat{\cc}$ can be obtained following the arguments in the proof of \cite[Proposition 4.6]{lim2021nonasymptotic}.
\end{proof}

By using the above result and $W_1 \leq w_{1,2}$ (see \cite[Lemma A.3]{lim2021nonasymptotic}), we can establish a non-asymptotic error bound for the second term on the RHS of \eqref{eq:mtsplit}. The explicit statement is given below.
\begin{lemma}\label{lem:w1converp2} Let Assumptions \ref{asm:AI}, \ref{asm:ALL}, and \ref{asm:AC} hold. Then, for any $0<\lambda\leq \lambda_{\max}$, $n \in \N_0$, and $t \in (nT, (n+1)T]$, we obtain
\[
W_1(\mathcal{L}(\bar{\zeta}_t^{\lambda,n}),\mathcal{L}(Z_t^\lambda)) \leq \lambda^{1+q/2}\left(e^{- \min\{\dot{\cc}, \ca_\cD \kappa, \ca_\cD  /2\} n/4}\cC_2\E[|\theta_0|^{16(\rho+1)}] +\cC_3 \right) ,
\]
where
\begin{align}\label{eq:w1converp2consts}
\begin{split}
\cC_2
& := \hat{\cc}\left(1+\frac{4}{\min\{\dot{\cc}, \ca_\cD \kappa, \ca_\cD  /2\} }\right)e^{\min\{\dot{\cc}, \ca_\cD \kappa, \ca_\cD  /2\} /4}\left(\cC_0 +2^{8\rho+11}\right),\\
\cC_3
& :=  2(\hat{\cc}/\dot{\cc})e^{\dot{\cc}/2}\left(\cC_1+15+2^{8\rho+11}(\cc_{8(\rho+1)}(1+1/( \ca_\cD \kappa))+1)\right. \left. +18\mathrm{v}_{16(\rho+1)}(\cM_V(16(\rho+1)))\right)
\end{split}
\end{align}
with $\dot{\cc}, \hat{\cc}$ given in Proposition \ref{prop:contractionw12}, $\cC_0, \cC_1$ given in \eqref{eq:w1converp1const} (see also Lemma \ref{lem:w1converp1}), $\kappa$, $\cc_{8(\rho+1)}$ given in Lemma \ref{lem:2ndpthmmt}, and $\cM_V(16(\rho+1))$ given in Lemma \ref{lem:zetaprocme}.
\end{lemma}
\begin{proof} See \cite[Lemma 4.7]{lim2021nonasymptotic}.
\end{proof}
To obtain an upper bound for the last term on the RHS of \eqref{eq:mtsplit}, we observe that $\pi_{\beta}$ is the invariant measure of the Langevin SDE \eqref{eq:tcsde}. Thus, by applying Proposition \ref{prop:contractionw12}, we have that
\begin{equation}\label{eq:w1converp3}
W_1(\mathcal{L}(Z^{\lambda}_t),\pi_\beta) \leq  \hat{\cc} e^{-\dot{\cc} \lambda t} w_{1,2}(\mathcal{L}(\theta_0), \pi_\beta)
\leq \hat{\cc} e^{-\dot{\cc} \lambda t} \left[1+\E[V_2(\theta_0)]+ \int_{\R^d} V_2(\theta) \pi_{\beta}(\rmd \theta)\right].
\end{equation}
By using Lemma \ref{lem:w1converp1}, \ref{lem:w1converp2} and \eqref{eq:w1converp3}, we can obtain an upper bound for each $W_2(\mathcal{L}(\theta_n^{\cHOLA}),\pi_{\beta})$, $n \in \N_0$, as stated in Theorem \ref{thm:mainw1}.

\begin{proof}[\textbf{Proof of Theorem \ref{thm:mainw1}}] Substituting the results in Lemma \ref{lem:w1converp1}, \ref{lem:w1converp2} and \eqref{eq:w1converp3} into \eqref{eq:mtsplit}, for any $0<\lambda\leq \lambda_{\max}$, $n \in \N_0$, and $t \in (nT, (n+1)T]$, we have that
\begin{align*}
W_1(\mathcal{L}(\overline{\theta}^{\lambda}_t),\pi_\beta) 
&\leq  \lambda^{1+q/2}\left(e^{- \ca_\cD \min\{ \kappa, 1 /2\} n/2}\cC_0\E[|\theta_0|^{16(\rho+1)}] +\cC_1 \right)^{1/2}\\
&\quad + \lambda^{1+q/2}\left(e^{- \min\{\dot{\cc}, \ca_\cD \kappa, \ca_\cD  /2\} n/4}\cC_2\E[|\theta_0|^{16(\rho+1)}] +\cC_3 \right) \\
&\quad +\hat{\cc} e^{-\dot{\cc} \lambda t} \left[1+\E[V_2(\theta_0)]+ \int_{\R^d} V_2(\theta) \pi_{\beta}(\rmd \theta)\right]\\
&\leq C_1 e^{-C_0 ( n+1)}(\E[|\theta_0|^{16(\rho+1)}]+1) +C_2\lambda^{1+q/2},
\end{align*}
where 
\begin{align}\label{eq:mainw1const}
\begin{split}
C_0&:=\min\{\dot{\cc}, \ca_\cD \kappa, \ca_\cD  /2\}  /4, \\
C_1 &:=e^{\min\{\dot{\cc}, \ca_\cD \kappa, \ca_\cD  /2\}  /4}\left[\cC_0^{1/2}+\cC_2+\hat{\cc}\left(3+\int_{\R^d} V_2(\theta) \pi_{\beta}(\rmd \theta)\right)\right], \\
C_2&:=\cC_1^{1/2}+\cC_3\\
\end{split}
\end{align}
with  $\dot{\cc}, \hat{\cc}$ given in Proposition \ref{prop:contractionw12}, $\kappa$ given in Lemma \ref{lem:2ndpthmmt}, $\cC_0, \cC_1$ given in \eqref{eq:w1converp1const} (see also Lemma \ref{lem:w1converp1}), $\cC_2, \cC_3$ given in \eqref{eq:w1converp2consts} (see also Lemma \ref{lem:w1converp2}). The above result implies that, for each $n \in \N_0$,
\[
W_1(\mathcal{L}(\overline{\theta}^{\lambda}_{nT}),\pi_\beta) \leq  C_1 e^{-C_0 n}(\E[|\theta_0|^{16(\rho+1)}]+1) +C_2\lambda^{1+q/2},
\]
which further yields, by setting $nT$ to $n$ on the LHS and $n$ to $n/T$ on the RHS, that
\[
W_1(\mathcal{L}(\overline{\theta}^{\lambda}_n),\pi_\beta) = W_1(\mathcal{L}(\theta^{\lambda}_n),\pi_\beta)  = W_1(\mathcal{L}(\theta_n^{\cHOLA}),\pi_{\beta}) \leq  C_1 e^{-C_0 \lambda n}(\E[|\theta_0|^{16(\rho+1)}]+1) +C_2\lambda^{1+q/2},
\]
where the inequality holds due to $n\lambda \leq n/T$. This completes the proof.
\end{proof}

By using similar arguments as in the proof of Theorem \ref{thm:mainw1}, we can obtain the upper bound for $W_2(\mathcal{L}(\theta_n^{\cHOLA}),\pi_{\beta}) $, $n \in \N_0$, as stated in Corollary \ref{crl:mainw2}.
\begin{proof}[\textbf{Proof of Corollary \ref{crl:mainw2}}] To establish a non-asymptotic error bound for $W_2(\mathcal{L}(\theta_n^{\cHOLA}),\pi_{\beta}) $, we consider the following splitting: for any $0<\lambda\leq \lambda_{\max}$, $n \in \N_0$, and $t \in (nT, (n+1)T]$,
\begin{equation}\label{eq:crlsplit}
W_2(\mathcal{L}(\overline{\theta}^{\lambda}_t),\pi_\beta) \leq W_2(\mathcal{L}(\overline{\theta}^{\lambda}_t),\mathcal{L}(\overline{\zeta}^{\lambda, n}_t))+W_2(\mathcal{L}(\overline{\zeta}^{\lambda, n}_t),\mathcal{L}(Z^{\lambda}_t))+W_2(\mathcal{L}(Z^{\lambda}_t),\pi_\beta).
\end{equation}
An upper bound for the first term on the RHS of \eqref{eq:crlsplit} is provided in Lemma \ref{lem:w1converp1}. To establish an estimate for the second term on the RHS of \eqref{eq:crlsplit}, we use $W_2 \leq \sqrt{2w_{1,2}}$ (see \cite[Lemma A.3]{lim2021nonasymptotic} for the proof) and follow the same arguments as that in the proof of \cite[Lemma 4.7]{lim2021nonasymptotic}. Consequently, for any $0<\lambda\leq \lambda_{\max}$, $n \in \N_0$, and $t \in (nT, (n+1)T]$, we obtain that,
\begin{equation}\label{eq:w2converp2} 
W_2(\mathcal{L}(\overline{\zeta}_t^{\lambda,n}),\mathcal{L}(Z_t^\lambda)) \leq \lambda^{1/2+q/4}\left(e^{- \min\{\dot{\cc}, \ca_\cD \kappa, \ca_\cD  /2\} n/8}\cC_4\E^{1/2}[|\theta_0|^{16(\rho+1)}] +\cC_5 \right) ,
\end{equation}
where
\begin{align}\label{eq:w2converp2consts}
\begin{split}
\cC_4
& := \sqrt{\hat{\cc}}\left(1+\frac{8}{\min\{\dot{\cc}, \ca_\cD \kappa, \ca_\cD  /2\} }\right)e^{\min\{\dot{\cc}, \ca_\cD \kappa, \ca_\cD  /2\} /8}\left(\sqrt{2}\cC_0^{1/2} +2^{4\rho+5}\right),\\
\cC_5
& :=  4(\sqrt{\hat{\cc}}/\dot{\cc})e^{\dot{\cc}/4}\left(\sqrt{2}\cC_1^{1/2}+3\sqrt{2}+2^{4\rho+5}(\cc_{8(\rho+1)}(1+1/( \ca_\cD \kappa))+1)^{1/2}\right.\\
&\qquad \left. +\sqrt{6}\mathrm{v}^{1/2}_{16(\rho+1)}(\cM_V(16(\rho+1)))\right)
\end{split}
\end{align}
with $\dot{\cc}, \hat{\cc}$ given in Proposition \ref{prop:contractionw12}, $\cC_0, \cC_1$ given in \eqref{eq:w1converp1const} (see also Lemma \ref{lem:w1converp1}), $\kappa$, $\cc_{8(\rho+1)}$ given in Lemma \ref{lem:2ndpthmmt}, and $\cM_V(16(\rho+1))$ given in Lemma \ref{lem:zetaprocme}. An upper bound for the last term on the RHS of  \eqref{eq:crlsplit} can be obtained by using $W_2 \leq \sqrt{2w_{1,2}}$ and Proposition \ref{prop:contractionw12}:
\begin{align}\label{eq:w2converp3} 
W_2(\mathcal{L}(Z^{\lambda}_t),\pi_\beta) 
&\leq  \sqrt{2\hat{\cc}} e^{-\dot{\cc} \lambda t/2} w_{1,2}^{1/2}(\mathcal{L}(\theta_0), \pi_\beta)\nonumber\\
&\leq  \sqrt{2\hat{\cc}} e^{-\dot{\cc} \lambda t/2}\left[1+\E[V_2(\theta_0)]+ \int_{\R^d} V_2(\theta) \pi_{\beta}(\rmd \theta)\right]^{1/2}. 
\end{align}
Applying the results in Lemma \ref{lem:w1converp1}, \eqref{eq:w2converp2}, \eqref{eq:w2converp3} to \eqref{eq:crlsplit} yields, for any $0<\lambda\leq \lambda_{\max}$, $n \in \N_0$, and $t \in (nT, (n+1)T]$, that
\begin{align*}
W_2(\mathcal{L}(\overline{\theta}^{\lambda}_t),\pi_\beta) 
&\leq \lambda^{1+q/2}\left(e^{- \ca_\cD \min\{ \kappa, 1 /2\} n/2}\cC_0\E[|\theta_0|^{16(\rho+1)}] +\cC_1 \right)^{1/2}\\
&\quad +\lambda^{1/2+q/4}\left(e^{- \min\{\dot{\cc}, \ca_\cD \kappa, \ca_\cD  /2\} n/8}\cC_4\E^{1/2}[|\theta_0|^{16(\rho+1)}] +\cC_5 \right)\\
&\quad +\sqrt{2\hat{\cc}} e^{-\dot{\cc} \lambda t/2}\left[1+\E[V_2(\theta_0)]+ \int_{\R^d} V_2(\theta) \pi_{\beta}(\rmd \theta)\right]^{1/2}\\
&\leq  C_4 e^{-C_3( n+1)}(\E[|\theta_0|^{16(\rho+1)}]+1)^{1/2} +C_5\lambda^{1/2+q/4},
\end{align*}
where 
\begin{align}\label{eq:mainw2const}
\begin{split}
C_3&:=\min\{\dot{\cc}, \ca_\cD \kappa, \ca_\cD  /2\}  /8, \\
C_4 &:=e^{\min\{\dot{\cc}, \ca_\cD \kappa, \ca_\cD  /2\}  /8}\left[\cC_0^{1/2}+\cC_4+\sqrt{2\hat{\cc}}\left(3+\int_{\R^d} V_2(\theta) \pi_{\beta}(\rmd \theta)\right)^{1/2}\right], \\
C_5&:=\cC_1^{1/2}+\cC_5\\
\end{split}
\end{align}
with  $\dot{\cc}, \hat{\cc}$ given in Proposition \ref{prop:contractionw12}, $\kappa$ given in Lemma \ref{lem:2ndpthmmt}, $\cC_0, \cC_1$ given in \eqref{eq:w1converp1const} (see also Lemma \ref{lem:w1converp1}), $\cC_4, \cC_5$ given in \eqref{eq:w2converp2consts}. This further implies that, for each $n \in \N_0$,
\begin{align*}
W_2(\mathcal{L}(\overline{\theta}^{\lambda}_n),\pi_\beta) = W_2(\mathcal{L}(\theta^{\lambda}_n),\pi_\beta)  
&= W_2(\mathcal{L}(\theta_n^{\cHOLA}),\pi_{\beta}) \\
&\leq  C_4 e^{-C_3\lambda n}(\E[|\theta_0|^{16(\rho+1)}]+1)^{1/2} +C_5\lambda^{1/2+q/4},
\end{align*}
which completes the proof.
\end{proof}

\newpage
\appendix
\section{Proof of auxiliary results}\label{appen:aux}
\subsection{Proof of auxiliary results in Section \ref{sec:main_superlinear}} 

\begin{proof}[\textbf{Proof of statements in Remark \ref{rmk:growthc}}]\label{rmk:growthcproof} We provide detailed proofs for inequalities \eqref{eq:growthc1}-\eqref{eq:growthc3}, the other inequalities can be obtained by using similar arguments. By Assumption \ref{asm:ALL}, for any $\theta, \overline{\theta} \in \R^d$, $i = 1, \dots, d$, we have that
\begin{align*}
|\nabla^2 h^{(i)}(\theta)|&\leq |\nabla^2 h^{(i)}(\theta)-\nabla^2 h^{(i)}(0)|+|\nabla^2 h^{(i)}(0)|\\
&\leq L(1+|\theta| )^{\rho-2}|\theta |^q+|\nabla^2 h^{(i)}(0)|\\
&\leq \cK_0(1+|\theta|)^{\rho+q-2}, 
\end{align*}
where $\cK_0 := 2^{1-q}\max\{L, |\nabla^2 h^{(1)}(0)|,\dots,|\nabla^2 h^{(d)}(0)|\}$. Moreover, fix $\theta, \overline{\theta} \in \R^d$, denote by $g(t): =\nabla  h^{(i)}(t\theta+(1-t)\overline{\theta}) $, $t \in [0,1]$. Then, by using the above inequality, we obtain that
\begin{align*}
|\nabla h^{(i)}(\theta) -  \nabla h^{(i)}(\overline{\theta}) |&=\left|\int_0^1\nabla^2 h^{(i)}(t\theta+(1-t)\overline{\theta}) (\theta-\overline{\theta})\rmd t  \right|\\
&\leq \int_0^1 \cK_0(1+|t\theta+(1-t)\overline{\theta}|)^{\rho+q-2}\rmd t|\theta-\overline{\theta}|\\
&\leq  \cK_0(1+|\theta|+|\overline{\theta}|)^{\rho+q-2}|\theta-\overline{\theta}|.
\end{align*}
In addition, for any $\theta, \overline{\theta} \in \R^d$, by Assumption \ref{asm:ALL}, we have that
\begin{align*}
|\Upsilon(\theta)-\Upsilon(\overline{\theta})| &= \left(\sum_{i = 1}^d\left(\sum_{j=1}^d (\partial^2_{\theta^{(j)}} h^{(i)}(\theta)- \partial^2_{\theta^{(j)}} h^{(i)}(\overline{\theta}))\right)^2\right)^{1/2}\\
&\leq  \left(d\sum_{i = 1}^d\sum_{j=1}^d\left( \partial^2_{\theta^{(j)}} h^{(i)}(\theta)- \partial^2_{\theta^{(j)}} h^{(i)}(\overline{\theta}) \right)^2\right)^{1/2}\\
&\leq  \left(d\sum_{i = 1}^d  |\nabla^2  h^{(i)}(\theta)-\nabla^2 h^{(i)}(\overline{\theta})|_{\cF}^2\right)^{1/2}\\
&\leq \left(d^2\sum_{i = 1}^d  |\nabla^2  h^{(i)}(\theta)-\nabla^2 h^{(i)}(\overline{\theta})|^2\right)^{1/2}\\
&\leq d^{3/2} L(1+|\theta|+|\overline{\theta}|)^{\rho-2}|\theta-\overline{\theta}|^q,
\end{align*}
which completes the proof.
\end{proof}

\begin{proof}[\textbf{Proof of statements in Remark \ref{rmk:dissipativityc}}]\label{rmk:dissipativitycproof} By Assumption \ref{asm:AC}, for any $\theta \in \R^d$, we have that
\[
\langle \theta , h(\theta)-h(0) \rangle \geq a|\theta|^2|\theta|^r - b|\theta|^2|\theta|^{\overline{r}},
\]
which implies that
\begin{align}
\langle \theta , h(\theta)  \rangle 
&\geq a|\theta|^{r+2}  - b |\theta|^{\overline{r}+2}+\langle \theta , h(0) \rangle \nonumber\\
&\geq a|\theta|^{r+2}  - b |\theta|^{\overline{r}+2}-|\langle \theta , h(0) \rangle| \nonumber\\
&\geq  (a/2)|\theta|^{r+2} +(  (a/4)|\theta|^{r+2} - b |\theta|^{\overline{r}+2})+( (a/4)|\theta|^{r+2} -   a|\theta|^2/2)-|h(0)|^2/(2a), \label{eq:dissipativityc1}
\end{align}
where the last inequality holds due to the fact that $yz \leq \varepsilon y^2/2+z^2/(2\varepsilon)$ for any $y, z \geq 0$ and $\varepsilon>0$. Denote by $\cR_\cD:=\max\{(4b/a)^{1/(r-\overline{r})}, 2^{1/r}\}>1$. We observe that for any $|\theta|>\cR_\cD$,
\[
  (a/4)|\theta|^{r+2} - b |\theta|^{\overline{r}+2} >0, \quad (a/4)|\theta|^{r+2} -   a|\theta|^2/2>0,
\]
and thus \eqref{eq:dissipativityc1} becomes
\begin{equation}\label{eq:dissipativityc2}
\langle \theta , h(\theta)  \rangle >  (a/2)|\theta|^{r+2} -|h(0)|^2/(2a).
\end{equation}
Moreover, for any $|\theta|\leq \cR_\cD$, \eqref{eq:dissipativityc1} becomes
\begin{equation}\label{eq:dissipativityc3}
\langle \theta , h(\theta)  \rangle \geq  (a/2)|\theta|^{r+2} -(b+a/2)\cR_\cD^{\overline{r}+2}-|h(0)|^2/(2a).
\end{equation}
Combining \eqref{eq:dissipativityc2} and \eqref{eq:dissipativityc3} yields
\begin{equation}\label{eq:dissipativityc4}
\langle \theta, h(\theta) \rangle \geq \ca_\cD|\theta|^{r+2}-\cb_\cD,
\end{equation}
where $\ca_\cD:=a/2$ and $\cb_\cD:=(a/2+b)\cR_\cD^{\overline{r}+2}+|h(0)|^2/(2a)$ with $\cR_\cD:=\max\{(4b/a)^{1/(r-\overline{r})}, 2^{1/r}\}$.
Moreover, by \eqref{eq:dissipativityc4}, for any $\theta \in \R^d$, we observe that, 
\[
\langle \theta, h(\theta) \rangle \geq \ca_\cD|\theta|^2-\overline{\cb}_\cD,
\]
where $\overline{\cb}_\cD:=\ca_\cD+\cb_\cD$. Indeed, for $|\theta|>1$, it holds that
\[
\langle \theta, h(\theta) \rangle \geq \ca_\cD|\theta|^{r+2}-\cb_\cD> \ca_\cD|\theta|^2- \cb_\cD> \ca_\cD|\theta|^2-\overline{\cb}_\cD,
\]
while for $|\theta| \leq 1$, we have that
\[
\langle \theta, h(\theta) \rangle \geq \ca_\cD|\theta|^{r+2}-\cb_\cD\geq  \ca_\cD-(\ca_\cD+\cb_\cD)\geq \ca_\cD|\theta|^2-\overline{\cb}_\cD,
\]
which completes the proof.
\end{proof}

\begin{proof}[\textbf{Proof of statements in Remark \ref{rmk:oslc}}] \label{rmk:oslcproof} For any $\theta, \overline{\theta}\in \R^d$, we observe that the result holds trivially when $\theta = \overline{\theta} $, and thus we consider only the case where $\theta \neq \overline{\theta}$. Denote by $\cR_{\cOS}: = (b/a)^{1/(r-\overline{r})}$. 
By Assumption \ref{asm:AC}, for any $|\theta|, |\overline{\theta}| >\cR_{\cOS}$, we have that
\begin{equation}\label{eq:oslc1}
\langle \theta - \overline{\theta}, h(\theta)-h(\overline{\theta}) \rangle >0.
\end{equation}
Furthermore, we note that by Remark \ref{rmk:growthc},
\[
|h(\theta)- h(\overline{\theta})|\leq \sqrt{d}\cK_1(1+|\theta|+|\overline{\theta}|)^{\rho+q-1}|\theta-\overline{\theta}|.
\]
Thus, for any $|\theta|, |\overline{\theta}| \leq \cR_{\cOS}$, we have that
\begin{equation}\label{eq:oslc2}
-\langle \theta - \overline{\theta}, h(\theta)-h(\overline{\theta}) \rangle \leq |\theta - \overline{\theta}|| h(\theta)-h(\overline{\theta})| \leq \cL_{\cOS}|\theta - \overline{\theta}|^2,
\end{equation}
where $\cL_{\cOS}:=\sqrt{d}\cK_1(1+2\cR_{\cOS})^{\rho+q-1}>0$. In addition, for $|\theta| \leq  \cR_{\cOS}$, $|\overline{\theta}| >\cR_{\cOS}$, we consider the following two cases:
\begin{enumerate}
\item For $|\theta| =  \cR_{\cOS}$, $|\overline{\theta}| >\cR_{\cOS}$, we obtain \eqref{eq:oslc1}.
\item For any $|\theta| <  \cR_{\cOS}$, $|\overline{\theta}| >\cR_{\cOS}$, there is a unique $\hat{\theta} \in \R^d$ with $|\hat{\theta}|= | \cR_{\cOS}|$ such that
\[\theta - \hat{\theta} = \cc_{\theta, \overline{\theta}}(\theta - \overline{\theta}), \quad \hat{\theta}-\overline{\theta} = (1-\cc_{\theta, \overline{\theta}})(\theta - \overline{\theta}), 
\]
where $\cc_{\theta, \overline{\theta}} \in (0,1)$. Then, we obtain that
\begin{align}
\langle \theta - \overline{\theta}, h(\theta)-h(\overline{\theta}) \rangle &=\langle \theta - \overline{\theta}, h(\theta)-h(\hat{\theta}) \rangle +\langle \theta - \hat{\theta}, h(\hat{\theta})-h(\overline{\theta}) \rangle +\langle \hat{\theta} - \overline{\theta}, h(\hat{\theta})-h(\overline{\theta})\rangle \nonumber\\ 
& =\langle \theta - \overline{\theta}, h(\theta)-h(\hat{\theta}) \rangle +(\cc_{\theta, \overline{\theta}}/(1-\cc_{\theta, \overline{\theta}}))\langle \hat{\theta}-\overline{\theta}, h(\hat{\theta})-h(\overline{\theta}) \rangle \nonumber\\
&> (1/\cc_{\theta, \overline{\theta}}) \langle \theta - \hat{\theta}, h(\theta)-h(\hat{\theta}) \rangle \label{eq:oslc4}\\
&\geq -(\cL_{\cOS}/\cc_{\theta, \overline{\theta}})|\theta - \hat{\theta}|^2 \label{eq:oslc5}\\
&= -\cL_{\cOS}\cc_{\theta, \overline{\theta}}|\theta - \hat{\theta}|^2 \nonumber\\
&\geq -\cL_{\cOS} |\theta - \hat{\theta}|^2,\label{eq:oslc6}
\end{align}
where \eqref{eq:oslc4} holds due to \eqref{eq:oslc1}, \eqref{eq:oslc5} holds due to \eqref{eq:oslc2}, and \eqref{eq:oslc6} holds due to $\cc_{\theta, \overline{\theta}} \in (0,1)$. 
\end{enumerate}
Thus, for $|\theta| \leq  \cR_{\cOS}$, $|\overline{\theta}| >\cR_{\cOS}$, combining the two cases yield
\begin{equation}\label{eq:oslc7}
\langle \theta - \overline{\theta}, h(\theta)-h(\overline{\theta}) \rangle\geq -\cL_{\cOS}|\theta - \overline{\theta}|^2.
\end{equation}
Finally, for $|\theta| >  \cR_{\cOS}$, $|\overline{\theta}| \leq \cR_{\cOS}$, we obtain \eqref{eq:oslc7} by applying the same arguments above. Combining \eqref{eq:oslc1}, \eqref{eq:oslc2}, and \eqref{eq:oslc7} yields the desired result.
\end{proof}

\subsection{Proof of auxiliary results in Section \ref{sec:app}}\label{sec:appapd}
\begin{proof}[\textbf{Proof of Proposition \ref{prop:example}}]\label{prop:exampleproof} We first show that the potentials given in \eqref{eq:gaudistr} and \eqref{eq:gaumixtdistr} satisfy Assumptions~\ref{asm:ALLlip} and \ref{asm:ADlip}.
\begin{enumerate}
\item For $U$ given in \eqref{eq:gaudistr}, we have that, for any $\theta \in \R^d$, $i= 1, \dots, d$,
\[
h(\theta) = \theta, \quad H(\theta) = I_d, \quad \nabla^2 h^{(i)}(\theta)= 0.
\]
Then, we observe that Assumption \ref{asm:ALLlip} holds with $\overline{L}_1=0$, $q  =1$, $\overline{L}_2=0$, and $\overline{L}_3=1$. Moreover, Assumption \ref{asm:ADlip} holds with $\overline{a}=1$, $\overline{b}=0$ since
\[
\langle \theta, h(\theta) \rangle =|\theta|^2.
\]
\item For $U$ given in \eqref{eq:gaumixtdistr}, we have that, for any $\theta \in \R^d$, $i= 1, \dots, d$,
\begin{align*}
h(\theta) &= \theta-\hat{a}+\frac{2\hat{a}}{1+\exp(2\langle\theta, \hat{a}\rangle)}, \\
H(\theta) & = I_d - \frac{4\hat{a}\hat{a}^{\cT}\exp(2\langle\theta, \hat{a}\rangle)}{(1+\exp(2\langle\theta, \hat{a}\rangle))^2}, \\
\nabla^2 h^{(i)}(\theta)&= \frac{8\hat{a}^{(i)}\hat{a}\hat{a}^{\cT}(\exp(4\langle\theta, \hat{a}\rangle)-\exp(2\langle\theta, \hat{a}\rangle))}{(1+\exp(2\langle\theta, \hat{a}\rangle))^3}.
\end{align*}
To show that Assumption \ref{asm:ALLlip} holds, we consider the following calculations. For any $\theta, \overline{\theta} \in \R^d$, we have that
\begin{align}
&|h(\theta) - h(\overline{\theta} )|\nonumber \\
&= \left|\theta-\hat{a}+\frac{2\hat{a}}{1+\exp(2\langle\theta, \hat{a}\rangle)} - \left(\overline{\theta}-\hat{a}+\frac{2\hat{a}}{1+\exp(2\langle\overline{\theta}, \hat{a}\rangle)},\right)\right|\nonumber\\
\begin{split}\label{eq:explip}
&\leq |\theta-\overline{\theta}|+2|\hat{a}|\frac{|\exp(2\langle\overline{\theta}, \hat{a}\rangle)) - \exp(2\langle\theta, \hat{a}\rangle))|}{(1+\exp(2\langle\theta, \hat{a}\rangle))(1+\exp(2\langle\overline{\theta}, \hat{a}\rangle))}\\
&= |\theta-\overline{\theta}|+2|\hat{a}|\frac{\max\{\exp(2\langle\theta, \hat{a}\rangle),\exp(2\langle\overline{\theta}, \hat{a}\rangle)\}| 1- \exp(-2|\langle\theta, \hat{a}\rangle-\langle\overline{\theta}, \hat{a}\rangle|)|}{(1+\exp(2\langle\theta, \hat{a}\rangle))(1+\exp(2\langle\overline{\theta}, \hat{a}\rangle))}\\
&\leq  |\theta-\overline{\theta}|+4|\hat{a}||\langle\theta, \hat{a}\rangle-\langle\overline{\theta}, \hat{a}\rangle|\\
&\leq (1+4|\hat{a}|^2) |\theta-\overline{\theta}|,
\end{split}
\end{align}
where the second last inequality holds due to $1-e^{-x}\leq x$, for all $x\in \R$, and the last inequality holds due to Cauchy-Schwarz inequality. Similarly, we have that, for any $\theta, \overline{\theta} \in \R^d$,
\begin{align*}
&|H(\theta) - H(\overline{\theta} )|\nonumber \\
&= 4|\hat{a}|^2\frac{|\exp(2\langle\overline{\theta}, \hat{a}\rangle)(1+\exp(2\langle\theta, \hat{a}\rangle))^2-\exp(2\langle\theta, \hat{a}\rangle)(1+\exp(2\langle\overline{\theta}, \hat{a}\rangle))^2|}{(1+\exp(2\langle\theta, \hat{a}\rangle))^2(1+\exp(2\langle\overline{\theta}, \hat{a}\rangle))^2}\\
&\leq 4|\hat{a}|^2\frac{(1+\exp(2\langle\overline{\theta}+\theta, \hat{a}\rangle))|\exp(2\langle\overline{\theta}, \hat{a}\rangle) - \exp(2\langle\theta, \hat{a}\rangle)|}{(1+\exp(2\langle\theta, \hat{a}\rangle))^2(1+\exp(2\langle\overline{\theta}, \hat{a}\rangle))^2}\\
&\leq 8|\hat{a}|^3 |\theta-\overline{\theta}|,
\end{align*}
where the last inequality holds due to the calculations in \eqref{eq:explip}. Furthermore, we note that, for any $\theta, \overline{\theta} \in \R^d$,
\begin{align*}
&|\nabla^2 h^{(i)}(\theta) - \nabla^2 h^{(i)}(\overline{\theta} )|\nonumber \\
&\leq 8|\hat{a}|^3\left|\frac{ \exp(4\langle\theta, \hat{a}\rangle)-\exp(2\langle\theta, \hat{a}\rangle)}{(1+\exp(2\langle\theta, \hat{a}\rangle))^3}- \frac{\exp(4\langle\overline{\theta}, \hat{a}\rangle)-\exp(2\langle\overline{\theta}, \hat{a}\rangle)}{(1+\exp(2\langle\overline{\theta}, \hat{a}\rangle))^3}\right|\\
&\leq 8|\hat{a}|^3\frac{(1+\exp(2\langle\overline{\theta}+\theta, \hat{a}\rangle))|\exp(4\langle\overline{\theta}, \hat{a}\rangle) - \exp(4\langle\theta, \hat{a}\rangle)|}{(1+\exp(2\langle\theta, \hat{a}\rangle))^3(1+\exp(2\langle\overline{\theta}, \hat{a}\rangle))^3}\\
&\quad +8|\hat{a}|^3\frac{(1+6\exp(2\langle\overline{\theta}+\theta, \hat{a}\rangle)+\exp(4\langle\overline{\theta}+\theta, \hat{a}\rangle))|\exp(2\langle\overline{\theta}, \hat{a}\rangle) - \exp(2\langle\theta, \hat{a}\rangle)|}{(1+\exp(2\langle\theta, \hat{a}\rangle))^3(1+\exp(2\langle\overline{\theta}, \hat{a}\rangle))^3}\\
&\leq 56|\hat{a}|^4|\theta-\overline{\theta}|,
\end{align*}
where the last inequality holds due to the calculations in \eqref{eq:explip}. Consequently, Assumption \ref{asm:ALLlip} holds with $\overline{L}_1=56|\hat{a}|^4$, $q  =1$, $\overline{L}_2=8|\hat{a}|^3$, and $\overline{L}_3=1+4|\hat{a}|^2$. Now, we show Assumption \ref{asm:ADlip} holds with $\overline{a}=1/2$, $\overline{b}=1/2$. Indeed, we have, for any $\theta \in \R^d$, that
\[
\langle \theta, h(\theta) \rangle =|\theta|^2+\frac{1-\exp(2\langle\theta, \hat{a}\rangle)}{1+\exp(2\langle\theta, \hat{a}\rangle)}\langle \theta, \hat{a} \rangle \geq |\theta|^2-\frac{1}{2}(|\theta|^2+|\hat{a}|^2)=\frac{1}{2}(|\theta|^2-|\hat{a}|^2).
\]
\end{enumerate}
Finally, we show that the potential \eqref{eq:dwdistr} satisfies Assumptions \ref{asm:ALL} and \ref{asm:AC}.
\begin{enumerate}
\item[(iii)] For $U$ given in \eqref{eq:dwdistr}, by Remark \ref{rmk:dwasmall}, we have that Assumption \ref{asm:ALL} is satisfied with $\rho=2$, $L=6$, $q = 1$, $K_h = 2$, and $K_H = 3$. In addition, Assumption \ref{asm:AC} is satisfied with $a=1/2$, $b = 1$, $r=2$, and $\overline{r}=0$. Indeed, for any $\theta\in\R^d$, by noticing
\[
h(\theta)=\theta(|\theta|^2-1), 
\]
we obtain that
\begin{align*}
\langle \theta - \overline{\theta}, h(\theta)-h(\overline{\theta}) \rangle 
&=\langle \theta - \overline{\theta}, \theta|\theta|^2-\theta-(\overline{\theta}|\overline{\theta}|^2-\overline{\theta})\rangle\\
&=\frac{1}{2}\langle \theta - \overline{\theta}, \theta|\theta|^2-\theta|\overline{\theta}|^2+\theta|\overline{\theta}|^2 -\overline{\theta}|\overline{\theta}|^2 \rangle\\
&\quad +\frac{1}{2}\langle \theta - \overline{\theta}, \theta|\theta|^2-\overline{\theta}|\theta |^2+\overline{\theta}|\theta |^2 -\overline{\theta}|\overline{\theta}|^2 \rangle-|\theta - \overline{\theta}|^2\\
& = \frac{1}{2}(|\theta|^2+|\overline{\theta}|^2)|\theta - \overline{\theta}|^2+ \frac{1}{2}(|\theta|^2-|\overline{\theta}|^2)^2-|\theta - \overline{\theta}|^2\\
&\geq \frac{1}{2}(|\theta|^2+|\overline{\theta}|^2)|\theta - \overline{\theta}|^2 -|\theta - \overline{\theta}|^2.
\end{align*}
\end{enumerate}
This completes the proof.
\end{proof}

\begin{proof}[\textbf{Proof of Proposition \ref{prop:lrexample}}]\label{prop:lrexampleproof} For $U$ defined in \eqref{eq:lrU}, we note, for any $\theta\in\R^d$, $l = 1, \dots, d$, that
\begin{align*}
&h(\theta) =  \theta  +\sum_{i=1}^N (1-y_i) z_i -\sum_{i=1}^N \frac{z_i}{1+e^{\langle \theta, z_i\rangle}}, \quad H(\theta) = I_d + \sum_{i=1}^N \frac{z_i z_i^{\cT}e^{\langle \theta, z_i\rangle}}{(1+e^{\langle \theta, z_i\rangle})^2},\\
&\nabla^2 h^{(l)}(\theta) =  \sum_{i=1}^N \frac{z_i^{(l)}z_i z_i^{\cT}e^{\langle \theta, z_i\rangle}(1-e^{\langle \theta, z_i\rangle})}{(1+e^{\langle \theta, z_i\rangle})^3}.
\end{align*}
We first show that Assumption~\ref{asm:ALLlip} holds. To this end, denote by 
\[
g_1(t):=(e^t-e^{2t})/(1+e^t)^3, \quad g_2(t): = e^t/(1+e^t)^2, \quad g_3(t) = 1/(1+e^t),
\]
for all $t\in\R$. We have that
\begin{align*}
&\sup_{t\in\R}|g_3'(t)| = \sup_{t\in\R}|g_2(t)| \leq 1, \quad \sup_{t\in\R}|g_2'(t)| = \sup_{t\in\R}|g_1(t)| \leq 1,\\
&\sup_{t\in\R}|g_1'(t)| = \sup_{t\in\R}|e^t(1-4e^t+e^{2t})/(1+e^t)^4|\leq 1,
\end{align*}
which implies, for all $l = 1, \dots, d$, $\theta, \overline{\theta} \in \R^d$, that
\begin{align*}
|\nabla^2 h^{(l)}(\theta)  - \nabla^2 h^{(l)}(\overline{\theta}) | &= \left| \sum_{i=1}^Nz_i^{(l)}z_i z_i^{\cT}(g_1(\langle \theta, z_i\rangle) -g_1( \langle \overline{\theta} , z_i\rangle))\right|\leq  \sum_{i=1}^N|z_i|^4|\theta- \overline{\theta} |,\\
| H(\theta)  -  H(\overline{\theta}) |& = \left| \sum_{i=1}^N z_i z_i^{\cT}(g_2(\langle \theta, z_i\rangle) -g_2( \langle \overline{\theta} , z_i\rangle))\right|\leq  \sum_{i=1}^N|z_i|^3|\theta- \overline{\theta} |,\\
| h(\theta)  -  h(\overline{\theta}) |& = \left| \theta- \overline{\theta}  -\sum_{i=1}^N z_i (g_3(\langle \theta, z_i\rangle) -g_3( \langle \overline{\theta} , z_i\rangle))\right|\leq (1+ \sum_{i=1}^N|z_i|^2)|\theta- \overline{\theta} |.
\end{align*}
Thus, Assumption~\ref{asm:ALLlip} holds with $\overline{L}_1 =  \sum_{i=1}^N|z_i|^4 , \overline{L}_2 =  \sum_{i=1}^N|z_i|^3, \overline{L}_3= \left(1+ \sum_{i=1}^N|z_i|^2\right)$.

Moreover, to show that Assumption~\ref{asm:ADlip} holds, we note, for any $\theta\in\R^d$, that
\begin{align*}
\langle \theta, h(\theta) \rangle 
 = | \theta|^2  +\sum_{i=1}^N (1-y_i) \langle z_i, \theta\rangle -\sum_{i=1}^N \frac{\langle z_i, \theta\rangle}{1+e^{\langle \theta, z_i\rangle}}
&\geq |\theta|^2/2-\left(\sum_{i=1}^N |1-y_i||z_i|\right)^2/2-N.
\end{align*}
This implies that Assumption~\ref{asm:ADlip} holds with $\overline{a} = 1/2$ and $\overline{b} =\left(\sum_{i=1}^N |1-y_i||z_i|\right)^2/2+N $.
\end{proof}

\subsection{Proof of auxiliary results in Section \ref{sec:mtslproofs}}\label{sec:mtslproofsapd}

\begin{lemma}\label{lem:sde2pmmt} Let Assumptions \ref{asm:AI}, \ref{asm:ALL}, and \ref{asm:AC} hold. Then, for any $p \in {\N}, t \geq 0$, we obtain that
\[
\E[|Z^{\lambda}_t|^{2p}] \leq e^{- \lambda p \ca_\cD t }\E[|\theta_0|^{2p}]+2(\overline{\cb}_\cD+\beta^{-1}(d+2(p-1))) \cM_0^{2p-2}/ \ca_\cD<\infty,
\]
where $\cM_0 :=(2(\overline{\cb}_\cD+\beta^{-1}(d+2(p-1)))/ \ca_\cD)^{1/2}$. 
\end{lemma}
\begin{proof}See \cite[Lemma A.1]{lim2021nonasymptotic}.
\end{proof}

\begin{proof}[\textbf{Proof of Lemma \ref{lem:2ndpthmmt}-\ref{lem:2ndpthmmti}}]\label{lem:2ndpthmmtproof}
For any $0<\lambda\leq \lambda_{\max}\leq 1$ with $\lambda_{\max}$  given in \eqref{eq:stepsizemax}, $t\in (n, n+1]$, $n \in \N_0$, we define
\begin{equation}\label{eq:delxinotation} 
\Delta_{n,t}^{\lambda}
 := \widetilde{\theta}^{\lambda}_n +  \lambda \phi^{\lambda}(\widetilde{\theta}^{\lambda}_n) (t-n),\quad
\Xi_{n,t}^{\lambda}
 :=  \sqrt{2\lambda\beta^{-1}}\psi^{\lambda}(\widetilde{\theta}^{\lambda}_n) (B_t^{\lambda} - B_n^{\lambda}),
\end{equation}
where, for all $\theta \in \R^d$, 
\begin{equation}\label{eq:delxinotationphi}
\phi^{\lambda}(\theta):= - h_{\lambda}( \theta) +(\lambda/2) \left(H_{\lambda}( \theta)h_{\lambda}( \theta)-\beta^{-1}\Upsilon_{\lambda}(\theta) \right),
\end{equation}
and 
\begin{equation}\label{eq:delxinotationpsi}
\psi^{\lambda}(\theta) :=\sqrt{\cI_d - \lambda H_{\lambda}(\theta)+(\lambda^2/3) (H_{\lambda}(\theta))^2}.
\end{equation}
By using \eqref{eq:aholaproc} together with \eqref{eq:delxinotation} -- \eqref{eq:delxinotationpsi}, and by noticing $\E\left[\left.\langle \Delta_{n,t}^{\lambda}, \Xi_{n,t}^{\lambda}\rangle \right|\widetilde{\theta}^{\lambda}_n \right] =0$, we obtain that
\begin{equation}\label{eq:2ndmmtexp}
\E\left[\left.|\widetilde{\theta}^{\lambda}_t|^2\right|\widetilde{\theta}^{\lambda}_n \right]  = |\Delta_{n,t}^{\lambda}|^2+\E\left[\left.|\Xi_{n,t}^{\lambda}|^2\right|\widetilde{\theta}^{\lambda}_n \right].
\end{equation}
Further calculations yield the following upper bound for the second term on the RHS of \eqref{eq:2ndmmtexp}:
\begin{align}\label{eq:xiub1}
\E\left[\left.|\Xi_{n,t}^{\lambda}|^2\right|\widetilde{\theta}^{\lambda}_n \right] 
&= 2\lambda\beta^{-1}\E\left[\left. \left\langle \psi^{\lambda}(\widetilde{\theta}^{\lambda}_n) (B_t^{\lambda} - B_n^{\lambda}),  \psi^{\lambda}(\widetilde{\theta}^{\lambda}_n) (B_t^{\lambda} - B_n^{\lambda})\right\rangle\right|\widetilde{\theta}^{\lambda}_n \right] \nonumber\\
& =  2\lambda\beta^{-1}\E\left[\left. \left\langle    B_t^{\lambda} - B_n^{\lambda} ,  \left(\cI_d - \lambda H_{\lambda}(\widetilde{\theta}^{\lambda}_n)+(\lambda^2/3) (H_{\lambda}(\widetilde{\theta}^{\lambda}_n))^2\right)  (B_t^{\lambda} - B_n^{\lambda})\right\rangle\right|\widetilde{\theta}^{\lambda}_n \right]  \nonumber\\
& \leq  2\lambda\beta^{-1}\left(d(t-n)+\lambda\sum_{i=1}^d |H^{(i,i)}_{\lambda}(\widetilde{\theta}^{\lambda}_n)| (t-n)+(\lambda^2/3) |H_{\lambda}(\widetilde{\theta}^{\lambda}_n)|^2_{\cF}(t-n)\right) \nonumber\\
& \leq  2\lambda\beta^{-1}(t-n)\left(d+\lambda \sqrt{d}|H_{\lambda}(\widetilde{\theta}^{\lambda}_n)|_{\cF}+(\lambda^2/3) |H_{\lambda}(\widetilde{\theta}^{\lambda}_n)|^2_{\cF}\right).
\end{align}
By Remark \ref{rmk:growthc}, we have that, for all $\theta \in \R^d$,
\[
|H (\theta)|_{\cF} = \left(\sum_{i=1}^d|\nabla h^{(i)}(\theta)|^2\right)^{1/2} \leq\left(\sum_{i=1}^d \cK_1^2(1+|\theta|)^{2(\rho+q-1)}\right)^{1/2} =\sqrt{d} \cK_1(1+|\theta|)^{\rho+q-1},
\]
which implies that, for any $\theta \in \R^d$,
\begin{align}\label{eq:xiubhlambda}
\begin{split}
|H_{\lambda}(\theta)|_{\cF}  
&= \frac{|H (\theta)|_{\cF}}{(1+\lambda^{3/2}|\theta|^{3(\rho+q-1)})^{1/3}}\leq  \frac{4^{1/3}|H (\theta)|_{\cF}}{ 1+\lambda^{1/2}|\theta|^{ (\rho+q-1)} }\leq   \frac{4^{1/3}\lambda^{-1/2}|H (\theta)|_{\cF}}{ 1+|\theta|^{ (\rho+q-1)} }\\
&\leq   \frac{4^{1/3}\lambda^{-1/2} \sqrt{d} \cK_1(1+|\theta|)^{\rho+q-1}}{ 1+|\theta|^{ (\rho+q-1)} }\leq 2^{\rho+q-1}\lambda^{-1/2}\sqrt{d}  \cK_1.
\end{split}
\end{align}
Substituting \eqref{eq:xiubhlambda} into \eqref{eq:xiub1} yields
\begin{align}\label{eq:xiubfin}
\E\left[\left.|\Xi_{n,t}^{\lambda}|^2\right|\widetilde{\theta}^{\lambda}_n \right]  
&\leq  2\lambda\beta^{-1} (t-n)\left(d+2^{\rho+q-1} d\cK_1 +2^{2(\rho+q-1)} d\cK_1^2  \right)\nonumber\\
&\leq 2^{2(\rho+q)}\lambda\beta^{-1} (t-n)d(1+\cK_1)^2,
\end{align}
where we use $1+2^{\rho+q-1}+2^{2(\rho+q-1)} \leq 2^{2\rho+2q-3}+2^{2\rho+2q-3}+2^{2(\rho+q-1)} \leq 2^{2(\rho+q)-1}$ with $\rho\in [2, \infty)\cap \N$ and $q \in (0,1]$ to obtain the last inequality. To upper bound the first term on the RHS of \eqref{eq:2ndmmtexp}, we use \eqref{eq:delxinotation} to obtain
\begin{align}\label{eq:deltaub1}
|\Delta_{n,t}^{\lambda}|^2
& = |\widetilde{\theta}^{\lambda}_n|^2+2\lambda(t-n)\left\langle \widetilde{\theta}^{\lambda}_n, \phi^{\lambda}(\widetilde{\theta}^{\lambda}_n) \right\rangle+ \lambda^2(t-n)^2|\phi^{\lambda}(\widetilde{\theta}^{\lambda}_n) |^2.
\end{align}
By using \eqref{eq:delxinotationphi}, the second term on the RHS of \eqref{eq:deltaub1} can be estimated as follows:
\begin{align} 
\left\langle \widetilde{\theta}^{\lambda}_n, \phi^{\lambda}(\widetilde{\theta}^{\lambda}_n) \right\rangle 
&=-\frac{\left\langle \widetilde{\theta}^{\lambda}_n, h(\widetilde{\theta}^{\lambda}_n) \right\rangle }{\left(1+\lambda^{3/2}|\widetilde{\theta}^{\lambda}_n|^{3(\rho+q-1)}\right)^{1/3}}
+\frac{(\lambda/2)\left\langle \widetilde{\theta}^{\lambda}_n, H(\widetilde{\theta}^{\lambda}_n) h(\widetilde{\theta}^{\lambda}_n) \right\rangle }{\left(1+\lambda^{3/2}|\widetilde{\theta}^{\lambda}_n|^{3(\rho+q-1)}\right)^{2/3}} \nonumber\\
&\quad -\frac{(\lambda/2)\beta^{-1}\left\langle \widetilde{\theta}^{\lambda}_n, \Upsilon(\widetilde{\theta}^{\lambda}_n) \right\rangle }{\left(1+\lambda^{3/2}|\widetilde{\theta}^{\lambda}_n|^{3(\rho+q-1)}\right)^{1/3}}\nonumber\\
&\leq -\frac{\ca_\cD|\widetilde{\theta}^{\lambda}_n|^{\rho+q+1}-\cb_\cD }{\left(1+\lambda^{3/2}|\widetilde{\theta}^{\lambda}_n|^{3(\rho+q-1)}\right)^{1/3}} 
+\frac{2\lambda K_hK_H\left(1+|\widetilde{\theta}^{\lambda}_n|^{2(\rho+q)}\right) }{\left(1+\lambda^{3/2}|\widetilde{\theta}^{\lambda}_n|^{3(\rho+q-1)}\right)^{2/3}} \nonumber\\
&\quad +\frac{\lambda\beta^{-1}2^{\rho+q-3}\cK_{3,d}\left(1+|\widetilde{\theta}^{\lambda}_n|^{\rho+q-1}\right)}{\left(1+\lambda^{3/2}|\widetilde{\theta}^{\lambda}_n|^{3(\rho+q-1)}\right)^{1/3}}\nonumber\\
\begin{split}\label{eq:deltaubcp}
&\leq  -\frac{\ca_\cD|\widetilde{\theta}^{\lambda}_n|^{\rho+q+1}}{\left(1+\lambda^{3/2}|\widetilde{\theta}^{\lambda}_n|^{3(\rho+q-1)}\right)^{1/3}}
+\frac{2\lambda K_hK_H |\widetilde{\theta}^{\lambda}_n|^{2(\rho+q)} }{\left(1+\lambda^{3/2}|\widetilde{\theta}^{\lambda}_n|^{3(\rho+q-1)}\right)^{2/3}} \\
&\quad +\cb_\cD +2K_hK_H+\beta^{-1}2^{\rho+q-2}\cK_{3,d},
\end{split}
\end{align}
where the first inequality holds due to Remark \ref{rmk:dissipativityc} with $r=\rho+q-1$ and the fact that: for all $\theta \in \R^d$, 
\begin{align*}
|\theta||H(\theta)||h(\theta)|
&\leq K_HK_h|\theta|(1+|\theta|^{\rho+q-1})(1+|\theta|^{\rho+q}) \\
&=K_HK_h(|\theta|+|\theta|^{\rho+q}+|\theta|^{\rho+q+1}+|\theta|^{2(\rho+q)})\leq 4K_hK_H(1+|\theta|^{2(\rho+q)}),\\
|\theta||\Upsilon(\theta)|
&\leq \cK_{3,d}|\theta|(1+|\theta|)^{\rho+q-2}\leq \cK_{3,d} (1+|\theta|)^{\rho+q-1}\leq 2^{\rho+q-2}\cK_{3,d} (1+|\theta|^{\rho+q-1}),
\end{align*}
where we recall from our assumptions that $\rho\in [2, \infty)\cap \N$ and $q \in (0,1]$, and where the last inequality holds since, for all $\theta \in \R^d$,
\[
\left(1+\lambda^{3/2}|\theta|^{3(\rho+q-1)}\right)^{1/3} \geq 4^{-1/3}(1+\lambda^{1/2}|\theta|^{ \rho+q-1})\geq 4^{-1/3}\lambda^{1/2}(1+|\theta|^{ \rho+q-1}).
\]
Similarly, the third term on the RHS of \eqref{eq:deltaub1} can be upper bounded as follows:
\begin{align}\label{eq:deltaubsq1}
|\phi^{\lambda}(\widetilde{\theta}^{\lambda}_n) |^2 
& = |- h_{\lambda}( \widetilde{\theta}^{\lambda}_n) +(\lambda/2) H_{\lambda}( \widetilde{\theta}^{\lambda}_n)h_{\lambda}( \widetilde{\theta}^{\lambda}_n)-(\lambda/2) \beta^{-1}\Upsilon_{\lambda}(\widetilde{\theta}^{\lambda}_n)|^2\nonumber\\
& = |h_{\lambda}( \widetilde{\theta}^{\lambda}_n) |^2+(\lambda^2/4)| H_{\lambda}( \widetilde{\theta}^{\lambda}_n)h_{\lambda}( \widetilde{\theta}^{\lambda}_n)|^2+(\lambda^2/4)\beta^{-2}|\Upsilon_{\lambda}(\widetilde{\theta}^{\lambda}_n)|^2\nonumber\\
&\quad  -\lambda \left\langle h_{\lambda}( \widetilde{\theta}^{\lambda}_n),H_{\lambda}( \widetilde{\theta}^{\lambda}_n)h_{\lambda}( \widetilde{\theta}^{\lambda}_n)\right\rangle
+\lambda\beta^{-1}\left\langle h_{\lambda}( \widetilde{\theta}^{\lambda}_n),\Upsilon_{\lambda}(\widetilde{\theta}^{\lambda}_n)\right\rangle\nonumber\\
&\quad -(\lambda^2/2)\beta^{-1}\left\langle H_{\lambda}( \widetilde{\theta}^{\lambda}_n)h_{\lambda}( \widetilde{\theta}^{\lambda}_n),\Upsilon_{\lambda}(\widetilde{\theta}^{\lambda}_n)\right\rangle.
\end{align}
By using Assumption \ref{asm:ALL}, Remark \ref{rmk:growthc}, and the inequalities $(u+w)^v \geq 2^{v-1}(u^v+w^v)$, $(u+w)^{\nu} \leq 2^{\nu-1}(u^{\nu}+w^{\nu})$, $u, w \geq 0$, $0<v\leq 1$, $\nu \geq 1$, $|\theta|^k \leq 1+|\theta|^l$, $\theta \in \R^d$, $k \leq l$, we obtain
\begin{align}
|h_{\lambda}( \widetilde{\theta}^{\lambda}_n) |^2
&\leq  \frac{K_h^2\left(1+|\widetilde{\theta}^{\lambda}_n|^{\rho+q}\right)^2}{\left(1+\lambda^{3/2}|\widetilde{\theta}^{\lambda}_n|^{3(\rho+q-1)}\right)^{2/3}} \leq  \frac{2K_h^2\left(1+|\widetilde{\theta}^{\lambda}_n|^{2(\rho+q)}\right) }{\left(1+\lambda^{3/2}|\widetilde{\theta}^{\lambda}_n|^{3(\rho+q-1)}\right)^{2/3}}\nonumber\\
\begin{split}
&\leq  \frac{2K_h^2 |\widetilde{\theta}^{\lambda}_n|^{2(\rho+q)}  }{\left(1+\lambda^{3/2}|\widetilde{\theta}^{\lambda}_n|^{3(\rho+q-1)}\right)^{2/3}}+2K_h^2,\label{eq:deltaubsqt1}
\end{split}\\
(\lambda^2/4)| H_{\lambda}( \widetilde{\theta}^{\lambda}_n)h_{\lambda}( \widetilde{\theta}^{\lambda}_n)|^2
&\leq  \frac{(\lambda^2/4)K_H^2K_h^2\left(1+|\widetilde{\theta}^{\lambda}_n|^{\rho+q-1}\right)^2\left(1+|\widetilde{\theta}^{\lambda}_n|^{\rho+q}\right)^2}{\left(1+\lambda^{3/2}|\widetilde{\theta}^{\lambda}_n|^{3(\rho+q-1)}\right)^{4/3}}\nonumber\\
&\leq  \frac{ \lambda^2 K_H^2K_h^2\left(1+|\widetilde{\theta}^{\lambda}_n|^{2(\rho+q-1)} +|\widetilde{\theta}^{\lambda}_n|^{2(\rho+q)}+|\widetilde{\theta}^{\lambda}_n|^{4(\rho+q)-2}\right) }{\left(1+\lambda^{3/2}|\widetilde{\theta}^{\lambda}_n|^{3(\rho+q-1)}\right)^{4/3}}\nonumber\\
&\leq  \frac{3 \lambda^2 K_H^2K_h^2\left(1 +|\widetilde{\theta}^{\lambda}_n|^{4(\rho+q)-2}\right) }{\left(1+\lambda^{3/2}|\widetilde{\theta}^{\lambda}_n|^{3(\rho+q-1)}\right)^{4/3}}\nonumber\\
\begin{split}
&\leq \frac{3 \lambda^2 K_H^2K_h^2 |\widetilde{\theta}^{\lambda}_n|^{4(\rho+q)-2}  }{\left(1+\lambda^{3/2}|\widetilde{\theta}^{\lambda}_n|^{3(\rho+q-1)}\right)^{4/3}}+3  K_H^2K_h^2,\label{eq:deltaubsqt2}
\end{split}\\
(\lambda^2/4)\beta^{-2}|\Upsilon_{\lambda}(\widetilde{\theta}^{\lambda}_n)|^2
&\leq \frac{(\lambda^2/4)\beta^{-2}\cK_{3,d}^2\left(1+|\widetilde{\theta}^{\lambda}_n|\right)^{2(\rho+q-2)}}{\left(1+\lambda^{3/2}|\widetilde{\theta}^{\lambda}_n|^{3(\rho+q-1)}\right)^{2/3}}\nonumber\\
&\leq \frac{2^{2(\rho+q)-7}\lambda^2 \beta^{-2}\cK_{3,d}^2\left(1+|\widetilde{\theta}^{\lambda}_n|^{2(\rho+q-2)}\right)}{\left(1+\lambda^{3/2}|\widetilde{\theta}^{\lambda}_n|^{3(\rho+q-1)}\right)^{2/3}}\nonumber\\
&\leq \frac{2^{2(\rho+q)-6}\lambda^2 \beta^{-2}\cK_{3,d}^2\left(1+|\widetilde{\theta}^{\lambda}_n|^{2(\rho+q-2)}\right)}{ 1+\lambda |\widetilde{\theta}^{\lambda}_n|^{2(\rho+q-1)} }\nonumber\\
\begin{split}
&\leq 2^{2(\rho+q)-6}  \beta^{-2}\cK_{3,d}^2,\label{eq:deltaubsqt3}
\end{split}\\
 -\lambda \left\langle h_{\lambda}( \widetilde{\theta}^{\lambda}_n),H_{\lambda}( \widetilde{\theta}^{\lambda}_n)h_{\lambda}( \widetilde{\theta}^{\lambda}_n)\right\rangle
&\leq \frac{\lambda K_h^2K_H\left(1+|\widetilde{\theta}^{\lambda}_n|^{\rho+q}\right)^2\left(1+|\widetilde{\theta}^{\lambda}_n|^{\rho+q-1}\right)}{ 1+\lambda^{3/2}|\widetilde{\theta}^{\lambda}_n|^{3(\rho+q-1)} }\nonumber\\
&\leq \frac{2\lambda K_h^2K_H\left(1+|\widetilde{\theta}^{\lambda}_n|^{2(\rho+q)} +|\widetilde{\theta}^{\lambda}_n|^{\rho+q-1}+|\widetilde{\theta}^{\lambda}_n|^{3(\rho+q)-1}\right)}{ 1+\lambda^{3/2}|\widetilde{\theta}^{\lambda}_n|^{3(\rho+q-1)} }\nonumber\\
&\leq  \frac{6\lambda K_h^2K_H\left(1 +|\widetilde{\theta}^{\lambda}_n|^{3(\rho+q)-1}\right)}{ 1+\lambda^{3/2}|\widetilde{\theta}^{\lambda}_n|^{3(\rho+q-1)} }\nonumber\\
\begin{split}
&\leq  \frac{6\lambda K_h^2K_H |\widetilde{\theta}^{\lambda}_n|^{3(\rho+q)-1} }{ 1+\lambda^{3/2}|\widetilde{\theta}^{\lambda}_n|^{3(\rho+q-1)} }+6  K_h^2K_H,\label{eq:deltaubsqt4}
\end{split}\\
\lambda\beta^{-1}\left\langle h_{\lambda}( \widetilde{\theta}^{\lambda}_n),\Upsilon_{\lambda}(\widetilde{\theta}^{\lambda}_n)\right\rangle
&\leq \frac{\lambda\beta^{-1} K_h\cK_{3,d}\left(1+|\widetilde{\theta}^{\lambda}_n|^{\rho+q}\right)\left(1+|\widetilde{\theta}^{\lambda}_n|\right)^{\rho+q-2} }{ \left(1+\lambda^{3/2}|\widetilde{\theta}^{\lambda}_n|^{3(\rho+q-1)}\right)^{2/3} }\nonumber\\
&\leq \frac{\lambda\beta^{-1} K_h\cK_{3,d}2^{\rho+q-3}\left(1+|\widetilde{\theta}^{\lambda}_n|^{\rho+q} +|\widetilde{\theta}^{\lambda}_n|^{\rho+q-2} +|\widetilde{\theta}^{\lambda}_n|^{2(\rho+q-1)}\right)}{ \left(1+\lambda^{3/2}|\widetilde{\theta}^{\lambda}_n|^{3(\rho+q-1)}\right)^{2/3} }\nonumber\\
&\leq\frac{3\lambda\beta^{-1} K_h\cK_{3,d}2^{\rho+q-2}\left(1  +|\widetilde{\theta}^{\lambda}_n|^{2(\rho+q-1)}\right)}{  1+\lambda |\widetilde{\theta}^{\lambda}_n|^{2(\rho+q-1)} }\nonumber\\
\begin{split}
&\leq \beta^{-1} K_h\cK_{3,d}2^{\rho+q},\label{eq:deltaubsqt5}
\end{split}\\
&\hspace{-10em} -(\lambda^2/2)\beta^{-1}\left\langle H_{\lambda}( \widetilde{\theta}^{\lambda}_n)h_{\lambda}( \widetilde{\theta}^{\lambda}_n),\Upsilon_{\lambda}(\widetilde{\theta}^{\lambda}_n)\right\rangle\nonumber\\
&\hspace{-8em}\leq \frac{(\lambda^2/2)\beta^{-1} K_H K_h\cK_{3,d}\left(1+|\widetilde{\theta}^{\lambda}_n|^{\rho+q-1}\right)\left(1+|\widetilde{\theta}^{\lambda}_n|^{\rho+q}\right) \left(1+|\widetilde{\theta}^{\lambda}_n|\right)^{\rho+q-2}}{ 1+\lambda^{3/2}|\widetilde{\theta}^{\lambda}_n|^{3(\rho+q-1)} }\nonumber\\
&\hspace{-8em}\leq \frac{ \lambda^2 \beta^{-1} K_H K_h\cK_{3,d}2^{\rho+q-4}\left(1+|\widetilde{\theta}^{\lambda}_n|^{\rho+q-1} +|\widetilde{\theta}^{\lambda}_n|^{\rho+q}+|\widetilde{\theta}^{\lambda}_n|^{2(\rho+q)-1}\right) \left(1+|\widetilde{\theta}^{\lambda}_n|^{\rho+q-2}\right)}{ 1+\lambda^{3/2}|\widetilde{\theta}^{\lambda}_n|^{3(\rho+q-1)} }\nonumber\\
\begin{split}
&\hspace{-8em}\leq \frac{7 \lambda^2 \beta^{-1} K_H K_h\cK_{3,d}2^{\rho+q-4} \left(1+|\widetilde{\theta}^{\lambda}_n|^{3(\rho+q-1)}\right)}{ 1+\lambda^{3/2}|\widetilde{\theta}^{\lambda}_n|^{3(\rho+q-1)} }\leq  \beta^{-1} K_H K_h\cK_{3,d}2^{\rho+q-1} .\label{eq:deltaubsqt6}
\end{split}
\end{align}
Substituting \eqref{eq:deltaubsqt1} -- \eqref{eq:deltaubsqt6} into \eqref{eq:deltaubsq1} yields
\begin{align}
\begin{split}\label{eq:deltaubsq2}
|\phi^{\lambda}(\widetilde{\theta}^{\lambda}_n) |^2 
&\leq \frac{2K_h^2 |\widetilde{\theta}^{\lambda}_n|^{2(\rho+q)}  }{\left(1+\lambda^{3/2}|\widetilde{\theta}^{\lambda}_n|^{3(\rho+q-1)}\right)^{2/3}}+2K_h^2+\frac{3 \lambda^2 K_H^2K_h^2 |\widetilde{\theta}^{\lambda}_n|^{4(\rho+q)-2}  }{\left(1+\lambda^{3/2}|\widetilde{\theta}^{\lambda}_n|^{3(\rho+q-1)}\right)^{4/3}}+3  K_H^2K_h^2\\
&\quad +\frac{6\lambda K_h^2K_H |\widetilde{\theta}^{\lambda}_n|^{3(\rho+q)-1} }{ 1+\lambda^{3/2}|\widetilde{\theta}^{\lambda}_n|^{3(\rho+q-1)} }+6  K_h^2K_H+2^{2(\rho+q)-6}  \beta^{-2}\cK_{3,d}^2+\beta^{-1} K_h\cK_{3,d}2^{\rho+q}\\ 
&\quad +\beta^{-1} K_H K_h\cK_{3,d}2^{\rho+q-1}.
\end{split}
\end{align}
Combining the results in \eqref{eq:deltaubcp} and \eqref{eq:deltaubsq2}, we obtain the following upper bound for \eqref{eq:deltaub1}:
\begin{align}\label{eq:deltaub2}
|\Delta_{n,t}^{\lambda}|^2
&\leq |\widetilde{\theta}^{\lambda}_n|^2 -\frac{2\lambda(t-n)\ca_\cD|\widetilde{\theta}^{\lambda}_n|^{\rho+q+1}}{\left(1+\lambda^{3/2}|\widetilde{\theta}^{\lambda}_n|^{3(\rho+q-1)}\right)^{1/3}}
+\frac{4\lambda^2(t-n) K_hK_H |\widetilde{\theta}^{\lambda}_n|^{2(\rho+q)} }{\left(1+\lambda^{3/2}|\widetilde{\theta}^{\lambda}_n|^{3(\rho+q-1)}\right)^{2/3}}\nonumber\\
&\quad +\frac{\lambda^2(t-n)2K_h^2 |\widetilde{\theta}^{\lambda}_n|^{2(\rho+q)}  }{\left(1+\lambda^{3/2}|\widetilde{\theta}^{\lambda}_n|^{3(\rho+q-1)}\right)^{2/3}}
+\frac{\lambda^4(t-n)3 K_H^2K_h^2 |\widetilde{\theta}^{\lambda}_n|^{4(\rho+q)-2}  }{\left(1+\lambda^{3/2}|\widetilde{\theta}^{\lambda}_n|^{3(\rho+q-1)}\right)^{4/3}}\nonumber\\
&\quad +\frac{\lambda^3(t-n)6 K_h^2K_H |\widetilde{\theta}^{\lambda}_n|^{3(\rho+q)-1} }{ 1+\lambda^{3/2}|\widetilde{\theta}^{\lambda}_n|^{3(\rho+q-1)} } 
+\lambda(t-n)\Big(2\cb_\cD +4K_hK_H\Big.\nonumber\\
&\quad +\beta^{-1}2^{\rho+q-1}\cK_{3,d}+2K_h^2+3  K_H^2K_h^2+6  K_h^2K_H\nonumber\\
&\quad \Big.+2^{2(\rho+q)-6}  \beta^{-2}\cK_{3,d}^2+\beta^{-1} K_h\cK_{3,d}2^{\rho+q}+\beta^{-1} K_H K_h\cK_{3,d}2^{\rho+q-1}\Big)\nonumber\\
& = \left(1 -\frac{\lambda(t-n)\ca_\cD|\widetilde{\theta}^{\lambda}_n|^{\rho+q-1}}{\left(1+\lambda^{3/2}|\widetilde{\theta}^{\lambda}_n|^{3(\rho+q-1)}\right)^{1/3}}\right)|\widetilde{\theta}^{\lambda}_n|^2-\lambda(t-n)\mathfrak{I}^{\lambda}_1(\widetilde{\theta}^{\lambda}_n)+ \lambda(t-n)\cc_1,
\end{align}
where for any $\theta \in \R^d$
\begin{align*}
\mathfrak{I}^{\lambda}_1(\theta)
&:= \frac{ \ca_\cD|\theta|^{\rho+q+1}}{\left(1+\lambda^{3/2}|\theta|^{3(\rho+q-1)}\right)^{1/3}} 
- \frac{\lambda (4 K_hK_H +2K_h^2)|\theta|^{2(\rho+q)} }{\left(1+\lambda^{3/2}|\theta|^{3(\rho+q-1)}\right)^{2/3}}\\
&\qquad -\frac{\lambda^3 3 K_H^2K_h^2 |\theta|^{4(\rho+q)-2}  }{\left(1+\lambda^{3/2}|\theta|^{3(\rho+q-1)}\right)^{4/3}}
-  \frac{\lambda^2 6 K_h^2K_H |\widetilde{\theta}^{\lambda}_n|^{3(\rho+q)-1} }{ 1+\lambda^{3/2}|\widetilde{\theta}^{\lambda}_n|^{3(\rho+q-1)} } ,
\end{align*}
and where
\begin{align*}
\cc_1&:=\Big(2\cb_\cD +4K_hK_H +\beta^{-1}2^{\rho+q-1}\cK_{3,d}+2K_h^2+3  K_H^2K_h^2+6  K_h^2K_H\Big.\\ 
&\qquad \Big.+2^{2(\rho+q)-6}  \beta^{-2}\cK_{3,d}^2+\beta^{-1} K_h\cK_{3,d}2^{\rho+q}+\beta^{-1} K_H K_h\cK_{3,d}2^{\rho+q-1}\Big).
\end{align*}
We note that, for $0<\lambda\leq \lambda_{\max}\leq \min\{(19\ca_\cD / 240K_hK_H)^2,(19\ca_\cD/240 K_h^2)^2,  (\ca_\cD/120K_h^2K_H^2)^{2/3}, \allowbreak \ca_\cD/(480K_h^2K_H)\}$, and for all $\theta \in \R^d$, 
\begin{equation}\label{eq:deltaub3}
\mathfrak{I}^{\lambda}_1(\theta) \geq 0.
\end{equation}
Indeed, by using the expression of $\mathfrak{I}^{\lambda}_1$, we have that, for all $\theta \in \R^d$, 
\begin{align*}
\mathfrak{I}^{\lambda}_1(\theta)
&=\frac{ (38\ca_\cD/60+19\ca_\cD/60)|\theta|^{\rho+q+1}\left(1+\lambda^{3/2}|\theta|^{3(\rho+q-1)}\right)^{1/3}-\lambda (4 K_hK_H +2K_h^2)|\theta|^{2(\rho+q)} }{\left(1+\lambda^{3/2}|\theta|^{3(\rho+q-1)}\right)^{2/3}}\\
&\quad +\frac{ (\ca_\cD/40)|\theta|^{\rho+q+1}\left(1+\lambda^{3/2}|\theta|^{3(\rho+q-1)}\right) -\lambda^3 3 K_H^2K_h^2 |\theta|^{4(\rho+q)-2}  }{\left(1+\lambda^{3/2}|\theta|^{3(\rho+q-1)}\right)^{4/3}}\\
&\quad+\frac{ (\ca_\cD/40)|\theta|^{\rho+q+1}\left(1+\lambda^{3/2}|\theta|^{3(\rho+q-1)}\right)^{2/3} -\lambda^2 6 K_h^2K_H |\theta|^{3(\rho+q)-1}  }{ 1+\lambda^{3/2}|\theta|^{3(\rho+q-1)} }\\
&\geq \frac{ (19\ca_\cD/60+19\ca_\cD/120)  \lambda^{1/2}|\theta|^{2(\rho+q)}-\lambda (4 K_hK_H +2K_h^2)|\theta|^{2(\rho+q)} }{\left(1+\lambda^{3/2}|\theta|^{3(\rho+q-1)}\right)^{2/3}}\\
& \quad+\frac{ (\ca_\cD/40) \lambda^{3/2}|\theta|^{4(\rho+q)-2} -\lambda^3 3 K_H^2K_h^2 |\theta|^{4(\rho+q)-2}  }{\left(1+\lambda^{3/2}|\theta|^{3(\rho+q-1)}\right)^{4/3}}\\
&\quad+\frac{ (\ca_\cD/80) \lambda |\theta|^{3(\rho+q)-1}  -\lambda^2 6 K_h^2K_H |\theta|^{3(\rho+q)-1}  }{ 1+\lambda^{3/2}|\theta|^{3(\rho+q-1)} }\\
&\geq 0,
\end{align*}
where the last inequality holds due to $0<\lambda\leq \lambda_{\max}$. Substituting \eqref{eq:deltaub3} into \eqref{eq:deltaub2} yields
\[
|\Delta_{n,t}^{\lambda}|^2
\leq \left(1 -\frac{\lambda(t-n)\ca_\cD|\widetilde{\theta}^{\lambda}_n|^{\rho+q-1}}{\left(1+\lambda^{3/2}|\widetilde{\theta}^{\lambda}_n|^{3(\rho+q-1)}\right)^{1/3}}\right)|\widetilde{\theta}^{\lambda}_n|^2 + \lambda(t-n)\cc_1.
\]
In addition, since $f(s) = s/(1+\lambda^{3/2}s^3)^{1/3}$ is non-decreasing for all $s \geq 0$, we obtain that, for all $|\theta|>\cM_1>0$,
\[
\frac{|\theta|^{\rho+q-1}}{\left(1+\lambda^{3/2}|\theta|^{3(\rho+q-1)}\right)^{1/3}}\geq \frac{\cM_1^{\rho+q-1}}{\left(1+\lambda^{3/2}\cM_1^{3(\rho+q-1)}\right)^{1/3}}\geq \frac{\cM_1^{\rho+q-1}}{\left(1+ \cM_1^{3(\rho+q-1)}\right)^{1/3}}=:\kappa.
\]
Denote by $\cS_{n,\cM_1}:=\{\omega \in \Omega:|\widetilde{\theta}^{\lambda}_n(\omega)|>\cM_1\}$. The above inequality further implies that,
\begin{equation*} 
|\Delta_{n,t}^{\lambda}|^2 \mathbbm{1}_{\cS_{n,\cM_1}} 
\leq \left(1 -\lambda(t-n)\ca_\cD\kappa \right)|\widetilde{\theta}^{\lambda}_n|^2\mathbbm{1}_{\cS_{n,\cM_1}} + \lambda(t-n)\cc_1 \mathbbm{1}_{\cS_{n,\cM_1}}.
\end{equation*}
Similarly, we have that
\begin{equation*} 
|\Delta_{n,t}^{\lambda}|^2\mathbbm{1}_{\cS_{n,\cM_1}^{\cc}}
\leq \left(1 -\lambda(t-n)\ca_\cD\kappa \right)|\widetilde{\theta}^{\lambda}_n|^2\mathbbm{1}_{\cS_{n,\cM_1}^{\cc}} + \lambda(t-n)(\ca_\cD\kappa \cM_1^2 +\cc_1)\mathbbm{1}_{\cS_{n,\cM_1}^{\cc}}.
\end{equation*}
Combining the two cases yields
\begin{equation} \label{eq:deltaubfin}
|\Delta_{n,t}^{\lambda}|^2 
\leq  \left(1 -\lambda(t-n)\ca_\cD\kappa \right)|\widetilde{\theta}^{\lambda}_n|^2+ \lambda(t-n)(\ca_\cD\kappa \cM_1^2 +\cc_1).
\end{equation}
Finally, by substituting \eqref{eq:xiubfin} and \eqref{eq:deltaubfin} into \eqref{eq:2ndmmtexp}, we obtain
\begin{equation}\label{eq:2ndmmtexpfin}
\E\left[\left.|\widetilde{\theta}^{\lambda}_t|^2\right|\widetilde{\theta}^{\lambda}_n \right]  
\leq  \left(1 -\lambda(t-n)\ca_\cD\kappa \right)|\widetilde{\theta}^{\lambda}_n|^2+ \lambda(t-n)\cc_0,
\end{equation}
where 
\begin{align}
\begin{split}\label{eq:2ndmmtexpconst} 
\kappa&:=\cM_1^{\rho+q-1}/\left(1+ \cM_1^{3(\rho+q-1)}\right)^{1/3} \quad (\text{with $ \cM_1>0$}),\\
\cc_0 &: =\ca_\cD\kappa \cM_1^2 +\cc_1+2^{2(\rho+q)} \beta^{-1} d(1+\cK_1)^2 , \\
\cc_1&:=\Big(2\cb_\cD +4K_hK_H +\beta^{-1}2^{\rho+q-1}\cK_{3,d}+2K_h^2+3  K_H^2K_h^2+6  K_h^2K_H\Big.\\ 
&\qquad \Big.+2^{2(\rho+q)-6}  \beta^{-2}\cK_{3,d}^2+\beta^{-1} K_h\cK_{3,d}2^{\rho+q}+\beta^{-1} K_H K_h\cK_{3,d}2^{\rho+q-1}\Big).
\end{split}
\end{align}
We observe that, for $0<\lambda\leq \lambda_{\max}\leq 1/\ca_\cD$,
\[
1>1 -\lambda(t-n)\ca_\cD\kappa > 1-\lambda \ca_\cD\geq 0,
\]
then, by induction, \eqref{eq:2ndmmtexpfin} implies, for $t\in (n, n+1]$, $n \in \N_0$, $0<\lambda\leq \lambda_{\max}\leq 1$, that, 
\begin{align*}
\E\left[ |\widetilde{\theta}^{\lambda}_t|^2 \right]  
&\leq  \left(1 -\lambda(t-n)\ca_\cD\kappa \right)\E\left[ |\widetilde{\theta}^{\lambda}_n|^2\right]+ \lambda(t-n)\cc_0\\
&\leq  \left(1 -\lambda(t-n)\ca_\cD\kappa \right)\left(1 -\lambda \ca_\cD\kappa \right)\E\left[ |\widetilde{\theta}^{\lambda}_{n-1}|^2\right]+  \cc_0+ \lambda \cc_0\\
&\leq  \left(1 -\lambda(t-n)\ca_\cD\kappa \right)\left(1 -\lambda \ca_\cD\kappa \right)^2\E\left[ |\widetilde{\theta}^{\lambda}_{n-2}|^2\right]+  \cc_0+ \lambda \cc_0\left(1+\left(1 -\lambda \ca_\cD\kappa \right)\right)\\
&\leq \dots\\
&\leq \left(1 -\lambda(t-n)\ca_\cD\kappa \right)\left(1 -\lambda \ca_\cD\kappa \right)^n\E\left[ |\theta_0|^2\right]+  \cc_0\left(1+1/(\ca_\cD\kappa)\right),
\end{align*}
which completes the proof.
\end{proof}

\begin{proof}[\textbf{Proof of Lemma \ref{lem:2ndpthmmt}-\ref{lem:2ndpthmmtii}}] For any $p\in [2, \infty)\cap {\N}$, $0<\lambda\leq \lambda_{\max}\leq 1$ with $\lambda_{\max}$  given in \eqref{eq:stepsizemax}, $t\in (n, n+1]$, $n \in \N_0$, by using the same arguments as in the proof of \cite[Lemma 4.2-(ii)]{lim2021nonasymptotic} up to the inequality before \cite[Eq. (134)]{lim2021nonasymptotic} and by using \eqref{eq:aholaproc} with \eqref{eq:delxinotation}, we obtain that
\begin{align}
\begin{split}\label{eq:2pthmmtexp}
\E\left[\left.|\widetilde{\theta}^{\lambda}_t|^{2p}\right|\widetilde{\theta}^{\lambda}_n \right] 
&\leq |\Delta_{n,t}^{\lambda}|^{2p} +2^{2p-3}p(2p-1)|\Delta_{n,t}^{\lambda}|^{2p-2}  \E\left[\left.|\Xi_{n,t}^{\lambda}|^2\right|\widetilde{\theta}^{\lambda}_n \right]  \\
&\quad + 2^{2p-3}p(2p-1)\E\left[\left.|\Xi_{n,t}^{\lambda}|^{2p}\right|\widetilde{\theta}^{\lambda}_n \right].
\end{split}
\end{align}
Then, by using \eqref{eq:delxinotation} and \eqref{eq:delxinotationpsi}, we can obtain an upper estimate for the last term in \eqref{eq:2pthmmtexp} as follows:
\begin{align}\label{eq:xi2pub} 
\E\left[\left.|\Xi_{n,t}^{\lambda}|^{2p}\right|\widetilde{\theta}^{\lambda}_n \right]
& = (2\lambda\beta^{-1})^p \E\left[\left. \left\langle  (B_t^{\lambda} - B_n^{\lambda}),  
\left(\psi^{\lambda}(\widetilde{\theta}^{\lambda}_n)\right)^2 (B_t^{\lambda} - B_n^{\lambda})\right\rangle^p\right|\widetilde{\theta}^{\lambda}_n \right] \nonumber\\
&\leq  (2\lambda\beta^{-1})^p\left|\cI_d - \lambda H_{\lambda}(\widetilde{\theta}^{\lambda}_n)+(\lambda^2/3) (H_{\lambda}(\widetilde{\theta}^{\lambda}_n))^2\right|^p\E\left[|B_t^{\lambda} - B_n^{\lambda}|^{2p}\right]\nonumber\\
&\leq  (2\lambda\beta^{-1}dp(2p-1)(t-n))^p3^{p-1}\left(1+\lambda^p|H_{\lambda}(\widetilde{\theta}^{\lambda}_n)|^p_{\cF}+(\lambda^2/3)^p |H_{\lambda}(\widetilde{\theta}^{\lambda}_n)|^{2p}_{\cF}\right)\nonumber\\
&\leq \lambda(t-n) (2\beta^{-1}dp(2p-1))^p3^{p-1}\left(1+\lambda^p|H_{\lambda}(\widetilde{\theta}^{\lambda}_n)|^p_{\cF}\right)^2\nonumber\\
&\leq \lambda(t-n) (2\beta^{-1}dp(2p-1))^p3^{p-1}\left(1+ 2^{p(\rho+q-1)} d^{p/2}  \cK_1^p  \right)^2,
\end{align}
where the last inequality holds due to \eqref{eq:xiubhlambda}. Substituting \eqref{eq:xiubfin} and \eqref{eq:xi2pub} into \eqref{eq:2pthmmtexp} yields
\begin{align}\label{eq:2pthmmtexpub1}
\begin{split}
\E\left[\left.|\widetilde{\theta}^{\lambda}_t|^{2p}\right|\widetilde{\theta}^{\lambda}_n \right] 
&\leq |\Delta_{n,t}^{\lambda}|^{2p} +\lambda(t-n) 2^{2(p+\rho+q)-3}p(2p-1) \beta^{-1} d(1+\cK_1)^2|\Delta_{n,t}^{\lambda}|^{2p-2}\\
&\quad + \lambda(t-n) \cc_{\Xi}(p),
\end{split}
\end{align}
where $\cc_{\Xi}(p):= 2^{2p-4}(\beta^{-1}d)^p(2p(2p-1))^{p+1}3^{p-1}\left(1+ 2^{p(\rho+q-1)} d^{p/2}  \cK_1^p  \right)^2$. Next, we apply \eqref{eq:deltaubfin} to obtain
\begin{align}\label{eq:2pthmmtexpub2}
|\Delta_{n,t}^{\lambda}|^{2p} 
&\leq \left( \left(1 -\lambda(t-n)\ca_\cD\kappa \right)|\widetilde{\theta}^{\lambda}_n|^2+ \lambda(t-n)(\ca_\cD\kappa \cM_1^2 +\cc_1)\right)^p\nonumber\\
&\leq \left(1+\lambda(t-n)\ca_\cD\kappa /2\right)^{p-1}\left(1 -\lambda(t-n)\ca_\cD\kappa \right)^p|\widetilde{\theta}^{\lambda}_n|^{2p}\nonumber\\
&\quad +\left(1+2/(\lambda(t-n)\ca_\cD\kappa )\right)^{p-1}(\lambda(t-n))^p(\ca_\cD\kappa \cM_1^2 +\cc_1)^p\nonumber\\
&\leq \left(1-\lambda(t-n)\ca_\cD\kappa /2\right)^{p-1}\left(1 -\lambda(t-n)\ca_\cD\kappa \right)|\widetilde{\theta}^{\lambda}_n|^{2p}\nonumber\\
&\quad +\lambda(t-n)\left(1+2/( \ca_\cD\kappa )\right)^{p-1}(\ca_\cD\kappa \cM_1^2 +\cc_1)^p\nonumber\\
&=\overline{\cc}_{n,t}^\lambda(p)|\widetilde{\theta}^{\lambda}_n|^{2p}+\widetilde{\cc}_{n,t}^\lambda(p),
\end{align}
where the second inequality holds due to $(u+v)^p\leq (1+\varepsilon)^{p-1}u^p+(1+\varepsilon^{-1})^{p-1}v^p$, $u,v \geq 0$, $\varepsilon>0$ with $\varepsilon = \lambda(t-n)\ca_\cD\kappa /2$, and where
\begin{align*}
\overline{\cc}_{n,t}^\lambda(p) &:=  \left(1-\lambda(t-n)\ca_\cD\kappa /2\right)^{p-1}\left(1 -\lambda(t-n)\ca_\cD\kappa \right),\\
\widetilde{\cc}_{n,t}^\lambda(p)&:=\lambda(t-n)\left(1+2/( \ca_\cD\kappa )\right)^{p-1}(\ca_\cD\kappa \cM_1^2 +\cc_1)^p.
\end{align*}
In addition, we observe that by \eqref{eq:2pthmmtexpub2}, 
\begin{equation}\label{eq:2pthmmtexpub3}
|\Delta_{n,t}^{\lambda}|^{2p-2}\leq  \overline{\cc}_{n,t}^\lambda(p-1)|\widetilde{\theta}^{\lambda}_n|^{2p-2}+\widetilde{\cc}_{n,t}^\lambda(p-1),
\end{equation}
and, in particular, when $p = 2$, \eqref{eq:2pthmmtexpub3} yields $|\Delta_{n,t}^{\lambda}|^2 \leq  \overline{\cc}_{n,t}^\lambda(1)|\widetilde{\theta}^{\lambda}_n|^{2}+\widetilde{\cc}_{n,t}^\lambda(1)$ which is exactly the upper bound \eqref{eq:deltaubfin}. By substituting \eqref{eq:2pthmmtexpub2} and \eqref{eq:2pthmmtexpub3} into \eqref{eq:2pthmmtexpub1}, we have that
\begin{align}\label{eq:2pthmmtexpub4}
\begin{split}
\E\left[\left.|\widetilde{\theta}^{\lambda}_t|^{2p}\right|\widetilde{\theta}^{\lambda}_n \right] 
&\leq \overline{\cc}_{n,t}^\lambda(p)|\widetilde{\theta}^{\lambda}_n|^{2p}+\widetilde{\cc}_{n,t}^\lambda(p) + \lambda(t-n) \cc_{\Xi}(p)\\
&\quad +\lambda(t-n) 2^{2(p+\rho+q)-3}p(2p-1) \beta^{-1} d(1+\cK_1)^2\\
&\qquad \times\left(\overline{\cc}_{n,t}^\lambda(p-1)|\widetilde{\theta}^{\lambda}_n|^{2p-2}+\widetilde{\cc}_{n,t}^\lambda(p-1)\right).
\end{split}
\end{align}
Denote by $\cM_2(p):=(2^{2(p+\rho+q)-1}p(2p-1) \beta^{-1} d(1+\cK_1)^2/(\ca_\cD\kappa))^{1/2}$. For all $|\theta|>\cM_2(p)$, we have that
\[
(\lambda(t-n)\ca_\cD\kappa /4)|\theta|^{2p}>(\lambda(t-n)2^{2(p+\rho+q)-1}p(2p-1) \beta^{-1} d(1+\cK_1)^2/4)|\theta|^{2p-2}.
\]
Denote by $\cS_{n,\cM_2(p)}:=\{\omega \in \Omega:|\widetilde{\theta}^{\lambda}_n(\omega)|>\cM_2(p)\}$. By using the above inequality, \eqref{eq:2pthmmtexpub4} can be further bounded as follows:
\begin{align}\label{eq:2pthmmtexpub5}
&\E\left[\left.|\widetilde{\theta}^{\lambda}_t|^{2p}\mathbbm{1}_{\cS_{n,\cM_2(p)}} \right|\widetilde{\theta}^{\lambda}_n \right] \nonumber\\
&\leq \left(1-\lambda(t-n)\ca_\cD\kappa /4\right)\overline{\cc}_{n,t}^\lambda(p-1)|\widetilde{\theta}^{\lambda}_n|^{2p}\mathbbm{1}_{\cS_{n,\cM_2(p)}}  +\left(\widetilde{\cc}_{n,t}^\lambda(p) + \lambda(t-n) \cc_{\Xi}(p)\right)\mathbbm{1}_{\cS_{n,\cM_2(p)}} \nonumber\\
&\quad +\lambda(t-n) 2^{2(p+\rho+q)-3}p(2p-1) \beta^{-1} d(1+\cK_1)^2\widetilde{\cc}_{n,t}^\lambda(p-1)\mathbbm{1}_{\cS_{n,\cM_2(p)}}\nonumber \\
&\quad -\left(\lambda(t-n)\ca_\cD\kappa /4\right)\overline{\cc}_{n,t}^\lambda(p-1)|\widetilde{\theta}^{\lambda}_n|^{2p}\mathbbm{1}_{\cS_{n,\cM_2(p)}}\nonumber \\
&\quad +(\lambda(t-n) 2^{2(p+\rho+q)-1}p(2p-1) \beta^{-1} d(1+\cK_1)^2/4)\overline{\cc}_{n,t}^\lambda(p-1)|\widetilde{\theta}^{\lambda}_n|^{2p-2}\mathbbm{1}_{\cS_{n,\cM_2(p)}}\nonumber \\
&\leq \left(1 -\lambda(t-n)\ca_\cD\kappa \right)|\widetilde{\theta}^{\lambda}_n|^{2p}\mathbbm{1}_{\cS_{n,\cM_2(p)}} +\lambda(t-n)\left[\left(1+2/( \ca_\cD\kappa )\right)^{p-1}(\ca_\cD\kappa \cM_1^2 +\cc_1)^p+\cc_{\Xi}(p)\right.\nonumber\\
&\quad \left. +2^{2(p+\rho+q)-3}p(2p-1) \beta^{-1} d(1+\cK_1)^2\right.\nonumber\\
&\qquad \left. \times\left(\left(1+2/( \ca_\cD\kappa )\right)^{p-2}(\ca_\cD\kappa \cM_1^2 +\cc_1)^{p-1}+\cM_2(p)^{2p-2}\right)\right]\mathbbm{1}_{\cS_{n,\cM_2(p)}} \nonumber\\
&\leq \left(1 -\lambda(t-n)\ca_\cD\kappa \right)|\widetilde{\theta}^{\lambda}_n|^{2p}\mathbbm{1}_{\cS_{n,\cM_2(p)}} +\lambda(t-n)\cc_p \mathbbm{1}_{\cS_{n,\cM_2(p)}},
\end{align}
where
\begin{align}\label{eq:2pthmmtexpconst} 
\begin{split}
\cc_p 
&:= \left(1+2/( \ca_\cD\kappa )\right)^{p-1}(\ca_\cD\kappa \cM_1^2 +\cc_1)^p+\cc_{\Xi}(p) +2^{2(p+\rho+q)-3}p(2p-1) \beta^{-1} d(1+\cK_1)^2\\
&\qquad \times\left(\left(1+2/( \ca_\cD\kappa )\right)^{p-2}(\ca_\cD\kappa \cM_1^2 +\cc_1)^{p-1}+\cM_2(p)^{2p-2}\right),\\
\cc_{\Xi}(p)
&:= 2^{2p-4}(\beta^{-1}d)^p(2p(2p-1))^{p+1}3^{p-1}\left(1+ 2^{p(\rho+q-1)} d^{p/2}  \cK_1^p  \right)^2,\\
\cM_2(p)
&:=(2^{2(p+\rho+q)-1}p(2p-1) \beta^{-1} d(1+\cK_1)^2/(\ca_\cD\kappa))^{1/2}
\end{split}
\end{align}
with $\kappa$ and $\cc_1$ given in \eqref{eq:2ndmmtexpconst}. Similarly, by using \eqref{eq:2pthmmtexpub4}, we have that
\begin{align}\label{eq:2pthmmtexpub6}
\E\left[\left.|\widetilde{\theta}^{\lambda}_t|^{2p}\mathbbm{1}_{\cS_{n,\cM_2(p)}^{\cc}} \right|\widetilde{\theta}^{\lambda}_n \right]\leq \left(1 -\lambda(t-n)\ca_\cD\kappa \right)|\widetilde{\theta}^{\lambda}_n|^{2p}\mathbbm{1}_{\cS_{n,\cM_2(p)}^{\cc}} +\lambda(t-n)\cc_p \mathbbm{1}_{\cS_{n,\cM_2(p)}^{\cc}}.
\end{align}
Combing \eqref{eq:2pthmmtexpub5} and \eqref{eq:2pthmmtexpub6} yields the desired result.
\end{proof}

\begin{lemma}\label{lem:oserroralg} Let Assumptions \ref{asm:AI}, \ref{asm:ALL}, and \ref{asm:AC} hold. Then, for any $0<\lambda\leq \lambda_{\max}$, $p > 0, t \geq 0$, we obtain
\begin{align}
\E\left[|\overline{\theta}^{\lambda}_t - \overline{\theta}^{\lambda}_{\lfrf{t}} |^{2p}\right]
&\leq \lambda^p\left(e^{-  \lambda\ca_\cD\kappa \lfrf{t}}\overline{\cC}_{\mathbf{A}0,p}\E[|\theta_0|^{4\lcrc{p}(\rho+1)}]+\widetilde{\cC}_{\mathbf{A}0,p} \right),\label{lem:oserroralgineq1}\\
\E\left[|\overline{\zeta}^{\lambda, n}_t - \overline{\zeta}^{\lambda, n}_{\lfrf{t}}|^{2p}\right]
&\leq \lambda^p\left(e^{-  \lambda\ca_\cD\min\{ \kappa, 1 /2\} \lfrf{t}}\overline{\cC}_{\mathbf{A}1,p}\E[|\theta_0|^{2\lcrc{p}(\rho+1)}]+\widetilde{\cC}_{\mathbf{A}1,p} \right),\label{lem:oserroralgineq2}
\end{align}
where $\overline{\cC}_{\mathbf{A}0,p}$ and $\widetilde{\cC}_{\mathbf{A}0,p}$ are given in \eqref{eq:oserroralgconst}, and $\overline{\cC}_{\mathbf{A}1,p}$ and $\widetilde{\cC}_{\mathbf{A}1,p}$ are given in \eqref{eq:oserrorauxconst}.
\end{lemma}
\begin{proof} To show that \eqref{lem:oserroralgineq1} holds, we use the definition of $(\overline{\theta}^{\lambda}_t)_{t \geq 0 }$ given in \eqref{eq:aholahoproc} and obtain that, for any $t \geq 0$,
\begin{align*}
&\E\left[|\overline{\theta}^{\lambda}_t - \overline{\theta}^{\lambda}_{\lfrf{t}} |^{2p}\right] \\
&= \E\left[\left|-\lambda \int_{\lfrf{t}}^t h_{\lambda}(\overline{\theta}^{\lambda}_{\lfrf{s}})\,\rmd s+\lambda^2\int_{\lfrf{t}}^t \int_{\lfrf{s}}^s\left(H_{\lambda}( \overline{\theta}^{\lambda}_{\lfrf{r}})h_{\lambda}( \overline{\theta}^{\lambda}_{\lfrf{r}})-\beta^{-1}\Upsilon_{\lambda}(\overline{\theta}^{\lambda}_{\lfrf{r}}) \right)\,\rmd r\,\rmd s \right.\right.\\
&\qquad \left.\left. -\lambda\sqrt{2\lambda\beta^{-1}} \int_{\lfrf{t}}^t \int_{{\lfrf{s}}}^s H_{\lambda}(\overline{\theta}_{\lfrf{r}}^{\lambda})\,\rmd B_r^{\lambda}\,\rmd s  +\sqrt{2\lambda\beta^{-1}} \int_{\lfrf{t}}^t \, \rmd B^{\lambda}_s\right|^{2p}\right] \\
&\leq 5^{2p}\Bigg(\lambda^{2p}\E\left[|h_{\lambda}(\overline{\theta}^{\lambda}_{\lfrf{t}})|^{2p}\right]+\lambda^{4p}\E\left[|H_{\lambda}( \overline{\theta}^{\lambda}_{\lfrf{r}})|^{2p}|h_{\lambda}( \overline{\theta}^{\lambda}_{\lfrf{t}})|^{2p}\right]\Bigg.\\
&\qquad +\left. \lambda^{4p}\beta^{-2p}\E\left[|\Upsilon_{\lambda}(\overline{\theta}^{\lambda}_{\lfrf{t}}) |^{2p}\right]+\lambda^p(2 \beta^{-1})^p\E\left[\left| \int_{\lfrf{t}}^t   \,\rmd B_s^{\lambda} \right|^{2p}\right]\right.\\
&\qquad \Bigg. +\lambda^{3p}(2 \beta^{-1})^p\E\left[|H_{\lambda}(\overline{\theta}_{\lfrf{t}}^{\lambda})|^{2p}\left| \int_{\lfrf{t}}^t \int_{{\lfrf{s}}}^s \,\rmd B_r^{\lambda}\,\rmd s\right|^{2p}\right]\Bigg) \\
&\leq 5^{2p}\lambda^p\left(K_h^{2p}\E\left[(1+|\overline{\theta}^{\lambda}_{\lfrf{t}}|^{\rho+q})^{2p}\right]+\beta^{-2p}\cK_{3,d}^{2p}\E\left[(1+|\overline{\theta}^{\lambda}_{\lfrf{t}}|)^{2p(\rho+q-2)}\right]\right.\\
&\qquad +K_H^{2p}K_h^{2p}\E\left[(1+|\overline{\theta}^{\lambda}_{\lfrf{t}}|^{\rho+q-1})^{2p}(1+|\overline{\theta}^{\lambda}_{\lfrf{t}}|^{\rho+q})^{2p}\right]+(2\beta^{-1}(p+1)(d+2p))^p\\
&\qquad \left.+ (\beta^{-1}(4p+2)(d+4p))^pK_H^{2p}\left(\E\left[(1+|\overline{\theta}^{\lambda}_{\lfrf{t}}|^{\rho+q-1})^{4p} \right] \right)^{1/2}\right)\\
&\leq \lambda^p5^{2p}\left(K_h^{2p}+\beta^{-2p}\cK_{3,d}^{2p}+K_H^{2p}K_h^{2p}+(\beta^{-1}(4p+2)(d+4p))^p(1+K_H^{2p})\right)\\
&\qquad \times 2^{4\lcrc{p}(\rho+1)}\left(\E\left[ |\overline{\theta}^{\lambda}_{\lfrf{t}}|^{4\lcrc{p}(\rho+1)}\right]+1\right)\\
&\leq \lambda^p\left(e^{-\lambda\ca_\cD\kappa \lfrf{t}}\overline{\cC}_{\mathbf{A}0,p}\E[|\theta_0|^{4\lcrc{p}(\rho+1)}]+\widetilde{\cC}_{\mathbf{A}0,p} \right),
\end{align*}
where the first inequality holds due to $(\sum_{l=1}^vu_l)^w\leq v^w\sum_{l=1}^vu_l^w$, $v \in \N$, $u_l \geq 0$, $w>0$, the second inequality holds due to Remark \ref{rmk:growthc}, Cauchy-Schwarz inequality, and the following inequality:
\[
\E\left[ \left| \int_{\lfrf{t}}^t   \,\rmd B_s^{\lambda} \right|^{2p}\right] \leq \max\{d^p, (p(d+2p-2))^p\}\leq ((p+1)(d+2p))^p,
\] 
the fourth inequality holds due to Lemma \ref{lem:2ndpthmmt}, and where
\begin{align}\label{eq:oserroralgconst}
\begin{split}
\overline{\cC}_{\mathbf{A}0,p}&: = 5^{2p}2^{4\lcrc{p}(\rho+1)}\left(K_h^{2p}+\beta^{-2p}\cK_{3,d}^{2p}+K_H^{2p}K_h^{2p}+(\beta^{-1}(4p+2)(d+4p))^p(1+K_H^{2p})\right),\\
\widetilde{\cC}_{\mathbf{A}0,p}&: = \overline{\cC}_{\mathbf{A}0,p}\left( \cc_{2\lcrc{p}(\rho+1)}\left(1+1/(\ca_\cD\kappa)\right)+1\right).
\end{split}
\end{align}
The inequality \eqref{lem:oserroralgineq2} can be obtained by using similar arguments. More precisely, by using Definition \ref{def:auxzeta} with \eqref{eq:auxproc}, we obtain that, for any $t \geq 0$,
\begin{align*}
\E\left[|\overline{\zeta}^{\lambda, n}_t - \overline{\zeta}^{\lambda, n}_{\lfrf{t}}|^{2p}\right]
& = \E\left[\left|-\lambda \int_{\lfrf{t}}^t h(\overline{\zeta}^{\lambda, n}_s)\,\rmd s+\sqrt{2\lambda\beta^{-1}} \int_{\lfrf{t}}^t \, \rmd B^{\lambda}_s \right|^{2p}\right]\\
&\leq 2^{2p}\left(\lambda^{2p}\E\left[\int_{\lfrf{t}}^t |h(\overline{\zeta}^{\lambda, n}_s)|^{2p}\,\rmd s\right]+\lambda^p(2\beta^{-1})^p\E\left[\left| \int_{\lfrf{t}}^t   \,\rmd B_s^{\lambda} \right|^{2p}\right]\right)\\
&\leq 2^{2p}\lambda^p\left(K_h^{2p}\int_{\lfrf{t}}^t \E\left[(1+|\overline{\zeta}^{\lambda, n}_s|^{\rho+q})^{2p}\right]\,\rmd s+ (2\beta^{-1}(p+1)(d+2p))^p\right)\\
&\leq \lambda^p\left( 2^{2p}K_h^{2p}\int_{\lfrf{t}}^t \E\left[V_{2\lcrc{p}(\rho+1)}(\overline{\zeta}^{\lambda, n}_s) \right]\,\rmd s+  2^{2p}(2\beta^{-1}(p+1)(d+2p))^p\right)\\
&\leq  \lambda^p\left(e^{-\lambda\ca_\cD \min\{ \kappa, 1 /2\} \lfrf{t}}\overline{\cC}_{\mathbf{A}1,p}\E[|\theta_0|^{2\lcrc{p}(\rho+1)}]+\widetilde{\cC}_{\mathbf{A}1,p} \right),
\end{align*}
where the last inequality holds due to Lemma \ref{lem:zetaprocme} and where
\begin{align}\label{eq:oserrorauxconst}
\begin{split}
\overline{\cC}_{\mathbf{A}1,p}&: = 2^{2p+\lcrc{p}(\rho+1)-1}K_h^{2p}, \\
\widetilde{\cC}_{\mathbf{A}1,p}
&: = \overline{\cC}_{\mathbf{A}1,p}\left( \cc_{\lcrc{p}(\rho+1)}\left(1+1/(\ca_\cD\kappa)\right)+1\right)+ 2^{2p}K_h^{2p}3\mathrm{v}_{2\lcrc{p}(\rho+1)}(\cM_V(2\lcrc{p}(\rho+1)))\\
&\quad +2^{2p}(2\beta^{-1}(p+1)(d+2p))^p.
\end{split}
\end{align}
This completes the proof.
\end{proof}

\begin{lemma}\label{lem:graditoest} Let Assumptions \ref{asm:AI}, \ref{asm:ALL}, and \ref{asm:AC} hold. Then, for any $0<\lambda\leq \lambda_{\max}$, $n \in \N_0$, $, t \geq nT$, we obtain the following inequalities:
\begin{align*}
\E\left[\left|-\lambda \int_{\lfrf{t}}^t  \left(H(\overline{\theta}^{\lambda}_s) - H_{\lambda}(\overline{\theta}^{\lambda}_{\lfrf{s}})\right)h_{\lambda}( \overline{\theta}^{\lambda}_{\lfrf{s}})\,\rmd s\right|^2\right] 
\leq \lambda^3\left(e^{- \ca_\cD  \kappa n/2} \overline{\cC}_{\mathbf{A}2}\E[|\theta_0|^{16(\rho+1)}] +\widetilde{\cC}_{\mathbf{A}2} \right),&\\
\E\left[\left| \lambda^2 \int_{\lfrf{t}}^t   H(\overline{\theta}^{\lambda}_s) \int_{\lfrf{s}}^s  H_{\lambda}(\overline{\theta}^{\lambda}_{\lfrf{r}}))h_{\lambda}( \overline{\theta}^{\lambda}_{\lfrf{r}}) \, \rmd r \,\rmd s\right|^2\right] 
\leq \lambda^3\left(e^{- \ca_\cD  \kappa n/2} \overline{\cC}_{\mathbf{A}2}\E[|\theta_0|^{16(\rho+1)}] +\widetilde{\cC}_{\mathbf{A}2} \right), &\\
\E\left[\left| -\lambda^2\beta^{-1} \int_{\lfrf{t}}^t   H(\overline{\theta}^{\lambda}_s) \int_{\lfrf{s}}^s  \Upsilon_{\lambda}( \overline{\theta}^{\lambda}_{\lfrf{r}}) \, \rmd r \,\rmd s\right|^2\right] 
\leq \lambda^3\left(e^{- \ca_\cD  \kappa n/2} \overline{\cC}_{\mathbf{A}2}\E[|\theta_0|^{16(\rho+1)}] +\widetilde{\cC}_{\mathbf{A}2} \right), &\\
\E\left[\left| -\lambda\sqrt{2\lambda\beta^{-1}} \int_{\lfrf{t}}^t   H(\overline{\theta}^{\lambda}_s) \int_{\lfrf{s}}^s  H_{\lambda}( \overline{\theta}^{\lambda}_{\lfrf{r}}) \, \rmd B_r^{\lambda} \,\rmd s\right|^2\right] 
\leq \lambda^3\left(e^{- \ca_\cD  \kappa n/2} \overline{\cC}_{\mathbf{A}2}\E[|\theta_0|^{16(\rho+1)}] +\widetilde{\cC}_{\mathbf{A}2} \right), &\\
\E\left[\left| \sqrt{2\lambda\beta^{-1}} \int_{\lfrf{t}}^t  \left( H(\overline{\theta}^{\lambda}_s) -  H_{\lambda}( \overline{\theta}^{\lambda}_{\lfrf{s}})  \right) \,\rmd B_s^{\lambda}\right|^2\right] 
\leq \lambda^2\left(e^{- \ca_\cD  \kappa n/2} \overline{\cC}_{\mathbf{A}2}\E[|\theta_0|^{16(\rho+1)}] +\widetilde{\cC}_{\mathbf{A}2} \right), &\\
\E\left[\left| \sqrt{2\lambda\beta^{-1}} \int_{\lfrf{t}}^t  \left( H(\overline{\zeta}^{\lambda, n}_s) -  H ( \overline{\zeta}^{\lambda, n}_{\lfrf{s}})\right)   \,\rmd B_s^{\lambda}\right|^2\right] \hspace{18em}&\\
\leq \lambda^2\left(e^{- \ca_\cD   \min\{ \kappa, 1 /2\} n/2} \overline{\cC}_{\mathbf{A}2}\E[|\theta_0|^{16(\rho+1)}] +\widetilde{\cC}_{\mathbf{A}2} \right),&\\
\E\left[\left| \lambda\beta^{-1} \int_{\lfrf{t}}^t  \left( \Upsilon(\overline{\theta}^{\lambda}_s) - \Upsilon_{\lambda}( \overline{\theta}^{\lambda}_{\lfrf{s}}) \right)  \,\rmd s\right|^2\right] 
\leq \lambda^{2+q}\left(e^{- \ca_\cD  \kappa n/2} \overline{\cC}_{\mathbf{A}2}\E[|\theta_0|^{16(\rho+1)}] +\widetilde{\cC}_{\mathbf{A}2} \right),&\\
\E\left[\left| -\lambda  \int_{\lfrf{t}}^t H(\overline{\theta}^{\lambda}_s) h_{\lambda}(\overline{\theta}^{\lambda}_{\lfrf{s}})\,\rmd s\right|^2+\left| \lambda\beta^{-1}\int_{\lfrf{t}}^t\Upsilon (\overline{\theta}^{\lambda}_s)\rmd s\right|^2\right] \hspace{15em}& \\
\leq \lambda^2\left(e^{- \ca_\cD  \kappa n/2} \overline{\cC}_{\mathbf{A}2}\E[|\theta_0|^{16(\rho+1)}] +\widetilde{\cC}_{\mathbf{A}2} \right), &\\
\E\left[\left|  -\lambda  \int_{\lfrf{t}}^t H(\overline{\zeta}^{\lambda, n}_s) h ( \overline{\zeta}^{\lambda, n}_s)\,\rmd s  \right|^2+\left| \lambda\beta^{-1}\int_{\lfrf{t}}^t\Upsilon (\overline{\zeta}^{\lambda, n}_s)\rmd s\right|^2\right] \hspace{15em}& \\
\leq \lambda^2\left(e^{- \ca_\cD \min\{ \kappa, 1 /2\} n/2} \overline{\cC}_{\mathbf{A}2}\E[|\theta_0|^{16(\rho+1)}] +\widetilde{\cC}_{\mathbf{A}2} \right),&
\end{align*}
where $\overline{\cC}_{\mathbf{A}2}$ and $\widetilde{\cC}_{\mathbf{A}2}$ are given in \eqref{eq:graditoestconst}.
\end{lemma}
\begin{proof} We note that, by using \eqref{eq:taming}, for any $f: \R^d \rightarrow \R^{i\times j}$, $i.j \in \N$, $\theta \in \R^d$,
\begin{equation}\label{eq:graditoestint1}
|f(\theta) - f_{\lambda}(\theta)|=\left| \frac{\left((1+\lambda^{3/2}|\theta|^{3(\rho+q-1)})^{1/3}  -1\right) f(\theta)}{(1+\lambda^{3/2}|\theta|^{3(\rho+q-1)})^{1/3}} \right|\leq \lambda^{3/2}|\theta|^{3(\rho+q-1)}| f(\theta)|.
\end{equation}
The inequalities can be obtained by using the following arguments:
\begin{enumerate}
\item To show that the first inequality holds, by using Remark \ref{rmk:growthc} and \eqref{eq:graditoestint1}, we have that
\begin{align*}
&\E\left[\left|-\lambda \int_{\lfrf{t}}^t \left(H(\overline{\theta}^{\lambda}_s) - H_{\lambda}(\overline{\theta}^{\lambda}_{\lfrf{s}})\right)h_{\lambda}( \overline{\theta}^{\lambda}_{\lfrf{s}})\,\rmd s\right|^2\right] \\
&\leq \lambda^2 \int_{\lfrf{t}}^t \E\left[|H(\overline{\theta}^{\lambda}_s) - H_{\lambda}(\overline{\theta}^{\lambda}_{\lfrf{s}})|^2|h_{\lambda}( \overline{\theta}^{\lambda}_{\lfrf{s}}) |^2\right]  \,\rmd s\\
&\leq 2\lambda^2 \int_{\lfrf{t}}^t \E\left[|H(\overline{\theta}^{\lambda}_s) - H (\overline{\theta}^{\lambda}_{\lfrf{s}})|_{\cF}^2|h_{\lambda}( \overline{\theta}^{\lambda}_{\lfrf{s}}) |^2\right]  \,\rmd s\\
&\quad + 2\lambda^2 \int_{\lfrf{t}}^t \E\left[|H(\overline{\theta}^{\lambda}_{\lfrf{s}}) - H_{\lambda}(\overline{\theta}^{\lambda}_{\lfrf{s}})|^2|h_{\lambda}( \overline{\theta}^{\lambda}_{\lfrf{s}}) |^2\right]  \,\rmd s\\
&\leq 2\lambda^2 \int_{\lfrf{t}}^t \E\left[d\cK_0^2(1+|\overline{\theta}^{\lambda}_s|+|\overline{\theta}^{\lambda}_{\lfrf{s}}|)^{2(\rho+q-2)}|\overline{\theta}^{\lambda}_s - \overline{\theta}^{\lambda}_{\lfrf{s}}|^2K_h^2(1+|\overline{\theta}^{\lambda}_{\lfrf{s}}|^{\rho+q})^2\right]  \,\rmd s\\
&\quad + 2\lambda^2 \int_{\lfrf{t}}^t \E\left[ \lambda^{3}|\overline{\theta}^{\lambda}_{\lfrf{s}}|^{6(\rho+q-1)} K_H^2(1+|\overline{\theta}^{\lambda}_{\lfrf{s}}|^{\rho+q-1})^2 K_h^2(1+|\overline{\theta}^{\lambda}_{\lfrf{s}}|^{\rho+q})^2\right]  \,\rmd s\\
&\leq  2\lambda^2d\cK_0^2K_h^2 \int_{\lfrf{t}}^t3^{8(\rho+1)} \left(\E\left[1+|\overline{\theta}^{\lambda}_s|^{16(\rho+1)}+|\overline{\theta}^{\lambda}_{\lfrf{s}}|^{16(\rho+1)}\right] \right)^{1/2} \\
&\qquad \times\left(\E\left[|\overline{\theta}^{\lambda}_s - \overline{\theta}^{\lambda}_{\lfrf{s}}|^4\right] \right)^{1/2} \,\rmd s\\
&\quad +\lambda^5 K_H^2K_h^2 \int_{\lfrf{t}}^t 2^{16(\rho+1)}\E\left[  1+|\overline{\theta}^{\lambda}_{\lfrf{s}}|^{16(\rho+1)} \right]  \,\rmd s\\
&\leq  \lambda^23^{8(\rho+1)+1} d\cK_0^2K_h^2\left(2e^{-\lambda \ca_\cD \kappa \lfrf{t}}\E\left[|\theta_0|^{16(\rho+1)}\right]  +2 \cc_{8(\rho+1)}\left(1+1/(\ca_\cD\kappa)\right)+1\right)^{1/2}\\
&\qquad \times \lambda \left(e^{-  \lambda\ca_\cD\kappa \lfrf{t}}\overline{\cC}_{\mathbf{A}0,2}\E[|\theta_0|^{8(\rho+1)}]+\widetilde{\cC}_{\mathbf{A}0,2} \right)^{1/2} \\
&\quad + \lambda^5 2^{16(\rho+1)}  K_H^2K_h^2\left(e^{-\lambda \ca_\cD \kappa \lfrf{t}}\E\left[|\theta_0|^{16(\rho+1)}\right]  + \cc_{8(\rho+1)}\left(1+1/(\ca_\cD\kappa)\right)+1\right)\\
&\leq \lambda^33^{8(\rho+1)+1} d\cK_0^2K_h^2\left(e^{-\lambda \ca_\cD \kappa \lfrf{t}}(2+\overline{\cC}_{\mathbf{A}0,2})\E\left[|\theta_0|^{16(\rho+1)}\right]  \right.\\
&\qquad \left.+2 \cc_{8(\rho+1)}\left(1+1/(\ca_\cD\kappa)\right)+2\widetilde{\cC}_{\mathbf{A}0,2}+1\right)\\
&\quad +\lambda^32^{16(\rho+1)}  K_H^2K_h^2\left(e^{-\lambda \ca_\cD \kappa \lfrf{t}}\E\left[|\theta_0|^{16(\rho+1)}\right]  + \cc_{8(\rho+1)}\left(1+1/(\ca_\cD\kappa)\right)+1\right)\\
&\leq \lambda^3\left(e^{-  \ca_\cD \kappa n/2}\overline{\cC}_{\mathbf{A}2}\E\left[|\theta_0|^{16(\rho+1)}\right]  + \widetilde{\cC}_{\mathbf{A}2}\right),
\end{align*}
where the fifth inequality holds by applying Lemma \ref{lem:2ndpthmmt} and \ref{lem:oserroralg}, the second last inequality holds by using $\overline{\cC}_{\mathbf{A}0,2}<\widetilde{\cC}_{\mathbf{A}0,2}$, the last inequality holds due to $\lambda \lfrf{t}\geq \lambda nT \geq n/2$, and where $\kappa$ is given in \eqref{eq:2ndmmtexpconst},
\begin{align}\label{eq:graditoestconst}
\begin{split}
\overline{\cC}_{\mathbf{A}2}&:=3^{24(\rho+1)}d^3(1+\cK_0)^2(1+K_h)^2(1+K_H)^4(1+\beta^{-1})^2(1+L)^2(1+ \cK_{3,d})^2\\
&\qquad \times \left(2+\max\{\overline{\cC}_{\mathbf{A}0,2},\overline{\cC}_{\mathbf{A}0,2q},\overline{\cC}_{\mathbf{A}1,2},\overline{\cC}_{\mathbf{A}0,2+2q}\}\right)\\
\widetilde{\cC}_{\mathbf{A}2}&:=3^{24(\rho+1)}d^3(1+\cK_0)^2(1+K_h)^2(1+K_H)^4(1+\beta^{-1})^2(1+L)^2(1+ \cK_{3,d})^2\\
&\qquad \times \Big(2 \cc_{8(\rho+1)}\left(1+1/(\ca_\cD\kappa)\right)+2\max\{\widetilde{\cC}_{\mathbf{A}0,2},\widetilde{\cC}_{\mathbf{A}0,2q},\widetilde{\cC}_{\mathbf{A}1,2},\widetilde{\cC}_{\mathbf{A}0,2+2q}\}+1 \Big.\\
&\qquad \Big. +\mathrm{v}_{16(\rho+1)}(\cM_V(16(\rho+1)))\Big)
\end{split}
\end{align}
with $\overline{\cC}_{\mathbf{A}0,p}$, $\widetilde{\cC}_{\mathbf{A}0,p}$, $\overline{\cC}_{\mathbf{A}1,p}$, $\widetilde{\cC}_{\mathbf{A}1,p}$, $p>0$, given in \eqref{eq:oserroralgconst} and \eqref{eq:oserrorauxconst}, and $ \cc_p$, $\cM_V(p)$, $p\in [2, \infty)\cap {\N}$ given in \eqref{eq:2pthmmtexpconst} (see also Lemma \ref{lem:2ndpthmmt}) and Lemma \ref{lem:driftcon}.
\item To establish the second inequality, we apply Lemma \ref{lem:2ndpthmmt} to obtain
\begin{align*}
&\E\left[\left| \lambda^2 \int_{\lfrf{t}}^t   H(\overline{\theta}^{\lambda}_s) \int_{\lfrf{s}}^s  H_{\lambda}(\overline{\theta}^{\lambda}_{\lfrf{r}}))h_{\lambda}( \overline{\theta}^{\lambda}_{\lfrf{r}}) \, \rmd r \,\rmd s\right|^2\right] \\
&\leq \lambda^4 \int_{\lfrf{t}}^t\E\left[|H(\overline{\theta}^{\lambda}_s)|^2|H_{\lambda}(\overline{\theta}^{\lambda}_{\lfrf{s}}))|^2|h_{\lambda}( \overline{\theta}^{\lambda}_{\lfrf{s}})|^2\right]\,\rmd s\\ 
&\leq \lambda^4 \int_{\lfrf{t}}^t\E\left[K_H^4(1+| \overline{\theta}^{\lambda}_s|^{\rho+q-1})^2(1+| \overline{\theta}^{\lambda}_{\lfrf{s}}|^{\rho+q-1})^2K_h^2(1+| \overline{\theta}^{\lambda}_{\lfrf{s}}|^{\rho+q})^2\right]\,\rmd s\\
&\leq \lambda^4K_H^4K_h^2 \int_{\lfrf{t}}^t3^{16(\rho+1)}\E\left[1+| \overline{\theta}^{\lambda}_s|^{16(\rho+1)}+| \overline{\theta}^{\lambda}_{\lfrf{s}}|^{16(\rho+1)}\right]\,\rmd s\\
&\leq \lambda^33^{16(\rho+1)}K_H^4K_h^2 \left(2e^{-\lambda \ca_\cD \kappa \lfrf{t}}\E\left[|\theta_0|^{16(\rho+1)}\right]  +2 \cc_{8(\rho+1)}\left(1+1/(\ca_\cD\kappa)\right)+1\right)\\
&\leq \lambda^3\left(e^{- \ca_\cD  \kappa n/2} \overline{\cC}_{\mathbf{A}2}\E[|\theta_0|^{16(\rho+1)}] +\widetilde{\cC}_{\mathbf{A}2} \right),
\end{align*} 
where $\overline{\cC}_{\mathbf{A}2}$ and $\widetilde{\cC}_{\mathbf{A}2}$ are given in \eqref{eq:graditoestconst}.
\item To obtain the third inequality, we apply Lemma \ref{lem:2ndpthmmt} and write
\begin{align*}
&\E\left[\left| -\lambda^2\beta^{-1} \int_{\lfrf{t}}^t   H(\overline{\theta}^{\lambda}_s) \int_{\lfrf{s}}^s  \Upsilon_{\lambda}( \overline{\theta}^{\lambda}_{\lfrf{r}}) \, \rmd r \,\rmd s\right|^2\right] \\
&\leq \lambda^4\beta^{-2} \int_{\lfrf{t}}^t \E\left[ |H(\overline{\theta}^{\lambda}_s)|^2| \Upsilon_{\lambda}( \overline{\theta}^{\lambda}_{\lfrf{s}})  |^2\right]  \,\rmd s\\
&\leq \lambda^4\beta^{-2} \int_{\lfrf{t}}^t \E\left[K_H^2(1+|\overline{\theta}^{\lambda}_s|^{\rho+q-1})^2  \cK_{3,d}^2(1+|\overline{\theta}^{\lambda}_{\lfrf{s}}|)^{2(\rho+q-2)} \right]  \,\rmd s\\
&\leq \lambda^4\beta^{-2} K_H^2\cK_{3,d}^2\int_{\lfrf{t}}^t 3^{16(\rho+1)}\E\left[1+| \overline{\theta}^{\lambda}_s|^{16(\rho+1)}+| \overline{\theta}^{\lambda}_{\lfrf{s}}|^{16(\rho+1)}\right]  \,\rmd s\\
&\leq  \lambda^3  3^{16(\rho+1)} \beta^{-2} K_H^2\cK_{3,d}^2\left(2e^{-\lambda \ca_\cD \kappa \lfrf{t}}\E\left[|\theta_0|^{16(\rho+1)}\right]  +2 \cc_{8(\rho+1)}\left(1+1/(\ca_\cD\kappa)\right)+1\right)\\
&\leq \lambda^3\left(e^{- \ca_\cD  \kappa n/2} \overline{\cC}_{\mathbf{A}2}\E[|\theta_0|^{16(\rho+1)}] +\widetilde{\cC}_{\mathbf{A}2} \right),
\end{align*}
where $\overline{\cC}_{\mathbf{A}2}$ and $\widetilde{\cC}_{\mathbf{A}2}$ are given in \eqref{eq:graditoestconst}.
\item To obtain the fourth inequality, we use Cauchy-Schwarz inequality and Lemma \ref{lem:2ndpthmmt}:
\begin{align*}
&\E\left[\left| -\lambda\sqrt{2\lambda\beta^{-1}} \int_{\lfrf{t}}^t   H(\overline{\theta}^{\lambda}_s) \int_{\lfrf{s}}^s  H_{\lambda}( \overline{\theta}^{\lambda}_{\lfrf{r}}) \, \rmd B_r^{\lambda} \,\rmd s\right|^2\right] \\
&\leq 2\lambda^3\beta^{-1}  \int_{\lfrf{t}}^t  \E\left[ |  H(\overline{\theta}^{\lambda}_s) |^2|H_{\lambda}( \overline{\theta}^{\lambda}_{\lfrf{s}}) |^2\left|\int_{\lfrf{s}}^s  \, \rmd B_r^{\lambda}  \right|^2\right] \,\rmd s\\
&\leq 2\lambda^3\beta^{-1}  \int_{\lfrf{t}}^t  \E\left[K_H^4(1+|\overline{\theta}^{\lambda}_s |^{\rho+q-1})^2(1+|\overline{\theta}^{\lambda}_{\lfrf{s}} |^{\rho+q-1})^2 \left|\int_{\lfrf{s}}^s  \, \rmd B_r^{\lambda}  \right|^2\right] \,\rmd s\\
&\leq 2\lambda^3\beta^{-1} K_H^4 \int_{\lfrf{t}}^t 3^{8(\rho+1)} \left( \E\left[1+|\overline{\theta}^{\lambda}_s |^{16(\rho+1)}+|\overline{\theta}^{\lambda}_{\lfrf{s}} |^{16(\rho+1)} \right]\right)^{1/2}\\
&\qquad \times \left(\E\left[ \left|\int_{\lfrf{s}}^s  \, \rmd B_r^{\lambda}  \right|^4\right]\right)^{1/2}\,\rmd s\\
&\leq \lambda^3\beta^{-1} K_H^43^{8(\rho+1)+2}(d+4)\left(2e^{-\lambda \ca_\cD \kappa \lfrf{t}}\E\left[|\theta_0|^{16(\rho+1)}\right]  +2 \cc_{8(\rho+1)}\left(1+1/(\ca_\cD\kappa)\right)+1\right) \\
&\leq \lambda^3\left(e^{- \ca_\cD  \kappa n/2} \overline{\cC}_{\mathbf{A}2}\E[|\theta_0|^{16(\rho+1)}] +\widetilde{\cC}_{\mathbf{A}2} \right),
\end{align*}
where $\overline{\cC}_{\mathbf{A}2}$ and $\widetilde{\cC}_{\mathbf{A}2}$ are given in \eqref{eq:graditoestconst}.
\item\label{item:graditoestfifth} To obtain the fifth inequality, we apply Remark \ref{rmk:growthc}, \eqref{eq:graditoestint1}, and  Cauchy-Schwarz inequality:
\begin{align*} 
&\E\left[\left| \sqrt{2\lambda\beta^{-1}} \int_{\lfrf{t}}^t   \left(H(\overline{\theta}^{\lambda}_s) -  H_{\lambda}( \overline{\theta}^{\lambda}_{\lfrf{s}})\right)   \,\rmd B_s^{\lambda}\right|^2\right] \\
& =  2\lambda\beta^{-1}  \int_{\lfrf{t}}^t \E\left[ |   H(\overline{\theta}^{\lambda}_s) -  H_{\lambda}( \overline{\theta}^{\lambda}_{\lfrf{s}})  |_{\cF}^2\right]    \,\rmd s\\
&\leq 4\lambda\beta^{-1}  \int_{\lfrf{t}}^t \E\left[ |   H(\overline{\theta}^{\lambda}_s) -  H ( \overline{\theta}^{\lambda}_{\lfrf{s}})  |_{\cF}^2\right]    \,\rmd s + 4\lambda\beta^{-1}  \int_{\lfrf{t}}^t \E\left[ |  H ( \overline{\theta}^{\lambda}_{\lfrf{s}})-  H_{\lambda}( \overline{\theta}^{\lambda}_{\lfrf{s}})  |_{\cF}^2\right]    \,\rmd s \\
&\leq 4\lambda\beta^{-1}  \int_{\lfrf{t}}^t \E\left[ d\cK_0^2(1+| \overline{\theta}^{\lambda}_s|+| \overline{\theta}^{\lambda}_{\lfrf{s}} |)^{2(\rho+q-2)}|    \overline{\theta}^{\lambda}_s  -  \overline{\theta}^{\lambda}_{\lfrf{s}}  |^2\right]    \,\rmd s\\
&\quad  + 4\lambda\beta^{-1}  \int_{\lfrf{t}}^t \E\left[  \lambda^{3}|\overline{\theta}^{\lambda}_{\lfrf{s}}|^{6(\rho+q-1)} dK_H^2(1+|\overline{\theta}^{\lambda}_{\lfrf{s}}|^{\rho+q-1})^2\right]    \,\rmd s \\
&\leq 4\lambda\beta^{-1}  d\cK_0^2 \int_{\lfrf{t}}^t 3^{8(\rho+1)} \left(\E\left[ 1+| \overline{\theta}^{\lambda}_s|^{16(\rho+1)} +| \overline{\theta}^{\lambda}_{\lfrf{s}} |^{16(\rho+1)}  \right] \right)^{1/2}  \\
&\qquad \times \left(\E\left[  |    \overline{\theta}^{\lambda}_s  -  \overline{\theta}^{\lambda}_{\lfrf{s}}  |^4\right] \right)^{1/2}  \,\rmd s\\
&\quad  + 4\lambda^4\beta^{-1} dK_H^2 \int_{\lfrf{t}}^t 2^{16(\rho+1)} \E\left[   1+|\overline{\theta}^{\lambda}_{\lfrf{s}}|^{16(\rho+1)}\right]    \,\rmd s \\
&\leq  \lambda\beta^{-1}  d\cK_0^2 3^{8(\rho+1)+2} \left(2e^{-\lambda \ca_\cD \kappa \lfrf{t}}\E\left[|\theta_0|^{16(\rho+1)}\right]  +2 \cc_{8(\rho+1)}\left(1+1/(\ca_\cD\kappa)\right)+1\right)^{1/2} \\
&\qquad \times  \lambda \left(e^{-  \lambda\ca_\cD\kappa \lfrf{t}}\overline{\cC}_{\mathbf{A}0,2}\E[|\theta_0|^{8(\rho+1)}]+\widetilde{\cC}_{\mathbf{A}0,2} \right)^{1/2} \\
&\quad +\lambda^4 \beta^{-1} dK_H^2 2^{16(\rho+1)+2} \left( e^{-\lambda \ca_\cD \kappa \lfrf{t}}\E\left[|\theta_0|^{16(\rho+1)}\right]  +  \cc_{8(\rho+1)}\left(1+1/(\ca_\cD\kappa)\right)+1\right) \\
&\leq  \lambda^2\beta^{-1}  d\cK_0^2 3^{8(\rho+1)+2} \left(e^{-\lambda \ca_\cD \kappa \lfrf{t}}(2+\overline{\cC}_{\mathbf{A}0,2})\E\left[|\theta_0|^{16(\rho+1)}\right]  \right.\\
&\qquad \left. +2 \cc_{8(\rho+1)}\left(1+1/(\ca_\cD\kappa)\right)+2\widetilde{\cC}_{\mathbf{A}0,2} +1\right)  \\ 
&\quad  +\lambda^2 \beta^{-1} dK_H^2 2^{16(\rho+1)+2} \left( e^{-\lambda \ca_\cD \kappa \lfrf{t}}\E\left[|\theta_0|^{16(\rho+1)}\right]  +  \cc_{8(\rho+1)}\left(1+1/(\ca_\cD\kappa)\right)+1\right) \\
&\leq \lambda^2\left(e^{- \ca_\cD  \kappa n/2} \overline{\cC}_{\mathbf{A}2}\E[|\theta_0|^{16(\rho+1)}] +\widetilde{\cC}_{\mathbf{A}2} \right),
\end{align*}
where the fourth inequality holds by applying Lemma \ref{lem:2ndpthmmt} and \ref{lem:oserroralg}, and where $\overline{\cC}_{\mathbf{A}2}$ and $\widetilde{\cC}_{\mathbf{A}2}$ are given in \eqref{eq:graditoestconst}.
\item To obtain the sixth inequality, we use the same arguments as in \ref{item:graditoestfifth}:
\begin{align*} 
&\E\left[\left| \sqrt{2\lambda\beta^{-1}} \int_{\lfrf{t}}^t   \left(H(\overline{\zeta}^{\lambda, n}_s) -  H ( \overline{\zeta}^{\lambda, n}_{\lfrf{s}}) \right)  \,\rmd B_s^{\lambda}\right|^2\right] \\
& =  2\lambda\beta^{-1}  \int_{\lfrf{t}}^t \E\left[ | H(\overline{\zeta}^{\lambda, n}_s) -  H ( \overline{\zeta}^{\lambda, n}_{\lfrf{s}})   |_{\cF}^2\right]    \,\rmd s\\
& \leq  2\lambda\beta^{-1}  \int_{\lfrf{t}}^t \E\left[d\cK_0^2(1+|\overline{\zeta}^{\lambda, n}_s|+|\overline{\zeta}^{\lambda, n}_{\lfrf{s}}|)^{2(\rho+q-2)} |\overline{\zeta}^{\lambda, n}_s- \overline{\zeta}^{\lambda, n}_{\lfrf{s}}|^2\right]    \,\rmd s\\
& \leq  2\lambda\beta^{-1}d\cK_0^2  \int_{\lfrf{t}}^t 3^{8(\rho+1)}\left( \E\left[1+|\overline{\zeta}^{\lambda, n}_s|^{16(\rho+1)}+|\overline{\zeta}^{\lambda, n}_{\lfrf{s}}|^{16(\rho+1)}   \right]  \right)^{1/2}\\
&\qquad \times \left( \E\left[ |\overline{\zeta}^{\lambda, n}_s- \overline{\zeta}^{\lambda, n}_{\lfrf{s}}|^4\right]  \right)^{1/2} \,\rmd s\\
&\leq  \lambda\beta^{-1}d\cK_0^23^{8(\rho+1)+1}\left(2^{8(\rho+1)}e^{-\lambda\ca_\cD \min\{ \kappa, 1 /2\} \lfrf{t}}\E[|\theta_0|^{16(\rho+1)}] \right.\\
&\qquad \left.+ 2^{8(\rho+1)}\left( \cc_{8(\rho+1)}\left(1+1/(\ca_\cD\kappa)\right)+1\right)+6\mathrm{v}_{16(\rho+1)}(\cM_V(16(\rho+1)))\right)^{1/2}\\
&\qquad \times \lambda\left(e^{-  \lambda\ca_\cD\min\{ \kappa, 1 /2\} \lfrf{t}}\overline{\cC}_{\mathbf{A}1,2}\E[|\theta_0|^{4(\rho+1)}]+\widetilde{\cC}_{\mathbf{A}1,2} \right)^{1/2}\\
&\leq  \lambda^2\beta^{-1}d\cK_0^23^{16(\rho+1)}\left(e^{-\lambda\ca_\cD \min\{ \kappa, 1 /2\} \lfrf{t}}(1+\overline{\cC}_{\mathbf{A}1,2})\E[|\theta_0|^{16(\rho+1)}] \right.\\
&\qquad \left.+  \cc_{8(\rho+1)}\left(1+1/(\ca_\cD\kappa)\right)+1 +2\widetilde{\cC}_{\mathbf{A}1,2}+\mathrm{v}_{16(\rho+1)}(\cM_V(16(\rho+1)))\right)\\
&\leq \lambda^2\left(e^{- \ca_\cD \min\{ \kappa, 1 /2\} n/2} \overline{\cC}_{\mathbf{A}2}\E[|\theta_0|^{16(\rho+1)}] +\widetilde{\cC}_{\mathbf{A}2} \right),
\end{align*}
where the third inequality holds due to Lemma \ref{lem:zetaprocme} and \ref{lem:oserroralg}, the fourth inequality holds due to  $\overline{\cC}_{\mathbf{A}1,2}<\widetilde{\cC}_{\mathbf{A}1,2}$, and where $\overline{\cC}_{\mathbf{A}2}$ and $\widetilde{\cC}_{\mathbf{A}2}$ are given in \eqref{eq:graditoestconst}.
\item To establish the seventh inequality, we apply Remark \ref{rmk:growthc}, \eqref{eq:graditoestint1}, and  Cauchy-Schwarz inequality to obtain
\begin{align*} 
&\E\left[\left| \lambda\beta^{-1} \int_{\lfrf{t}}^t  \left( \Upsilon(\overline{\theta}^{\lambda}_s) - \Upsilon_{\lambda}( \overline{\theta}^{\lambda}_{\lfrf{s}})  \right) \,\rmd s\right|^2\right] \\
&\leq \lambda^2\beta^{-2} \int_{\lfrf{t}}^t \E\left[ |   \Upsilon(\overline{\theta}^{\lambda}_s) - \Upsilon_{\lambda}( \overline{\theta}^{\lambda}_{\lfrf{s}})  |^2\right]    \,\rmd s\\
&\leq 2\lambda^2\beta^{-2} \int_{\lfrf{t}}^t \E\left[ |   \Upsilon(\overline{\theta}^{\lambda}_s) - \Upsilon ( \overline{\theta}^{\lambda}_{\lfrf{s}})  |^2\right]    \,\rmd s +2\lambda^2\beta^{-2} \int_{\lfrf{t}}^t \E\left[ |  \Upsilon ( \overline{\theta}^{\lambda}_{\lfrf{s}}) - \Upsilon_{\lambda}( \overline{\theta}^{\lambda}_{\lfrf{s}})  |^2\right]    \,\rmd s \\
&\leq 2\lambda^2\beta^{-2} \int_{\lfrf{t}}^t \E\left[ d^3L^2(1+|\overline{\theta}^{\lambda}_s|+|\overline{\theta}^{\lambda}_{\lfrf{s}}|)^{2(\rho-2)}|\overline{\theta}^{\lambda}_s- \overline{\theta}^{\lambda}_{\lfrf{s}}|^{2q} \right]    \,\rmd s \\
&\quad +2\lambda^2\beta^{-2} \int_{\lfrf{t}}^t \E\left[  \lambda^3|\overline{\theta}^{\lambda}_{\lfrf{s}}|^{6(\rho+q-1)} \cK_{3,d}^2(1+|\overline{\theta}^{\lambda}_{\lfrf{s}}|)^{2(\rho+q-2)} \right]    \,\rmd s \\
&\leq  2\lambda^2\beta^{-2}d^3L^2 \int_{\lfrf{t}}^t 3^{8(\rho+1)}\left(\E\left[ 1+|\overline{\theta}^{\lambda}_s|^{16(\rho+1)}+|\overline{\theta}^{\lambda}_{\lfrf{s}}|^{16(\rho+1)} \right]\right)^{1/2}  \\
&\qquad \times\left(\E\left[ |\overline{\theta}^{\lambda}_s- \overline{\theta}^{\lambda}_{\lfrf{s}}|^{4q} \right]\right)^{1/2}   \,\rmd s \\
&\quad +2\lambda^5\beta^{-2}\cK_{3,d}^2 \int_{\lfrf{t}}^t2^{16(\rho+1)} \E\left[  1+|\overline{\theta}^{\lambda}_{\lfrf{s}}|^{16(\rho+1)} \right]    \,\rmd s \\
&\leq \lambda^2\beta^{-2}d^3L^2 3^{8(\rho+1)+1} \left(2e^{-\lambda \ca_\cD \kappa \lfrf{t}}\E\left[|\theta_0|^{16(\rho+1)}\right]  +2 \cc_{8(\rho+1)}\left(1+1/(\ca_\cD\kappa)\right)+1\right)^{1/2} \\
&\qquad \times  \lambda^q \left(e^{-  \lambda\ca_\cD\kappa \lfrf{t}}\overline{\cC}_{\mathbf{A}0,2q}\E[|\theta_0|^{4\lcrc{2q}(\rho+1)}]+\widetilde{\cC}_{\mathbf{A}0,2q} \right)^{1/2} \\
&\quad + \lambda^5\beta^{-2}\cK_{3,d}^22^{16(\rho+1)+1} \left( e^{-\lambda \ca_\cD \kappa \lfrf{t}}\E\left[|\theta_0|^{16(\rho+1)}\right]  +  \cc_{8(\rho+1)}\left(1+1/(\ca_\cD\kappa)\right)+1\right)\\
&\leq \lambda^{2+q}\beta^{-2}d^3L^2 3^{8(\rho+1)+1} \\
&\qquad \times\left( e^{-\lambda \ca_\cD \kappa \lfrf{t}}(2+\overline{\cC}_{\mathbf{A}0,2q})\E\left[|\theta_0|^{16(\rho+1)}\right]  +2 \cc_{8(\rho+1)}\left(1+1/(\ca_\cD\kappa)\right)+2\widetilde{\cC}_{\mathbf{A}0,2q}+1\right)  \\
&\quad + \lambda^{2+q}\beta^{-2}\cK_{3,d}^22^{16(\rho+1)+1} \left( e^{-\lambda \ca_\cD \kappa \lfrf{t}}\E\left[|\theta_0|^{16(\rho+1)}\right]  +  \cc_{8(\rho+1)}\left(1+1/(\ca_\cD\kappa)\right)+1\right)\\
&\leq \lambda^{2+q}\left(e^{- \ca_\cD  \kappa n/2} \overline{\cC}_{\mathbf{A}2}\E[|\theta_0|^{16(\rho+1)}] +\widetilde{\cC}_{\mathbf{A}2} \right),
\end{align*}
where the fourth inequality holds due to Lemma \ref{lem:2ndpthmmt} and \ref{lem:oserroralg}, and $\overline{\cC}_{\mathbf{A}2}$, $\widetilde{\cC}_{\mathbf{A}2}$ are given in \eqref{eq:graditoestconst}.
\item\label{item:graditoesteighth} To obtain the eighth inequality, we apply Remark \ref{rmk:growthc} and write
\begin{align*} 
&\E\left[\left| -\lambda  \int_{\lfrf{t}}^t H(\overline{\theta}^{\lambda}_s) h_{\lambda}(\overline{\theta}^{\lambda}_{\lfrf{s}})\,\rmd s\right|^2+\left| \lambda\beta^{-1}\int_{\lfrf{t}}^t\Upsilon (\overline{\theta}^{\lambda}_s)\,\rmd s\right|^2\right]  \\
&\leq \lambda^2  \int_{\lfrf{t}}^t\E\left[|H(\overline{\theta}^{\lambda}_s)|^2| h(\overline{\theta}^{\lambda}_{\lfrf{s}})|^2\right]\, \rmd s + \lambda^2 \beta^{-2} \int_{\lfrf{t}}^t\E\left[|\Upsilon (\overline{\theta}^{\lambda}_s)|^2\right]\, \rmd s\\
&\leq \lambda^2  \int_{\lfrf{t}}^t\E\left[K_H^2(1+|\overline{\theta}^{\lambda}_s|^{\rho+q-1})^2K_h^2(1+|\overline{\theta}^{\lambda}_{\lfrf{s}}|^{\rho+q})^2\right]\, \rmd s \\
&\quad + \lambda^2 \beta^{-2} \int_{\lfrf{t}}^t\E\left[\cK_{3,d}^2(1+|\overline{\theta}^{\lambda}_s|)^{2(\rho+q-2)} \right]\, \rmd s\\
&\leq \lambda^2 (1+K_H)^2(1+K_h)^2 (1+\beta^{-1})^2(1+\cK_{3,d})^23^{16(\rho+1)}\\
&\qquad \times \int_{\lfrf{t}}^t\E\left[ 1+|\overline{\theta}^{\lambda}_s|^{16(\rho+1)} +|\overline{\theta}^{\lambda}_{\lfrf{s}}|^{16(\rho+1)} \right]\, \rmd s\\
&\leq \lambda^2 (1+K_H)^2(1+K_h)^2 (1+\beta^{-1})^2(1+\cK_{3,d})^23^{16(\rho+1)} \\
&\qquad \times \left( 2e^{-\lambda \ca_\cD \kappa \lfrf{t}}\E\left[|\theta_0|^{16(\rho+1)}\right]  +  2\cc_{8(\rho+1)}\left(1+1/(\ca_\cD\kappa)\right)+1\right)\\
&\leq \lambda^2\left(e^{- \ca_\cD  \kappa n/2} \overline{\cC}_{\mathbf{A}2}\E[|\theta_0|^{16(\rho+1)}] +\widetilde{\cC}_{\mathbf{A}2} \right),
\end{align*}
where the fourth inequality holds due to Lemma \ref{lem:2ndpthmmt}, and where $\overline{\cC}_{\mathbf{A}2}$, $\widetilde{\cC}_{\mathbf{A}2}$ are given in \eqref{eq:graditoestconst}.
\item To establish the last inequality, we follow the arguments in \ref{item:graditoesteighth}:
\begin{align*} 
&\E\left[\left|  -\lambda  \int_{\lfrf{t}}^t H(\overline{\zeta}^{\lambda, n}_s) h ( \overline{\zeta}^{\lambda, n}_s)\,\rmd s  \right|^2+\left| \lambda\beta^{-1}\int_{\lfrf{t}}^t\Upsilon (\overline{\zeta}^{\lambda, n}_s)\rmd s\right|^2\right]  \\
&\leq \lambda^2 (1+K_H)^2(1+K_h)^2 (1+\beta^{-1})^2(1+\cK_{3,d})^23^{16(\rho+1)}\\
&\qquad \times \int_{\lfrf{t}}^t\E\left[ 1+|\overline{\zeta}^{\lambda, n}_s|^{16(\rho+1)} +|\overline{\zeta}^{\lambda, n}_{\lfrf{s}}|^{16(\rho+1)} \right]\, \rmd s\\
&\leq \lambda^2 (1+K_H)^2(1+K_h)^2 (1+\beta^{-1})^2(1+\cK_{3,d})^23^{16(\rho+1)} \\
&\qquad \times \left(2^{8(\rho+1)}e^{-\lambda\ca_\cD \min\{ \kappa, 1 /2\} \lfrf{t}}\E[|\theta_0|^{16(\rho+1)}] \right.\\
&\qquad \left.+ 2^{8(\rho+1)}\left( \cc_{8(\rho+1)}\left(1+1/(\ca_\cD\kappa)\right)+1\right)+6\mathrm{v}_{16(\rho+1)}(\cM_V(16(\rho+1)))\right)\\
&\leq \lambda^2\left(e^{- \ca_\cD \min\{ \kappa, 1 /2\} n/2} \overline{\cC}_{\mathbf{A}2}\E[|\theta_0|^{16(\rho+1)}] +\widetilde{\cC}_{\mathbf{A}2} \right),
\end{align*}
where the second inequality holds due to Lemma \ref{lem:zetaprocme}, and where $\overline{\cC}_{\mathbf{A}2}$ and $\widetilde{\cC}_{\mathbf{A}2}$ are given in \eqref{eq:graditoestconst}.
\end{enumerate}
This completes the proof.
\end{proof}

\begin{corollary}\label{cor:graditoub} Let Assumptions \ref{asm:AI}, \ref{asm:ALL}, and \ref{asm:AC} hold. Then, for any $0<\lambda\leq \lambda_{\max}$, $n \in \N_0$, $, t \geq nT$, we obtain the following inequalities:
\begin{align*}
&\E\left[\left| h(\overline{\theta}^{\lambda}_t) - h (\overline{\theta}^{\lambda}_{\lfrf{t}}) + \lambda \int_{\lfrf{t}}^t\left(H_{\lambda}( \overline{\theta}^{\lambda}_{\lfrf{s}})h_{\lambda}( \overline{\theta}^{\lambda}_{\lfrf{s}})-\beta^{-1}\Upsilon_{\lambda}(\overline{\theta}^{\lambda}_{\lfrf{s}}) \right)\,\rmd s\right.\right.\\
&\qquad \left.\left. - \sqrt{2\lambda\beta^{-1}} \int_{{\lfrf{t}}}^t H_{\lambda}(\overline{\theta}_{\lfrf{s}}^{\lambda})\,\rmd B_s^{\lambda}  \right|^2\right]  
\leq \lambda^2\left(e^{- \ca_\cD  \kappa n/2}36 \overline{\cC}_{\mathbf{A}2}\E[|\theta_0|^{16(\rho+1)}] +36\widetilde{\cC}_{\mathbf{A}2} \right),\\
&\E\left[\left| h( \overline{\zeta}^{\lambda, n}_t) - h ( \overline{\zeta}^{\lambda, n}_{\lfrf{t}})   - \sqrt{2\lambda\beta^{-1}} \int_{{\lfrf{t}}}^t H (\overline{\zeta}^{\lambda, n}_{\lfrf{s}})\,\rmd B_s^{\lambda}  -\Bigg(h(\overline{\theta}^{\lambda}_t) - h (\overline{\theta}^{\lambda}_{\lfrf{t}})  \Bigg. \right.\right.\\
&\qquad \left.\left. \Bigg.- \sqrt{2\lambda\beta^{-1}} \int_{{\lfrf{t}}}^t H (\overline{\theta}_{\lfrf{s}}^{\lambda})\,\rmd B_s^{\lambda}  \Bigg)\right|^2\right]  \leq \lambda^2\left(e^{- \ca_\cD \min\{ \kappa, 1 /2\} n/2}72 \overline{\cC}_{\mathbf{A}2}\E[|\theta_0|^{16(\rho+1)}] +72\widetilde{\cC}_{\mathbf{A}2} \right),
\end{align*}
where $\overline{\cC}_{\mathbf{A}2}$ and $\widetilde{\cC}_{\mathbf{A}2}$ are given in \eqref{eq:graditoestconst}.
\end{corollary}

\begin{proof}
For any $t \geq nT$, by applying It\^o's formula to $h(\overline{\theta}^{\lambda}_t)$, we obtain, almost surely
\begin{align}
\begin{split}\label{eq:holaprocito}
h(\overline{\theta}^{\lambda}_t) - h(\overline{\theta}^{\lambda}_{\lfrf{t}})  
&=  -\lambda  \int_{\lfrf{t}}^t H(\overline{\theta}^{\lambda}_s) h_{\lambda}(\overline{\theta}^{\lambda}_{\lfrf{s}})\,\rmd s+\lambda^2\int_{\lfrf{t}}^t H(\overline{\theta}^{\lambda}_s)\int_{\lfrf{s}}^s H_{\lambda}( \overline{\theta}^{\lambda}_{\lfrf{r}})h_{\lambda}( \overline{\theta}^{\lambda}_{\lfrf{r}}) \,\rmd r\,\rmd s\\
&\quad -\lambda^2\beta^{-1}\int_{\lfrf{t}}^t H(\overline{\theta}^{\lambda}_s)\int_{\lfrf{s}}^s  \Upsilon_{\lambda}(\overline{\theta}^{\lambda}_{\lfrf{r}}) \,\rmd r\,\rmd s\\
&\quad-\lambda\sqrt{2\lambda\beta^{-1}} \int_{\lfrf{t}}^t  H(\overline{\theta}^{\lambda}_s) \int_{{\lfrf{s}}}^s H_{\lambda}(\overline{\theta}_{\lfrf{r}}^{\lambda})\,\rmd B_r^{\lambda}\,\rmd s  \\
&\quad +\sqrt{2\lambda\beta^{-1}}\int_{\lfrf{t}}^tH(\overline{\theta}^{\lambda}_s) \, \rmd B^{\lambda}_s+\lambda\beta^{-1}\int_{\lfrf{t}}^t\Upsilon (\overline{\theta}^{\lambda}_s)\rmd s.
\end{split}
\end{align}
Similarly, applying It\^o's formula to $h(\overline{\zeta}^{\lambda, n}_t)$ yields, almost surely
\begin{align}
\begin{split}\label{eq:auxprocito}
h( \overline{\zeta}^{\lambda, n}_t) - h ( \overline{\zeta}^{\lambda, n}_{\lfrf{t}})
&=  -\lambda  \int_{\lfrf{t}}^t H(\overline{\zeta}^{\lambda, n}_s) h ( \overline{\zeta}^{\lambda, n}_s)\,\rmd s   +\sqrt{2\lambda\beta^{-1}}\int_{\lfrf{t}}^tH(\overline{\zeta}^{\lambda, n}_s) \, \rmd B^{\lambda}_s\\
&\quad +\lambda\beta^{-1}\int_{\lfrf{t}}^t\Upsilon (\overline{\zeta}^{\lambda, n}_s)\rmd s.
\end{split}
\end{align}

\begin{enumerate}
\item To obtain the first inequality, by using \eqref{eq:holaprocito} with $(\sum_{l=1}^vu_l)^2 \leq v \sum_{l=1}^vu_l^2$, $v \in \N$, $u_l \geq 0$, we obtain that
\begin{align*}
&\E\left[\left| h(\overline{\theta}^{\lambda}_t) - h (\overline{\theta}^{\lambda}_{\lfrf{t}}) + \lambda \int_{\lfrf{t}}^t\left(H_{\lambda}( \overline{\theta}^{\lambda}_{\lfrf{s}})h_{\lambda}( \overline{\theta}^{\lambda}_{\lfrf{s}})-\beta^{-1}\Upsilon_{\lambda}(\overline{\theta}^{\lambda}_{\lfrf{s}}) \right)\,\rmd s\right.\right.\\
&\qquad \left.\left. - \sqrt{2\lambda\beta^{-1}} \int_{{\lfrf{t}}}^t H_{\lambda}(\overline{\theta}_{\lfrf{s}}^{\lambda})\,\rmd B_s^{\lambda}  \right|^2\right]\\
&\leq 6\E\left[\left|-\lambda \int_{\lfrf{t}}^t  \left(H(\overline{\theta}^{\lambda}_s) - H_{\lambda}(\overline{\theta}^{\lambda}_{\lfrf{s}})\right)h_{\lambda}( \overline{\theta}^{\lambda}_{\lfrf{s}})\,\rmd s\right|^2\right] \\
&\quad+  6\E\left[\left| \lambda^2 \int_{\lfrf{t}}^t   H(\overline{\theta}^{\lambda}_s) \int_{\lfrf{s}}^s  H_{\lambda}(\overline{\theta}^{\lambda}_{\lfrf{r}}))h_{\lambda}( \overline{\theta}^{\lambda}_{\lfrf{r}}) \, \rmd r \,\rmd s\right|^2\right] \\
&\quad+ 6\E\left[\left| -\lambda^2\beta^{-1} \int_{\lfrf{t}}^t   H(\overline{\theta}^{\lambda}_s) \int_{\lfrf{s}}^s  \Upsilon_{\lambda}( \overline{\theta}^{\lambda}_{\lfrf{r}}) \, \rmd r \,\rmd s\right|^2\right]  \\
&\quad+  6\E\left[\left| -\lambda\sqrt{2\lambda\beta^{-1}} \int_{\lfrf{t}}^t   H(\overline{\theta}^{\lambda}_s) \int_{\lfrf{s}}^s  H_{\lambda}( \overline{\theta}^{\lambda}_{\lfrf{r}}) \, \rmd B_r^{\lambda} \,\rmd s\right|^2\right] \\
&\quad+  6\E\left[\left| \sqrt{2\lambda\beta^{-1}} \int_{\lfrf{t}}^t  \left( H(\overline{\theta}^{\lambda}_s) -  H_{\lambda}( \overline{\theta}^{\lambda}_{\lfrf{s}}) \right)  \,\rmd B_s^{\lambda}\right|^2\right] \\
&\quad+  6\E\left[\left| \lambda\beta^{-1} \int_{\lfrf{t}}^t   \left(\Upsilon(\overline{\theta}^{\lambda}_s) - \Upsilon_{\lambda}( \overline{\theta}^{\lambda}_{\lfrf{s}})\right)  \,\rmd s\right|^2\right] \\
& \leq \lambda^2\left(e^{- \ca_\cD  \kappa n/2}36 \overline{\cC}_{\mathbf{A}2}\E[|\theta_0|^{16(\rho+1)}] +36\widetilde{\cC}_{\mathbf{A}2} \right),
\end{align*}
where the last inequality holds due to Lemma \ref{lem:graditoest}.
\item To establish the second inequality, we use \eqref{eq:holaprocito} and \eqref{eq:auxprocito} to obtain
\begin{align*}
&\E\left[\left| h( \overline{\zeta}^{\lambda, n}_t) - h ( \overline{\zeta}^{\lambda, n}_{\lfrf{t}})   - \sqrt{2\lambda\beta^{-1}} \int_{{\lfrf{t}}}^t H (\overline{\zeta}^{\lambda, n}_{\lfrf{s}})\,\rmd B_s^{\lambda}  \right.\right.\\
&\qquad \left.\left. -\left(h(\overline{\theta}^{\lambda}_t) - h (\overline{\theta}^{\lambda}_{\lfrf{t}})  - \sqrt{2\lambda\beta^{-1}} \int_{{\lfrf{t}}}^t H (\overline{\theta}_{\lfrf{s}}^{\lambda})\,\rmd B_s^{\lambda}  \right)\right|^2\right]\\
&\leq 2 \E\left[\left| h( \overline{\zeta}^{\lambda, n}_t) - h ( \overline{\zeta}^{\lambda, n}_{\lfrf{t}})   - \sqrt{2\lambda\beta^{-1}} \int_{{\lfrf{t}}}^t H (\overline{\zeta}^{\lambda, n}_{\lfrf{s}})\,\rmd B_s^{\lambda} \right|^2\right]\\
&\quad +2 \E\left[\left| h(\overline{\theta}^{\lambda}_t) - h (\overline{\theta}^{\lambda}_{\lfrf{t}})  - \sqrt{2\lambda\beta^{-1}} \int_{{\lfrf{t}}}^t H (\overline{\theta}_{\lfrf{s}}^{\lambda})\,\rmd B_s^{\lambda} \right|^2\right]\\
&\leq 6\E\left[\left| -\lambda  \int_{\lfrf{t}}^t H(\overline{\zeta}^{\lambda, n}_s) h ( \overline{\zeta}^{\lambda, n}_s)\,\rmd s   \right|^2\right]  +6\E\left[\left| \lambda\beta^{-1}\int_{\lfrf{t}}^t\Upsilon (\overline{\zeta}^{\lambda, n}_s)\rmd s  \right|^2\right] \\
&\quad +6\E\left[\left| \sqrt{2\lambda\beta^{-1}}\int_{\lfrf{t}}^t \left(H(\overline{\zeta}^{\lambda, n}_s) -  H(\overline{\zeta}^{\lambda, n}_{\lfrf{s}}) \right) \, \rmd B^{\lambda}_s
  \right|^2\right] \\ 
&\quad +12\E\left[\left| -\lambda  \int_{\lfrf{t}}^t H(\overline{\theta}^{\lambda}_s) h_{\lambda}(\overline{\theta}^{\lambda}_{\lfrf{s}})\,\rmd s\right|^2\right] + 12\E\left[\left|\lambda\beta^{-1}\int_{\lfrf{t}}^t\Upsilon (\overline{\theta}^{\lambda}_s)\rmd s\right|^2\right]\\
&\quad + 12\E\left[\left|\lambda^2\int_{\lfrf{t}}^t H(\overline{\theta}^{\lambda}_s)\int_{\lfrf{s}}^s H_{\lambda}( \overline{\theta}^{\lambda}_{\lfrf{r}})h_{\lambda}( \overline{\theta}^{\lambda}_{\lfrf{r}}) \,\rmd r\,\rmd s \right|^2\right] \\
&\quad + 12\E\left[\left|-\lambda^2\beta^{-1}\int_{\lfrf{t}}^t H(\overline{\theta}^{\lambda}_s)\int_{\lfrf{s}}^s  \Upsilon_{\lambda}(\overline{\theta}^{\lambda}_{\lfrf{r}}) \,\rmd r\,\rmd s\right|^2\right] \\
&\quad + 12\E\left[\left|-\lambda\sqrt{2\lambda\beta^{-1}} \int_{\lfrf{t}}^t  H(\overline{\theta}^{\lambda}_s) \int_{{\lfrf{s}}}^s H_{\lambda}(\overline{\theta}_{\lfrf{r}}^{\lambda})\,\rmd B_r^{\lambda}\,\rmd s \right|^2\right] \\
&\quad + 12\E\left[\left|\sqrt{2\lambda\beta^{-1}}\int_{\lfrf{t}}^t \left(H(\overline{\theta}^{\lambda}_s) -  H(\overline{\theta}_{\lfrf{s}}^{\lambda}) \right) \, \rmd B^{\lambda}_s \right|^2\right]   \\
&\leq  12\lambda^2\left(e^{- \ca_\cD \min\{ \kappa, 1 /2\} n/2} \overline{\cC}_{\mathbf{A}2}\E[|\theta_0|^{16(\rho+1)}] +\widetilde{\cC}_{\mathbf{A}2} \right)\\
&\quad + 60 \lambda^2\left(e^{- \ca_\cD \min\{ \kappa, 1 /2\} n/2} \overline{\cC}_{\mathbf{A}2}\E[|\theta_0|^{16(\rho+1)}] +\widetilde{\cC}_{\mathbf{A}2} \right)\\
&\leq \lambda^2\left(e^{- \ca_\cD \min\{ \kappa, 1 /2\} n/2}72 \overline{\cC}_{\mathbf{A}2}\E[|\theta_0|^{16(\rho+1)}] +72\widetilde{\cC}_{\mathbf{A}2} \right),
\end{align*}
where the second last inequality holds by using Lemma \ref{lem:graditoest} with the fact that
\begin{align*}
&\E\left[\left| \sqrt{2\lambda\beta^{-1}} \int_{\lfrf{t}}^t   \left(H(\overline{\theta}^{\lambda}_s) -  H( \overline{\theta}^{\lambda}_{\lfrf{s}})\right)   \,\rmd B_s^{\lambda}\right|^2\right] \\
& =  2\lambda\beta^{-1}  \int_{\lfrf{t}}^t \E\left[ |   H(\overline{\theta}^{\lambda}_s) -  H( \overline{\theta}^{\lambda}_{\lfrf{s}})  |_{\cF}^2\right]    \,\rmd s\\
&\leq \lambda^2\left(e^{- \ca_\cD \min\{ \kappa, 1 /2\} n/2} \overline{\cC}_{\mathbf{A}2}\E[|\theta_0|^{16(\rho+1)}] +\widetilde{\cC}_{\mathbf{A}2} \right)
\end{align*}
as indicated in the calculations in \ref{item:graditoestfifth} of Lemma \ref{lem:graditoest}.
\end{enumerate}
This completes the proof.
\end{proof}

\begin{definition}\label{def:Mdef} Define $\mathfrak{M}=(\mathfrak{M}^{(i,j)})_{i,j=1,\dots,d}:\R^d \times \R^d \rightarrow \R^{d\times d}$ by setting, for every $i,j = 1, \dots, d$,
\[
\mathfrak{M}^{(i,j)}(\theta, \overline{\theta}) = \langle\nabla H^{(i,j)}(\overline{\theta}), \theta-\overline{\theta} \rangle, \qquad \theta, \overline{\theta} \in \R^d.
\]
\end{definition}
\begin{lemma}\label{lem:Mest} Let Assumptions \ref{asm:AI}, \ref{asm:ALL}, and \ref{asm:AC} hold. Then, for any $0<\lambda\leq \lambda_{\max}$, $n \in \N_0$, $, t \geq nT$, we obtain the following inequalities:
\begin{align} 
\begin{split}
&\E\left[\left| \sqrt{2\lambda\beta^{-1}} \int_{\lfrf{t}}^t   \left(H(\overline{\theta}^{\lambda}_s) -  H ( \overline{\theta}^{\lambda}_{\lfrf{s}})  -\mathfrak{M}(\overline{\theta}^{\lambda}_s, \overline{\theta}^{\lambda}_{\lfrf{s}})\right)  \,\rmd B_s^{\lambda}\right|^2\right] \\
& \leq  \lambda^{2+q}\left(e^{- \ca_\cD  \kappa n/2} \overline{\cC}_{\mathbf{A}2}\E[|\theta_0|^{16(\rho+1)}] +\widetilde{\cC}_{\mathbf{A}2} \right),\label{lem:Mestineq1}
\end{split}\\
&\E\left[\left| \sqrt{2\lambda\beta^{-1}} \int_{\lfrf{t}}^t   \ \mathfrak{M}(\overline{\theta}^{\lambda}_s, \overline{\theta}^{\lambda}_{\lfrf{s}})   \,\rmd B_s^{\lambda}\right|^2\right] 
\leq  \lambda^2\left(e^{- \ca_\cD  \kappa n/2} \overline{\cC}_{\mathbf{A}2}\E[|\theta_0|^{16(\rho+1)}] +\widetilde{\cC}_{\mathbf{A}2} \right),\label{lem:Mestineq2}\\
\begin{split}
&\E\left[2\lambda\beta^{-1} \left\langle  \int_{{\lfrf{t}}}^t  \int_{{\lfrf{s}}}^s \left(H (\overline{\zeta}^{\lambda, n}_{\lfrf{r}}) - H (\overline{\theta}_{\lfrf{r}}^{\lambda})\right) \,\rmd B_r^{\lambda} \,\rmd s ,  \int_{\lfrf{t}}^t   \ \mathfrak{M}(\overline{\theta}^{\lambda}_s, \overline{\theta}^{\lambda}_{\lfrf{s}})   \,\rmd B_s^{\lambda}\right\rangle\right]\\
&\leq \lambda^2\left(e^{- \ca_\cD \min\{ \kappa, 1 /2\} n/2}10 \overline{\cC}_{\mathbf{A}2}\E[|\theta_0|^{16(\rho+1)}] +10\widetilde{\cC}_{\mathbf{A}2} \right),\label{lem:Mestineq3}
\end{split}
\end{align}
where $\overline{\cC}_{\mathbf{A}2}$ and $\widetilde{\cC}_{\mathbf{A}2}$ are given in \eqref{eq:graditoestconst}.
\end{lemma}
\begin{proof}To show the inequalities hold, we follow the arguments below:
\begin{enumerate}
\item To establish the first inequality \eqref{lem:Mestineq1}, by using Definition \ref{def:Mdef}, we observe that, for fixed $\theta, \overline{\theta} \in \R^d$, 
\begin{align}\label{eq:multidmvt}
&|H(\theta) - H(\overline{\theta}) - \mathfrak{M} (\theta, \overline{\theta}) |_{\cF}^2 \nonumber\\
&=\sum_{i,j= 1}^d|H^{(i,j)}(\theta) - H^{(i,j)}(\overline{\theta}) - \mathfrak{M}^{(i,j)}(\theta, \overline{\theta}) |^2 \nonumber\\
& = \sum_{i,j= 1}^d\left|\int_0^1 \langle \nabla H^{(i,j)}(\nu\theta+(1-\nu)\overline{\theta}), \theta-\overline{\theta} \rangle\, \rmd \nu - \langle\nabla H^{(i,j)}(\overline{\theta}), \theta-\overline{\theta} \rangle\right|^2 \nonumber\\
&\leq \int_0^1\sum_{i,j= 1}^d |\nabla H^{(i,j)}(\nu\theta+(1-\nu)\overline{\theta}) - \nabla H^{(i,j)}(\overline{\theta})|^2\, \rmd \nu |\theta-\overline{\theta} |^2 \nonumber\\
&\leq \int_0^1\sum_{i= 1}^d d |\nabla^2 h^{(i)}(\nu\theta+(1-\nu)\overline{\theta}) - \nabla^2 h^{(i)}(\overline{\theta})| ^2\, \rmd \nu |\theta-\overline{\theta} |^2 \nonumber\\
&\leq 2^{2(\rho-2)} d^2L^2(1+|\theta|+|\overline{\theta}|)^{2(\rho-2)}|\theta-\overline{\theta}|^{2+2q}.
\end{align}
By using \eqref{eq:multidmvt}, we obtain that
\begin{align*}
&\E\left[\left| \sqrt{2\lambda\beta^{-1}} \int_{\lfrf{t}}^t   \left(H(\overline{\theta}^{\lambda}_s) -  H ( \overline{\theta}^{\lambda}_{\lfrf{s}})  -\mathfrak{M}(\overline{\theta}^{\lambda}_s, \overline{\theta}^{\lambda}_{\lfrf{s}})\right)  \,\rmd B_s^{\lambda}\right|^2\right] \\
& =2\lambda\beta^{-1}\int_{\lfrf{t}}^t  \E\left[\left| H(\overline{\theta}^{\lambda}_s) -  H ( \overline{\theta}^{\lambda}_{\lfrf{s}})  -\mathfrak{M}(\overline{\theta}^{\lambda}_s, \overline{\theta}^{\lambda}_{\lfrf{s}}) \right|_{\cF}^2\right]  \,\rmd s\\
&\leq  \lambda\beta^{-1} 2^{2\rho-3} d^2L^2 \int_{\lfrf{t}}^t  \E\left[(1+|\overline{\theta}^{\lambda}_s|+|\overline{\theta}^{\lambda}_{\lfrf{s}}|)^{2(\rho-2)}|\overline{\theta}^{\lambda}_s-\overline{\theta}^{\lambda}_{\lfrf{s}}|^{2+2q}\right]  \,\rmd s\\
&\leq  \lambda\beta^{-1} 2^{2\rho-3} d^2L^2 \int_{\lfrf{t}}^t 3^{8(\rho+1)}  \left(\E\left[1+|\overline{\theta}^{\lambda}_s|^{16(\rho+1)}+|\overline{\theta}^{\lambda}_{\lfrf{s}}|^{16(\rho+1)} \right]\right)^{1/2}  \\
&\qquad \times \left(\E\left[ |\overline{\theta}^{\lambda}_s-\overline{\theta}^{\lambda}_{\lfrf{s}}|^{4+4q}\right]\right)^{1/2} \,\rmd s\\
&\leq  \lambda\beta^{-1} 3^{16(\rho+1)}  d^2L^2 \left(2e^{-\lambda \ca_\cD \kappa \lfrf{t}}\E\left[|\theta_0|^{16(\rho+1)}\right]  +2 \cc_{8(\rho+1)}\left(1+1/(\ca_\cD\kappa)\right)+1\right)^{1/2}\\
&\qquad \times \lambda^{1+q} \left(e^{-  \lambda\ca_\cD\kappa \lfrf{t}}\overline{\cC}_{\mathbf{A}0,2+2q}\E[|\theta_0|^{4\lcrc{2+2q}(\rho+1)}]+\widetilde{\cC}_{\mathbf{A}0,2+2q} \right)^{1/2} \\
&\leq \lambda^{2+q} \left(e^{- \ca_\cD  \kappa n/2} \overline{\cC}_{\mathbf{A}2}\E[|\theta_0|^{16(\rho+1)}] +\widetilde{\cC}_{\mathbf{A}2} \right),
\end{align*}
where the third inequality holds due to Lemma \ref{lem:2ndpthmmt} and \ref{lem:oserroralg}, and $\overline{\cC}_{\mathbf{A}2}$, $\widetilde{\cC}_{\mathbf{A}2}$ are given in \eqref{eq:graditoestconst}.
\item To establish the second inequality \eqref{lem:Mestineq2}, we recall Definition \ref{def:Mdef} and use Remark \ref{rmk:growthc} to obtain
\begin{align*}
&\E\left[\left| \sqrt{2\lambda\beta^{-1}} \int_{\lfrf{t}}^t    \mathfrak{M}(\overline{\theta}^{\lambda}_s, \overline{\theta}^{\lambda}_{\lfrf{s}})   \,\rmd B_s^{\lambda}\right|^2\right] \\
& = 2\lambda\beta^{-1}\int_{\lfrf{t}}^t  \E\left[\sum_{i,j=1}^d\left|   \left\langle\nabla H^{(i,j)}( \overline{\theta}^{\lambda}_{\lfrf{s}}) , \overline{\theta}^{\lambda}_s- \overline{\theta}^{\lambda}_{\lfrf{s}} \right\rangle\right|^2\right]  \,\rmd s\\
&\leq 2\lambda\beta^{-1}\int_{\lfrf{t}}^t  \E\left[d^2\cK_0^2(1+|\overline{\theta}^{\lambda}_{\lfrf{s}}|)^{2(\rho+q-2)} |\overline{\theta}^{\lambda}_s- \overline{\theta}^{\lambda}_{\lfrf{s}} |^2\right]  \,\rmd s\\
&\leq  \lambda\beta^{-1}d^2\cK_0^2\int_{\lfrf{t}}^t 2^{8(\rho+1)+1} \left(\E\left[1+|\overline{\theta}^{\lambda}_{\lfrf{s}}|^{16(\rho+1)}  \right] \right)^{1/2} \left(\E\left[ |\overline{\theta}^{\lambda}_s- \overline{\theta}^{\lambda}_{\lfrf{s}} |^4\right] \right)^{1/2}\,\rmd s\\
&\leq  \lambda\beta^{-1}d^2\cK_0^2  2^{8(\rho+1)+1}  \left(e^{-\lambda \ca_\cD \kappa \lfrf{t}}\E\left[|\theta_0|^{16(\rho+1)}\right]  +\cc_{8(\rho+1)}\left(1+1/(\ca_\cD\kappa)\right)+1\right)^{1/2}\\
&\qquad \times   \lambda \left(e^{-  \lambda\ca_\cD\kappa \lfrf{t}}\overline{\cC}_{\mathbf{A}0,2}\E[|\theta_0|^{8(\rho+1)}]+\widetilde{\cC}_{\mathbf{A}0,2} \right)^{1/2}  \\
&\leq \lambda^2 \left(e^{- \ca_\cD  \kappa n/2} \overline{\cC}_{\mathbf{A}2}\E[|\theta_0|^{16(\rho+1)}] +\widetilde{\cC}_{\mathbf{A}2} \right),
\end{align*}
where the third inequality holds due to Lemma \ref{lem:2ndpthmmt} and \ref{lem:oserroralg}, and $\overline{\cC}_{\mathbf{A}2}$, $\widetilde{\cC}_{\mathbf{A}2}$ are given in \eqref{eq:graditoestconst}.
\item To obtain the third inequality \eqref{lem:Mestineq3}, we use Definition \ref{def:Mdef} and write the following
\begin{align}
&\E\left[2\lambda\beta^{-1} \left\langle  \int_{{\lfrf{t}}}^t  \int_{{\lfrf{s}}}^s \left(H (\overline{\zeta}^{\lambda, n}_{\lfrf{r}}) - H (\overline{\theta}_{\lfrf{r}}^{\lambda})\right) \,\rmd B_r^{\lambda} \,\rmd s ,  \int_{\lfrf{t}}^t    \mathfrak{M}(\overline{\theta}^{\lambda}_s, \overline{\theta}^{\lambda}_{\lfrf{s}})   \,\rmd B_s^{\lambda}\right\rangle\right]\nonumber\\
\begin{split}\label{eq:Mestthdineqexpsn} 
& = 2\lambda\beta^{-1} \E\left[\left\langle  \int_{{\lfrf{t}}}^t  \int_{{\lfrf{s}}}^s \left(H (\overline{\zeta}^{\lambda, n}_{\lfrf{r}}) - H (\overline{\theta}_{\lfrf{r}}^{\lambda})\right) \,\rmd B_r^{\lambda} \,\rmd s ,  \right.\right.\\
&\qquad \left.\left. \int_{\lfrf{t}}^t   \left\langle\nabla H(\overline{\theta}^{\lambda}_{\lfrf{s}}),    
\int_{\lfrf{s}}^s - \lambda h_{\lambda}(\overline{\theta}^{\lambda}_{\lfrf{r}})   \,\rmd r
\right\rangle \,\rmd B_s^{\lambda}\right\rangle\right] \\
& \quad+ 2\lambda\beta^{-1} \E\left[\left\langle  \int_{{\lfrf{t}}}^t  \int_{{\lfrf{s}}}^s \left(H (\overline{\zeta}^{\lambda, n}_{\lfrf{r}}) - H (\overline{\theta}_{\lfrf{r}}^{\lambda})\right) \,\rmd B_r^{\lambda} \,\rmd s ,   \right.\right.\\
&\qquad \left.\left.  \int_{\lfrf{t}}^t   \left\langle    \nabla H(\overline{\theta}^{\lambda}_{\lfrf{s}}),    
\int_{\lfrf{s}}^s     \int_{\lfrf{r}}^r \lambda^2 H_{\lambda}( \overline{\theta}^{\lambda}_{\lfrf{\nu}})h_{\lambda}( \overline{\theta}^{\lambda}_{\lfrf{\nu}}) \,\rmd \nu   \,\rmd r
\right\rangle \,\rmd B_s^{\lambda}\right\rangle\right] \\
& \quad+ 2\lambda\beta^{-1} \E\left[\left\langle  \int_{{\lfrf{t}}}^t  \int_{{\lfrf{s}}}^s \left(H (\overline{\zeta}^{\lambda, n}_{\lfrf{r}}) - H (\overline{\theta}_{\lfrf{r}}^{\lambda})\right) \,\rmd B_r^{\lambda} \,\rmd s ,   \right.\right.\\
&\qquad \left.\left.  \int_{\lfrf{t}}^t   \left\langle    \nabla H(\overline{\theta}^{\lambda}_{\lfrf{s}}),    
\int_{\lfrf{s}}^s    \int_{\lfrf{r}}^r -\lambda^2 \beta^{-1}\Upsilon_{\lambda}(\overline{\theta}^{\lambda}_{\lfrf{\nu}}) \,\rmd \nu   \,\rmd r
\right\rangle \,\rmd B_s^{\lambda}\right\rangle\right] \\
& \quad+ 2\lambda\beta^{-1} \E\left[\left\langle  \int_{{\lfrf{t}}}^t  \int_{{\lfrf{s}}}^s \left(H (\overline{\zeta}^{\lambda, n}_{\lfrf{r}}) - H (\overline{\theta}_{\lfrf{r}}^{\lambda})\right) \,\rmd B_r^{\lambda} \,\rmd s ,   \right.\right.\\
&\qquad \left.\left. \int_{\lfrf{t}}^t   \left\langle   \nabla H(\overline{\theta}^{\lambda}_{\lfrf{s}}),    
\int_{\lfrf{s}}^s -\lambda\sqrt{2\lambda\beta^{-1}} \int_{{\lfrf{r}}}^r H_{\lambda}(\overline{\theta}_{\lfrf{\nu}}^{\lambda})\,\rmd B_{\nu}^{\lambda}    \,\rmd r
\right\rangle \,\rmd B_s^{\lambda}\right\rangle\right] \\
& \quad+ 2\lambda\beta^{-1} \E\left[\left\langle  \int_{{\lfrf{t}}}^t  \int_{{\lfrf{s}}}^s \left(H (\overline{\zeta}^{\lambda, n}_{\lfrf{r}}) - H (\overline{\theta}_{\lfrf{r}}^{\lambda})\right) \,\rmd B_r^{\lambda} \,\rmd s ,   \right.\right.\\
&\qquad \left.\left. \int_{\lfrf{t}}^t   \left\langle\nabla H(\overline{\theta}^{\lambda}_{\lfrf{s}}),    
\sqrt{2\lambda\beta^{-1}}\int_{\lfrf{s}}^s \, \rmd B^{\lambda}_r
\right\rangle \,\rmd B_s^{\lambda}\right\rangle\right]  .
\end{split}
\end{align}
Recall that $(\mathcal{F}^{\lambda}_t)_{t \geq 0}$ is the completed natural filtration of $(B^{\lambda}_t)_{t\geq0}$. Then, we note that, for any $i,j,k=1, \dots, d$, it holds that
\begin{align*}
&\E\left[ \left( \int_{{\lfrf{t}}}^t   \int_{{\lfrf{s}}}^s \left(H^{(i,j)} (\overline{\zeta}^{\lambda, n}_{\lfrf{r}}) - H^{(i,j)} (\overline{\theta}_{\lfrf{r}}^{\lambda})\right) \,\rmd (B_r^{\lambda})^{(j)} \,\rmd s   \right)\right.\\
&\qquad \times \left. \left(  \int_{\lfrf{t}}^t   \left(  \partial_{\theta^{(k)}} H^{(i,j)}(\overline{\theta}^{\lambda}_{\lfrf{s}})   
\int_{\lfrf{s}}^s \, \rmd (B^{\lambda}_r)^{(k)}\right)
\,\rmd (B_s^{\lambda})^{(j)}\right) \right]\\
& = \E\Bigg[ \left(H^{(i,j)} (\overline{\zeta}^{\lambda, n}_{\lfrf{t}}) - H^{(i,j)} (\overline{\theta}_{\lfrf{t}}^{\lambda})\right)\partial_{\theta^{(k)}} H^{(i,j)}(\overline{\theta}^{\lambda}_{\lfrf{t}})    \Bigg.\\
&\qquad \times \left. \E\left[  \left.\int_{{\lfrf{t}}}^t   \int_{{\lfrf{s}}}^s \,\rmd (B_r^{\lambda})^{(j)} \,\rmd s     \int_{\lfrf{t}}^t    
\int_{\lfrf{s}}^s \, \rmd (B^{\lambda}_r)^{(k)}  \,\rmd (B_s^{\lambda})^{(j)} 
\right|\mathcal{F}^{\lambda}_{\lfrf{t}}\right]\right]\\
& = \E\Bigg[ \left(H^{(i,j)} (\overline{\zeta}^{\lambda, n}_{\lfrf{t}}) - H^{(i,j)} (\overline{\theta}_{\lfrf{t}}^{\lambda})\right)\partial_{\theta^{(k)}} H^{(i,j)}(\overline{\theta}^{\lambda}_{\lfrf{t}})    \Bigg.\\
&\qquad \times \left. \E\left[  \left. \int_{{\lfrf{t}}}^t (t-s) \,\rmd (B_s^{\lambda})^{(j)}  \int_{\lfrf{t}}^t    
\int_{\lfrf{s}}^s \, \rmd (B^{\lambda}_r)^{(k)}  \,\rmd (B_s^{\lambda})^{(j)} 
\right|\mathcal{F}^{\lambda}_{\lfrf{t}}\right]\right]\\
& = \E\Bigg[ \left(H^{(i,j)} (\overline{\zeta}^{\lambda, n}_{\lfrf{t}}) - H^{(i,j)} (\overline{\theta}_{\lfrf{t}}^{\lambda})\right)\partial_{\theta^{(k)}} H^{(i,j)}(\overline{\theta}^{\lambda}_{\lfrf{t}})    \Bigg.\\
&\qquad \times \left.\left( t\int_{{\lfrf{t}}}^t   \E\left[  \left.  \int_{\lfrf{t}}^s \, \rmd (B^{\lambda}_r)^{(k)}  
\right|\mathcal{F}^{\lambda}_{\lfrf{t}}\right]\,\rmd s  -  \int_{{\lfrf{t}}}^t s  \E\left[  \left.  \int_{\lfrf{t}}^s \, \rmd (B^{\lambda}_r)^{(k)}  
\right|\mathcal{F}^{\lambda}_{\lfrf{t}}\right]\,\rmd s  \right)\right]\\
& = 0.
\end{align*}
This implies that the last term in \eqref{eq:Mestthdineqexpsn} is zero. Indeed, we have that
\begin{align*}&2\lambda\beta^{-1} \E\left[\left\langle  \int_{{\lfrf{t}}}^t  \int_{{\lfrf{s}}}^s \left(H (\overline{\zeta}^{\lambda, n}_{\lfrf{r}}) - H (\overline{\theta}_{\lfrf{r}}^{\lambda})\right) \,\rmd B_r^{\lambda} \,\rmd s ,   \right.\right.\\
&\qquad \left.\left. \int_{\lfrf{t}}^t   \left\langle\nabla H(\overline{\theta}^{\lambda}_{\lfrf{s}}),    
\sqrt{2\lambda\beta^{-1}}\int_{\lfrf{s}}^s \, \rmd B^{\lambda}_r
\right\rangle \,\rmd B_s^{\lambda}\right\rangle\right]\\
& = (2\lambda\beta^{-1})^{3/2} \E\left[\sum_{i=1}^d\left( \int_{{\lfrf{t}}}^t \sum_{j=1}^d  \int_{{\lfrf{s}}}^s \left(H^{(i,j)} (\overline{\zeta}^{\lambda, n}_{\lfrf{r}}) - H^{(i,j)} (\overline{\theta}_{\lfrf{r}}^{\lambda})\right) \,\rmd (B_r^{\lambda})^{(j)} \,\rmd s   \right)\right.\\
&\qquad \times \left. \left( \sum_{j = 1}^d \int_{\lfrf{t}}^t   \left(\sum_{k=1}^d \partial_{\theta^{(k)}} H^{(i,j)}(\overline{\theta}^{\lambda}_{\lfrf{s}})   
\int_{\lfrf{s}}^s \, \rmd (B^{\lambda}_r)^{(k)}\right)
 \,\rmd (B_s^{\lambda})^{(j)}\right) \right]=0.
\end{align*}
Then, by using Remark \ref{rmk:growthc} and \eqref{eq:Mestthdineqexpsn} with the result above, we obtain that
\begin{align*}
&\E\left[2\lambda\beta^{-1} \left\langle  \int_{{\lfrf{t}}}^t  \int_{{\lfrf{s}}}^s \left(H (\overline{\zeta}^{\lambda, n}_{\lfrf{r}}) - H (\overline{\theta}_{\lfrf{r}}^{\lambda})\right) \,\rmd B_r^{\lambda} \,\rmd s ,  \int_{\lfrf{t}}^t    \mathfrak{M}(\overline{\theta}^{\lambda}_s, \overline{\theta}^{\lambda}_{\lfrf{s}})   \,\rmd B_s^{\lambda}\right\rangle\right] \\ 
&\leq 4\lambda^2\beta^{-2} \E\left[\left| \int_{{\lfrf{t}}}^t  \int_{{\lfrf{s}}}^s \left(H (\overline{\zeta}^{\lambda, n}_{\lfrf{r}}) - H (\overline{\theta}_{\lfrf{r}}^{\lambda})\right) \,\rmd B_r^{\lambda} \,\rmd s  \right|^2\right]\\
&\quad + \lambda^2\E\left[\left|\int_{\lfrf{t}}^t   \left\langle\nabla H(\overline{\theta}^{\lambda}_{\lfrf{s}}),    
\int_{\lfrf{s}}^s -   h_{\lambda}(\overline{\theta}^{\lambda}_{\lfrf{r}})   \,\rmd r      \right\rangle \,\rmd B_s^{\lambda}  \right|^2\right]\\ 
&\quad + \lambda^4\E\left[\left|\int_{\lfrf{t}}^t   \left\langle    \nabla H(\overline{\theta}^{\lambda}_{\lfrf{s}}),    
\int_{\lfrf{s}}^s     \int_{\lfrf{r}}^r  H_{\lambda}( \overline{\theta}^{\lambda}_{\lfrf{\nu}})h_{\lambda}( \overline{\theta}^{\lambda}_{\lfrf{\nu}}) \,\rmd \nu   \,\rmd r   \right\rangle \,\rmd B_s^{\lambda} \right|^2\right]\\ 
&\quad + \lambda^4\beta^{-2}\E\left[\left| \int_{\lfrf{t}}^t   \left\langle    \nabla H(\overline{\theta}^{\lambda}_{\lfrf{s}}),    
\int_{\lfrf{s}}^s    \int_{\lfrf{r}}^r - \Upsilon_{\lambda}(\overline{\theta}^{\lambda}_{\lfrf{\nu}}) \,\rmd \nu   \,\rmd r     \right\rangle \,\rmd B_s^{\lambda} \right|^2\right]\\ 
&\quad +2\lambda^3\beta^{-1}\E\left[\left|  \int_{\lfrf{t}}^t   \left\langle   \nabla H(\overline{\theta}^{\lambda}_{\lfrf{s}}),    
\int_{\lfrf{s}}^s -\int_{{\lfrf{r}}}^r H_{\lambda}(\overline{\theta}_{\lfrf{\nu}}^{\lambda})\,\rmd B_{\nu}^{\lambda}    \,\rmd r \right\rangle \,\rmd B_s^{\lambda} \right|^2\right]\\ 
&\leq 8\lambda^2\beta^{-2} \int_{{\lfrf{t}}}^t  \int_{{\lfrf{s}}}^s \E\left[|H (\overline{\zeta}^{\lambda, n}_{\lfrf{r}}) |_{\cF}^2+|H (\overline{\theta}_{\lfrf{r}}^{\lambda})|_{\cF}^2\right]\,\rmd r \,\rmd s \\
&\quad + \lambda^2\int_{\lfrf{t}}^t \E\left[\sum_{i,j=1}^d\left|  \left\langle\nabla H^{(i,j)}(\overline{\theta}^{\lambda}_{\lfrf{s}}),    h_{\lambda}(\overline{\theta}^{\lambda}_{\lfrf{s}})   \right\rangle  \right|^2\right]\, \rmd s \\ 
&\quad + \lambda^4\int_{\lfrf{t}}^t \E\left[\sum_{i,j=1}^d\left|  \left\langle\nabla H^{(i,j)}(\overline{\theta}^{\lambda}_{\lfrf{s}}),    H_{\lambda}( \overline{\theta}^{\lambda}_{\lfrf{s}})h_{\lambda}( \overline{\theta}^{\lambda}_{\lfrf{s}})   \right\rangle  \right|^2\right]\, \rmd s \\ 
&\quad + \lambda^4\beta^{-2}\int_{\lfrf{t}}^t \E\left[\sum_{i,j=1}^d\left|  \left\langle\nabla H^{(i,j)}(\overline{\theta}^{\lambda}_{\lfrf{s}}),   \Upsilon_{\lambda}(\overline{\theta}^{\lambda}_{\lfrf{s}})  \right\rangle  \right|^2\right]\, \rmd s \\ 
&\quad +2\lambda^3\beta^{-1} \int_{\lfrf{t}}^t \E\left[\sum_{i,j=1}^d\left|  \left\langle\nabla H^{(i,j)}(\overline{\theta}^{\lambda}_{\lfrf{s}}),  H_{\lambda}(\overline{\theta}_{\lfrf{s}}^{\lambda}) \int_{\lfrf{s}}^s \int_{{\lfrf{r}}}^r \,\rmd B_{\nu}^{\lambda}    \,\rmd r  \right\rangle  \right|^2\right]\, \rmd s \\ 
&\leq 8\lambda^2\beta^{-2} \int_{{\lfrf{t}}}^t  \E\left[dK_H^2(1+|\overline{\zeta}^{\lambda, n}_{\lfrf{s}}|^{\rho+q-1})^2 +dK_H^2(1+|\overline{\theta}_{\lfrf{s}}^{\lambda}|^{\rho+q-1})^2  \right] \,\rmd s \\
&\quad +  \lambda^2\int_{\lfrf{t}}^t \E\left[d^2\cK_0^2(1+|\overline{\theta}^{\lambda}_{\lfrf{s}}|)^{2(\rho+q-2)}K_h^2(1+|\overline{\theta}^{\lambda}_{\lfrf{s}}|^{\rho+q})^2  \right]\, \rmd s \\ 
&\quad +  \lambda^4\int_{\lfrf{t}}^t \E\left[d^2\cK_0^2(1+|\overline{\theta}^{\lambda}_{\lfrf{s}}|)^{2(\rho+q-2)}K_H^2(1+|\overline{\theta}^{\lambda}_{\lfrf{s}}|^{\rho+q-1})^2  K_h^2(1+|\overline{\theta}^{\lambda}_{\lfrf{s}}|^{\rho+q})^2  \right]\, \rmd s \\ 
&\quad +  \lambda^4\beta^{-2}\int_{\lfrf{t}}^t \E\left[d^2\cK_0^2(1+|\overline{\theta}^{\lambda}_{\lfrf{s}}|)^{2(\rho+q-2)}\cK_{3,d}^2(1+|\overline{\theta}^{\lambda}_{\lfrf{s}}|)^{2(\rho+q-2)}   \right]\, \rmd s \\ 
&\quad + 2\lambda^3\beta^{-1}  \int_{\lfrf{t}}^t \E\left[d^2\cK_0^2(1+|\overline{\theta}^{\lambda}_{\lfrf{s}}|)^{2(\rho+q-2)}K_H^2(1+|\overline{\theta}^{\lambda}_{\lfrf{s}}|^{\rho+q-1})^2 \left| \int_{\lfrf{s}}^s \int_{{\lfrf{r}}}^r \,\rmd B_{\nu}^{\lambda}    \,\rmd r  \right|^2   \right]\, \rmd s \\ 
&\leq 6\lambda^2\beta^{-2} dK_H^2 3^{16(\rho+1) } \int_{{\lfrf{t}}}^t  \E\left[ 1+|\overline{\zeta}^{\lambda, n}_{\lfrf{s}}|^{16(\rho+1)}+|\overline{\theta}_{\lfrf{s}}^{\lambda}|^{16(\rho+1)}   \right] \,\rmd s \\
&\quad +  \lambda^2( K_h^2 + K_H^2K_h^2+\beta^{-2}\cK_{3,d}^2)d^2\cK_0^22^{16(\rho+1)}\int_{\lfrf{t}}^t \E\left[1+|\overline{\theta}^{\lambda}_{\lfrf{s}}|^{16(\rho+1)}   \right]\, \rmd s \\ 
&\quad + \lambda^2 \beta^{-1}d^2\cK_0^2  K_H^2 2^{8(\rho+1)+1}\int_{\lfrf{t}}^t \left(\E\left[1+|\overline{\theta}^{\lambda}_{\lfrf{s}}|^{16(\rho+1)} \right]\right)^{1/2}(3(d+4))\, \rmd s \\ 
&\leq 6 \lambda^2\beta^{-2} dK_H^2 3^{24(\rho+1)} \left( e^{-\lambda\ca_\cD \min\{ \kappa, 1 /2\} \lfrf{t}}\E[|\theta_0|^{16(\rho+1)}] + \left( \cc_{8(\rho+1)}\left(1+1/(\ca_\cD\kappa)\right)+1\right)\right.\\
&\qquad \left. +3\mathrm{v}_{16(\rho+1)}(\cM_V(16(\rho+1)))\right)\\
&\quad +  \lambda^2( K_h^2 + K_H^2K_h^2+\beta^{-2}\cK_{3,d}^2+\beta^{-1}K_H^2)d^3\cK_0^22^{16(\rho+1)} \\
&\qquad \times \left(e^{-\lambda \ca_\cD \kappa \lfrf{t}}\E\left[|\theta_0|^{16(\rho+1)}\right]  +\cc_{8(\rho+1)}\left(1+1/(\ca_\cD\kappa)\right)+1\right)\\
&\leq \lambda^2\left(e^{- \ca_\cD \min\{ \kappa, 1 /2\} n/2} 10\overline{\cC}_{\mathbf{A}2}\E[|\theta_0|^{16(\rho+1)}] +10\widetilde{\cC}_{\mathbf{A}2} \right),
\end{align*}
where the fifth inequality holds due to Lemma \ref{lem:2ndpthmmt}, \ref{lem:zetaprocme} and \ref{lem:oserroralg}, and $\overline{\cC}_{\mathbf{A}2}$, $\widetilde{\cC}_{\mathbf{A}2}$ are given in \eqref{eq:graditoestconst}.
\end{enumerate}
This completes the proof.
\end{proof}

\begin{proof}[\textbf{Proof of Lemma \ref{lem:w1converp1}}]
\label{lem:w1converp1proof}
By using the definitions of  $\overline{\theta}^{\lambda}_t$ in \eqref{eq:aholahoproc} 
and $\overline{\zeta}^{\lambda, n}_t $ in Definition \ref{def:auxzeta}, and by applying It\^o's formula, we obtain, for any $n \in \N_0$, $t \in (nT, (n+1)T]$,
\begin{align}
&W_2^2(\mathcal{L}(\overline{\theta}^{\lambda}_t),\mathcal{L}(\overline{\zeta}^{\lambda, n}_t)) \nonumber\\
&\leq \E\left[\left|\overline{\theta}^{\lambda}_t - \overline{\zeta}^{\lambda, n}_t\right|^2\right] \nonumber\\
&=-2\lambda \E\left[\int_{nT}^t \left\langle\overline{\theta}^{\lambda}_s - \overline{\zeta}^{\lambda, n}_s, h_{\lambda}(\overline{\theta}^{\lambda}_{\lfrf{s}}) -h(\overline{\zeta}^{\lambda, n}_s)  \right\rangle\, \rmd s\right]\nonumber\\
&\quad +2\lambda \E\left[\int_{nT}^t \left\langle\overline{\theta}^{\lambda}_s - \overline{\zeta}^{\lambda, n}_s,  \lambda \int_{\lfrf{s}}^s\left(H_{\lambda}( \overline{\theta}^{\lambda}_{\lfrf{r}})h_{\lambda}( \overline{\theta}^{\lambda}_{\lfrf{r}})-\beta^{-1}\Upsilon_{\lambda}(\overline{\theta}^{\lambda}_{\lfrf{r}}) \right)\,\rmd r  \right\rangle\, \rmd s\right]\nonumber\\
&\quad -2\lambda \E\left[\int_{nT}^t \left\langle\overline{\theta}^{\lambda}_s - \overline{\zeta}^{\lambda, n}_s,     \sqrt{2\lambda\beta^{-1}} \int_{{\lfrf{s}}}^s H_{\lambda}(\overline{\theta}_{\lfrf{r}}^{\lambda}) \,\rmd B_r^{\lambda}  \right\rangle\, \rmd s\right]\nonumber\\
\begin{split}\label{eq:L2convsplting}
& =-2\lambda \E\left[\int_{nT}^t \left\langle\overline{\theta}^{\lambda}_s - \overline{\zeta}^{\lambda, n}_s, h(\overline{\theta}^{\lambda}_s) - h(\overline{\zeta}^{\lambda, n}_s)  \right\rangle\, \rmd s\right] \\
& \quad -2\lambda \E\left[\int_{nT}^t \left\langle\overline{\theta}^{\lambda}_s - \overline{\zeta}^{\lambda, n}_s, h (\overline{\theta}^{\lambda}_{\lfrf{s}})  - h(\overline{\theta}^{\lambda}_s)  \right\rangle\, \rmd s\right] \\
& \quad -2\lambda \E\left[\int_{nT}^t \left\langle\overline{\theta}^{\lambda}_s - \overline{\zeta}^{\lambda, n}_s, h_{\lambda}(\overline{\theta}^{\lambda}_{\lfrf{s}})  - h(\overline{\theta}^{\lambda}_{\lfrf{s}}) \right\rangle\, \rmd s\right] \\
&\quad +2\lambda \E\left[\int_{nT}^t \left\langle\overline{\theta}^{\lambda}_s - \overline{\zeta}^{\lambda, n}_s,  \lambda \int_{\lfrf{s}}^s\left(H_{\lambda}( \overline{\theta}^{\lambda}_{\lfrf{r}})h_{\lambda}( \overline{\theta}^{\lambda}_{\lfrf{r}})-\beta^{-1}\Upsilon_{\lambda}(\overline{\theta}^{\lambda}_{\lfrf{r}}) \right)\,\rmd r  \right\rangle\, \rmd s\right] \\
&\quad - 2\lambda \E\left[\int_{nT}^t \left\langle\overline{\theta}^{\lambda}_s - \overline{\zeta}^{\lambda, n}_s, \sqrt{2\lambda\beta^{-1}} \int_{{\lfrf{s}}}^s H_{\lambda}(\overline{\theta}_{\lfrf{r}}^{\lambda})\,\rmd B_r^{\lambda}  \right\rangle\, \rmd s\right].
\end{split}
\end{align}
By applying It\^o's formula to $ h(\overline{\theta}^{\lambda}_s)$, we obtain \eqref{eq:holaprocito}. Substituting \eqref{eq:holaprocito} into \eqref{eq:L2convsplting}, applying Remark \ref{rmk:oslc} and Young's inequality yield
\begin{align}
&\E\left[\left|\overline{\theta}^{\lambda}_t - \overline{\zeta}^{\lambda, n}_t\right|^2\right] \nonumber\\
&\leq 2\lambda \cL_{\cOS}\int_{nT}^t  \E\left[|\overline{\theta}^{\lambda}_s - \overline{\zeta}^{\lambda, n}_s|^2 \right] \, \rmd s   \nonumber\\
& \quad +2\lambda \E\left[\int_{nT}^t \left\langle\overline{\theta}^{\lambda}_s - \overline{\zeta}^{\lambda, n}_s, 
-\lambda  \int_{\lfrf{s}}^s \left(H(\overline{\theta}^{\lambda}_r) - H_{\lambda}( \overline{\theta}^{\lambda}_{\lfrf{r}}) \right)h_{\lambda}(\overline{\theta}^{\lambda}_{\lfrf{r}})\,\rmd r  \right\rangle\, \rmd s\right]   \nonumber\\
& \quad +2\lambda \E\left[\int_{nT}^t \left\langle\overline{\theta}^{\lambda}_s - \overline{\zeta}^{\lambda, n}_s, 
\lambda^2\int_{\lfrf{s}}^s H(\overline{\theta}^{\lambda}_r)\int_{\lfrf{r}}^r H_{\lambda}( \overline{\theta}^{\lambda}_{\lfrf{\nu}})h_{\lambda}( \overline{\theta}^{\lambda}_{\lfrf{\nu}}) \,\rmd \nu\,\rmd r   \right\rangle\, \rmd s\right]  \nonumber \\
& \quad +2\lambda \E\left[\int_{nT}^t \left\langle\overline{\theta}^{\lambda}_s - \overline{\zeta}^{\lambda, n}_s, 
-\lambda^2\beta^{-1}\int_{\lfrf{s}}^s H(\overline{\theta}^{\lambda}_r)\int_{\lfrf{r}}^r  \Upsilon_{\lambda}(\overline{\theta}^{\lambda}_{\lfrf{\nu}}) \,\rmd \nu\,\rmd r  \right\rangle\, \rmd s\right] \nonumber\\
& \quad +2\lambda \E\left[\int_{nT}^t \left\langle\overline{\theta}^{\lambda}_s - \overline{\zeta}^{\lambda, n}_s, 
-\lambda\sqrt{2\lambda\beta^{-1}} \int_{\lfrf{s}}^s  H(\overline{\theta}^{\lambda}_r) \int_{{\lfrf{r}}}^r H_{\lambda}(\overline{\theta}_{\lfrf{\nu}}^{\lambda})\,\rmd B_{\nu}^{\lambda}\,\rmd r    \right\rangle\, \rmd s\right]  \nonumber \\
& \quad +2\lambda \E\left[\int_{nT}^t \left\langle\overline{\theta}^{\lambda}_s - \overline{\zeta}^{\lambda, n}_s, 
\sqrt{2\lambda\beta^{-1}}\int_{\lfrf{s}}^s \left( H(\overline{\theta}^{\lambda}_r) - H_{\lambda}(\overline{\theta}_{\lfrf{r}}^{\lambda})\right) \, \rmd B^{\lambda}_r   \right\rangle\, \rmd s\right]  \nonumber \\
& \quad +2\lambda \E\left[\int_{nT}^t \left\langle\overline{\theta}^{\lambda}_s - \overline{\zeta}^{\lambda, n}_s, 
\lambda\beta^{-1}\int_{\lfrf{s}}^s\left(\Upsilon (\overline{\theta}^{\lambda}_r) - \Upsilon_{\lambda}(\overline{\theta}^{\lambda}_{\lfrf{r}}) \right)\rmd r   \right\rangle\, \rmd s\right]  \nonumber\\
& \quad -2\lambda \E\left[\int_{nT}^t \left\langle\overline{\theta}^{\lambda}_s - \overline{\zeta}^{\lambda, n}_s, h_{\lambda}(\overline{\theta}^{\lambda}_{\lfrf{s}})  - h(\overline{\theta}^{\lambda}_{\lfrf{s}}) \right\rangle\, \rmd s\right]  \nonumber\\
\begin{split}\label{eq:L2convspltingitoY}
&\leq  \lambda (2\cL_{\cOS}+6)\int_{nT}^t  \E\left[|\overline{\theta}^{\lambda}_s - \overline{\zeta}^{\lambda, n}_s|^2 \right] \, \rmd s\\
& \quad +\lambda\int_{nT}^t \E\left[\left|-\lambda  \int_{\lfrf{s}}^s \left(H(\overline{\theta}^{\lambda}_r) - H_{\lambda}( \overline{\theta}^{\lambda}_{\lfrf{r}}) \right)h_{\lambda}(\overline{\theta}^{\lambda}_{\lfrf{r}})\,\rmd r \right|^2 \right] \, \rmd s \\
& \quad +\lambda\int_{nT}^t \E\left[\left|\lambda^2\int_{\lfrf{s}}^s H(\overline{\theta}^{\lambda}_r)\int_{\lfrf{r}}^r H_{\lambda}( \overline{\theta}^{\lambda}_{\lfrf{\nu}})h_{\lambda}( \overline{\theta}^{\lambda}_{\lfrf{\nu}}) \,\rmd \nu\,\rmd r   \right|^2 \right] \, \rmd s \\ 
& \quad +\lambda\int_{nT}^t \E\left[\left|-\lambda^2\beta^{-1}\int_{\lfrf{s}}^s H(\overline{\theta}^{\lambda}_r)\int_{\lfrf{r}}^r  \Upsilon_{\lambda}(\overline{\theta}^{\lambda}_{\lfrf{\nu}}) \,\rmd \nu\,\rmd r  \right|^2 \right] \, \rmd s \\ 
& \quad +\lambda\int_{nT}^t \E\left[\left|-\lambda\sqrt{2\lambda\beta^{-1}} \int_{\lfrf{s}}^s  H(\overline{\theta}^{\lambda}_r) \int_{{\lfrf{r}}}^r H_{\lambda}(\overline{\theta}_{\lfrf{\nu}}^{\lambda})\,\rmd B_{\nu}^{\lambda}\,\rmd r   \right|^2 \right] \, \rmd s \\ 
& \quad +\lambda\int_{nT}^t \E\left[\left|\lambda\beta^{-1}\int_{\lfrf{s}}^s\left(\Upsilon (\overline{\theta}^{\lambda}_r) - \Upsilon_{\lambda}(\overline{\theta}^{\lambda}_{\lfrf{r}}) \right)\rmd r   \right|^2 \right] \, \rmd s \\  
& \quad +\lambda\int_{nT}^t \E\left[\left|h(\overline{\theta}^{\lambda}_{\lfrf{s}}) -h_{\lambda}(\overline{\theta}^{\lambda}_{\lfrf{s}})   \right|^2 \right] \, \rmd s \\  
& \quad +2\lambda \E\left[\int_{nT}^t \left\langle\overline{\theta}^{\lambda}_s - \overline{\zeta}^{\lambda, n}_s, 
\sqrt{2\lambda\beta^{-1}}\int_{\lfrf{s}}^s \left( H(\overline{\theta}^{\lambda}_r) - H_{\lambda}(\overline{\theta}_{\lfrf{r}}^{\lambda})\right) \, \rmd B^{\lambda}_r   \right\rangle\, \rmd s\right] .
\end{split} 
\end{align}
We note that, by using \eqref{eq:graditoestint1}, for any $t \geq nT$,
\begin{align}\label{eq:htamingrate}
\E\left[|  h(\overline{\theta}^{\lambda}_{\lfrf{t}}) -h_{\lambda}(\overline{\theta}^{\lambda}_{\lfrf{t}})|^2\right]
& \leq \E\left[\lambda^3| \overline{\theta}^{\lambda}_{\lfrf{t}} |^{6(\rho+q-1)}K_h^2(1+|\overline{\theta}^{\lambda}_{\lfrf{t}}|^{\rho+q})^2\right]\nonumber\\
&\leq \lambda^3K_h^22^{16(\rho+1)}\E\left[ 1+|\overline{\theta}^{\lambda}_{\lfrf{t}}|^{16(\rho+1)}\right]\nonumber\\
&\leq  \lambda^3K_h^22^{16(\rho+1)}\left(e^{-\lambda \ca_\cD \kappa \lfrf{t}}\E\left[|\theta_0|^{16(\rho+1)}\right]  +\cc_{8(\rho+1)}\left(1+1/(\ca_\cD\kappa)\right)+1\right)\nonumber\\
&\leq \lambda^3 \left(e^{- \ca_\cD  \kappa n/2} \overline{\cC}_{\mathbf{A}2}\E[|\theta_0|^{16(\rho+1)}] +\widetilde{\cC}_{\mathbf{A}2} \right),
\end{align}
where the third inequality holds due to Lemma \ref{lem:2ndpthmmt}, and $\overline{\cC}_{\mathbf{A}2}$, $\widetilde{\cC}_{\mathbf{A}2}$ are given in \eqref{eq:graditoestconst}. By using Lemma \ref{lem:graditoest} and \eqref{eq:htamingrate}, \eqref{eq:L2convspltingitoY} becomes
\begin{align}
\begin{split}\label{eq:L2convspltingMterm}
\E\left[\left|\overline{\theta}^{\lambda}_t - \overline{\zeta}^{\lambda, n}_t\right|^2\right]
&\leq  \lambda (2\cL_{\cOS}+6)\int_{nT}^t  \E\left[|\overline{\theta}^{\lambda}_s - \overline{\zeta}^{\lambda, n}_s|^2 \right] \, \rmd s\\
&\quad + 6\lambda^{2+q} \left(e^{- \ca_\cD  \kappa n/2} \overline{\cC}_{\mathbf{A}2}\E[|\theta_0|^{16(\rho+1)}] +\widetilde{\cC}_{\mathbf{A}2} \right)\\
&\quad + \mathfrak{J}_1+\mathfrak{J}_2,
\end{split}
\end{align}
where
\begin{align*}
\mathfrak{J}_1^{\lambda}(t)&:= 2\lambda \E\left[\int_{nT}^t \left\langle\overline{\theta}^{\lambda}_s - \overline{\zeta}^{\lambda, n}_s, 
\sqrt{2\lambda\beta^{-1}}\int_{\lfrf{s}}^s \left( H(\overline{\theta}^{\lambda}_r) - H_{\lambda}(\overline{\theta}_{\lfrf{r}}^{\lambda}) - \mathfrak{M}(\overline{\theta}^{\lambda}_r, \overline{\theta}^{\lambda}_{\lfrf{r}})     \right)   \, \rmd B^{\lambda}_r   \right\rangle\, \rmd s\right],\\
\mathfrak{J}_2^{\lambda}(t)&:= 2\lambda \E\left[\int_{nT}^t \left\langle\overline{\theta}^{\lambda}_s - \overline{\zeta}^{\lambda, n}_s, 
\sqrt{2\lambda\beta^{-1}}\int_{\lfrf{s}}^s \mathfrak{M}(\overline{\theta}^{\lambda}_r, \overline{\theta}^{\lambda}_{\lfrf{r}}) \, \rmd B^{\lambda}_r   \right\rangle\, \rmd s\right]
\end{align*}
with $\mathfrak{M}$ defined in Definition \ref{def:Mdef}. By using Young's inequality and Lemma \ref{lem:Mest}, we have that
\begin{align}\label{eq:L2convspltingMtermJ1}
\mathfrak{J}_1^{\lambda}(t) \leq \lambda \int_{nT}^t  \E\left[|\overline{\theta}^{\lambda}_s - \overline{\zeta}^{\lambda, n}_s|^2 \right] \, \rmd s +  \lambda^{2+q}\left(e^{- \ca_\cD  \kappa n/2} \overline{\cC}_{\mathbf{A}2}\E[|\theta_0|^{16(\rho+1)}] +\widetilde{\cC}_{\mathbf{A}2} \right).
\end{align}
To establish an upper bound for $\mathfrak{J}_2^{\lambda}(t)$, we recall the definitions of $(\overline{\theta}^{\lambda}_t)_{t\geq 0}$ and $(\overline{\zeta}^{\lambda, n}_t)_{t \geq 0}$ given in \eqref{eq:aholahoproc} and Definition \ref{def:auxzeta}, respectively, and consider the following splitting:
\begin{align*}
\mathfrak{J}_2^{\lambda}(t)
&= 2\lambda \int_{nT}^t\E\left[ \left\langle \overline{\theta}^{\lambda}_s - \overline{\theta}^{\lambda}_{\lfrf{s}} -( \overline{\zeta}^{\lambda, n}_s- \overline{\zeta}^{\lambda, n}_{\lfrf{s}}), 
\sqrt{2\lambda\beta^{-1}}\int_{\lfrf{s}}^s \mathfrak{M}(\overline{\theta}^{\lambda}_r, \overline{\theta}^{\lambda}_{\lfrf{r}}) \, \rmd B^{\lambda}_r   \right\rangle\right] \, \rmd s\\
&= 2\lambda \int_{nT}^t\E\left[ \left\langle \int_{\lfrf{s}}^s \left(\lambda (h(\overline{\zeta}^{\lambda, n}_r) - h_{\lambda}(\overline{\theta}^{\lambda}_{\lfrf{r}})) +\lambda^2\int_{\lfrf{r}}^r\left(H_{\lambda}( \overline{\theta}^{\lambda}_{\lfrf{\nu}})h_{\lambda}( \overline{\theta}^{\lambda}_{\lfrf{\nu}})-\beta^{-1}\Upsilon_{\lambda}(\overline{\theta}^{\lambda}_{\lfrf{\nu}}) \right)\,\rmd \nu \right.\right. \right.\\
&\qquad \left.\left.\left.  - \lambda\sqrt{2\lambda\beta^{-1}} \int_{{\lfrf{r}}}^r H_{\lambda}(\overline{\theta}_{\lfrf{\nu}}^{\lambda})\,\rmd B_{\nu}^{\lambda}\right)\,\rmd r, \sqrt{2\lambda\beta^{-1}}\int_{\lfrf{s}}^s \mathfrak{M}(\overline{\theta}^{\lambda}_r, \overline{\theta}^{\lambda}_{\lfrf{r}}) \, \rmd B^{\lambda}_r   \right\rangle\right] \, \rmd s\\
&= 
2\lambda \int_{nT}^t\E\left[ \left\langle \int_{\lfrf{s}}^s \left(\lambda (h(\overline{\theta}^{\lambda}_r)- h(\overline{\theta}^{\lambda}_{\lfrf{r}})) +\lambda^2\int_{\lfrf{r}}^r\left(H_{\lambda}( \overline{\theta}^{\lambda}_{\lfrf{\nu}})h_{\lambda}( \overline{\theta}^{\lambda}_{\lfrf{\nu}})-\beta^{-1}\Upsilon_{\lambda}(\overline{\theta}^{\lambda}_{\lfrf{\nu}}) \right)\,\rmd \nu \right.\right. \right.\\
&\qquad \left.\left.\left.  - \lambda\sqrt{2\lambda\beta^{-1}} \int_{{\lfrf{r}}}^r H_{\lambda}(\overline{\theta}_{\lfrf{\nu}}^{\lambda})\,\rmd B_{\nu}^{\lambda}\right)\,\rmd r, \sqrt{2\lambda\beta^{-1}}\int_{\lfrf{s}}^s \mathfrak{M}(\overline{\theta}^{\lambda}_r, \overline{\theta}^{\lambda}_{\lfrf{r}}) \, \rmd B^{\lambda}_r   \right\rangle\right] \, \rmd s\\
&\quad + 2\lambda \int_{nT}^t\E\left[ \left\langle \int_{\lfrf{s}}^s \lambda (  h(\overline{\theta}^{\lambda}_{\lfrf{r}}) - h_{\lambda}(\overline{\theta}^{\lambda}_{\lfrf{r}})) \,\rmd r,
\sqrt{2\lambda\beta^{-1}}\int_{\lfrf{s}}^s \mathfrak{M}(\overline{\theta}^{\lambda}_r, \overline{\theta}^{\lambda}_{\lfrf{r}}) \, \rmd B^{\lambda}_r   \right\rangle\right] \, \rmd s+\mathfrak{J}_{2,1}^{\lambda}(t),
\end{align*}
where the first equality holds due to the following:
\begin{align}\label{eq:gridptmeasmartingale}
\begin{split}
&\E\left[ \left\langle  \overline{\theta}^{\lambda}_{\lfrf{s}} +\overline{\zeta}^{\lambda, n}_{\lfrf{s}}, 
\sqrt{2\lambda\beta^{-1}}\int_{\lfrf{s}}^s \mathfrak{M}(\overline{\theta}^{\lambda}_r, \overline{\theta}^{\lambda}_{\lfrf{r}}) \, \rmd B^{\lambda}_r   \right\rangle\right] \\
&=\E\left[ \left\langle  \overline{\theta}^{\lambda}_{\lfrf{s}} +\overline{\zeta}^{\lambda, n}_{\lfrf{s}}, \E\left[\left.
\sqrt{2\lambda\beta^{-1}}\int_{\lfrf{s}}^s \mathfrak{M}(\overline{\theta}^{\lambda}_r, \overline{\theta}^{\lambda}_{\lfrf{r}}) \, \rmd B^{\lambda}_r\right|\mathcal{F}^{\lambda}_{\lfrf{s}}\right]   \right\rangle\right]  = 0,
\end{split}
\end{align}
and where
\[
\mathfrak{J}_{2,1}^{\lambda}(t) =  2\lambda \int_{nT}^t\E\left[ \left\langle \int_{\lfrf{s}}^s \lambda (h(\overline{\zeta}^{\lambda, n}_r) - h(\overline{\theta}^{\lambda}_r)) \,\rmd r,
\sqrt{2\lambda\beta^{-1}}\int_{\lfrf{s}}^s \mathfrak{M}(\overline{\theta}^{\lambda}_r, \overline{\theta}^{\lambda}_{\lfrf{r}}) \, \rmd B^{\lambda}_r   \right\rangle\right] \, \rmd s.
\]
By applying Cauchy-Schwarz inequality, Corollary \ref{cor:graditoub}, Lemma \ref{lem:Mest} and \eqref{eq:htamingrate}, we further obtain that
\begin{align}\label{eq:L2convspltingMtermJ2}
\mathfrak{J}_2^{\lambda}(t)
&\leq 2\lambda \int_{nT}^t\left(\E\left[ \left| \int_{\lfrf{s}}^s \lambda \left( h(\overline{\theta}^{\lambda}_r)- h(\overline{\theta}^{\lambda}_{\lfrf{r}}) +\lambda \int_{\lfrf{r}}^r\left(H_{\lambda}( \overline{\theta}^{\lambda}_{\lfrf{\nu}})h_{\lambda}( \overline{\theta}^{\lambda}_{\lfrf{\nu}})-\beta^{-1}\Upsilon_{\lambda}(\overline{\theta}^{\lambda}_{\lfrf{\nu}}) \right)\,\rmd \nu \right.\right. \right.\right. \nonumber\\
&\qquad \left.\left.\left.\left.  -  \sqrt{2\lambda\beta^{-1}} \int_{{\lfrf{r}}}^r H_{\lambda}(\overline{\theta}_{\lfrf{\nu}}^{\lambda})\,\rmd B_{\nu}^{\lambda}\right)\,\rmd r  \right|^2\right]\right)^{1/2} \nonumber \\
&\qquad \times \left(\E\left[\left| \sqrt{2\lambda\beta^{-1}}\int_{\lfrf{s}}^s \mathfrak{M}(\overline{\theta}^{\lambda}_r, \overline{\theta}^{\lambda}_{\lfrf{r}}) \, \rmd B^{\lambda}_r\right|^2\right]\right)^{1/2} \, \rmd s \nonumber\\
&\quad + 2\lambda \int_{nT}^t \left(\E\left[ \left| \int_{\lfrf{s}}^s \lambda (  h(\overline{\theta}^{\lambda}_{\lfrf{r}}) - h_{\lambda}(\overline{\theta}^{\lambda}_{\lfrf{r}})) \,\rmd r   \right|^2\right] \right)^{1/2}  \nonumber  \\
&\qquad \times \left(\E\left[ \left| \sqrt{2\lambda\beta^{-1}}\int_{\lfrf{s}}^s \mathfrak{M}(\overline{\theta}^{\lambda}_r, \overline{\theta}^{\lambda}_{\lfrf{r}}) \, \rmd B^{\lambda}_r   \right|^2\right] \right)^{1/2}\, \rmd s+\mathfrak{J}_{2,1}^{\lambda}(t) \nonumber\\
&\leq  2\lambda \int_{nT}^t\left(\lambda^4\left(e^{- \ca_\cD  \kappa n/2}36 \overline{\cC}_{\mathbf{A}2}\E[|\theta_0|^{16(\rho+1)}] +36\widetilde{\cC}_{\mathbf{A}2} \right)\right)^{1/2} \nonumber\\
&\qquad \times \left(\lambda^2\left(e^{- \ca_\cD  \kappa n/2} \overline{\cC}_{\mathbf{A}2}\E[|\theta_0|^{16(\rho+1)}] +\widetilde{\cC}_{\mathbf{A}2} \right)\right)^{1/2} \, \rmd s \nonumber\\
&\quad +2\lambda \int_{nT}^t\left(\lambda^5 \left(e^{- \ca_\cD  \kappa n/2} \overline{\cC}_{\mathbf{A}2}\E[|\theta_0|^{16(\rho+1)}] +\widetilde{\cC}_{\mathbf{A}2} \right)\right)^{1/2} \nonumber\\
&\qquad \times \left(\lambda^2\left(e^{- \ca_\cD  \kappa n/2} \overline{\cC}_{\mathbf{A}2}\E[|\theta_0|^{16(\rho+1)}] +\widetilde{\cC}_{\mathbf{A}2} \right)\right)^{1/2} \, \rmd s+\mathfrak{J}_{2,1}^{\lambda}(t) \nonumber\\
&\leq 14\lambda^3\left(e^{- \ca_\cD  \kappa n/2} \overline{\cC}_{\mathbf{A}2}\E[|\theta_0|^{16(\rho+1)}] +\widetilde{\cC}_{\mathbf{A}2} \right)+\mathfrak{J}_{2,1}^{\lambda}(t).
\end{align}
At this stage, the task reduces to upper bound $\mathfrak{J}_{2,1}^{\lambda}(t)$. To achieve this, we apply Cauchy-Schwarz inequality, Corollary \ref{cor:graditoub} and Lemma \ref{lem:Mest} to obtain
\begin{align}\label{eq:L2convspltingMtermJ21}
\mathfrak{J}_{2,1}^{\lambda}(t)
& =  2\lambda \int_{nT}^t\E\left[ \left\langle \int_{\lfrf{s}}^s \lambda \left(h(\overline{\zeta}^{\lambda, n}_r) - h ( \overline{\zeta}^{\lambda, n}_{\lfrf{r}})  - \sqrt{2\lambda\beta^{-1}} \int_{{\lfrf{r}}}^r H (\overline{\zeta}^{\lambda, n}_{\lfrf{\nu}})\,\rmd B_{\nu}^{\lambda}  
  - \Big(h(\overline{\theta}^{\lambda}_r) \Big.\right.\right.\right. \nonumber\\
&\qquad \left.\left.\left.\Big. - h (\overline{\theta}^{\lambda}_{\lfrf{r}}) - \sqrt{2\lambda\beta^{-1}} \int_{{\lfrf{r}}}^r H (\overline{\theta}_{\lfrf{\nu}}^{\lambda})\,\rmd B_{\nu}^{\lambda} \Big)\right) \,\rmd r, 
\sqrt{2\lambda\beta^{-1}}\int_{\lfrf{s}}^s \mathfrak{M}(\overline{\theta}^{\lambda}_r, \overline{\theta}^{\lambda}_{\lfrf{r}}) \, \rmd B^{\lambda}_r   \right\rangle\right] \, \rmd s \nonumber\\
&\quad +2\lambda \int_{nT}^t\E\left[ \left\langle \int_{\lfrf{s}}^s \lambda\sqrt{2\lambda\beta^{-1}} \int_{{\lfrf{r}}}^r (H (\overline{\zeta}^{\lambda, n}_{\lfrf{\nu}}) -  H (\overline{\theta}_{\lfrf{\nu}}^{\lambda}))\,\rmd B_{\nu}^{\lambda}  \,\rmd r,\right.\right. \nonumber\\
&\qquad \left.\left. \sqrt{2\lambda\beta^{-1}}\int_{\lfrf{s}}^s \mathfrak{M}(\overline{\theta}^{\lambda}_r, \overline{\theta}^{\lambda}_{\lfrf{r}}) \, \rmd B^{\lambda}_r   \right\rangle\right] \, \rmd s\nonumber\\
&\leq 2\lambda \int_{nT}^t\left(\E\left[ \left| \int_{\lfrf{s}}^s \lambda \left(h(\overline{\zeta}^{\lambda, n}_r) - h ( \overline{\zeta}^{\lambda, n}_{\lfrf{r}})  - \sqrt{2\lambda\beta^{-1}} \int_{{\lfrf{r}}}^r H (\overline{\zeta}^{\lambda, n}_{\lfrf{\nu}})\,\rmd B_{\nu}^{\lambda}  
  - \Big(h(\overline{\theta}^{\lambda}_r) \Big.\right.\right.\right.\right.\nonumber\\
&\qquad \left.\left.\left.\left.\Big. - h (\overline{\theta}^{\lambda}_{\lfrf{r}}) - \sqrt{2\lambda\beta^{-1}} \int_{{\lfrf{r}}}^r H (\overline{\theta}_{\lfrf{\nu}}^{\lambda})\,\rmd B_{\nu}^{\lambda} \Big)\right) \,\rmd r
   \right|^2\right]\right)^{1/2}  \nonumber\\
&\qquad \times  \left(\E\left[\left|\sqrt{2\lambda\beta^{-1}}\int_{\lfrf{s}}^s \mathfrak{M}(\overline{\theta}^{\lambda}_r, \overline{\theta}^{\lambda}_{\lfrf{r}}) \, \rmd B^{\lambda}_r\right|^2\right]\right)^{1/2}\, \rmd s\nonumber\\
&\quad +2\lambda^3\left(e^{- \ca_\cD \min\{ \kappa, 1 /2\} n/2}10 \overline{\cC}_{\mathbf{A}2}\E[|\theta_0|^{16(\rho+1)}] +10\widetilde{\cC}_{\mathbf{A}2} \right)\nonumber\\
&\leq 2\lambda \int_{nT}^t\left(\lambda^4\left(e^{- \ca_\cD \min\{ \kappa, 1 /2\} n/2}72 \overline{\cC}_{\mathbf{A}2}\E[|\theta_0|^{16(\rho+1)}] +72\widetilde{\cC}_{\mathbf{A}2} \right)\right)^{1/2}\nonumber\\
&\qquad \times \left( \lambda^2\left(e^{- \ca_\cD  \kappa n/2} \overline{\cC}_{\mathbf{A}2}\E[|\theta_0|^{16(\rho+1)}] +\widetilde{\cC}_{\mathbf{A}2} \right)\right)^{1/2}\, \rmd s\nonumber\\
&\quad +20\lambda^3\left(e^{- \ca_\cD \min\{ \kappa, 1 /2\} n/2} \overline{\cC}_{\mathbf{A}2}\E[|\theta_0|^{16(\rho+1)}] +\widetilde{\cC}_{\mathbf{A}2} \right)\nonumber\\
&\leq 44\lambda^3\left(e^{- \ca_\cD \min\{ \kappa, 1 /2\} n/2} \overline{\cC}_{\mathbf{A}2}\E[|\theta_0|^{16(\rho+1)}] +\widetilde{\cC}_{\mathbf{A}2} \right).
\end{align}
Substituting \eqref{eq:L2convspltingMtermJ21} into \eqref{eq:L2convspltingMtermJ2} yields
\begin{equation}\label{eq:L2convspltingMtermJ2ub}
\mathfrak{J}_2^{\lambda}(t) \leq 58\lambda^3\left(e^{- \ca_\cD \min\{ \kappa, 1 /2\} n/2} \overline{\cC}_{\mathbf{A}2}\E[|\theta_0|^{16(\rho+1)}] +\widetilde{\cC}_{\mathbf{A}2} \right).
\end{equation}
By applying \eqref{eq:L2convspltingMtermJ1} and \eqref{eq:L2convspltingMtermJ2ub} to \eqref{eq:L2convspltingMterm}, we obtain that
\begin{align*}
\E\left[\left|\overline{\theta}^{\lambda}_t - \overline{\zeta}^{\lambda, n}_t\right|^2\right]
&\leq  \lambda (2\cL_{\cOS}+7)\int_{nT}^t  \E\left[|\overline{\theta}^{\lambda}_s - \overline{\zeta}^{\lambda, n}_s|^2 \right] \, \rmd s\\
&\quad  +   65\lambda^{2+q}\left(e^{- \ca_\cD \min\{ \kappa, 1 /2\} n/2} \overline{\cC}_{\mathbf{A}2}\E[|\theta_0|^{16(\rho+1)}] +\widetilde{\cC}_{\mathbf{A}2} \right),
\end{align*}
which, by applying Gr\"{o}nwall's lemma, yields
\[
W_2^2(\mathcal{L}(\overline{\theta}^{\lambda}_t),\mathcal{L}(\overline{\zeta}^{\lambda, n}_t))
\leq \E\left[\left|\overline{\theta}^{\lambda}_t - \overline{\zeta}^{\lambda, n}_t\right|^2\right]
\leq \lambda^{2+q}\left(e^{- \ca_\cD \min\{ \kappa, 1 /2\} n/2}\cC_0\E[|\theta_0|^{16(\rho+1)}] +\cC_1 \right),
\]
where $\kappa$ is given in \eqref{eq:2ndmmtexpconst} and
\begin{equation}\label{eq:w1converp1const}
\cC_0:=65e^{2\cL_{\cOS}+7} \overline{\cC}_{\mathbf{A}2}, \quad  \cC_1: = 65e^{2\cL_{\cOS}+7}\widetilde{\cC}_{\mathbf{A}2}
\end{equation}
with $\overline{\cC}_{\mathbf{A}2}$, $\widetilde{\cC}_{\mathbf{A}2}$ given in \eqref{eq:graditoestconst}.
\end{proof}

\newpage
\section{Analytic Expression of Constants for aHOLA}
\label{appen:constexp}
{\footnotesize
\begin{tabularx}{\textwidth}[H]
{ 
	>{\hsize=0.48\hsize\linewidth=\hsize}X
	>{\hsize=0.2\hsize\linewidth=\hsize}X|
	>{\hsize=2.32\hsize\linewidth=\hsize}X
} 
\hline
\hline
{CONSTANT} & & FULL EXPRESSION\\
\hline
Remark \ref{rmk:growthc}		& $\cK_0$		& 		$2^{1-q}\max\{L, |\nabla^2 h^{(1)}(0)|,\dots, |\nabla^2 h^{(d)}(0)|\}$\\
										& $\cK_1$  		& 		$\max\{\cK_0,|\nabla  h^{(1)}(0)|,\dots, |\nabla  h^{(d)}(0)|\}$ \\
										& $\cK_2$  		& 		$\max\{\cK_1,|h^{(1)}(0)|,\dots, |h^{(d)}(0)|\}$ \\
										&$\cK_{3,d}$  	& 		$\max\{d^{3/2}L, \Upsilon(0)\}$ \\
\hline
Remark \ref{rmk:dissipativityc}	&$\ca_\cD$				& 		$a/2$ \\
										&$\cb_\cD$				& 		$(a/2+b)\cR_\cD^{\overline{r}+2}+|h(0)|^2/(2a)$   \\
										&$\cR_\cD$				& 		$\max\{(4b/a)^{1/(r-\overline{r})}, 2^{1/r}\}$   \\
										&$\overline{\cb}_\cD$& 		$\ca_\cD+\cb_\cD$   \\
\hline
Remark \ref{rmk:oslc} 				&$\cL_{\cOS}$			& 		$\sqrt{d}\cK_1(1+2\cR_{\cOS})^{\rho+q-1}$  \\
										&$\cR_{\cOS}$			& 		$ (b/a)^{1/(r-\overline{r})}$  \\
\hline
Theorem \ref{thm:mainw1} 		& $C_0$ &	$\min\{\dot{\cc}, \ca_\cD \kappa, \ca_\cD  /2\}  /4$\\
										& $C_1$ &	$e^{\min\{\dot{\cc}, \ca_\cD \kappa, \ca_\cD  /2\}  /4}\left[\cC_0^{1/2}+\cC_2+\hat{\cc}\left(3+\int_{\R^d} V_2(\theta) \pi_{\beta}(\rmd \theta)\right)\right]$\\
										& $C_2$ &	$\cC_1^{1/2}+\cC_3$\\
\hline	
Corollary \ref{crl:mainw2}		& $C_3$ 	&	$\min\{\dot{\cc}, \ca_\cD \kappa, \ca_\cD  /2\}  /8$\\
										& $C_4$ 	&	$e^{\min\{\dot{\cc}, \ca_\cD \kappa, \ca_\cD  /2\}  /8}\left[\cC_0^{1/2}+\cC_4+\sqrt{2\hat{\cc}}\left(3+\int_{\R^d} V_2(\theta) \pi_{\beta}(\rmd \theta)\right)^{1/2}\right]$\\
										& $C_5$ 	&	$\cC_1^{1/2}+\cC_5$\\
										&$\cC_4$ &	$\sqrt{\hat{\cc}}\left(1+\frac{8}{\min\{\dot{\cc}, \ca_\cD \kappa, \ca_\cD  /2\} }\right)e^{\min\{\dot{\cc}, \ca_\cD \kappa, \ca_\cD  /2\} /8}\left(\sqrt{2}\cC_0^{1/2} +2^{4\rho+5}\right)$\\
										&$\cC_5$ &	$4(\sqrt{\hat{\cc}}/\dot{\cc})e^{\dot{\cc}/4}(\sqrt{2}\cC_1^{1/2}+3\sqrt{2}+2^{4\rho+5}(\cc_{8(\rho+1)}(1+1/( \ca_\cD \kappa))+1)^{1/2} +\sqrt{6}\mathrm{v}^{1/2}_{16(\rho+1)}(\cM_V(16(\rho+1))))$\\
\hline	
Lemma \ref{lem:2ndpthmmt}	& $\kappa$			& $\cM_1^{\rho+q-1}/\left(1+ \cM_1^{3(\rho+q-1)}\right)^{1/3} \quad (\text{with $ \cM_1>0$})$\\
										& $\cc_0$			& $\ca_\cD\kappa \cM_1^2 +\cc_1+2^{2(\rho+q)} \beta^{-1} d(1+\cK_1)^2$\\
										& $\cc_1$			& $ 2\cb_\cD +4K_hK_H +\beta^{-1}2^{\rho+q-1}\cK_{3,d}+2K_h^2+3  K_H^2K_h^2+6  K_h^2K_H +2^{2(\rho+q)-6}  \beta^{-2}\cK_{3,d}^2+\beta^{-1} K_h\cK_{3,d}2^{\rho+q}+\beta^{-1} K_H K_h\cK_{3,d}2^{\rho+q-1} $\\
										& $\cc_p$  			& $\left(1+2/( \ca_\cD\kappa )\right)^{p-1}(\ca_\cD\kappa \cM_1^2 +\cc_1)^p+\cc_{\Xi}(p) +2^{2(p+\rho+q)-3}p(2p-1) \beta^{-1} d(1+\cK_1)^2 (\left(1+2/( \ca_\cD\kappa )\right)^{p-2}(\ca_\cD\kappa \cM_1^2 +\cc_1)^{p-1}+\cM_2(p)^{2p-2})$\\
										&$\cc_{\Xi}(p)$ 	& $ 2^{2p-4}(\beta^{-1}d)^p(2p(2p-1))^{p+1}3^{p-1}\left(1+ 2^{p(\rho+q-1)} d^{p/2}  \cK_1^p  \right)^2$\\
										& $\cM_2(p)$ 		& $(2^{2(p+\rho+q)-1}p(2p-1) \beta^{-1} d(1+\cK_1)^2/(\ca_\cD\kappa))^{1/2}$\\
\hline	
Lemma \ref{lem:driftcon}			& $\cc_{V,1}(p) $	& $ \ca_\cD p/4$\\
										& $\cc_{V,2}(p)$ 	& $(3/4)\ca_\cD p\mathrm{v}_p(\cM_V(p))$\\
										& $\cM_V(p) $ 		& $(1/3+4\overline{\cb}_\cD/(3\ca_\cD)+4d/(3\ca_\cD\beta)+4(p-2)/(3\ca_\cD\beta))^{1/2}$\\
\hline	
Lemma \ref{lem:w1converp1}		&$\cC_0$ 									& $65e^{2\cL_{\cOS}+7} \overline{\cC}_{\mathbf{A}2}$\\
											&$\cC_1$ 									& $65e^{2\cL_{\cOS}+7}\widetilde{\cC}_{\mathbf{A}2}$\\
											&$\overline{\cC}_{\mathbf{A}2}$		& $3^{24(\rho+1)}d^3(1+\cK_0)^2(1+K_h)^2(1+K_H)^4(1+\beta^{-1})^2(1+L)^2(1+ \cK_{3,d})^2   (2+\max\{\overline{\cC}_{\mathbf{A}0,2},\overline{\cC}_{\mathbf{A}0,2q},\overline{\cC}_{\mathbf{A}1,2},\overline{\cC}_{\mathbf{A}0,2+2q}\} )$\\
											&$\widetilde{\cC}_{\mathbf{A}2}$		& $3^{24(\rho+1)}d^3(1+\cK_0)^2(1+K_h)^2(1+K_H)^4(1+\beta^{-1})^2(1+L)^2(1+ \cK_{3,d})^2   (2 \cc_{8(\rho+1)}\left(1+1/(\ca_\cD\kappa)\right)+2\max\{\widetilde{\cC}_{\mathbf{A}0,2},\widetilde{\cC}_{\mathbf{A}0,2q},\widetilde{\cC}_{\mathbf{A}1,2},\widetilde{\cC}_{\mathbf{A}0,2+2q}\}+1  +\mathrm{v}_{16(\rho+1)}(\cM_V(16(\rho+1))) )$\\
											&$\overline{\cC}_{\mathbf{A}0,p}$ 	& $5^{2p}2^{4\lcrc{p}(\rho+1)}(K_h^{2p}+\beta^{-2p}\cK_{3,d}^{2p}+K_H^{2p}K_h^{2p}+(\beta^{-1}(4p+2)(d+4p))^p(1+K_H^{2p}))$\\
											&$\widetilde{\cC}_{\mathbf{A}0,p}$ 	& $ \overline{\cC}_{\mathbf{A}0,p}\left( \cc_{2\lcrc{p}(\rho+1)}\left(1+1/(\ca_\cD\kappa)\right)+1\right)$\\
											&$\overline{\cC}_{\mathbf{A}1,p}$  	& $2^{2p+\lcrc{p}(\rho+1)-1}K_h^{2p}$\\
											&$\widetilde{\cC}_{\mathbf{A}1,p}$  	& $\overline{\cC}_{\mathbf{A}1,p}\left( \cc_{\lcrc{p}(\rho+1)}\left(1+1/(\ca_\cD\kappa)\right)+1\right)+ 2^{2p}K_h^{2p}3\mathrm{v}_{2\lcrc{p}(\rho+1)}(\cM_V(2\lcrc{p}(\rho+1))) +2^{2p}(2\beta^{-1}(p+1)(d+2p))^p$\\
\hline	
Proposition \ref{prop:contractionw12}		& $\dot{\cc}$ 				& $\min\{\bar{\phi}, \cc_{V,1}(2), 4\cc_{V,2}(2) \epsilon \cc_{V,1}(2)\}/2$\\
													& $\bar{\phi}$				& $\left(\sqrt{8\pi/(\beta \cL_{\cOS})} \dot{\cc}_0  \exp \left( \left(\dot{\cc}_0 \sqrt{\beta \cL_{\cOS}/8} + \sqrt{8/(\beta \cL_{\cOS})} \right)^2 \right) \right)^{-1}$\\
													& $\epsilon  $  				& $1 \wedge     (4 \cc_{V,2}(2) \sqrt{2 \beta\pi/  \cL_{\cOS} }\int_0^{\dot{\cc}_1}\exp  \left( \left(s \sqrt{\beta \cL_{\cOS}/8}+\sqrt{8/(\beta \cL_{\cOS})}\right)^2 \right) \,\rmd s  )^{-1}$\\
													& $\dot{\cc}_0$   			& $2(4\cc_{V,2}(2)(1+\cc_{V,1}(2))/\cc_{V,1}(2)-1)^{1/2}$\\
													& $\dot{\cc}_1$			   	& $2(2 \cc_{V,2}(2)/\cc_{V,1}(2)-1)^{1/2}$\\
													&$\hat{\cc}$	   				& $2(1+ \dot{\cc}_0)\exp(\beta \cL_{\cOS} \dot{\cc}_0^2/8+2\dot{\cc}_0)/\epsilon$\\
\hline	
Lemma \ref{lem:w1converp2}		& $\cC_2$ & $  \hat{\cc}\left(1+\frac{4}{\min\{\dot{\cc}, \ca_\cD \kappa, \ca_\cD  /2\} }\right)e^{\min\{\dot{\cc}, \ca_\cD \kappa, \ca_\cD  /2\} /4}\left(\cC_0 +2^{8\rho+11}\right)$\\
											& $\cC_3$ & $2(\hat{\cc}/\dot{\cc})e^{\dot{\cc}/2} (\cC_1+15+2^{8\rho+11}(\cc_{8(\rho+1)}(1+1/( \ca_\cD \kappa))+1) +18\mathrm{v}_{16(\rho+1)}(\cM_V(16(\rho+1))) )$\\
\hline
\hline
\end{tabularx}
}

\newpage
\section{Analytic Expression of Constants for aHOLLA}
\label{appen:constexplip}

{\footnotesize
\begin{tabularx}{\textwidth}[H]
{ 
	>{\hsize=0.48\hsize\linewidth=\hsize}X
	>{\hsize=0.2\hsize\linewidth=\hsize}X|
	>{\hsize=2.32\hsize\linewidth=\hsize}X
} 
\hline
\hline
{CONSTANT} & & FULL EXPRESSION\\
\hline
Remark \ref{rmk:growthclip}		& $\cKl_0$  & 		$ \max\{\overline{L}_1, |\nabla^2 h^{(1)}(0)|,\dots, |\nabla^2 h^{(d)}(0)|\}$\\
										& $\cKl_1$ & 		$\max\{\overline{L}_3,|h(0)|\}$ \\
\hline
Theorem \ref{thm:mainw1lip} 	& $C_{\Lin,0}$ &	$\min\{\dot{\cc}_{\Lin},\overline{a}/2\} /4$\\
										& $C_{\Lin,1}$ &	$e^{ \min\{\dot{\cc}_{\Lin},\overline{a}/2\} /4}\left[\cC_{\Lin,0}^{1/2}+\cC_{\Lin,2}+\hat{\cc}_{\Lin}\left(3+\int_{\R^d} V_2(\theta) \pi_{\beta}(\rmd \theta)\right)\right]$\\
										& $C_{\Lin,2}$ &	$\cC_{\Lin,1}^{1/2}+\cC_{\Lin,3}$\\
\hline	
Corollary \ref{crl:mainw2lip}		& $C_{\Lin,3}$ 	&	$ \min\{\dot{\cc}_{\Lin},\overline{a}/2\}  /8$\\
										& $C_{\Lin,4}$ 	&	$e^{ \min\{\dot{\cc}_{\Lin},\overline{a}/2\}   /8}\left[\cC_{\Lin,0}^{1/2}+\cC_{\Lin,4}+\sqrt{2\hat{\cc}_{\Lin}}\left(3+\int_{\R^d} V_2(\theta) \pi_{\beta}(\rmd \theta)\right)^{1/2}\right]$\\
										& $C_{\Lin,5}$ 	&	$\cC_{\Lin,1}^{1/2}+\cC_{\Lin,5}$\\
										&$\cC_{\Lin,4}$ &	$\sqrt{ \hat{\cc}_{\Lin}}\left(1+\frac{8}{\min\{\dot{\cc}_{\Lin},\overline{a}/2\} }\right)e^{\min\{\dot{\cc}_{\Lin},\overline{a}/2\} /8}\left( \cC_{\Lin,0}^{1/2} +3\right)$\\
										&$\cC_{\Lin,5}$ &	$4(\sqrt{\hat{\cc}_{\Lin}}/\dot{\cc}_{\Lin})e^{\dot{\cc}_{\Lin}/4}\left( \cC_{\Lin,1}^{1/2}+1+3(\ccl_2(1+1/\overline{a})+1)^{1/2}+\sqrt{3}\mathrm{v}_4^{1/2}(\cMl_V(4))\right)$\\
\hline	
Lemma \ref{lem:2ndpthmmtlip}	& $\ccl_0$			& $\beta^{-1} d\left( \overline{L}_2 + \overline{L}_2\cKl_1+ \overline{L}_2\overline{L}_3\cKl_1/2\right) \cMl_1 +\ccl_1+2 \beta^{-1} d(1+ \overline{L}_3)^2$\\
										& $\ccl_1$			& $2\overline{b}+2\overline{L}_3\cKl_1+2 \cKl_1^2+\overline{L}_3^2\cKl_1^2/2+\beta^{-2}d^2\overline{L}_2^2/4+2\overline{L}_3\cKl_1^2+\beta^{-1} d(\overline{L}_2\cKl_1 +\overline{L}_2\overline{L}_3\cKl_1/2)$\\
										& $\cMl_1$ 			& $2(\overline{a}\beta)^{-1} d(\overline{L}_2 +\overline{L}_2\cKl_1+ \overline{L}_2\overline{L}_3\cKl_1/2)$\\
										& $\ccl_p$  			& $\ccl_{\Xil}(p)+\left(1+2/\overline{a}\right)^{p-1}  (\ccl_1+\beta^{-1} d\left(\overline{L}_2 + \overline{L}_2\cKl_1+ \overline{L}_2\overline{L}_3\cKl_1/2\right) \cMl_1)^p  +2^{2p-2}p(2p-1) \beta^{-1} d(1+\overline{L}_3)^2  (\left(1+2/\overline{a}\right)^{p-2}(\ccl_1+\beta^{-1} d\left(\overline{L}_2 + \overline{L}_2\cKl_1+ \overline{L}_2\overline{L}_3\cKl_1/2\right) \cMl_1)^{p-1}+\cMl_2(p)^{2p-2})$\\
										&$\ccl_{\Xil}(p)$ 	& $2^{2p-3}p(2p-1)  (2\beta^{-1}dp(2p-1)(1+\overline{L}_3)^2)^p$\\
										& $\cMl_2(p)$ 		& $(2^{2p}p(2p-1) \beta^{-1} d(1+\overline{L}_3)^2/\overline{a})^{1/2}$\\
\hline	
Lemma \ref{lem:driftconlip}		& $\ccl_{V,1}(p)$	& $\overline{a} p/4$\\
										& $\ccl_{V,2}(p)$ 	& $(3/4)\overline{a} p\mathrm{v}_p(\cMl_V(p))$\\
										& $\cMl_V(p)$ 		& $ (1/3+4\overline{b}/(3\overline{a})+4d/(3\overline{a}\beta)+4(p-2)/(3\overline{a}\beta))^{1/2}$\\
\hline	
Lemma \ref{lem:w1converlipp1}		&$\cC_{\Lin,0}$ 							& $52e^{2\overline{L}_3+6} \overline{\cC}_{\mathbf{S}2}$\\
											&$\cC_{\Lin,1}$ 							& $52e^{2\overline{L}_3+6}\widetilde{\cC}_{\mathbf{S}2}$\\
											&$\overline{\cC}_{\mathbf{S}2}$		& $2^7d^4(1+\beta^{-1})^2(1+\overline{L}_3)^4(1+\overline{L}_2)^2 (1+\overline{L}_1)^2( 1+\cKl_1)^2( 1+\cKl_0)^2 (1+\max\{\overline{\cC}_{\mathbf{S}0,2},\overline{\cC}_{\mathbf{S}0,2q},\overline{\cC}_{\mathbf{S}0,1+q},\overline{\cC}_{\mathbf{S}1,2}\})$\\
											&$\widetilde{\cC}_{\mathbf{S}2}$		& $2^7d^4(1+\beta^{-1})^2(1+\overline{L}_3)^4(1+\overline{L}_2)^2(1+\overline{L}_1)^2( 1+\cKl_1)^2( 1+\cKl_0)^2  (\ccl_2 (1+1/\overline{a})+\max\{\widetilde{\cC}_{\mathbf{S}0,2},\widetilde{\cC}_{\mathbf{S}0,2q},\widetilde{\cC}_{\mathbf{S}0,1+q},\widetilde{\cC}_{\mathbf{S}1,2}\}+1)$\\
											&$\overline{\cC}_{\mathbf{S}0,p}$ 		& $10^{2p} (\cKl_1^{2p}+(\overline{L}_3\cKl_1)^{2p})$\\
											&$\widetilde{\cC}_{\mathbf{S}0,p}$ 	& $\overline{\cC}_{\mathbf{S}0,p}( \ccl_{\lcrc{p}}(1+1/\overline{a})+1)+5^{2p}( (\beta^{-1}d \overline{L}_2)^{2p}+(2\beta^{-1}(p+1)(d+2p))^p(1+\overline{L}_3^{2p}))$\\
											&$\overline{\cC}_{\mathbf{S}1,p}$  	& $2^{4\lcrc{p}}(1+\cKl_1)^{2p}$\\
											&$\widetilde{\cC}_{\mathbf{S}1,p}$  	& $\overline{\cC}_{\mathbf{S}1,p}\left( \ccl_{\lcrc{p}}(1+1/\overline{a})+1\right)+ 2^{3\lcrc{p}}(1+\cKl_1)^{2p}3\mathrm{v}_{2\lcrc{p}}(\cMl_V(2\lcrc{p})) +2^{2p}(2\beta^{-1}(p+1)(d+2p))^p$\\
\hline	
Proposition \ref{prop:contractionw12lip}	& $\dot{\cc}_{\Lin}$ 		& $\min\{\bar{\phi}_{\Lin}, \ccl_{V,1}(2), 4\ccl_{V,2}(2) \overline{\epsilon} \ccl_{V,1}(2)\}/2$\\
													& $\bar{\phi}_{\Lin}$		& $\left(\sqrt{8\pi/(\beta \overline{L}_3)} \dot{\cc}_{\Lin,0}  \exp \left( \left(\dot{\cc}_{\Lin,0} \sqrt{\beta \overline{L}_3/8} + \sqrt{8/(\beta \overline{L}_3)} \right)^2 \right) \right)^{-1}$\\
													& $\overline{\epsilon}$  	& $1 \wedge    \left(4 \ccl_{V,2}(2) \sqrt{2 \beta\pi/  \overline{L}_3 }\int_0^{\dot{\cc}_{\Lin,1}}\exp  \left( \left(s \sqrt{\beta \overline{L}_3/8}+\sqrt{8/(\beta \overline{L}_3)}\right)^2 \right) \,\rmd s \right)^{-1}$\\
													& $\dot{\cc}_{\Lin,0}$   	& $2(4\ccl_{V,2}(2)(1+\ccl_{V,1}(2))/\ccl_{V,1}(2)-1)^{1/2}$\\
													& $\dot{\cc}_{\Lin,1}$   	& $2(2 \ccl_{V,2}(2)/\ccl_{V,1}(2)-1)^{1/2}$\\
													&$\hat{\cc}_{\Lin}$	   	& $2(1+ \dot{\cc}_{\Lin,0})\exp(\beta \overline{L}_3 \dot{\cc}_{\Lin,0}^2/8+2\dot{\cc}_{\Lin,0})/\overline{\epsilon}$\\
\hline	
Lemma \ref{lem:w1converlipp2}		& $\cC_{\Lin,2}$ & $ \hat{\cc}_{\Lin}\left(1+\frac{4}{\min\{\dot{\cc}_{\Lin},\overline{a}/2\} }\right)e^{\min\{\dot{\cc}_{\Lin},\overline{a}/2\} /4}\left( \cC_{\Lin,0} +9\right)$\\
											& $\cC_{\Lin,3}$ & $2(\hat{\cc}_{\Lin}/\dot{\cc}_{\Lin})e^{\dot{\cc}_{\Lin}/2}\left( \cC_{\Lin,1}+9+9\ccl_2(1+1/\overline{a})+9\mathrm{v}_4(\cMl_V(4))\right)$\\
\hline
\hline
\end{tabularx}
}

\newpage
\section{Proof of main results for aHOLLA}
\label{appen:mtlinearproofs}
We provide the proofs for Theorem \ref{thm:mainw1lip} and Corollary \ref{crl:mainw2lip}, which follow the same lines as those for aHOLA. We first introduce auxiliary processes which we use throughout the convergence analysis in Appendix \ref{appen:auxproclipintro}. Then, we provide moment estimates for the newly introduced processes in Appendix \ref{appen:melip}, which are followed by the detailed proofs for the main results in Appendix \ref{appen:mtlipdetalis}. We postpone the proofs of the results in Appendices \ref{appen:melip} and \ref{appen:mtlipdetalis} to Appendix \ref{appen:mtlinearproofsapd}.

\subsection{Auxiliary processes}\label{appen:auxproclipintro} Fix $\beta>0$. Consider the Langevin SDE $(Z_t)_{t \geq 0}$ given by
\begin{equation} \label{eq:sdelip}
Z_0 := \theta_0, \quad \rmd Z_t=-h\left(Z_t\right) \rmd t+ \sqrt{2\beta^{-1}} \rmd B_t,
\end{equation}
where $(B_t)_{t \geq 0}$ is a $d$-dimensional Brownian motion on $(\Omega,\mathcal{F},P)$ with its completed natural filtration denoted by $(\mathcal{F}_t)_{t\geq 0}$. Moreover, we assume that $(\mathcal{F}_t)_{t\geq 0}$ is independent of $\sigma(\theta_0)$. Under Assumption~\ref{asm:ALLlip}, it is a well-known result that the Langevin SDE \eqref{eq:sdelip} admits a unique solution, which is adapted to $\mathcal{F}_t \vee \sigma(\theta_0)$, $t\geq 0$. Its $2p$-th moment estimate with $p \in \N$ is provided in Lemma \ref{lem:sde2pmmtlip}, which can be used to further deduce the $2p$-th moment estimate of $\pi_{\beta}(\theta)\propto e^{-\beta U(\theta)}$. 

For each $\lambda>0$, recall $ B^{\lambda}_t :=B_{\lambda t}/\sqrt{\lambda}, t \geq 0$. Denote by $(\mathcal{F}^{\lambda}_t)_{t \geq 0}$ with $\mathcal{F}^{\lambda}_t := \mathcal{F}_{\lambda t}$, $t \geq 0$, its completed natural filtration, which is independent of $\sigma(\theta_0)$. Moreover, we denote by $Z^{\lambda}_t := Z_{\lambda t}$, $t \geq 0$, the time-changed version of Langevin SDE \eqref{eq:sdelip}, which is given by
\begin{equation} \label{eq:tcsdelip}
Z^{\lambda}_0:=\theta_0, \quad \rmd Z^{\lambda}_t = -\lambda h(Z^{\lambda}_t)\, \rmd t +\sqrt{2\lambda\beta^{-1}}\,\rmd B^{\lambda}_t.
\end{equation}

Furthermore, we denote by $(\widetilde{\Theta}^{\lambda}_t)_{t \geq 0}$ the continuous-time interpolation of aHOLLA \eqref{eq:aholalip1}-\eqref{eq:aholalip3} given by 
\begin{equation}\label{eq:aholaproclip}
\widetilde{\Theta}^{\lambda}_0 := \theta_0, \quad \rmd \widetilde{\Theta}^{\lambda}_t= \lambda \phi^{\lambda}_{\Lin}(\widetilde{\Theta}^{\lambda}_{\lfrf{t}})\, \rmd t
+ \sqrt{2\lambda\beta^{-1}}\psi^{\lambda}_{\Lin}(\widetilde{\Theta}^{\lambda}_{\lfrf{t}}) \,\rmd B^{\lambda}_t,
\end{equation}
where $\phi^{\lambda}_{\Lin}$ and $\psi^{\lambda}_{\Lin}$ are defined in \eqref{eq:aholalip2} and \eqref{eq:aholalip3}, respectively. 
\begin{remark}
Similarly, denote by $(\overline{\Theta}^{\lambda}_t)_{t \geq 0 }$ the continuous-time interpolation of the order 1.5 scheme \eqref{eq:aholaohlip} given by
\begin{align}\label{eq:aholahoproclip}
\begin{split}
 \rmd \overline{\Theta}^{\lambda}_t&=  -\lambda h (\overline{\Theta}^{\lambda}_{\lfrf{t}})\,\rmd t+\lambda^2\int_{\lfrf{t}}^t\left(H ( \overline{\Theta}^{\lambda}_{\lfrf{s}})h ( \overline{\Theta}^{\lambda}_{\lfrf{s}})-\beta^{-1}\Upsilon (\overline{\Theta}^{\lambda}_{\lfrf{s}}) \right)\,\rmd s\,\rmd t\\
&\quad-\lambda\sqrt{2\lambda\beta^{-1}} \int_{{\lfrf{t}}}^t H (\overline{\Theta}_{\lfrf{s}}^{\lambda})\,\rmd B_s^{\lambda}\,\rmd t  +\sqrt{2\lambda\beta^{-1}} \, \rmd B^{\lambda}_t
\end{split}
\end{align}
with $\overline{\Theta}^{\lambda}_0 := \theta_0$. We note that $\mathcal{L}(\widetilde{\Theta}^{\lambda}_n)=\mathcal{L}(\Theta_n^{\cHOLLA})=\mathcal{L}(\Theta_n^{\lambda})=\mathcal{L}(\overline{\Theta}^{\lambda}_n)$, for each $n\in\N_0$.
\end{remark}

Finally, for any $s \geq 0$, consider a continuous-time process $(\zeta^{s,v, \lambda}_t)_{t \geq s}$ defined by
\begin{equation}\label{eq:auxproclip}
\zeta^{s,v, \lambda}_s := v \in \R^d, \quad \rmd\zeta^{s,v, \lambda}_t = -\lambda h(\zeta^{s,v, \lambda}_t)\, \rmd t +\sqrt{2\lambda\beta^{-1}}\,\rmd B^{\lambda}_t.
\end{equation}
\begin{definition}\label{def:auxzetalip} Fix $\lambda >0$. Define $T\equiv T(\lambda) : = \lfrf{1/\lambda}$. Then, for any $n \in \N_0$ and $t \geq nT$, define 
\[
\widetilde{\zeta}^{\lambda, n}_t := \zeta^{nT,\overline{\Theta}^{\lambda}_{nT}, \lambda}_t.
\]
\end{definition}

\subsection{Moment estimates}\label{appen:melip} Recall the following Lyapunov functions: for each $p\in [2, \infty)\cap {\N}$, define $V_p(\theta) := (1+|\theta|^2)^{p/2}$, for all $\theta \in \R^d$, and moreover, define $\mathrm{v}_p(w) := (1+w^2)^{p/2}$, for all $w \geq 0$. We observe that $V_p$ is twice continuously differentiable and satisfies:
\begin{equation}\label{eq:contrassumptionlip}
\sup_{\theta \in \R^d}|\nabla V_p(\theta)|/V_p(\theta) <\infty, \quad \lim_{|\theta|\to\infty} \nabla V_p(\theta)/V_p(\theta)=0.
\end{equation}
Furthermore, we denote by $\mathcal{P}_{V_p}(\R^d)$ the set of probability measures $\mu \in \mathcal{P}(\R^d)$ which satisfies $\int_{\R^d} V_p(\theta)\, \mu(\rmd \theta) <\infty$.

Next, we establish moment estimates for $(\widetilde{\Theta}^{\lambda}_t)_{t\geq 0}$ given in \eqref{eq:aholaproclip}. The results with explicit constants are provided below. We note that for any $p\in [2, \infty)\cap {\N}$ and $t\geq 0$, we have that $\E[|\widetilde{\Theta}^{\lambda}_t|^{2p}] = \E[| \overline{\Theta}^{\lambda}_t|^{2p}]$.
\begin{lemma}\label{lem:2ndpthmmtlip} Let Assumptions \ref{asm:AIlip}, \ref{asm:ALLlip}, and \ref{asm:ADlip} hold.  Then, we obtain the following estimates:
\begin{enumerate}[leftmargin=*]
\item For any $0<\lambda\leq \overline{\lambda}_{\max}$, $n \in \N_0$, and $t \in (n, n+1]$,
\[
\E\left[ |\widetilde{\Theta}^{\lambda}_t|^2  \right]  \leq \left(1 -\lambda(t-n)\overline{a} \right)\left(1 -\lambda \overline{a} \right)^n\E\left[ |\theta_0|^2\right]+  \ccl_0\left(1+1/\overline{a}\right),
\]
where the constant $\ccl_0$ is given explicitly in \eqref{eq:2ndmmtlipexpconst}. In particular, the above inequality implies $\sup_{t\geq 0}\E\left[|\widetilde{\Theta}^{\lambda}_t|^2\right]  \leq  \E\left[ |\theta_0|^2\right]+  \ccl_0\left(1+1/\overline{a}\right)<\infty$.
\label{lem:2ndpthmmtlipi}
\item  For any $p\in [2, \infty)\cap {\N}$, $0<\lambda\leq \overline{\lambda}_{\max}$, $n \in \N_0$, and $t \in (n, n+1]$,
\[
\E\left[ |\widetilde{\Theta}^{\lambda}_t|^{2p} \right]  \leq \left(1 -\lambda(t-n)\overline{a} \right)\left(1 -\lambda \overline{a} \right)^n\E\left[ |\theta_0|^{2p}\right]+  \ccl_p\left(1+1/\overline{a}\right),
\]
where the constant $\ccl_p$ is given explicitly in \eqref{eq:2pthmmtlipexpconst}. In particular, the above estimate implies $\sup_{t\geq 0}\E\left[|\widetilde{\Theta}^{\lambda}_t|^{2p} \right]  \leq \E\left[ |\theta_0|^{2p}\right]+  \ccl_p\left(1+1/\overline{a}\right)<\infty$. \label{lem:2ndpthmmtlipii}
\end{enumerate}
\end{lemma}
\begin{proof} See Appendix \ref{lem:2ndpthmmtlipproof}.
\end{proof}

We provide below a drift condition for $V_p$ (defined in the beginning of Appendix \ref{appen:melip}).
\begin{lemma}
\label{lem:driftconlip} Let Assumptions \ref{asm:AIlip}, \ref{asm:ALLlip}, and \ref{asm:ADlip} hold. Then, for any $p\in [2, \infty)\cap {\N}$, $\theta \in \R^d$, we obtain
\[
\Delta V_p(\theta)/\beta - \langle \nabla V_p(\theta), h(\theta) \rangle \leq -\ccl_{V,1}(p) V_p(\theta) +\ccl_{V,2}(p),
\]
where $\ccl_{V,1}(p) := \overline{a} p/4$, $\ccl_{V,2}(p) := (3/4)\overline{a} p\mathrm{v}_p(\cMl_V(p))$ with $\cMl_V(p) := (1/3+4\overline{b}/(3\overline{a})+4d/(3\overline{a}\beta)+4(p-2)/(3\overline{a}\beta))^{1/2}$.
\end{lemma}
\begin{proof} See \cite[Lemma 3.5]{nonconvex}.
\end{proof}

By applying Lemma \ref{lem:2ndpthmmtlip} and \ref{lem:driftconlip}, we obtain moment estimates for $(\widetilde{\zeta}^{\lambda, n}_t)_{t \geq nT}$ defined in Definition \ref{def:auxzetalip}.
\begin{lemma}\label{lem:zetaprocmelip} Let Assumptions \ref{asm:AIlip}, \ref{asm:ALLlip}, and \ref{asm:ADlip} hold. Then, for any $p \in \N$, $0<\lambda\leq \overline{\lambda}_{\max}$, $n \in \N_0$, and $t\geq nT$, we obtain
\[
\E[V_{2p}(\widetilde{\zeta}^{\lambda, n}_t)] \leq 2^{p-1}e^{-\lambda \overline{a}  t/2}\E[|\theta_0|^{2p}] + 2^{p-1}\left( \ccl_p(1+1/\overline{a})+1\right)+3\mathrm{v}_{2p}(\cMl_V(2p)),
\]
where $\ccl_p$ is given in \eqref{eq:2pthmmtlipexpconst} (see also Lemma \ref{lem:2ndpthmmtlip}) and $\cMl_V(2p)$ is given in Lemma \ref{lem:driftconlip}. 
\end{lemma}
\begin{proof} See \cite[Corollary 4.6]{mtula}.
\end{proof}

\subsection{Proof of main results}\label{appen:mtlipdetalis} In this section, we present key results used to obtain Theorem \ref{thm:mainw1lip}. To this end, we split $W_1(\mathcal{L}(\overline{\Theta}^{\lambda}_t),\pi_\beta)$, for any $t \in (nT, (n+1)T]$ and $n \in \N_0$, by using the law of $\widetilde{\zeta}^{\lambda, n}_t $ defined in Definition \ref{def:auxzetalip} and the law of $Z^{\lambda}_t$ given in \eqref{eq:tcsdelip} as follows:
\begin{equation}\label{eq:mtsplitlip}
W_1(\mathcal{L}(\overline{\Theta}^{\lambda}_t),\pi_\beta) \leq W_1(\mathcal{L}(\overline{\Theta}^{\lambda}_t),\mathcal{L}(\widetilde{\zeta}^{\lambda, n}_t))+W_1(\mathcal{L}(\widetilde{\zeta}^{\lambda, n}_t),\mathcal{L}(Z^{\lambda}_t))+W_1(\mathcal{L}(Z^{\lambda}_t),\pi_\beta).
\end{equation}

In the following lemma, we provide a non-asymptotic estimate for $W_2(\mathcal{L}(\overline{\Theta}^{\lambda}_t),\mathcal{L}(\widetilde{\zeta}^{\lambda, n}_t))$, which can be used to upper bound the first term on the RHS of \eqref{eq:mtsplitlip}.
\begin{lemma}\label{lem:w1converlipp1} Let Assumptions \ref{asm:AIlip}, \ref{asm:ALLlip}, and \ref{asm:ADlip} hold. Then, for any $0<\lambda\leq \overline{\lambda}_{\max}$, $n \in \N_0$, and $t \in (nT, (n+1)T]$, we obtain
\[
W_2(\mathcal{L}(\overline{\Theta}^{\lambda}_t),\mathcal{L}(\widetilde{\zeta}^{\lambda, n}_t)) 
\leq \lambda^{1+q/2}\left(e^{- \overline{a} n/4}  \cC_{\Lin,0}\E[|\theta_0|^4] + \cC_{\Lin,1} \right)^{1/2},
\]
where $ \cC_{\Lin,0},  \cC_{\Lin,1}$ are given explicitly in \eqref{eq:w1converlipp1const}.
\end{lemma}
\begin{proof} See Appendix \ref{lem:w1converlipp1proof}.
\end{proof}

For the last two terms on the RHS of \eqref{eq:mtsplitlip}, we observe that they can be viewed as Wasserstein-1 distances between distributions of Langevin processes starting from different initial points. Therefore, to obtain their upper bounds, we introduce a semi-metric which allows us to establish a contraction result for the Langevin SDE \eqref{eq:sdelip} under our assumptions.

We consider the following semi-metric: for any $p\in [2, \infty)\cap {\N}$, $ \mu,\nu \in \mathcal{P}_{V_p}(\R^d)$, let
\begin{equation}\label{eq:semimetricw1plip}
w_{1,p}(\mu,\nu):=\inf_{\zeta\in\mathcal{C}(\mu,\nu)}\int_{\mathbb{R}^d}\int_{\mathbb{R}^d} [1\wedge |\theta-\theta'|](1+V_p(\theta)+V_p(\theta'))\zeta(\rmd\theta, \rmd\theta').
\end{equation}
Then, we provide a result which states the contraction property of the Langevin SDE \eqref{eq:sdelip} in $w_{1,2}$.
\begin{proposition}\label{prop:contractionw12lip} Let Assumptions \ref{asm:AIlip}, \ref{asm:ALLlip}, and \ref{asm:ADlip} hold.  Moreover, let $\theta_0' \in L^2$, and let $(Z_t')_{t \geq 0}$ be the solution of SDE \eqref{eq:sdelip} whose starting point $Z'_0 := \theta'_0$ is assumed to be independent of $\mathcal{F}_{\infty} := \sigma(\bigcup_{t \geq 0} \mathcal{F}_t)$. Then, we obtain
\begin{equation}\label{eq:w12contractionlip}
w_{1,2}(\mathcal{L}(Z_t),\mathcal{L}(Z'_t)) \leq \hat{\cc}_{\Lin} e^{-\dot{\cc}_{\Lin} t} w_{1,2}(\mathcal{L}(\theta_0),\mathcal{L}(\theta_0')),
\end{equation}
where the explicit expressions for $\dot{\cc}_{\Lin}, \hat{\cc}_{\Lin}$ are given below.

The contraction constant $\dot{\cc}_{\Lin}$ is given by:
\begin{equation*}
\dot{\cc}_{\Lin}:=\min\{\bar{\phi}_{\Lin}, \ccl_{V,1}(2), 4\ccl_{V,2}(2) \overline{\epsilon} \ccl_{V,1}(2)\}/2,
\end{equation*}
where $\ccl_{V,1}(2) :=  \overline{a}/2$, $\ccl_{V,2}(2) :=3\overline{a} \mathrm{v}_2(\cMl_V(2))/2$ with $\cMl_V(2) := (1/3+4\overline{b}/(3\overline{a})+4d/(3\overline{a}\beta))^{1/2}$, the constant $\bar{\phi}_{\Lin} $ is given by
\begin{equation*}
\bar{\phi}_{\Lin} := \left(\sqrt{8\pi/(\beta \overline{L}_3)} \dot{\cc}_{\Lin,0}  \exp \left( \left(\dot{\cc}_{\Lin,0} \sqrt{\beta \overline{L}_3/8} + \sqrt{8/(\beta \overline{L}_3)} \right)^2 \right) \right)^{-1},
\end{equation*}
and $\overline{\epsilon} >0$ is chosen such that
\begin{equation*}
\overline{\epsilon}  \leq 1 \wedge    \left(4 \ccl_{V,2}(2) \sqrt{2 \beta\pi/  \overline{L}_3 }\int_0^{\dot{\cc}_{\Lin,1}}\exp  \left( \left(s \sqrt{\beta \overline{L}_3/8}+\sqrt{8/(\beta \overline{L}_3)}\right)^2 \right) \,\rmd s \right)^{-1}
\end{equation*}
with $\dot{\cc}_{\Lin,0} := 2(4\ccl_{V,2}(2)(1+\ccl_{V,1}(2))/\ccl_{V,1}(2)-1)^{1/2}$ and $\dot{\cc}_{\Lin,1}:=2(2 \ccl_{V,2}(2)/\ccl_{V,1}(2)-1)^{1/2}$.

Moreover, the constant $\hat{\cc}_{\Lin}$ is given by:
\begin{equation*}
\hat{\cc}_{\Lin}: =2(1+ \dot{\cc}_{\Lin,0})\exp(\beta \overline{L}_3 \dot{\cc}_{\Lin,0}^2/8+2\dot{\cc}_{\Lin,0})/\overline{\epsilon}.
\end{equation*}
\end{proposition}

\begin{proof} We note that \cite[Assumption 2.1]{eberle2019quantitative} holds with $\kappa = \overline{L}_3$ due to Assumption \ref{asm:ALLlip}, \cite[Assumption 2.2]{eberle2019quantitative} holds with $V=V_2$ due to Remark \ref{lem:driftconlip}, and \cite[Assumptions 2.4 and 2.5]{eberle2019quantitative} hold due to \eqref{eq:contrassumptionlip}. Therefore, we can obtain \eqref{eq:w12contractionlip} following the same arguments as in the proof of \cite[Proposition~3.14]{nonconvex} based on \cite[Theorem 2.2, Corollary 2.3]{eberle2019quantitative}. In addition, $\dot{\cc}_{\Lin}, \hat{\cc}_{\mathsf{Lin}}$ can be obtained following the arguments in the proof of \cite[Proposition 4.6]{lim2021nonasymptotic}.
\end{proof}

By using the above result and $W_1 \leq w_{1,2}$ (see \cite[Lemma A.3]{lim2021nonasymptotic}), we can establish a non-asymptotic error bound for the second term on the RHS of \eqref{eq:mtsplitlip}. The explicit statement is given below.
\begin{lemma}\label{lem:w1converlipp2} Let Assumptions \ref{asm:AIlip}, \ref{asm:ALLlip}, and \ref{asm:ADlip} hold. Then, for any $0<\lambda\leq \overline{\lambda}_{\max}$, $n \in \N_0$, and $t \in (nT, (n+1)T]$, we obtain
\[
W_1(\mathcal{L}(\widetilde{\zeta}_t^{\lambda,n}),\mathcal{L}(Z_t^\lambda)) \leq \lambda^{1+q/2}\left(e^{- \min\{\dot{\cc}_{\Lin},\overline{a}/2\} n/4}\cC_{\Lin,2}\E[|\theta_0|^4] +\cC_{\Lin,3} \right) ,
\]
where
\begin{align}\label{eq:w1converlipp2consts}
\begin{split}
\cC_{\Lin,2}
& := \hat{\cc}_{\Lin}\left(1+\frac{4}{\min\{\dot{\cc}_{\Lin},\overline{a}/2\} }\right)e^{\min\{\dot{\cc}_{\Lin},\overline{a}/2\} /4}\left( \cC_{\Lin,0} +9\right),\\
\cC_{\Lin,3}
& :=  2(\hat{\cc}_{\Lin}/\dot{\cc}_{\Lin})e^{\dot{\cc}_{\Lin}/2}\left( \cC_{\Lin,1}+9+9\ccl_2(1+1/\overline{a})+9\mathrm{v}_4(\cMl_V(4))\right)
\end{split}
\end{align}
with $\dot{\cc}_{\Lin}, \hat{\cc}_{\Lin}$ given in Proposition \ref{prop:contractionw12lip}, $\cC_{\Lin,0}, \cC_{\Lin,1}$ given in \eqref{eq:w1converlipp1const} (see also Lemma \ref{lem:w1converlipp1}), $\ccl_2$ given in Lemma \ref{lem:2ndpthmmtlip}, and $\cMl_V(4)$ given in Lemma \ref{lem:zetaprocmelip}.
\end{lemma}
\begin{proof} See \cite[Lemma 4.7]{lim2021nonasymptotic}.
\end{proof}
To obtain an upper bound for the last term on the RHS of \eqref{eq:mtsplitlip}, we observe that $\pi_{\beta}$ is the invariant measure of the Langevin SDE \eqref{eq:tcsdelip}. Thus, by applying Proposition \ref{prop:contractionw12lip}, we have that
\begin{equation}\label{eq:w1converlipp3}
W_1(\mathcal{L}(Z^{\lambda}_t),\pi_\beta) \leq  \hat{\cc}_{\Lin} e^{-\dot{\cc}_{\Lin} \lambda t} w_{1,2}(\mathcal{L}(\theta_0), \pi_\beta)
\leq \hat{\cc}_{\Lin} e^{-\dot{\cc}_{\Lin} \lambda t} \left[1+\E[V_2(\theta_0)]+ \int_{\R^d} V_2(\theta) \pi_{\beta}(\rmd \theta)\right].
\end{equation}
By using Lemma \ref{lem:w1converlipp1}, \ref{lem:w1converlipp2} and \eqref{eq:w1converlipp3}, we can obtain an upper bound for each $W_2(\mathcal{L}(\Theta_n^{\cHOLLA}),\pi_{\beta})$, $n \in \N_0$, as stated in Theorem \ref{thm:mainw1lip}.

\begin{proof}[\textbf{Proof of Theorem \ref{thm:mainw1lip}}]  Substituting the results in Lemma \ref{lem:w1converlipp1}, \ref{lem:w1converlipp2} and \eqref{eq:w1converlipp3} into \eqref{eq:mtsplitlip}, for any $0<\lambda\leq \overline{\lambda}_{\max}$, $n \in \N_0$, and $t \in (nT, (n+1)T]$, we have that
\begin{align*}
W_1(\mathcal{L}(\overline{\Theta}^{\lambda}_t),\pi_\beta) 
&\leq  
\lambda^{1+q/2}\left(e^{- \overline{a} n/4}  \cC_{\Lin,0}\E[|\theta_0|^4] + \cC_{\Lin,1} \right)^{1/2}\\
&\quad +\lambda^{1+q/2}\left(e^{- \min\{\dot{\cc}_{\Lin},\overline{a}/2\} n/4}\cC_{\Lin,2}\E[|\theta_0|^4] +\cC_{\Lin,3} \right)\\
&\quad + \hat{\cc}_{\Lin} e^{-\dot{\cc}_{\Lin} \lambda t} \left[1+\E[V_2(\theta_0)]+ \int_{\R^d} V_2(\theta) \pi_{\beta}(\rmd \theta)\right]\\
&\leq C_{\Lin,1} e^{-C_{\Lin,0} ( n+1)}(\E[|\theta_0|^4]+1) +C_{\Lin,2}\lambda^{1+q/2},
\end{align*}
where 
\begin{align}\label{eq:mainw1lipconst}
\begin{split}
C_{\Lin,0}&:= \min\{\dot{\cc}_{\Lin},\overline{a}/2\} /4, \\
C_{\Lin,1} &:=e^{ \min\{\dot{\cc}_{\Lin},\overline{a}/2\} /4}\left[\cC_{\Lin,0}^{1/2}+\cC_{\Lin,2}+\hat{\cc}_{\Lin}\left(3+\int_{\R^d} V_2(\theta) \pi_{\beta}(\rmd \theta)\right)\right], \\
C_{\Lin,2}&:=\cC_{\Lin,1}^{1/2}+\cC_{\Lin,3}\\
\end{split}
\end{align}
with  $\dot{\cc}_{\Lin}, \hat{\cc}_{\Lin}$ given in Proposition \ref{prop:contractionw12lip}, $\cC_{\Lin,0}, \cC_{\Lin,1}$ given in \eqref{eq:w1converlipp1const} (see also Lemma \ref{lem:w1converlipp1}), $\cC_{\Lin,2}, \cC_{\Lin,3}$ given in \eqref{eq:w1converlipp2consts} (see also Lemma \ref{lem:w1converlipp2}). The above result implies that, for each $n \in \N_0$,
\[
W_1(\mathcal{L}(\overline{\Theta}^{\lambda}_{nT}),\pi_\beta) \leq  C_{\Lin,1} e^{-C_{\Lin,0} n}(\E[|\theta_0|^4]+1) +C_{\Lin,2}\lambda^{1+q/2},
\]
which further yields, by setting $nT$ to $n$ on the LHS and $n$ to $n/T$ on the RHS, that
\begin{align*}
W_1(\mathcal{L}(\overline{\Theta}^{\lambda}_n),\pi_\beta) = W_1(\mathcal{L}(\Theta^{\lambda}_n),\pi_\beta)  &= W_1(\mathcal{L}(\Theta_n^{\cHOLLA}),\pi_{\beta})\\
& \leq   C_{\Lin,1} e^{-C_{\Lin,0} \lambda n}(\E[|\theta_0|^4]+1) +C_{\Lin,2}\lambda^{1+q/2},
\end{align*}
where the inequality holds due to $n\lambda \leq n/T$. This completes the proof.
\end{proof}

We can obtain the upper bound for $W_2(\mathcal{L}(\Theta_n^{\cHOLLA}),\pi_{\beta}) $, $n \in \N_0$, as stated in Corollary \ref{crl:mainw2lip}, by applying similar arguments as in the proof of Theorem \ref{thm:mainw1lip}.
\begin{proof}[\textbf{Proof of Corollary \ref{crl:mainw2lip}}] To establish a non-asymptotic error bound for $W_2(\mathcal{L}(\Theta_n^{\cHOLLA}),\pi_{\beta}) $, we consider the following splitting: for any $0<\lambda\leq \overline{\lambda}_{\max}$, $n \in \N_0$, and $t \in (nT, (n+1)T]$,
\begin{equation}\label{eq:crllipsplit}
W_2(\mathcal{L}(\overline{\Theta}^{\lambda}_t),\pi_\beta) \leq W_2(\mathcal{L}(\overline{\Theta}^{\lambda}_t),\mathcal{L}(\widetilde{\zeta}^{\lambda, n}_t))+W_2(\mathcal{L}(\widetilde{\zeta}^{\lambda, n}_t),\mathcal{L}(Z^{\lambda}_t))+W_2(\mathcal{L}(Z^{\lambda}_t),\pi_\beta).
\end{equation}
An upper bound for the first term on the RHS of \eqref{eq:crllipsplit} is provided in Lemma \ref{lem:w1converlipp1}. To establish an estimate for the second term on the RHS of \eqref{eq:crllipsplit}, we use $W_2 \leq \sqrt{2w_{1,2}}$ (see \cite[Lemma A.3]{lim2021nonasymptotic} for the proof) and follow the same arguments as that in the proof of \cite[Lemma 4.7]{lim2021nonasymptotic}. Consequently, for any $0<\lambda\leq \overline{\lambda}_{\max}$, $n \in \N_0$, and $t \in (nT, (n+1)T]$, we obtain that,
\begin{equation}\label{eq:w2converlipp2} 
W_2(\mathcal{L}(\widetilde{\zeta}_t^{\lambda,n}),\mathcal{L}(Z_t^\lambda)) \leq \lambda^{1/2+q/4}\left(e^{-  \min\{\dot{\cc}_{\Lin},\overline{a}/2\}  n/8}\cC_{\Lin,4}\E^{1/2}[|\theta_0|^4] +\cC_{\Lin,5} \right) ,
\end{equation}
where
\begin{align}\label{eq:w2converlipp2consts}
\begin{split}
\cC_{\Lin,4}
& :=\sqrt{ \hat{\cc}_{\Lin}}\left(1+\frac{8}{\min\{\dot{\cc}_{\Lin},\overline{a}/2\} }\right)e^{\min\{\dot{\cc}_{\Lin},\overline{a}/2\} /8}\left( \cC_{\Lin,0}^{1/2} +3\right),\\
\cC_{\Lin,5}
& :=  4(\sqrt{\hat{\cc}_{\Lin}}/\dot{\cc}_{\Lin})e^{\dot{\cc}_{\Lin}/4}\left( \cC_{\Lin,1}^{1/2}+1+3(\ccl_2(1+1/\overline{a})+1)^{1/2}+\sqrt{3}\mathrm{v}_4^{1/2}(\cMl_V(4))\right)
\end{split}
\end{align}
with $\dot{\cc}_{\Lin}, \hat{\cc}_{\Lin}$ given in Proposition \ref{prop:contractionw12lip}, $\cC_{\Lin,0}, \cC_{\Lin,1}$ given in \eqref{eq:w1converlipp1const} (see also Lemma \ref{lem:w1converlipp1}), $\ccl_2$ given in Lemma \ref{lem:2ndpthmmtlip}, and $\cMl_V(4)$ given in Lemma \ref{lem:zetaprocmelip}. An upper bound for the last term on the RHS of  \eqref{eq:crllipsplit} can be obtained by using $W_2 \leq \sqrt{2w_{1,2}}$ and Proposition \ref{prop:contractionw12lip}:
\begin{align}\label{eq:w2converlipp3} 
W_2(\mathcal{L}(Z^{\lambda}_t),\pi_\beta) 
&\leq  \sqrt{2\hat{\cc}_{\Lin}} e^{-\dot{\cc}_{\Lin} \lambda t/2} w_{1,2}^{1/2}(\mathcal{L}(\theta_0), \pi_\beta)\nonumber\\
&\leq  \sqrt{2\hat{\cc}_{\Lin}} e^{-\dot{\cc}_{\Lin} \lambda t/2}\left[1+\E[V_2(\theta_0)]+ \int_{\R^d} V_2(\theta) \pi_{\beta}(\rmd \theta)\right]^{1/2}. 
\end{align}
Applying the results in Lemma \ref{lem:w1converlipp1}, \eqref{eq:w2converlipp2}, \eqref{eq:w2converlipp3} to \eqref{eq:crllipsplit} yields, for any $0<\lambda\leq \overline{\lambda}_{\max}$, $n \in \N_0$, and $t \in (nT, (n+1)T]$, that
\begin{align*}
W_2(\mathcal{L}(\overline{\Theta}^{\lambda}_t),\pi_\beta) 
&\leq \lambda^{1+q/2}\left(e^{- \overline{a} n/4}  \cC_{\Lin,0}\E[|\theta_0|^4] + \cC_{\Lin,1} \right)^{1/2}\\
&\quad + \lambda^{1/2+q/4}\left(e^{-  \min\{\dot{\cc}_{\Lin},\overline{a}/2\}  n/8}\cC_{\Lin,4}\E^{1/2}[|\theta_0|^4] +\cC_{\Lin,5} \right) \\
&\quad +\sqrt{2\hat{\cc}_{\Lin}} e^{-\dot{\cc}_{\Lin} \lambda t/2}\left[1+\E[V_2(\theta_0)]+ \int_{\R^d} V_2(\theta) \pi_{\beta}(\rmd \theta)\right]^{1/2}\\
&\leq  C_{\Lin,4} e^{-C_{\Lin,3}( n+1)}(\E[|\theta_0|^4]+1)^{1/2} +C_{\Lin,5}\lambda^{1/2+q/4},
\end{align*}
where 
\begin{align}\label{eq:mainw2lipconst}
\begin{split}
C_{\Lin,3}&:= \min\{\dot{\cc}_{\Lin},\overline{a}/2\}  /8, \\
C_{\Lin,4} &:=e^{ \min\{\dot{\cc}_{\Lin},\overline{a}/2\}   /8}\left[\cC_{\Lin,0}^{1/2}+\cC_{\Lin,4}+\sqrt{2\hat{\cc}_{\Lin}}\left(3+\int_{\R^d} V_2(\theta) \pi_{\beta}(\rmd \theta)\right)^{1/2}\right], \\
C_{\Lin,5}&:=\cC_{\Lin,1}^{1/2}+\cC_{\Lin,5}\\
\end{split}
\end{align}
with  $\dot{\cc}_{\Lin}, \hat{\cc}_{\Lin}$ given in Proposition \ref{prop:contractionw12lip}, $\cC_{\Lin,0}, \cC_{\Lin,1}$ given in \eqref{eq:w1converlipp1const} (see also Lemma \ref{lem:w1converlipp1}), $\cC_{\Lin,4}, \cC_{\Lin,5}$ given in \eqref{eq:w2converlipp2consts}. This further implies that, for each $n \in \N_0$,
\begin{align*}
W_2(\mathcal{L}(\overline{\Theta}^{\lambda}_n),\pi_\beta) = W_2(\mathcal{L}(\Theta^{\lambda}_n),\pi_\beta)  
&= W_2(\mathcal{L}(\Theta_n^{\cHOLLA}),\pi_{\beta}) \\
&\leq  C_{\Lin,4} e^{-C_{\Lin,3}\lambda n}(\E[|\theta_0|^4]+1)^{1/2} +C_{\Lin,5}\lambda^{1/2+q/4},
\end{align*}
which completes the proof.
\end{proof}


\subsection{Proof of auxiliary results in Appendices \ref{appen:melip} and \ref{appen:mtlipdetalis}}\label{appen:mtlinearproofsapd} 

\begin{lemma}\label{lem:sde2pmmtlip} Let Assumption 
 \ref{asm:ADlip} hold. Then, for any $p \in {\N}, t \geq 0$, we obtain that
\[
\E[|Z^{\lambda}_t|^{2p}] \leq e^{- \lambda p \overline{a} t }\E[|\theta_0|^{2p}]+2(\overline{b}+\beta^{-1}(d+2(p-1))) \cMl_0^{2p-2}/ \overline{a}<\infty,
\]
where $\cMl_0 :=(2(\overline{b}+\beta^{-1}(d+2(p-1)))/ \overline{a})^{1/2}$. 
\end{lemma}
\begin{proof}See \cite[Lemma A.1]{lim2021nonasymptotic}.
\end{proof}

\begin{proof}[\textbf{Proof of Lemma \ref{lem:2ndpthmmtlip}-\ref{lem:2ndpthmmtlipi}}]
\label{lem:2ndpthmmtlipproof}
For any $0<\lambda\leq \overline{\lambda}_{\max}\leq 1$ with $\overline{\lambda}_{\max}$  given in \eqref{eq:stepsizemaxlip}, $t\in (n, n+1]$, $n \in \N_0$, we define
\begin{equation}\label{eq:delxinotationlip} 
\Deltal_{n,t}^{\lambda}
 := \widetilde{\Theta}^{\lambda}_n +  \lambda \phi^{\lambda}_{\Lin}(\widetilde{\Theta}^{\lambda}_n) (t-n),\quad
\Xil_{n,t}^{\lambda}
 :=  \sqrt{2\lambda\beta^{-1}}\psi^{\lambda}_{\Lin}(\widetilde{\Theta}^{\lambda}_n) (B_t^{\lambda} - B_n^{\lambda}),
\end{equation}
where for all $\theta \in \R^d$, 
\begin{equation}\label{eq:delxinotationlipphi}
\phi^{\lambda}_{\Lin}(\theta):= - h( \theta) +(\lambda/2) \left(H ( \theta)h ( \theta)-\beta^{-1}\Upsilon (\theta) \right),
\end{equation}
and 
\begin{equation}\label{eq:delxinotationlippsi}
\psi^{\lambda}_{\Lin}(\theta) :=\sqrt{\cI_d - \lambda H (\theta)+(\lambda^2/3) (H (\theta))^2}.
\end{equation}
Then, by using \eqref{eq:aholaproclip}, \eqref{eq:delxinotationlip} -- \eqref{eq:delxinotationlippsi}, and by noticing $\E\left[\left.\langle \Deltal_{n,t}^{\lambda}, \Xil_{n,t}^{\lambda}\rangle \right|\widetilde{\Theta}^{\lambda}_n \right] =0$, we have that
\begin{equation}\label{eq:2ndmmtlipexp}
\E\left[\left.|\widetilde{\Theta}^{\lambda}_t|^2\right|\widetilde{\Theta}^{\lambda}_n \right]  = |\Deltal_{n,t}^{\lambda}|^2+\E\left[\left.|\Xil_{n,t}^{\lambda}|^2\right|\widetilde{\Theta}^{\lambda}_n \right].
\end{equation}
The second term on the RHS of \eqref{eq:2ndmmtlipexp} can be further upper bounded as follows:
\begin{align}\label{eq:xilipubfin}
\E\left[\left.|\Xil_{n,t}^{\lambda}|^2\right|\widetilde{\Theta}^{\lambda}_n \right] 
&= 2\lambda\beta^{-1}\E\left[\left. \left\langle \psi^{\lambda}_{\Lin}(\widetilde{\Theta}^{\lambda}_n) (B_t^{\lambda} - B_n^{\lambda}),  \psi^{\lambda}_{\Lin}(\widetilde{\Theta}^{\lambda}_n) (B_t^{\lambda} - B_n^{\lambda})\right\rangle\right|\widetilde{\Theta}^{\lambda}_n \right] \nonumber\\
& =  2\lambda\beta^{-1}\E\left[\left. \left\langle    B_t^{\lambda} - B_n^{\lambda} ,  \left(\cI_d - \lambda H(\widetilde{\Theta}^{\lambda}_n)+(\lambda^2/3) (H(\widetilde{\Theta}^{\lambda}_n))^2\right)  (B_t^{\lambda} - B_n^{\lambda})\right\rangle\right|\widetilde{\Theta}^{\lambda}_n \right]  \nonumber\\
& \leq  2\lambda\beta^{-1}\left(d(t-n)+\lambda \overline{L}_3  d(t-n)+(\lambda^2/3) \overline{L}_3^2 d(t-n)\right)\nonumber\\
&\leq 2\lambda\beta^{-1}(t-n)d(1+ \overline{L}_3)^2,
\end{align}
where the first inequality holds due to Remark \ref{rmk:growthclip}. Next, to upper bound the first term on the RHS of \eqref{eq:2ndmmtlipexp}, we use \eqref{eq:delxinotationlip} to obtain
\begin{align}\label{eq:deltalipub1}
|\Deltal_{n,t}^{\lambda}|^2
& = |\widetilde{\Theta}^{\lambda}_n|^2+2\lambda(t-n)\left\langle \widetilde{\Theta}^{\lambda}_n, \phi^{\lambda}_{\Lin}(\widetilde{\Theta}^{\lambda}_n) \right\rangle+ \lambda^2(t-n)^2|\phi^{\lambda}_{\Lin}(\widetilde{\Theta}^{\lambda}_n) |^2.
\end{align}
By using \eqref{eq:delxinotationlipphi}, the second term on the RHS of \eqref{eq:deltalipub1} can be estimated as follows:
\begin{align} \label{eq:deltalipubcp}
\left\langle \widetilde{\Theta}^{\lambda}_n, \phi^{\lambda}_{\Lin}(\widetilde{\Theta}^{\lambda}_n) \right\rangle
&= - \left\langle \widetilde{\Theta}^{\lambda}_n, h(\widetilde{\Theta}^{\lambda}_n) \right\rangle +(\lambda/2)\left\langle \widetilde{\Theta}^{\lambda}_n, H(\widetilde{\Theta}^{\lambda}_n) h(\widetilde{\Theta}^{\lambda}_n) \right\rangle -(\lambda/2)\beta^{-1}\left\langle \widetilde{\Theta}^{\lambda}_n, \Upsilon(\widetilde{\Theta}^{\lambda}_n) \right\rangle \nonumber\\
&\leq -\overline{a}|\widetilde{\Theta}^{\lambda}_n|^2+\overline{b}+\lambda\overline{L}_3\cKl_1 (1+|\widetilde{\Theta}^{\lambda}_n|^2)+(\lambda/2)\beta^{-1}d\overline{L}_2|\widetilde{\Theta}^{\lambda}_n|,
\end{align}
where the last inequality holds due to Assumption \ref{asm:ADlip}, Remark \ref{rmk:growthclip} and the following calculations: for all $\theta \in \R^d$, 
\[
|\theta||H(\theta)h(\theta)|
\leq \overline{L}_3|\theta||h(\theta)| =\overline{L}_3\cKl_1(|\theta|+|\theta|^2)\leq 2\overline{L}_3\cKl_1(1+|\theta|^2).
\]
To provide an upper bound for the third term on the RHS of \eqref{eq:deltalipub1}, we write by straightforward calculations that
\begin{align}\label{eq:deltalipubsq1}
|\phi^{\lambda}_{\Lin}(\widetilde{\Theta}^{\lambda}_n) |^2
& = |- h ( \widetilde{\Theta}^{\lambda}_n) +(\lambda/2) H ( \widetilde{\Theta}^{\lambda}_n)h ( \widetilde{\Theta}^{\lambda}_n)-(\lambda/2) \beta^{-1}\Upsilon (\widetilde{\Theta}^{\lambda}_n)|^2\nonumber\\
& = |h ( \widetilde{\Theta}^{\lambda}_n) |^2+(\lambda^2/4)| H ( \widetilde{\Theta}^{\lambda}_n)h ( \widetilde{\Theta}^{\lambda}_n)|^2+(\lambda^2/4)\beta^{-2}|\Upsilon (\widetilde{\Theta}^{\lambda}_n)|^2\nonumber\\
&\quad  -\lambda \left\langle h ( \widetilde{\Theta}^{\lambda}_n),H ( \widetilde{\Theta}^{\lambda}_n)h ( \widetilde{\Theta}^{\lambda}_n)\right\rangle
+\lambda\beta^{-1}\left\langle h ( \widetilde{\Theta}^{\lambda}_n),\Upsilon (\widetilde{\Theta}^{\lambda}_n)\right\rangle\nonumber\\
&\quad -(\lambda^2/2)\beta^{-1}\left\langle H ( \widetilde{\Theta}^{\lambda}_n)h ( \widetilde{\Theta}^{\lambda}_n),\Upsilon (\widetilde{\Theta}^{\lambda}_n)\right\rangle.
\end{align}
We then provide upper bounds for each of the terms on the RHS of \eqref{eq:deltalipubsq1}. By using Remark \ref{rmk:growthclip}, we obtain that
\begin{align}\label{eq:deltalipubsqterms}
\begin{split}
|h ( \widetilde{\Theta}^{\lambda}_n) |^2
&\leq 2\cKl_1^2(1+| \widetilde{\Theta}^{\lambda}_n|^2) ,\\
(\lambda^2/4)| H ( \widetilde{\Theta}^{\lambda}_n)h ( \widetilde{\Theta}^{\lambda}_n)|^2
&\leq (\lambda^2/2)\overline{L}_3^2\cKl_1^2(1+| \widetilde{\Theta}^{\lambda}_n|^2),\\
(\lambda^2/4)\beta^{-2}|\Upsilon (\widetilde{\Theta}^{\lambda}_n)|^2
&\leq (\lambda^2/4)\beta^{-2}d^2\overline{L}_2^2 ,\\
 -\lambda \left\langle h ( \widetilde{\Theta}^{\lambda}_n),H ( \widetilde{\Theta}^{\lambda}_n)h ( \widetilde{\Theta}^{\lambda}_n)\right\rangle
&\leq 2\lambda\overline{L}_3\cKl_1^2(1+| \widetilde{\Theta}^{\lambda}_n|^2),\\
\lambda\beta^{-1}\left\langle h ( \widetilde{\Theta}^{\lambda}_n),\Upsilon (\widetilde{\Theta}^{\lambda}_n)\right\rangle
&\leq \lambda\beta^{-1} d\overline{L}_2\cKl_1 (1+| \widetilde{\Theta}^{\lambda}_n| ),\\
-(\lambda^2/2)\beta^{-1}\left\langle H ( \widetilde{\Theta}^{\lambda}_n)h ( \widetilde{\Theta}^{\lambda}_n),\Upsilon (\widetilde{\Theta}^{\lambda}_n)\right\rangle
&\leq (\lambda^2/2)\beta^{-1}d\overline{L}_2\overline{L}_3\cKl_1 (1+| \widetilde{\Theta}^{\lambda}_n| ).
\end{split}
\end{align}
Substituting \eqref{eq:deltalipubsqterms} into \eqref{eq:deltalipubsq1} yields
\begin{align}
\begin{split}\label{eq:deltalipubsq2}
|\phi^{\lambda}_{\Lin}(\widetilde{\Theta}^{\lambda}_n) |^2
&\leq 2\cKl_1^2(1+| \widetilde{\Theta}^{\lambda}_n|^2) +  (\lambda^2/2)\overline{L}_3^2\cKl_1^2(1+| \widetilde{\Theta}^{\lambda}_n|^2)+(\lambda^2/4)\beta^{-2}d^2\overline{L}_2^2\\
&\quad + 2\lambda\overline{L}_3\cKl_1^2(1+| \widetilde{\Theta}^{\lambda}_n|^2)+\lambda\beta^{-1} d\overline{L}_2\cKl_1 (1+| \widetilde{\Theta}^{\lambda}_n| )+(\lambda^2/2)\beta^{-1}d\overline{L}_2\overline{L}_3\cKl_1 (1+| \widetilde{\Theta}^{\lambda}_n| ).
\end{split}
\end{align}
Combining the results in \eqref{eq:deltalipubcp} and \eqref{eq:deltalipubsq2}, we obtain the following upper bound for \eqref{eq:deltalipub1}:
\begin{align}\label{eq:deltalipub2}
|\Deltal_{n,t}^{\lambda}|^2
&\leq |\widetilde{\Theta}^{\lambda}_n|^2 -2\lambda(t-n)\overline{a}|\widetilde{\Theta}^{\lambda}_n|^2 +\lambda(t-n)\left(2\lambda \overline{L}_3\cKl_1+2\lambda \cKl_1^2+\lambda^3\overline{L}_3^2\cKl_1^2/2+2\lambda^2\overline{L}_3\cKl_1^2\right)|\widetilde{\Theta}^{\lambda}_n|^2\nonumber\\
&\quad +\lambda(t-n)\left(\beta^{-1} d\overline{L}_2 +\beta^{-1} d\overline{L}_2\cKl_1+\beta^{-1}d\overline{L}_2\overline{L}_3\cKl_1/2\right) |\widetilde{\Theta}^{\lambda}_n| +\lambda(t-n)\Big(2\overline{b}+2\overline{L}_3\cKl_1\Big.\nonumber\\
&\quad\Big.+2 \cKl_1^2+\overline{L}_3^2\cKl_1^2/2+\beta^{-2}d^2\overline{L}_2^2/4+2\overline{L}_3\cKl_1^2+\beta^{-1} d\overline{L}_2\cKl_1 +\beta^{-1}d\overline{L}_2\overline{L}_3\cKl_1/2\Big)\nonumber\\
& = (1-\lambda(t-n)\overline{a}) |\widetilde{\Theta}^{\lambda}_n|^2 -\lambda(t-n)\overline{\mathfrak{I}}^{\lambda}_1(\widetilde{\Theta}^{\lambda}_n)-\lambda(t-n)\overline{\mathfrak{I}}^{\lambda}_2(\widetilde{\Theta}^{\lambda}_n)+ \lambda(t-n)\ccl_1,
\end{align}
where, for all $\theta \in \R^d$,
\begin{align*}
\overline{\mathfrak{I}}^{\lambda}_1(\theta)
&:=(\overline{a}/2)|\theta|^2-\left(2\lambda \overline{L}_3\cKl_1+2\lambda \cKl_1^2+\lambda^3\overline{L}_3^2\cKl_1^2/2+2\lambda^2\overline{L}_3\cKl_1^2\right)|\theta|^2\\
\overline{\mathfrak{I}}^{\lambda}_2(\theta)
&:=(\overline{a}/2)|\theta|^2-\left(\beta^{-1} d\overline{L}_2 +\beta^{-1} d\overline{L}_2\cKl_1+\beta^{-1}d\overline{L}_2\overline{L}_3\cKl_1/2\right) |\theta|,
\end{align*}
and where $\ccl_1:= 2\overline{b}+2\overline{L}_3\cKl_1+2 \cKl_1^2+\overline{L}_3^2\cKl_1^2/2+\beta^{-2}d^2\overline{L}_2^2/4+2\overline{L}_3\cKl_1^2+\beta^{-1} d\overline{L}_2\cKl_1 +\beta^{-1}d\overline{L}_2\overline{L}_3\cKl_1/2$. We note that, for all $\theta \in \R^d$, $0<\lambda\leq \overline{\lambda}_{\max}\leq \min\{\overline{a}/(16\overline{L}_3\cKl_1),\overline{a}/(16\cKl_1^2), \overline{a}^{1/3}/(4\overline{L}_3^2\cKl_1^2)^{1/3},\allowbreak \overline{a}^{1/2}/(16\overline{L}_3\cKl_1^2)^{1/3}\}$,
\begin{equation}\label{eq:deltalipub3}
\overline{\mathfrak{I}}^{\lambda}_1(\theta) =\left((\overline{a}/8-2\lambda \overline{L}_3\cKl_1)+(\overline{a}/8-2\lambda \cKl_1^2)+(\overline{a}/8-\lambda^3\overline{L}_3^2\cKl_1^2/2)+(\overline{a}/8-2\lambda^2\overline{L}_3\cKl_1^2)\right)|\theta|^2\geq 0.
\end{equation}
Substituting \eqref{eq:deltalipub3} into \eqref{eq:deltalipub2} yields
\[
|\Deltal_{n,t}^{\lambda}|^2 \leq  (1-\lambda(t-n)\overline{a}) |\widetilde{\Theta}^{\lambda}_n|^2-\lambda(t-n)\overline{\mathfrak{I}}^{\lambda}_2(\widetilde{\Theta}^{\lambda}_n)+ \lambda(t-n)\ccl_1.
\]
Denote by $\cMl_1: = 2(\overline{a}\beta)^{-1} d(\overline{L}_2 +\overline{L}_2\cKl_1+ \overline{L}_2\overline{L}_3\cKl_1/2)$ and $\cSl_{n,\cMl_1}:=\{\omega \in \Omega:|\widetilde{\Theta}^{\lambda}_n(\omega)|>\cMl_1\}$. By observing the fact that
\[
\overline{\mathfrak{I}}^{\lambda}_2(\theta) > 0 \qquad \Longleftrightarrow \qquad |\theta|>\cMl,
\]
we obtain the following:
\[
|\Deltal_{n,t}^{\lambda}|^2 \mathbbm{1}_{\cSl_{n,\cMl_1}}  \leq  (1-\lambda(t-n)\overline{a}) |\widetilde{\Theta}^{\lambda}_n|^2\mathbbm{1}_{\cSl_{n,\cMl_1}} + \lambda(t-n)\ccl_1\mathbbm{1}_{\cSl_{n,\cMl_1}}.
\]
Similarly, we have that
\begin{align*}
|\Deltal_{n,t}^{\lambda}|^2 \mathbbm{1}_{\cSl_{n,\cMl_1}^{\cc}}
&\leq  (1-\lambda(t-n)\overline{a}) |\widetilde{\Theta}^{\lambda}_n|^2\mathbbm{1}_{\cSl_{n,\cMl_1}^{\cc}}+ \lambda(t-n)\ccl_1\mathbbm{1}_{\cSl_{n,\cMl_1}^{\cc}}\\
&\quad +\lambda(t-n)\left(\beta^{-1} d\overline{L}_2 +\beta^{-1} d\overline{L}_2\cKl_1+\beta^{-1}d\overline{L}_2\overline{L}_3\cKl_1/2\right) \cMl_1\mathbbm{1}_{\cSl_{n,\cMl_1}^{\cc}}.
\end{align*}
Combining the two cases yields
\begin{align}\label{eq:deltalipubfin}
\begin{split}
|\Deltal_{n,t}^{\lambda}|^2 
&\leq  (1-\lambda(t-n)\overline{a}) |\widetilde{\Theta}^{\lambda}_n|^2+ \lambda(t-n)\ccl_1  +\lambda(t-n)\beta^{-1} d\left(\overline{L}_2 + \overline{L}_2\cKl_1+ \overline{L}_2\overline{L}_3\cKl_1/2\right) \cMl_1.
\end{split}
\end{align}
Finally, by substituting \eqref{eq:xilipubfin} and \eqref{eq:deltalipubfin} into \eqref{eq:2ndmmtlipexp}, we obtain
\begin{equation}\label{eq:2ndmmtlipexpfin}
\E\left[\left.|\widetilde{\Theta}^{\lambda}_t|^2\right|\widetilde{\Theta}^{\lambda}_n \right]
\leq  (1-\lambda(t-n)\overline{a}) |\widetilde{\Theta}^{\lambda}_n|^2+ \lambda(t-n)\ccl_0,
\end{equation}
where
\begin{align}
\begin{split}\label{eq:2ndmmtlipexpconst} 
\ccl_0 &: =\beta^{-1} d\left( \overline{L}_2 + \overline{L}_2\cKl_1+ \overline{L}_2\overline{L}_3\cKl_1/2\right) \cMl_1 +\ccl_1+2 \beta^{-1} d(1+ \overline{L}_3)^2 , \\
\ccl_1&:=2\overline{b}+2\overline{L}_3\cKl_1+2 \cKl_1^2+\overline{L}_3^2\cKl_1^2/2+\beta^{-2}d^2\overline{L}_2^2/4+2\overline{L}_3\cKl_1^2+\beta^{-1} d(\overline{L}_2\cKl_1 +\overline{L}_2\overline{L}_3\cKl_1/2),\\
\cMl_1&: = 2(\overline{a}\beta)^{-1} d(\overline{L}_2 +\overline{L}_2\cKl_1+ \overline{L}_2\overline{L}_3\cKl_1/2).
\end{split}
\end{align}
We observe that, for $0<\lambda\leq \overline{\lambda}_{\max}\leq 1/\overline{a}$,
\[
1>1-\lambda(t-n)\overline{a} > 1-\lambda \overline{a}\geq 0,
\]
then, by induction, \eqref{eq:2ndmmtlipexpfin} implies, for $t\in (n, n+1]$, $n \in \N_0$, $0<\lambda\leq \overline{\lambda}_{\max}\leq 1$, that, 
\begin{align*}
\E\left[ |\widetilde{\Theta}^{\lambda}_t|^2 \right]  
&\leq  (1-\lambda(t-n)\overline{a}) \E\left[ |\widetilde{\Theta}^{\lambda}_n|^2\right]+ \lambda(t-n)\ccl_0\\
&\leq  \left(1 -\lambda(t-n)\overline{a} \right)\left(1 -\lambda \overline{a} \right)\E\left[ |\widetilde{\Theta}^{\lambda}_{n-1}|^2\right]+  \ccl_0+ \lambda \ccl_0\\
&\leq   \left(1 -\lambda(t-n)\overline{a} \right)\left(1 -\lambda \overline{a} \right)^2\E\left[ |\widetilde{\Theta}^{\lambda}_{n-2}|^2\right]+  \ccl_0+ \lambda \ccl_0\left(1+\left(1 -\lambda \overline{a} \right)\right)\\
&\leq \dots\\
&\leq \left(1 -\lambda(t-n)\overline{a} \right)\left(1 -\lambda \overline{a} \right)^n\E\left[ |\theta_0|^2\right]+  \ccl_0\left(1+1/\overline{a}\right),
\end{align*}
which completes the proof.
\end{proof}

\begin{proof}[\textbf{Proof of Lemma \ref{lem:2ndpthmmtlip}-\ref{lem:2ndpthmmtlipii}}] 
For any $p\in [2, \infty)\cap {\N}$, $0<\lambda\leq \overline{\lambda}_{\max}\leq 1$ with $\overline{\lambda}_{\max}$  given in \eqref{eq:stepsizemaxlip}, $t\in (n, n+1]$, $n \in \N_0$, by using the same arguments as in the proof of \cite[Lemma 4.2-(ii)]{lim2021nonasymptotic} up to the inequality before \cite[Eq. (134)]{lim2021nonasymptotic} and by using \eqref{eq:aholaproclip} with \eqref{eq:delxinotationlip}, we obtain that
\begin{align}
\begin{split}\label{eq:2pthmmtlipexp}
\E\left[\left.|\widetilde{\Theta}^{\lambda}_t|^{2p}\right|\widetilde{\Theta}^{\lambda}_n \right] 
&\leq |\Deltal_{n,t}^{\lambda}|^{2p} +2^{2p-3}p(2p-1)|\Deltal_{n,t}^{\lambda}|^{2p-2}  \E\left[\left.|\Xil_{n,t}^{\lambda}|^2\right|\widetilde{\Theta}^{\lambda}_n \right]  \\
&\quad + 2^{2p-3}p(2p-1)\E\left[\left.|\Xil_{n,t}^{\lambda}|^{2p}\right|\widetilde{\Theta}^{\lambda}_n \right].
\end{split}
\end{align}
Then, by using \eqref{eq:delxinotationlip} and \eqref{eq:delxinotationlippsi}, we can obtain an upper estimate for the last term in \eqref{eq:2pthmmtlipexp} as follows:
\begin{align}\label{eq:xi2plipub} 
\E\left[\left.|\Xil_{n,t}^{\lambda}|^{2p}\right|\widetilde{\Theta}^{\lambda}_n \right]
& = (2\lambda\beta^{-1})^p \E\left[\left. \left\langle  (B_t^{\lambda} - B_n^{\lambda}), 
\left(\psi^{\lambda}_{\Lin}(\widetilde{\Theta}^{\lambda}_n)\right)^2 (B_t^{\lambda} - B_n^{\lambda})\right\rangle^p\right|\widetilde{\Theta}^{\lambda}_n \right] \nonumber\\
&\leq (2\lambda\beta^{-1})^p \E\left[\left. \left(  |B_t^{\lambda} - B_n^{\lambda}|^2+ \lambda\overline{L}_3 |B_t^{\lambda} - B_n^{\lambda}|^2+(\lambda^2/3) \overline{L}_3^2 |B_t^{\lambda} - B_n^{\lambda}|^2 
\right)^p\right|\widetilde{\Theta}^{\lambda}_n \right] \nonumber\\
&\leq  (2\lambda\beta^{-1})^p(1+\overline{L}_3)^{2p}\E\left[|B_t^{\lambda} - B_n^{\lambda}|^{2p}\right]\nonumber\\
&\leq  (2\lambda\beta^{-1}dp(2p-1)(t-n))^p(1+\overline{L}_3)^{2p}\nonumber\\
&\leq \lambda(t-n) (2\beta^{-1}dp(2p-1)(1+\overline{L}_3)^2)^p.
\end{align} 
Substituting \eqref{eq:xilipubfin} and \eqref{eq:xi2plipub} into \eqref{eq:2pthmmtlipexp} yields
\begin{align}\label{eq:2pthmmtlipexpub1}
\begin{split}
\E\left[\left.|\widetilde{\Theta}^{\lambda}_t|^{2p}\right|\widetilde{\Theta}^{\lambda}_n \right] 
&\leq |\Deltal_{n,t}^{\lambda}|^{2p} +\lambda(t-n) 2^{2p-2}p(2p-1) \beta^{-1} d(1+\overline{L}_3)^2|\Deltal_{n,t}^{\lambda}|^{2p-2} + \lambda(t-n) \ccl_{\Xil}(p),
\end{split}
\end{align}
where $\ccl_{\Xil}(p):= 2^{2p-3}p(2p-1)  (2\beta^{-1}dp(2p-1)(1+\overline{L}_3)^2)^p$. Next, we apply \eqref{eq:deltalipubfin} to obtain
\begin{align}\label{eq:2pthmmtlipexpub2}
|\Deltal_{n,t}^{\lambda}|^{2p} 
&\leq \left( (1-\lambda(t-n)\overline{a}) |\widetilde{\Theta}^{\lambda}_n|^2+ \lambda(t-n)\ccl_1  +\lambda(t-n)\beta^{-1} d\left(\overline{L}_2 + \overline{L}_2\cKl_1+ \overline{L}_2\overline{L}_3\cKl_1/2\right) \cMl_1\right)^p\nonumber\\
&\leq \left(1+\lambda(t-n)\overline{a} /2\right)^{p-1}\left(1 -\lambda(t-n)\overline{a} \right)^p|\widetilde{\Theta}^{\lambda}_n|^{2p}\nonumber\\
&\quad +\left(1+2/(\lambda(t-n)\overline{a} )\right)^{p-1}(\lambda(t-n))^p(\ccl_1+\beta^{-1} d\left(\overline{L}_2 + \overline{L}_2\cKl_1+ \overline{L}_2\overline{L}_3\cKl_1/2\right) \cMl_1)^p\nonumber\\
&\leq \left(1-\lambda(t-n)\overline{a} /2\right)^{p-1}\left(1 -\lambda(t-n)\overline{a} \right)|\widetilde{\Theta}^{\lambda}_n|^{2p}\nonumber\\
&\quad +\lambda(t-n)\left(1+2/\overline{a}\right)^{p-1}(\ccl_1+\beta^{-1} d\left(\overline{L}_2 + \overline{L}_2\cKl_1+ \overline{L}_2\overline{L}_3\cKl_1/2\right) \cMl_1)^p\nonumber\\
&=\overline{\cc}_{n,t}^\lambda(p)|\widetilde{\Theta}^{\lambda}_n|^{2p}+\widetilde{\cc}_{n,t}^\lambda(p),
\end{align}
where the second inequality holds due to $(u+v)^p\leq (1+\varepsilon)^{p-1}u^p+(1+\varepsilon^{-1})^{p-1}v^p$, $u,v \geq 0$, $\varepsilon>0$ with $\varepsilon = \lambda(t-n)\overline{a} /2$, and where
\begin{align*}
\overline{\cc}_{n,t}^\lambda(p) &:=  \left(1-\lambda(t-n)\overline{a} /2\right)^{p-1}\left(1 -\lambda(t-n)\overline{a}  \right),\\
\widetilde{\cc}_{n,t}^\lambda(p)&:=\lambda(t-n)\left(1+2/\overline{a}\right)^{p-1}(\ccl_1+\beta^{-1} d\left(\overline{L}_2 + \overline{L}_2\cKl_1+ \overline{L}_2\overline{L}_3\cKl_1/2\right) \cMl_1)^p.
\end{align*}
In addition, we observe that by \eqref{eq:2pthmmtlipexpub2}, 
\begin{equation}\label{eq:2pthmmtlipexpub3}
|\Deltal_{n,t}^{\lambda}|^{2p-2}\leq  \overline{\cc}_{n,t}^\lambda(p-1)|\widetilde{\Theta}^{\lambda}_n|^{2p-2}+\widetilde{\cc}_{n,t}^\lambda(p-1),
\end{equation}
and, in particular, when $p = 2$, \eqref{eq:2pthmmtlipexpub3} yields $|\Deltal_{n,t}^{\lambda}|^2 \leq  \overline{\cc}_{n,t}^\lambda(1)|\widetilde{\Theta}^{\lambda}_n|^{2}+\widetilde{\cc}_{n,t}^\lambda(1)$ which is exactly the upper bound \eqref{eq:deltalipubfin}. By substituting \eqref{eq:2pthmmtlipexpub2} and \eqref{eq:2pthmmtlipexpub3} into \eqref{eq:2pthmmtlipexpub1}, we have that
\begin{align}\label{eq:2pthmmtlipexpub4}
\begin{split}
\E\left[\left.|\widetilde{\Theta}^{\lambda}_t|^{2p}\right|\widetilde{\Theta}^{\lambda}_n \right] 
&\leq \overline{\cc}_{n,t}^\lambda(p)|\widetilde{\Theta}^{\lambda}_n|^{2p}+\widetilde{\cc}_{n,t}^\lambda(p) + \lambda(t-n) \ccl_{\Xil}(p)\\
&\quad +\lambda(t-n) 2^{2p-2}p(2p-1) \beta^{-1} d(1+\overline{L}_3)^2 \left(\overline{\cc}_{n,t}^\lambda(p-1)|\widetilde{\Theta}^{\lambda}_n|^{2p-2}+\widetilde{\cc}_{n,t}^\lambda(p-1)\right).
\end{split}
\end{align}
Denote by $\cMl_2(p):=(2^{2p}p(2p-1) \beta^{-1} d(1+\overline{L}_3)^2/\overline{a})^{1/2}$. For all $|\theta|>\cMl_2(p)$, we have that
\[
(\lambda(t-n)\overline{a} /4)|\theta|^{2p}>(\lambda(t-n)2^{2p}p(2p-1) \beta^{-1} d(1+\overline{L}_3)^2/4)|\theta|^{2p-2}.
\]
Denote by $\cSl_{n,\cMl_2(p)}:=\{\omega \in \Omega:|\widetilde{\Theta}^{\lambda}_n(\omega)|>\cMl_2(p)\}$. By using the above inequality, \eqref{eq:2pthmmtlipexpub4} can be further bounded as follows:
\begin{align}\label{eq:2pthmmtlipexpub5}
&\E\left[\left.|\widetilde{\Theta}^{\lambda}_t|^{2p}\mathbbm{1}_{\cSl_{n,\cMl_2(p)}} \right|\widetilde{\Theta}^{\lambda}_n \right] \nonumber\\
&\leq \left(1-\lambda(t-n)\overline{a} /4\right)\overline{\cc}_{n,t}^\lambda(p-1)|\widetilde{\Theta}^{\lambda}_n|^{2p}\mathbbm{1}_{\cSl_{n,\cMl_2(p)}}  +\left(\widetilde{\cc}_{n,t}^\lambda(p) + \lambda(t-n) \ccl_{\Xil}(p)\right)\mathbbm{1}_{\cSl_{n,\cMl_2(p)}} \nonumber\\
&\quad +\lambda(t-n) 2^{2p-2}p(2p-1) \beta^{-1} d(1+\overline{L}_3)^2\widetilde{\cc}_{n,t}^\lambda(p-1)\mathbbm{1}_{\cSl_{n,\cMl_2(p)}}\nonumber \\
&\quad -\left(\lambda(t-n)\overline{a} /4\right)\overline{\cc}_{n,t}^\lambda(p-1)|\widetilde{\Theta}^{\lambda}_n|^{2p}\mathbbm{1}_{\cSl_{n,\cMl_2(p)}}\nonumber \\
&\quad +(\lambda(t-n) 2^{2p}p(2p-1) \beta^{-1} d(1+\overline{L}_3)^2/4)\overline{\cc}_{n,t}^\lambda(p-1)|\widetilde{\Theta}^{\lambda}_n|^{2p-2}\mathbbm{1}_{\cSl_{n,\cMl_2(p)}}\nonumber \\
&\leq \left(1 -\lambda(t-n)\overline{a} \right)|\widetilde{\Theta}^{\lambda}_n|^{2p}\mathbbm{1}_{\cSl_{n,\cMl_2(p)}} +\lambda(t-n)\left[\ccl_{\Xil}(p)+\left(1+2/\overline{a}\right)^{p-1}\right.\nonumber\\
&\quad \left. \times (\ccl_1+\beta^{-1} d\left(\overline{L}_2 + \overline{L}_2\cKl_1+ \overline{L}_2\overline{L}_3\cKl_1/2\right) \cMl_1)^p +2^{2p-2}p(2p-1) \beta^{-1} d(1+\overline{L}_3)^2\right.\nonumber\\
&\qquad \left. \times\left(\left(1+2/\overline{a}\right)^{p-2}(\ccl_1+\beta^{-1} d\left(\overline{L}_2 + \overline{L}_2\cKl_1+ \overline{L}_2\overline{L}_3\cKl_1/2\right) \cMl_1)^{p-1}+\cMl_2(p)^{2p-2}\right)\right]\mathbbm{1}_{\cSl_{n,\cMl_2(p)}} \nonumber\\
&\leq \left(1 -\lambda(t-n)\overline{a} \right)|\widetilde{\Theta}^{\lambda}_n|^{2p}\mathbbm{1}_{\cSl_{n,\cMl_2(p)}} +\lambda(t-n)\ccl_p \mathbbm{1}_{\cSl_{n,\cMl_2(p)}},
\end{align}
where
\begin{align}\label{eq:2pthmmtlipexpconst} 
\begin{split}
\ccl_p 
&:=\ccl_{\Xil}(p)+\left(1+2/\overline{a}\right)^{p-1}  (\ccl_1+\beta^{-1} d\left(\overline{L}_2 + \overline{L}_2\cKl_1+ \overline{L}_2\overline{L}_3\cKl_1/2\right) \cMl_1)^p\\
&\quad  +2^{2p-2}p(2p-1) \beta^{-1} d(1+\overline{L}_3)^2 \\
&\qquad  \times\left(\left(1+2/\overline{a}\right)^{p-2}(\ccl_1+\beta^{-1} d\left(\overline{L}_2 + \overline{L}_2\cKl_1+ \overline{L}_2\overline{L}_3\cKl_1/2\right) \cMl_1)^{p-1}+\cMl_2(p)^{2p-2}\right),\\
\ccl_{\Xil}(p)
&:= 2^{2p-3}p(2p-1)  (2\beta^{-1}dp(2p-1)(1+\overline{L}_3)^2)^p,\\
\cMl_2(p)
&:=(2^{2p}p(2p-1) \beta^{-1} d(1+\overline{L}_3)^2/\overline{a})^{1/2}
\end{split}
\end{align}
with  $\ccl_1$ given in \eqref{eq:2ndmmtlipexpconst}. Similarly, by using \eqref{eq:2pthmmtlipexpub4}, we have that
\begin{align}\label{eq:2pthmmtlipexpub6}
\E\left[\left.|\widetilde{\Theta}^{\lambda}_t|^{2p}\mathbbm{1}_{\cSl_{n,\cMl_2(p)}^{\cc}} \right|\widetilde{\Theta}^{\lambda}_n \right]\leq \left(1 -\lambda(t-n)\overline{a} \right)|\widetilde{\Theta}^{\lambda}_n|^{2p}\mathbbm{1}_{\cSl_{n,\cMl_2(p)}^{\cc}} +\lambda(t-n)\ccl_p \mathbbm{1}_{\cSl_{n,\cMl_2(p)}^{\cc}}.
\end{align}
Combing \eqref{eq:2pthmmtlipexpub5} and \eqref{eq:2pthmmtlipexpub6} yields the desired result.
\end{proof}


\begin{lemma}\label{lem:oserroralglip} Let Assumptions \ref{asm:AIlip}, \ref{asm:ALLlip}, and \ref{asm:ADlip} hold. Then, for any $0<\lambda\leq \overline{\lambda}_{\max}$, $p > 0, t \geq 0$, we obtain
\begin{align}
\E\left[|\overline{\Theta}^{\lambda}_t - \overline{\Theta}^{\lambda}_{\lfrf{t}} |^{2p}\right]
&\leq \lambda^p\left(e^{-\lambda\overline{a} \lfrf{t}}\overline{\cC}_{\mathbf{S}0,p}\E[|\theta_0|^{2\lcrc{p}}]+\widetilde{\cC}_{\mathbf{S}0,p} \right),\label{lem:oserroralglipineq1}\\
\E\left[|\widetilde{\zeta}^{\lambda, n}_t - \widetilde{\zeta}^{\lambda, n}_{\lfrf{t}}|^{2p}\right]
&\leq \lambda^p\left(e^{-  \lambda \overline{a}  \lfrf{t}/2}\overline{\cC}_{\mathbf{S}1,p}\E[|\theta_0|^{2\lcrc{p}}]+\widetilde{\cC}_{\mathbf{S}1,p} \right),\label{lem:oserroralglipineq2}
\end{align}
where $\overline{\cC}_{\mathbf{S}0,p}$ and $\widetilde{\cC}_{\mathbf{S}0,p}$ are given in \eqref{eq:oserroralglipconst}, and $\overline{\cC}_{\mathbf{S}1,p}$ and $\widetilde{\cC}_{\mathbf{S}1,p}$ are given in \eqref{eq:oserrorauxlipconst}.
\end{lemma}
\begin{proof} To show that \eqref{lem:oserroralglipineq1} holds, we use the definition of $(\overline{\Theta}^{\lambda}_t)_{t \geq 0 }$ given in \eqref{eq:aholahoproclip} and obtain that, for any $t \geq 0$,
\begin{align*}
&\E\left[|\overline{\Theta}^{\lambda}_t - \overline{\Theta}^{\lambda}_{\lfrf{t}} |^{2p}\right] \\
&= \E\left[\left|-\lambda \int_{\lfrf{t}}^t h (\overline{\Theta}^{\lambda}_{\lfrf{s}})\,\rmd s+\lambda^2\int_{\lfrf{t}}^t \int_{\lfrf{s}}^s\left(H ( \overline{\Theta}^{\lambda}_{\lfrf{r}})h ( \overline{\Theta}^{\lambda}_{\lfrf{r}})-\beta^{-1}\Upsilon (\overline{\Theta}^{\lambda}_{\lfrf{r}}) \right)\,\rmd r\,\rmd s \right.\right.\\
&\qquad \left.\left. -\lambda\sqrt{2\lambda\beta^{-1}} \int_{\lfrf{t}}^t \int_{{\lfrf{s}}}^s H (\overline{\Theta}_{\lfrf{r}}^{\lambda})\,\rmd B_r^{\lambda}\,\rmd s  +\sqrt{2\lambda\beta^{-1}} \int_{\lfrf{t}}^t \, \rmd B^{\lambda}_s\right|^{2p}\right] \\
&\leq 5^{2p}\Bigg(\lambda^{2p}\E\left[|h (\overline{\Theta}^{\lambda}_{\lfrf{t}})|^{2p}\right]+\lambda^{4p}\E\left[|H ( \overline{\Theta}^{\lambda}_{\lfrf{r}})h ( \overline{\Theta}^{\lambda}_{\lfrf{t}})|^{2p}\right]\Bigg.\\
&\qquad +\left. \lambda^{4p}\beta^{-2p}\E\left[|\Upsilon (\overline{\Theta}^{\lambda}_{\lfrf{t}}) |^{2p}\right]+\lambda^p(2 \beta^{-1})^p\E\left[\left| \int_{\lfrf{t}}^t   \,\rmd B_s^{\lambda} \right|^{2p}\right]\right.\\
&\qquad \Bigg. +\lambda^{3p}(2 \beta^{-1})^p\E\left[\left|H (\overline{\Theta}_{\lfrf{t}}^{\lambda})  \int_{\lfrf{t}}^t \int_{{\lfrf{s}}}^s \,\rmd B_r^{\lambda}\,\rmd s\right|^{2p}\right]\Bigg) \\
&\leq 5^{2p}\lambda^p\left(\cKl_1^{2p}\E\left[(1+|\overline{\Theta}^{\lambda}_{\lfrf{t}}|)^{2p}\right]+(\overline{L}_3\cKl_1)^{2p}\E\left[(1+|\overline{\Theta}^{\lambda}_{\lfrf{t}}|)^{2p}\right]+(\beta^{-1}d \overline{L}_2)^{2p}\right.\\
&\qquad \left. +(2\beta^{-1}(p+1)(d+2p))^p+(2\beta^{-1}\overline{L}_3^2(p+1)(d+2p))^p\right)\\
&\leq \lambda^p10^{2p}(\cKl_1^{2p}+(\overline{L}_3\cKl_1)^{2p}) \left(\E\left[ |\overline{\Theta}^{\lambda}_{\lfrf{t}}|^{2\lcrc{p}}\right]+1\right)\\
&\quad +\lambda^p5^{2p}( (\beta^{-1}d \overline{L}_2)^{2p}+(2\beta^{-1}(p+1)(d+2p))^p+(2\beta^{-1}\overline{L}_3^2(p+1)(d+2p))^p)\\
&\leq \lambda^p\left(e^{-\lambda\overline{a} \lfrf{t}}\overline{\cC}_{\mathbf{S}0,p}\E[|\theta_0|^{2\lcrc{p}}]+\widetilde{\cC}_{\mathbf{S}0,p} \right),
\end{align*}
where the first inequality holds due to $(\sum_{l=1}^vu_l)^w\leq v^w\sum_{l=1}^vu_l^w$, $v \in \N$, $u_l \geq 0$, $w>0$, the second inequality holds due to Remark \ref{rmk:growthclip}, Cauchy-Schwarz inequality, and the following inequality:
\[
\E\left[ \left| \int_{\lfrf{t}}^t   \,\rmd B_s^{\lambda} \right|^{2p}\right] \leq \max\{d^p, (p(d+2p-2))^p\}\leq ((p+1)(d+2p))^p,
\] 
the fourth inequality holds due to Lemma \ref{lem:2ndpthmmtlip}, and where
\begin{align}\label{eq:oserroralglipconst}
\begin{split}
\overline{\cC}_{\mathbf{S}0,p}&: = 10^{2p} (\cKl_1^{2p}+(\overline{L}_3\cKl_1)^{2p}),\\
\widetilde{\cC}_{\mathbf{S}0,p}&: =  \overline{\cC}_{\mathbf{S}0,p}\left( \ccl_{\lcrc{p}}(1+1/\overline{a})+1\right)+5^{2p}\left( (\beta^{-1}d \overline{L}_2)^{2p}+(2\beta^{-1}(p+1)(d+2p))^p(1+\overline{L}_3^{2p})\right).
\end{split}
\end{align}
The inequality \eqref{lem:oserroralglipineq2} can be obtained by using similar arguments. More precisely, by using Definition \ref{def:auxzetalip} with \eqref{eq:auxproclip}, we obtain that, for any $t \geq 0$,
\begin{align*}
\E\left[|\widetilde{\zeta}^{\lambda, n}_t - \widetilde{\zeta}^{\lambda, n}_{\lfrf{t}}|^{2p}\right]
& = \E\left[\left|-\lambda \int_{\lfrf{t}}^t h(\widetilde{\zeta}^{\lambda, n}_s)\,\rmd s+\sqrt{2\lambda\beta^{-1}} \int_{\lfrf{t}}^t \, \rmd B^{\lambda}_s \right|^{2p}\right]\\
&\leq 2^{2p}\left(\lambda^{2p}\E\left[\int_{\lfrf{t}}^t |h(\widetilde{\zeta}^{\lambda, n}_s)|^{2p}\,\rmd s\right]+\lambda^p(2\beta^{-1})^p\E\left[\left| \int_{\lfrf{t}}^t   \,\rmd B_s^{\lambda} \right|^{2p}\right]\right)\\
&\leq 2^{2p}\lambda^p\left(\cKl_1^{2p}\int_{\lfrf{t}}^t \E\left[(1+|\widetilde{\zeta}^{\lambda, n}_s|)^{2\lcrc{p}}\right]\,\rmd s+ (2\beta^{-1}(p+1)(d+2p))^p\right)\\
&\leq \lambda^p\left( 2^{3\lcrc{p}}\cKl_1^{2p}\int_{\lfrf{t}}^t \E\left[V_{2\lcrc{p}}(\widetilde{\zeta}^{\lambda, n}_s) \right]\,\rmd s+  2^{2p}(2\beta^{-1}(p+1)(d+2p))^p\right)\\
&\leq  \lambda^p\left(e^{-\lambda\overline{a} \lfrf{t}/2}\overline{\cC}_{\mathbf{S}1,p}\E[|\theta_0|^{2\lcrc{p}}]+\widetilde{\cC}_{\mathbf{S}1,p} \right),
\end{align*}
where the second inequality holds due to Remark \ref{rmk:growthclip} and the last inequality holds due to Lemma \ref{lem:zetaprocmelip} and where
\begin{align}\label{eq:oserrorauxlipconst}
\begin{split}
\overline{\cC}_{\mathbf{S}1,p}&: = 2^{4\lcrc{p}}(1+\cKl_1)^{2p}, \\
\widetilde{\cC}_{\mathbf{S}1,p}
&: = \overline{\cC}_{\mathbf{S}1,p}\left( \ccl_{\lcrc{p}}(1+1/\overline{a})+1\right)+ 2^{3\lcrc{p}}(1+\cKl_1)^{2p}3\mathrm{v}_{2\lcrc{p}}(\cMl_V(2\lcrc{p}))\\
&\qquad +2^{2p}(2\beta^{-1}(p+1)(d+2p))^p.
\end{split}
\end{align}
This completes the proof.
\end{proof}


\begin{lemma}\label{lem:graditoestlip} Let Assumptions \ref{asm:AIlip}, \ref{asm:ALLlip}, and \ref{asm:ADlip} hold. Then, for any $0<\lambda\leq \overline{\lambda}_{\max}$, $n \in \N_0$, $, t \geq nT$, we obtain the following inequalities:
\begin{align*}
\E\left[\left|-\lambda \int_{\lfrf{t}}^t  \left(H(\overline{\Theta}^{\lambda}_s) - H (\overline{\Theta}^{\lambda}_{\lfrf{s}})\right)h ( \overline{\Theta}^{\lambda}_{\lfrf{s}})\,\rmd s\right|^2\right] 
\leq \lambda^3\left(e^{- \overline{a}n/2} \overline{\cC}_{\mathbf{S}2}\E[|\theta_0|^4] +\widetilde{\cC}_{\mathbf{S}2} \right),&\\
\E\left[\left| \lambda^2 \int_{\lfrf{t}}^t   H(\overline{\Theta}^{\lambda}_s) \int_{\lfrf{s}}^s  H (\overline{\Theta}^{\lambda}_{\lfrf{r}}))h ( \overline{\Theta}^{\lambda}_{\lfrf{r}}) \, \rmd r \,\rmd s\right|^2\right] 
\leq \lambda^3\left(e^{- \overline{a}n/2} \overline{\cC}_{\mathbf{S}2}\E[|\theta_0|^4] +\widetilde{\cC}_{\mathbf{S}2} \right), &\\
\E\left[\left| -\lambda^2\beta^{-1} \int_{\lfrf{t}}^t   H(\overline{\Theta}^{\lambda}_s) \int_{\lfrf{s}}^s  \Upsilon ( \overline{\Theta}^{\lambda}_{\lfrf{r}}) \, \rmd r \,\rmd s\right|^2\right] 
\leq \lambda^3\left(e^{- \overline{a}n/2} \overline{\cC}_{\mathbf{S}2}\E[|\theta_0|^4] +\widetilde{\cC}_{\mathbf{S}2} \right), &\\
\E\left[\left| -\lambda\sqrt{2\lambda\beta^{-1}} \int_{\lfrf{t}}^t   H(\overline{\Theta}^{\lambda}_s) \int_{\lfrf{s}}^s  H ( \overline{\Theta}^{\lambda}_{\lfrf{r}}) \, \rmd B_r^{\lambda} \,\rmd s\right|^2\right] 
\leq \lambda^3\left(e^{- \overline{a}n/2} \overline{\cC}_{\mathbf{S}2}\E[|\theta_0|^4] +\widetilde{\cC}_{\mathbf{S}2} \right), &\\
\E\left[\left| \sqrt{2\lambda\beta^{-1}} \int_{\lfrf{t}}^t  \left( H(\overline{\Theta}^{\lambda}_s) -  H ( \overline{\Theta}^{\lambda}_{\lfrf{s}})  \right) \,\rmd B_s^{\lambda}\right|^2\right] 
\leq \lambda^2\left(e^{- \overline{a}n/2} \overline{\cC}_{\mathbf{S}2}\E[|\theta_0|^4] +\widetilde{\cC}_{\mathbf{S}2} \right), &\\
\E\left[\left| \sqrt{2\lambda\beta^{-1}} \int_{\lfrf{t}}^t  \left( H(\widetilde{\zeta}^{\lambda, n}_s) -  H ( \widetilde{\zeta}^{\lambda, n}_{\lfrf{s}})\right)   \,\rmd B_s^{\lambda}\right|^2\right] 
\leq \lambda^2\left(e^{- \overline{a}n/4} \overline{\cC}_{\mathbf{S}2}\E[|\theta_0|^4] +\widetilde{\cC}_{\mathbf{S}2} \right),&\\
\E\left[\left| \lambda\beta^{-1} \int_{\lfrf{t}}^t  \left( \Upsilon(\overline{\Theta}^{\lambda}_s) - \Upsilon ( \overline{\Theta}^{\lambda}_{\lfrf{s}}) \right)  \,\rmd s\right|^2\right] 
\leq \lambda^{2+q}\left(e^{- \overline{a}n/2} \overline{\cC}_{\mathbf{S}2}\E[|\theta_0|^4] +\widetilde{\cC}_{\mathbf{S}2} \right),&\\
\E\left[\left| -\lambda  \int_{\lfrf{t}}^t H(\overline{\Theta}^{\lambda}_s) h (\overline{\Theta}^{\lambda}_{\lfrf{s}})\,\rmd s\right|^2+\left| \lambda\beta^{-1}\int_{\lfrf{t}}^t\Upsilon (\overline{\Theta}^{\lambda}_s)\rmd s\right|^2\right] 
\leq \lambda^2\left(e^{- \overline{a}n/2} \overline{\cC}_{\mathbf{S}2}\E[|\theta_0|^4] +\widetilde{\cC}_{\mathbf{S}2} \right), &\\
\E\left[\left|  -\lambda  \int_{\lfrf{t}}^t H(\widetilde{\zeta}^{\lambda, n}_s) h ( \widetilde{\zeta}^{\lambda, n}_s)\,\rmd s  \right|^2+\left| \lambda\beta^{-1}\int_{\lfrf{t}}^t\Upsilon (\widetilde{\zeta}^{\lambda, n}_s)\rmd s\right|^2\right] 
\leq \lambda^2\left(e^{- \overline{a} n/4} \overline{\cC}_{\mathbf{S}2}\E[|\theta_0|^4] +\widetilde{\cC}_{\mathbf{S}2} \right),&
\end{align*}
where $\overline{\cC}_{\mathbf{S}2}$ and $\widetilde{\cC}_{\mathbf{S}2}$ are given in \eqref{eq:graditoestlipconst}.
\end{lemma}

\begin{proof} The inequalities can be obtained by using the following arguments:
\begin{enumerate}
\item To show that the first inequality holds, by using Assumption \ref{asm:ALLlip}, Remark \ref{rmk:growthclip}, we have that
\begin{align*}
&\E\left[\left|-\lambda \int_{\lfrf{t}}^t \left(H(\overline{\Theta}^{\lambda}_s) - H (\overline{\Theta}^{\lambda}_{\lfrf{s}})\right)h ( \overline{\Theta}^{\lambda}_{\lfrf{s}})\,\rmd s\right|^2\right] \\
&\leq \lambda^2 \int_{\lfrf{t}}^t \E\left[|H(\overline{\Theta}^{\lambda}_s) - H (\overline{\Theta}^{\lambda}_{\lfrf{s}})|^2|h ( \overline{\Theta}^{\lambda}_{\lfrf{s}}) |^2\right]  \,\rmd s\\
&\leq \lambda^2 \int_{\lfrf{t}}^t \E\left[ \overline{L}_2^2|\overline{\Theta}^{\lambda}_s - \overline{\Theta}^{\lambda}_{\lfrf{s}}|^2 \cKl_1^2(1+|\overline{\Theta}^{\lambda}_{\lfrf{s}}|)^2\right]  \,\rmd s\\
&\leq  2^{3/2}\lambda^2 (\overline{L}_2 \cKl_1)^2 \int_{\lfrf{t}}^t  \left(\E\left[1+|\overline{\Theta}^{\lambda}_{\lfrf{s}}|^4\right] \right)^{1/2} \left(\E\left[|\overline{\Theta}^{\lambda}_s - \overline{\Theta}^{\lambda}_{\lfrf{s}}|^4\right] \right)^{1/2} \,\rmd s\\
&\leq  2^{3/2}\lambda^2 (\overline{L}_2 \cKl_1)^2\left(e^{-\lambda \overline{a} \lfrf{t}}\E\left[|\theta_0|^4\right]  + \ccl_2(1+1/\overline{a})+1\right)^{1/2}\\
&\qquad \times \lambda \left(e^{-  \lambda\overline{a} \lfrf{t}}\overline{\cC}_{\mathbf{S}0,2}\E[|\theta_0|^4]+\widetilde{\cC}_{\mathbf{S}0,2} \right)^{1/2} \\
&\leq \lambda^33 (\overline{L}_2 \cKl_1)^2\left(e^{-\lambda \overline{a} \lfrf{t}}(1+\overline{\cC}_{\mathbf{S}0,2})\E\left[|\theta_0|^4\right]  + \ccl_2 (1+1/\overline{a})+\widetilde{\cC}_{\mathbf{S}0,2}+1\right)\\
&\leq \lambda^3\left(e^{- \overline{a} n/2}\overline{\cC}_{\mathbf{S}2}\E\left[|\theta_0|^4\right]  + \widetilde{\cC}_{\mathbf{S}2}\right),
\end{align*}
where the fourth inequality holds by applying Lemma \ref{lem:2ndpthmmtlip} and \ref{lem:oserroralglip}, the last inequality holds due to $\lambda \lfrf{t}\geq \lambda nT \geq n/2$, and where
\begin{align}\label{eq:graditoestlipconst}
\begin{split}
\overline{\cC}_{\mathbf{S}2}&:=2^7d^4(1+\beta^{-1})^2(1+\overline{L}_3)^4(1+\overline{L}_2)^2 (1+\overline{L}_1)^2( 1+\cKl_1)^2( 1+\cKl_0)^2\\
&\qquad \times (1+\max\{\overline{\cC}_{\mathbf{S}0,2},\overline{\cC}_{\mathbf{S}0,2q},\overline{\cC}_{\mathbf{S}0,1+q},\overline{\cC}_{\mathbf{S}1,2}\}), \\
\widetilde{\cC}_{\mathbf{S}2}&:=2^7d^4(1+\beta^{-1})^2(1+\overline{L}_3)^4(1+\overline{L}_2)^2(1+\overline{L}_1)^2( 1+\cKl_1)^2( 1+\cKl_0)^2\\
&\qquad \times (\ccl_2 (1+1/\overline{a})+\max\{\widetilde{\cC}_{\mathbf{S}0,2},\widetilde{\cC}_{\mathbf{S}0,2q},\widetilde{\cC}_{\mathbf{S}0,1+q},\widetilde{\cC}_{\mathbf{S}1,2}\}+1)
\end{split}
\end{align}
with $\overline{\cC}_{\mathbf{S}0,p}$, $\widetilde{\cC}_{\mathbf{S}0,p}$, $\overline{\cC}_{\mathbf{S}1,p}$, $\widetilde{\cC}_{\mathbf{S}1,p}$, $p>0$, given in \eqref{eq:oserroralglipconst} and \eqref{eq:oserrorauxlipconst}, and $ \ccl_p$, $\cMl_V(p)$, $p\in [2, \infty)\cap {\N}$ given in \eqref{eq:2pthmmtlipexpconst} (see also Lemma \ref{lem:2ndpthmmtlip}) and Lemma \ref{lem:driftconlip}.
\item To establish the second inequality, we apply Remark \ref{rmk:growthclip} and Lemma \ref{lem:2ndpthmmtlip} to obtain
\begin{align*}
&\E\left[\left| \lambda^2 \int_{\lfrf{t}}^t   H(\overline{\Theta}^{\lambda}_s) \int_{\lfrf{s}}^s  H (\overline{\Theta}^{\lambda}_{\lfrf{r}}))h ( \overline{\Theta}^{\lambda}_{\lfrf{r}}) \, \rmd r \,\rmd s\right|^2\right] \\
&\leq \lambda^4 \int_{\lfrf{t}}^t\E\left[|H(\overline{\Theta}^{\lambda}_s) H (\overline{\Theta}^{\lambda}_{\lfrf{s}}) h ( \overline{\Theta}^{\lambda}_{\lfrf{s}})|^2\right]\,\rmd s\\ 
&\leq \lambda^4 \int_{\lfrf{t}}^t\E\left[\overline{L}_3^4 \cKl_1^2(1+| \overline{\Theta}^{\lambda}_{\lfrf{s}}|)^2\right]\,\rmd s\\
&\leq 2\lambda^4\overline{L}_3^4 \cKl_1^2 \int_{\lfrf{t}}^t \E\left[1 +| \overline{\Theta}^{\lambda}_{\lfrf{s}}|^2\right]\,\rmd s\\
&\leq 4\lambda^4\overline{L}_3^4 \cKl_1^2 \int_{\lfrf{t}}^t \E\left[1 +| \overline{\Theta}^{\lambda}_{\lfrf{s}}|^4\right]\,\rmd s\\
&\leq \lambda^34 \overline{L}_3^4 \cKl_1^2 \left(e^{-\lambda \overline{a} \lfrf{t}}\E\left[ |\theta_0|^4\right]+  \ccl_2 (1+1/\overline{a})+1\right)\\
&\leq \lambda^3\left(e^{- \overline{a} n/2}\overline{\cC}_{\mathbf{S}2}\E\left[|\theta_0|^4\right]  + \widetilde{\cC}_{\mathbf{S}2}\right),
\end{align*} 
where $\overline{\cC}_{\mathbf{S}2}$ and $\widetilde{\cC}_{\mathbf{S}2}$ are given in \eqref{eq:graditoestlipconst}.
\item To obtain the third inequality, we apply Remark \ref{rmk:growthclip} and write
\begin{align*}
&\E\left[\left| -\lambda^2\beta^{-1} \int_{\lfrf{t}}^t   H(\overline{\Theta}^{\lambda}_s) \int_{\lfrf{s}}^s  \Upsilon ( \overline{\Theta}^{\lambda}_{\lfrf{r}}) \, \rmd r \,\rmd s\right|^2\right] \\
&\leq \lambda^4\beta^{-2} \int_{\lfrf{t}}^t \E\left[ |H(\overline{\Theta}^{\lambda}_s) \Upsilon ( \overline{\Theta}^{\lambda}_{\lfrf{s}})  |^2\right]  \,\rmd s\\
&\leq \lambda^4\beta^{-2} \int_{\lfrf{t}}^t \E\left[\overline{L}_3^2 |\Upsilon ( \overline{\Theta}^{\lambda}_{\lfrf{s}})  |^2\right]  \,\rmd s\\
&\leq \lambda^4(\beta^{-1}  \overline{L}_3d\overline{L}_2)^2\\
&\leq \lambda^3\left(e^{- \overline{a} n/2} \overline{\cC}_{\mathbf{S}2}\E[|\theta_0|^4] +\widetilde{\cC}_{\mathbf{S}2} \right),
\end{align*}
where $\overline{\cC}_{\mathbf{S}2}$ and $\widetilde{\cC}_{\mathbf{S}2}$ are given in \eqref{eq:graditoestlipconst}.
\item To obtain the fourth inequality, we use Cauchy-Schwarz inequality and Lemma \ref{lem:2ndpthmmtlip}:
\begin{align*}
&\E\left[\left| -\lambda\sqrt{2\lambda\beta^{-1}} \int_{\lfrf{t}}^t   H(\overline{\Theta}^{\lambda}_s) \int_{\lfrf{s}}^s  H ( \overline{\Theta}^{\lambda}_{\lfrf{r}}) \, \rmd B_r^{\lambda} \,\rmd s\right|^2\right] \\
&\leq 2\lambda^3\beta^{-1}  \int_{\lfrf{t}}^t  \E\left[\overline{L}_3^2 \left|H ( \overline{\Theta}^{\lambda}_{\lfrf{s}})  \int_{\lfrf{s}}^s  \, \rmd B_r^{\lambda}  \right|^2\right] \,\rmd s\\
&\leq 2\lambda^3\beta^{-1}  \int_{\lfrf{t}}^t  \E\left[\overline{L}_3^4 \left|\int_{\lfrf{s}}^s  \, \rmd B_r^{\lambda}  \right|^2\right] \,\rmd s\\
&\leq 2\lambda^3\beta^{-1} \overline{L}_3^4d\\
&\leq \lambda^3\left(e^{- \overline{a} n/2} \overline{\cC}_{\mathbf{S}2}\E[|\theta_0|^4] +\widetilde{\cC}_{\mathbf{S}2} \right),
\end{align*}
where $\overline{\cC}_{\mathbf{S}2}$ and $\widetilde{\cC}_{\mathbf{S}2}$ are given in \eqref{eq:graditoestlipconst}.
\item\label{item:graditoestlipfifth} To obtain the fifth inequality, we apply Assumption \ref{asm:ALLlip} and Cauchy-Schwarz inequality:
\begin{align*} 
&\E\left[\left| \sqrt{2\lambda\beta^{-1}} \int_{\lfrf{t}}^t   \left(H(\overline{\Theta}^{\lambda}_s) -  H ( \overline{\Theta}^{\lambda}_{\lfrf{s}})\right)   \,\rmd B_s^{\lambda}\right|^2\right] \\
& =  2\lambda\beta^{-1}  \int_{\lfrf{t}}^t \E\left[ |   H(\overline{\Theta}^{\lambda}_s) -  H ( \overline{\Theta}^{\lambda}_{\lfrf{s}})  |_{\cF}^2\right]    \,\rmd s\\
&\leq 2\lambda\beta^{-1}  \int_{\lfrf{t}}^t \E\left[ d\overline{L}_2^2 | \overline{\Theta}^{\lambda}_s  -  \overline{\Theta}^{\lambda}_{\lfrf{s}}  |^2\right]    \,\rmd s\\
&\leq 2\lambda\beta^{-1} d\overline{L}_2^2 \int_{\lfrf{t}}^t \left(\E\left[  |    \overline{\Theta}^{\lambda}_s  -  \overline{\Theta}^{\lambda}_{\lfrf{s}}  |^4\right] \right)^{1/2}  \,\rmd s\\
&\leq  2\lambda^2\beta^{-1} d\overline{L}_2^2   \left(e^{-\lambda\overline{a} \lfrf{t}}\overline{\cC}_{\mathbf{S}0,2}\E[|\theta_0|^4]+\widetilde{\cC}_{\mathbf{S}0,2}\right)^{1/2} \\
&\leq \lambda^2\left(e^{- \overline{a} n/2} \overline{\cC}_{\mathbf{S}2}\E[|\theta_0|^4] +\widetilde{\cC}_{\mathbf{S}2} \right),
\end{align*}
where the third inequality holds due to Lemma \ref{lem:oserroralglip}, and where $\overline{\cC}_{\mathbf{S}2}$ and $\widetilde{\cC}_{\mathbf{S}2}$ are given in \eqref{eq:graditoestlipconst}.
\item To obtain the sixth inequality, we use the same arguments as in \ref{item:graditoestlipfifth}:
\begin{align*} 
&\E\left[\left| \sqrt{2\lambda\beta^{-1}} \int_{\lfrf{t}}^t   \left(H(\widetilde{\zeta}^{\lambda, n}_s) -  H ( \widetilde{\zeta}^{\lambda, n}_{\lfrf{s}}) \right)  \,\rmd B_s^{\lambda}\right|^2\right] \\
& =  2\lambda\beta^{-1}  \int_{\lfrf{t}}^t \E\left[ | H(\widetilde{\zeta}^{\lambda, n}_s) -  H ( \widetilde{\zeta}^{\lambda, n}_{\lfrf{s}})   |_{\cF}^2\right]    \,\rmd s\\
& \leq  2\lambda\beta^{-1}  \int_{\lfrf{t}}^t \E\left[ d\overline{L}_2^2  |\widetilde{\zeta}^{\lambda, n}_s- \widetilde{\zeta}^{\lambda, n}_{\lfrf{s}}|^2\right]    \,\rmd s\\
& \leq  2\lambda\beta^{-1} d\overline{L}_2^2  \int_{\lfrf{t}}^t  \left( \E\left[ |\widetilde{\zeta}^{\lambda, n}_s- \widetilde{\zeta}^{\lambda, n}_{\lfrf{s}}|^4\right]  \right)^{1/2} \,\rmd s\\
&\leq 2 \lambda^2\beta^{-1} d\overline{L}_2^2  \left(e^{-  \lambda \overline{a}  \lfrf{t}/2}\overline{\cC}_{\mathbf{S}1,2}\E[|\theta_0|^4]+\widetilde{\cC}_{\mathbf{S}1,2} \right)^{1/2}\\
&\leq  \lambda^2\left(e^{- \overline{a} n/4} \overline{\cC}_{\mathbf{S}2}\E[|\theta_0|^4] +\widetilde{\cC}_{\mathbf{S}2} \right),
\end{align*}
where the third inequality holds due to Lemma \ref{lem:oserroralglip}, and where $\overline{\cC}_{\mathbf{S}2}$ and $\widetilde{\cC}_{\mathbf{S}2}$ are given in \eqref{eq:graditoestlipconst}.
\item To establish the seventh inequality, we apply Remark \ref{rmk:growthclip} and  Cauchy-Schwarz inequality to obtain
\begin{align*} 
&\E\left[\left| \lambda\beta^{-1} \int_{\lfrf{t}}^t  \left( \Upsilon(\overline{\Theta}^{\lambda}_s) - \Upsilon ( \overline{\Theta}^{\lambda}_{\lfrf{s}})  \right) \,\rmd s\right|^2\right] \\
&\leq \lambda^2\beta^{-2} \int_{\lfrf{t}}^t \E\left[ |   \Upsilon(\overline{\Theta}^{\lambda}_s) - \Upsilon ( \overline{\Theta}^{\lambda}_{\lfrf{s}})  |^2\right]    \,\rmd s\\
&\leq  \lambda^2\beta^{-2} \int_{\lfrf{t}}^t \E\left[ d^3\overline{L}_1^2|\overline{\Theta}^{\lambda}_s- \overline{\Theta}^{\lambda}_{\lfrf{s}}|^{2q} \right]    \,\rmd s \\
&\leq  \lambda^2\beta^{-2}  d^3\overline{L}_1^2\int_{\lfrf{t}}^t  \left(\E\left[ |\overline{\Theta}^{\lambda}_s- \overline{\Theta}^{\lambda}_{\lfrf{s}}|^{4q} \right]\right)^{1/2}   \,\rmd s \\
&\leq  \lambda^{2+q}\beta^{-2}  d^3\overline{L}_1^2  \left(e^{-\lambda\overline{a} \lfrf{t}}\overline{\cC}_{\mathbf{S}0,2q}\E[|\theta_0|^4]+\widetilde{\cC}_{\mathbf{S}0,2q} \right)^{1/2} \\
&\leq  \lambda^{2+q}\left(e^{- \overline{a} n/2} \overline{\cC}_{\mathbf{S}2}\E[|\theta_0|^4] +\widetilde{\cC}_{\mathbf{S}2} \right),
\end{align*}
where the fourth inequality holds due to Lemma \ref{lem:oserroralglip}, and $\overline{\cC}_{\mathbf{S}2}$, $\widetilde{\cC}_{\mathbf{S}2}$ are given in \eqref{eq:graditoestlipconst}.
\item\label{item:graditoestlipeighth} To obtain the eighth inequality, we apply Remark \ref{rmk:growthclip} and write
\begin{align*} 
&\E\left[\left| -\lambda  \int_{\lfrf{t}}^t H(\overline{\Theta}^{\lambda}_s) h (\overline{\Theta}^{\lambda}_{\lfrf{s}})\,\rmd s\right|^2+\left| \lambda\beta^{-1}\int_{\lfrf{t}}^t\Upsilon (\overline{\Theta}^{\lambda}_s)\,\rmd s\right|^2\right]  \\
&\leq \lambda^2  \int_{\lfrf{t}}^t\E\left[\overline{L}_3^2| h(\overline{\Theta}^{\lambda}_{\lfrf{s}})|^2\right]\, \rmd s + \lambda^2 \beta^{-2} d^2\overline{L}_2^2\\
&\leq \lambda^2  \overline{L}_3^2 \int_{\lfrf{t}}^t\E\left[ \cKl_1^2(1+|\overline{\Theta}^{\lambda}_{\lfrf{s}}| )^2\right]\, \rmd s + \lambda^2 \beta^{-2} d^2\overline{L}_2^2\\
&\leq \lambda^2\left(4  \overline{L}_3^2 \cKl_1^2 \int_{\lfrf{t}}^t\E\left[(1+|\overline{\Theta}^{\lambda}_{\lfrf{s}}|^4 ) \right]\, \rmd s +\beta^{-2} d^2\overline{L}_2^2\right)\\
&\leq \lambda^2\left(4 \overline{L}_3^2 \cKl_1^2\left( e^{-\lambda\overline{a} \lfrf{t}}\E\left[ |\theta_0|^{4}\right]+  \ccl_2\left(1+1/\overline{a}\right)+1\right)+\beta^{-2} d^2\overline{L}_2^2\right)\\
&\leq \lambda^24(1+ \overline{L}_3)^2(1+ \cKl_1)^2(1+\beta^{-1})^2 d^2(1+\overline{L}_2)^2\left( e^{-\lambda\overline{a} \lfrf{t}}\E\left[ |\theta_0|^{4}\right]+  \ccl_2\left(1+1/\overline{a}\right)+1\right)\\
&\leq \lambda^2\left(e^{- \overline{a} n/2} \overline{\cC}_{\mathbf{S}2}\E[|\theta_0|^4] +\widetilde{\cC}_{\mathbf{S}2} \right),
\end{align*}
where the fourth inequality holds due to Lemma \ref{lem:2ndpthmmtlip}, and where $\overline{\cC}_{\mathbf{S}2}$, $\widetilde{\cC}_{\mathbf{S}2}$ are given in \eqref{eq:graditoestlipconst}.
\item To establish the last inequality, we follow the arguments in \ref{item:graditoestlipeighth}:
\begin{align*} 
&\E\left[\left|  -\lambda  \int_{\lfrf{t}}^t H(\widetilde{\zeta}^{\lambda, n}_s) h ( \widetilde{\zeta}^{\lambda, n}_s)\,\rmd s  \right|^2+\left| \lambda\beta^{-1}\int_{\lfrf{t}}^t\Upsilon (\widetilde{\zeta}^{\lambda, n}_s)\rmd s\right|^2\right]  \\
&\leq \lambda^2\left(2  \overline{L}_3^2 \cKl_1^2 \int_{\lfrf{t}}^t\E\left[(1+|\widetilde{\zeta}^{\lambda, n}_s|^2) \right]\, \rmd s +\beta^{-2} d^2\overline{L}_2^2\right)\\
&\leq \lambda^2\left(2  \overline{L}_3^2 \cKl_1^2 \int_{\lfrf{t}}^t\E\left[V_4(\widetilde{\zeta}^{\lambda, n}_s) \right]\, \rmd s +\beta^{-2} d^2\overline{L}_2^2\right)\\
&\leq \lambda^2\left(2 \overline{L}_3^2 \cKl_1^2\left( 2e^{-\lambda \overline{a} \lfrf{t}/2}\E[|\theta_0|^4] +2\left( \ccl_2(1+1/\overline{a})+1\right)+3\mathrm{v}_{4}(\cMl_V(4)) \right)+\beta^{-2} d^2\overline{L}_2^2\right)\\
&\leq \lambda^24(1+ \overline{L}_3)^2(1+ \cKl_1)^2(1+\beta^{-1})^2 d^2(1+\overline{L}_2)^2\\
&\quad \times \left( e^{-\lambda\overline{a} \lfrf{t}/2}\E\left[ |\theta_0|^{4}\right]+  \ccl_2\left(1+1/\overline{a}\right)+3\mathrm{v}_{4}(\cMl_V(4))/2+2\right)\\
&\leq \lambda^2\left(e^{- \overline{a} n/4} \overline{\cC}_{\mathbf{S}2}\E[|\theta_0|^4] +\widetilde{\cC}_{\mathbf{S}2} \right),
\end{align*}
where the second inequality holds due to Lemma \ref{lem:zetaprocmelip}, and where $\overline{\cC}_{\mathbf{S}2}$ and $\widetilde{\cC}_{\mathbf{S}2}$ are given in \eqref{eq:graditoestlipconst}.
\end{enumerate}
This completes the proof.
\end{proof}

\begin{corollary}\label{cor:graditoublip} Let Assumptions \ref{asm:AIlip}, \ref{asm:ALLlip}, and \ref{asm:ADlip} hold. Then, for any $0<\lambda\leq \overline{\lambda}_{\max}$, $n \in \N_0$, $, t \geq nT$, we obtain the following inequalities:
\begin{align*}
&\E\left[\left| h(\overline{\Theta}^{\lambda}_t) - h (\overline{\Theta}^{\lambda}_{\lfrf{t}}) + \lambda \int_{\lfrf{t}}^t\left(H ( \overline{\Theta}^{\lambda}_{\lfrf{s}})h ( \overline{\Theta}^{\lambda}_{\lfrf{s}})-\beta^{-1}\Upsilon (\overline{\Theta}^{\lambda}_{\lfrf{s}}) \right)\,\rmd s\right.\right.\\
&\qquad \left.\left. - \sqrt{2\lambda\beta^{-1}} \int_{{\lfrf{t}}}^t H (\overline{\Theta}_{\lfrf{s}}^{\lambda})\,\rmd B_s^{\lambda}  \right|^2\right]  
\leq \lambda^2\left(e^{- \overline{a}n/2}36 \overline{\cC}_{\mathbf{S}2}\E[|\theta_0|^4] +36\widetilde{\cC}_{\mathbf{S}2} \right),\\
&\E\left[\left| h( \widetilde{\zeta}^{\lambda, n}_t) - h ( \widetilde{\zeta}^{\lambda, n}_{\lfrf{t}})   - \sqrt{2\lambda\beta^{-1}} \int_{{\lfrf{t}}}^t H (\widetilde{\zeta}^{\lambda, n}_{\lfrf{s}})\,\rmd B_s^{\lambda}  -\Bigg(h(\overline{\Theta}^{\lambda}_t) - h (\overline{\Theta}^{\lambda}_{\lfrf{t}})  \Bigg. \right.\right.\\
&\qquad \left.\left. \Bigg.- \sqrt{2\lambda\beta^{-1}} \int_{{\lfrf{t}}}^t H (\overline{\Theta}_{\lfrf{s}}^{\lambda})\,\rmd B_s^{\lambda}  \Bigg)\right|^2\right]  \leq \lambda^2\left(e^{- \overline{a}n/4}72 \overline{\cC}_{\mathbf{S}2}\E[|\theta_0|^4] +72\widetilde{\cC}_{\mathbf{S}2} \right),
\end{align*}
where $\overline{\cC}_{\mathbf{S}2}$ and $\widetilde{\cC}_{\mathbf{S}2}$ are given in \eqref{eq:graditoestlipconst}.
\end{corollary}

\begin{proof}
For any $t \geq nT$, by applying It\^o's formula to $h(\overline{\Theta}^{\lambda}_t)$, we obtain, almost surely
\begin{align}
\begin{split}\label{eq:holaproclipito}
h(\overline{\Theta}^{\lambda}_t) - h(\overline{\Theta}^{\lambda}_{\lfrf{t}})  
&=  -\lambda  \int_{\lfrf{t}}^t H(\overline{\Theta}^{\lambda}_s) h (\overline{\Theta}^{\lambda}_{\lfrf{s}})\,\rmd s+\lambda^2\int_{\lfrf{t}}^t H(\overline{\Theta}^{\lambda}_s)\int_{\lfrf{s}}^s H ( \overline{\Theta}^{\lambda}_{\lfrf{r}})h ( \overline{\Theta}^{\lambda}_{\lfrf{r}}) \,\rmd r\,\rmd s\\
&\quad -\lambda^2\beta^{-1}\int_{\lfrf{t}}^t H(\overline{\Theta}^{\lambda}_s)\int_{\lfrf{s}}^s  \Upsilon (\overline{\Theta}^{\lambda}_{\lfrf{r}}) \,\rmd r\,\rmd s\\
&\quad-\lambda\sqrt{2\lambda\beta^{-1}} \int_{\lfrf{t}}^t  H(\overline{\Theta}^{\lambda}_s) \int_{{\lfrf{s}}}^s H (\overline{\Theta}_{\lfrf{r}}^{\lambda})\,\rmd B_r^{\lambda}\,\rmd s  \\
&\quad +\sqrt{2\lambda\beta^{-1}}\int_{\lfrf{t}}^tH(\overline{\Theta}^{\lambda}_s) \, \rmd B^{\lambda}_s+\lambda\beta^{-1}\int_{\lfrf{t}}^t\Upsilon (\overline{\Theta}^{\lambda}_s)\rmd s.
\end{split}
\end{align}
Similarly, applying It\^o's formula to $h(\widetilde{\zeta}^{\lambda, n}_t)$ yields, almost surely
\begin{align}
\begin{split}\label{eq:auxproclipito}
h( \widetilde{\zeta}^{\lambda, n}_t) - h ( \widetilde{\zeta}^{\lambda, n}_{\lfrf{t}})
&=  -\lambda  \int_{\lfrf{t}}^t H(\widetilde{\zeta}^{\lambda, n}_s) h ( \widetilde{\zeta}^{\lambda, n}_s)\,\rmd s   +\sqrt{2\lambda\beta^{-1}}\int_{\lfrf{t}}^tH(\widetilde{\zeta}^{\lambda, n}_s) \, \rmd B^{\lambda}_s\\
&\quad +\lambda\beta^{-1}\int_{\lfrf{t}}^t\Upsilon (\widetilde{\zeta}^{\lambda, n}_s)\rmd s.
\end{split}
\end{align}

\begin{enumerate}
\item To obtain the first inequality, by using \eqref{eq:holaproclipito} with $(\sum_{l=1}^vu_l)^2 \leq v \sum_{l=1}^vu_l^2$, $v \in \N$, $u_l \geq 0$, we obtain that
\begin{align*}
&\E\left[\left| h(\overline{\Theta}^{\lambda}_t) - h (\overline{\Theta}^{\lambda}_{\lfrf{t}}) + \lambda \int_{\lfrf{t}}^t\left(H ( \overline{\Theta}^{\lambda}_{\lfrf{s}})h ( \overline{\Theta}^{\lambda}_{\lfrf{s}})-\beta^{-1}\Upsilon (\overline{\Theta}^{\lambda}_{\lfrf{s}}) \right)\,\rmd s\right.\right.\\
&\qquad \left.\left. - \sqrt{2\lambda\beta^{-1}} \int_{{\lfrf{t}}}^t H (\overline{\Theta}_{\lfrf{s}}^{\lambda})\,\rmd B_s^{\lambda}  \right|^2\right]\\
&\leq 6\E\left[\left|-\lambda \int_{\lfrf{t}}^t  \left(H(\overline{\Theta}^{\lambda}_s) - H (\overline{\Theta}^{\lambda}_{\lfrf{s}})\right)h ( \overline{\Theta}^{\lambda}_{\lfrf{s}})\,\rmd s\right|^2\right] \\
&\quad+  6\E\left[\left| \lambda^2 \int_{\lfrf{t}}^t   H(\overline{\Theta}^{\lambda}_s) \int_{\lfrf{s}}^s  H (\overline{\Theta}^{\lambda}_{\lfrf{r}}))h ( \overline{\Theta}^{\lambda}_{\lfrf{r}}) \, \rmd r \,\rmd s\right|^2\right] \\
&\quad+ 6\E\left[\left| -\lambda^2\beta^{-1} \int_{\lfrf{t}}^t   H(\overline{\Theta}^{\lambda}_s) \int_{\lfrf{s}}^s  \Upsilon ( \overline{\Theta}^{\lambda}_{\lfrf{r}}) \, \rmd r \,\rmd s\right|^2\right]  \\
&\quad+  6\E\left[\left| -\lambda\sqrt{2\lambda\beta^{-1}} \int_{\lfrf{t}}^t   H(\overline{\Theta}^{\lambda}_s) \int_{\lfrf{s}}^s  H ( \overline{\Theta}^{\lambda}_{\lfrf{r}}) \, \rmd B_r^{\lambda} \,\rmd s\right|^2\right] \\
&\quad+  6\E\left[\left| \sqrt{2\lambda\beta^{-1}} \int_{\lfrf{t}}^t  \left( H(\overline{\Theta}^{\lambda}_s) -  H ( \overline{\Theta}^{\lambda}_{\lfrf{s}}) \right)  \,\rmd B_s^{\lambda}\right|^2\right] \\
&\quad+  6\E\left[\left| \lambda\beta^{-1} \int_{\lfrf{t}}^t   \left(\Upsilon(\overline{\Theta}^{\lambda}_s) - \Upsilon ( \overline{\Theta}^{\lambda}_{\lfrf{s}})\right)  \,\rmd s\right|^2\right] \\
& \leq \lambda^2\left(e^{- \overline{a}n/2} 36\overline{\cC}_{\mathbf{S}2}\E[|\theta_0|^4] +36\widetilde{\cC}_{\mathbf{S}2}  \right),
\end{align*}
where the last inequality holds due to Lemma \ref{lem:graditoestlip}.
\item To establish the second inequality, we use \eqref{eq:holaproclipito} and \eqref{eq:auxproclipito} to obtain
\begin{align*}
&\E\left[\left| h( \widetilde{\zeta}^{\lambda, n}_t) - h ( \widetilde{\zeta}^{\lambda, n}_{\lfrf{t}})   - \sqrt{2\lambda\beta^{-1}} \int_{{\lfrf{t}}}^t H (\widetilde{\zeta}^{\lambda, n}_{\lfrf{s}})\,\rmd B_s^{\lambda}  \right.\right.\\
&\qquad \left.\left. -\left(h(\overline{\Theta}^{\lambda}_t) - h (\overline{\Theta}^{\lambda}_{\lfrf{t}})  - \sqrt{2\lambda\beta^{-1}} \int_{{\lfrf{t}}}^t H (\overline{\Theta}_{\lfrf{s}}^{\lambda})\,\rmd B_s^{\lambda}  \right)\right|^2\right]\\
&\leq 2 \E\left[\left| h( \widetilde{\zeta}^{\lambda, n}_t) - h ( \widetilde{\zeta}^{\lambda, n}_{\lfrf{t}})   - \sqrt{2\lambda\beta^{-1}} \int_{{\lfrf{t}}}^t H (\widetilde{\zeta}^{\lambda, n}_{\lfrf{s}})\,\rmd B_s^{\lambda} \right|^2\right]\\
&\quad +2 \E\left[\left| h(\overline{\Theta}^{\lambda}_t) - h (\overline{\Theta}^{\lambda}_{\lfrf{t}})  - \sqrt{2\lambda\beta^{-1}} \int_{{\lfrf{t}}}^t H (\overline{\Theta}_{\lfrf{s}}^{\lambda})\,\rmd B_s^{\lambda} \right|^2\right]\\
&\leq 6\E\left[\left| -\lambda  \int_{\lfrf{t}}^t H(\widetilde{\zeta}^{\lambda, n}_s) h ( \widetilde{\zeta}^{\lambda, n}_s)\,\rmd s   \right|^2\right]  +6\E\left[\left| \lambda\beta^{-1}\int_{\lfrf{t}}^t\Upsilon (\widetilde{\zeta}^{\lambda, n}_s)\rmd s  \right|^2\right] \\
&\quad +6\E\left[\left| \sqrt{2\lambda\beta^{-1}}\int_{\lfrf{t}}^t \left(H(\widetilde{\zeta}^{\lambda, n}_s) -  H(\widetilde{\zeta}^{\lambda, n}_{\lfrf{s}}) \right) \, \rmd B^{\lambda}_s
  \right|^2\right] \\ 
&\quad +12\E\left[\left| -\lambda  \int_{\lfrf{t}}^t H(\overline{\Theta}^{\lambda}_s) h (\overline{\Theta}^{\lambda}_{\lfrf{s}})\,\rmd s\right|^2\right] + 12\E\left[\left|\lambda\beta^{-1}\int_{\lfrf{t}}^t\Upsilon (\overline{\Theta}^{\lambda}_s)\rmd s\right|^2\right]\\
&\quad + 12\E\left[\left|\lambda^2\int_{\lfrf{t}}^t H(\overline{\Theta}^{\lambda}_s)\int_{\lfrf{s}}^s H ( \overline{\Theta}^{\lambda}_{\lfrf{r}})h ( \overline{\Theta}^{\lambda}_{\lfrf{r}}) \,\rmd r\,\rmd s \right|^2\right] \\
&\quad + 12\E\left[\left|-\lambda^2\beta^{-1}\int_{\lfrf{t}}^t H(\overline{\Theta}^{\lambda}_s)\int_{\lfrf{s}}^s  \Upsilon (\overline{\Theta}^{\lambda}_{\lfrf{r}}) \,\rmd r\,\rmd s\right|^2\right] \\
&\quad + 12\E\left[\left|-\lambda\sqrt{2\lambda\beta^{-1}} \int_{\lfrf{t}}^t  H(\overline{\Theta}^{\lambda}_s) \int_{{\lfrf{s}}}^s H (\overline{\Theta}_{\lfrf{r}}^{\lambda})\,\rmd B_r^{\lambda}\,\rmd s \right|^2\right] \\
&\quad + 12\E\left[\left|\sqrt{2\lambda\beta^{-1}}\int_{\lfrf{t}}^t \left(H(\overline{\Theta}^{\lambda}_s) -  H(\overline{\Theta}_{\lfrf{s}}^{\lambda}) \right) \, \rmd B^{\lambda}_s \right|^2\right]   \\
&\leq  12\lambda^2\left(e^{- \overline{a} n/4} \overline{\cC}_{\mathbf{S}2}\E[|\theta_0|^4] +\widetilde{\cC}_{\mathbf{S}2} \right)+ 60 \lambda^2\left(e^{- \overline{a} n/4} \overline{\cC}_{\mathbf{S}2}\E[|\theta_0|^4] +\widetilde{\cC}_{\mathbf{S}2}\right)\\
&\leq \lambda^2\left(e^{- \overline{a} n/4} 72\overline{\cC}_{\mathbf{S}2}\E[|\theta_0|^4] +72\widetilde{\cC}_{\mathbf{S}2}\right),
\end{align*}
where the second last inequality holds by using Lemma \ref{lem:graditoestlip}.
\end{enumerate}
This completes the proof.
\end{proof}

\begin{definition}\label{def:Mdeflip} Define $\mathfrak{M}=(\mathfrak{M}^{(i,j)})_{i,j=1,\dots,d}:\R^d \times \R^d \rightarrow \R^{d\times d}$ by setting, for every $i,j = 1, \dots, d$,
\[
\mathfrak{M}^{(i,j)}(\theta, \overline{\theta}) = \langle\nabla H^{(i,j)}(\overline{\theta}), \theta-\overline{\theta} \rangle, \qquad \theta, \overline{\theta} \in \R^d.
\]
\end{definition}
\begin{lemma}\label{lem:Mestlip} Let Assumptions \ref{asm:AIlip}, \ref{asm:ALLlip}, and \ref{asm:ADlip} hold. Then, for any $0<\lambda\leq \overline{\lambda}_{\max}$, $n \in \N_0$, $, t \geq nT$, we obtain the following inequalities:
\begin{align} 
\begin{split}
&\E\left[\left| \sqrt{2\lambda\beta^{-1}} \int_{\lfrf{t}}^t   \left(H(\overline{\Theta}^{\lambda}_s) -  H ( \overline{\Theta}^{\lambda}_{\lfrf{s}})  -\mathfrak{M}(\overline{\Theta}^{\lambda}_s, \overline{\Theta}^{\lambda}_{\lfrf{s}})\right)  \,\rmd B_s^{\lambda}\right|^2\right] \\
& \leq  \lambda^{2+q}\left(e^{- \overline{a} n/2}  \overline{\cC}_{\mathbf{S}2}\E[|\theta_0|^4] + \widetilde{\cC}_{\mathbf{S}2} \right),\label{lem:Mestlipineq1}
\end{split}\\
&\E\left[\left| \sqrt{2\lambda\beta^{-1}} \int_{\lfrf{t}}^t   \ \mathfrak{M}(\overline{\Theta}^{\lambda}_s, \overline{\Theta}^{\lambda}_{\lfrf{s}})   \,\rmd B_s^{\lambda}\right|^2\right] 
\leq  \lambda^2\left(e^{- \overline{a} n/2}  \overline{\cC}_{\mathbf{S}2}\E[|\theta_0|^4] + \widetilde{\cC}_{\mathbf{S}2} \right),\label{lem:Mestlipineq2}\\
\begin{split}
&\E\left[2\lambda\beta^{-1} \left\langle  \int_{{\lfrf{t}}}^t  \int_{{\lfrf{s}}}^s \left(H (\widetilde{\zeta}^{\lambda, n}_{\lfrf{r}}) - H (\overline{\Theta}_{\lfrf{r}}^{\lambda})\right) \,\rmd B_r^{\lambda} \,\rmd s ,  \int_{\lfrf{t}}^t   \ \mathfrak{M}(\overline{\Theta}^{\lambda}_s, \overline{\Theta}^{\lambda}_{\lfrf{s}})   \,\rmd B_s^{\lambda}\right\rangle\right]\\
&\leq \lambda^2\left(e^{- \overline{a} n/4}5  \overline{\cC}_{\mathbf{S}2}\E[|\theta_0|^4] +5 \widetilde{\cC}_{\mathbf{S}2} \right),\label{lem:Mestlipineq3}
\end{split}
\end{align}
where $\overline{\cC}_{\mathbf{S}2}$ and $\widetilde{\cC}_{\mathbf{S}2}$ are given in \eqref{eq:graditoestlipconst}.
\end{lemma}

\begin{proof}To show the inequalities hold, we follow the arguments below:
\begin{enumerate}
\item To establish the first inequality \eqref{lem:Mestlipineq1}, by using Definition \ref{def:Mdeflip}, we observe that, for fixed $\theta, \overline{\theta} \in \R^d$, 
\begin{align}\label{eq:multidmvtlip}
&|H(\theta) - H(\overline{\theta}) - \mathfrak{M} (\theta, \overline{\theta}) |_{\cF}^2 \nonumber\\
&=\sum_{i,j= 1}^d|H^{(i,j)}(\theta) - H^{(i,j)}(\overline{\theta}) - \mathfrak{M}^{(i,j)}(\theta, \overline{\theta}) |^2 \nonumber\\
& = \sum_{i,j= 1}^d\left|\int_0^1 \langle \nabla H^{(i,j)}(\nu\theta+(1-\nu)\overline{\theta}), \theta-\overline{\theta} \rangle\, \rmd \nu - \langle\nabla H^{(i,j)}(\overline{\theta}), \theta-\overline{\theta} \rangle\right|^2 \nonumber\\
&\leq \int_0^1\sum_{i,j= 1}^d |\nabla H^{(i,j)}(\nu\theta+(1-\nu)\overline{\theta}) - \nabla H^{(i,j)}(\overline{\theta})|^2\, \rmd \nu |\theta-\overline{\theta} |^2 \nonumber\\
&\leq \int_0^1\sum_{i= 1}^d d |\nabla^2 h^{(i)}(\nu\theta+(1-\nu)\overline{\theta}) - \nabla^2 h^{(i)}(\overline{\theta})| ^2\, \rmd \nu |\theta-\overline{\theta} |^2 \nonumber\\
&\leq  d^2\overline{L}_1^2 |\theta-\overline{\theta}|^{2+2q},
\end{align}
where the last inequality holds due to Assumption \ref{asm:ALLlip}. By using \eqref{eq:multidmvtlip}, we obtain that
\begin{align*}
&\E\left[\left| \sqrt{2\lambda\beta^{-1}} \int_{\lfrf{t}}^t   \left(H(\overline{\Theta}^{\lambda}_s) -  H ( \overline{\Theta}^{\lambda}_{\lfrf{s}})  -\mathfrak{M}(\overline{\Theta}^{\lambda}_s, \overline{\Theta}^{\lambda}_{\lfrf{s}})\right)  \,\rmd B_s^{\lambda}\right|^2\right] \\
& =2\lambda\beta^{-1}\int_{\lfrf{t}}^t  \E\left[\left| H(\overline{\Theta}^{\lambda}_s) -  H ( \overline{\Theta}^{\lambda}_{\lfrf{s}})  -\mathfrak{M}(\overline{\Theta}^{\lambda}_s, \overline{\Theta}^{\lambda}_{\lfrf{s}}) \right|_{\cF}^2\right]  \,\rmd s\\
&\leq  \lambda\beta^{-1} 2 d^2\overline{L}_1^2 \int_{\lfrf{t}}^t  \E\left[|\overline{\Theta}^{\lambda}_s-\overline{\Theta}^{\lambda}_{\lfrf{s}}|^{2+2q}\right]  \,\rmd s\\
&\leq   \lambda^{2+q}\beta^{-1} 2 d^2\overline{L}_1^2 \left(e^{-\lambda\overline{a} \lfrf{t}}\overline{\cC}_{\mathbf{S}0,1+q}\E[|\theta_0|^4]+\widetilde{\cC}_{\mathbf{S}0,1+q}  \right)  \\
&\leq \lambda^{2+q} \left(e^{- \overline{a} n/2}   \overline{\cC}_{\mathbf{S}2}\E[|\theta_0|^4] + \widetilde{\cC}_{\mathbf{S}2} \right),
\end{align*}
where the third inequality holds due to Lemma \ref{lem:2ndpthmmtlip} and \ref{lem:oserroralglip}, and $\overline{\cC}_{\mathbf{S}2}$, $\widetilde{\cC}_{\mathbf{S}2}$ are given in \eqref{eq:graditoestlipconst}.
\item To establish the second inequality \eqref{lem:Mestlipineq2}, we recall Definition \ref{def:Mdeflip} and use Remark \ref{rmk:growthclip} to obtain
\begin{align*}
&\E\left[\left| \sqrt{2\lambda\beta^{-1}} \int_{\lfrf{t}}^t    \mathfrak{M}(\overline{\Theta}^{\lambda}_s, \overline{\Theta}^{\lambda}_{\lfrf{s}})   \,\rmd B_s^{\lambda}\right|^2\right] \\
& = 2\lambda\beta^{-1}\int_{\lfrf{t}}^t  \E\left[\sum_{i,j=1}^d\left|   \left\langle\nabla H^{(i,j)}( \overline{\Theta}^{\lambda}_{\lfrf{s}}) , \overline{\Theta}^{\lambda}_s- \overline{\Theta}^{\lambda}_{\lfrf{s}} \right\rangle\right|^2\right]  \,\rmd s\\
&\leq 2^{3-2q}\lambda\beta^{-1}\int_{\lfrf{t}}^t  \E\left[d^2\overline{\cK}_0^2(1+|\overline{\Theta}^{\lambda}_{\lfrf{s}}|)^{2q} |\overline{\Theta}^{\lambda}_s- \overline{\Theta}^{\lambda}_{\lfrf{s}} |^2\right]  \,\rmd s\\
&\leq  \lambda\beta^{-1}2^5d^2\overline{\cK}_0^2\int_{\lfrf{t}}^t   \left(\E\left[1+|\overline{\Theta}^{\lambda}_{\lfrf{s}}|^4  \right] \right)^{1/2} \left(\E\left[ |\overline{\Theta}^{\lambda}_s- \overline{\Theta}^{\lambda}_{\lfrf{s}} |^4\right] \right)^{1/2}\,\rmd s\\
&\leq  \lambda\beta^{-1}2^5d^2\overline{\cK}_0^2 \left(e^{-\lambda \overline{a}  \lfrf{t}}\E\left[ |\theta_0|^4\right]+  \ccl_2\left(1+1/\overline{a}\right)+1\right)^{1/2}\\
&\qquad \times   \lambda \left(e^{-\lambda\overline{a} \lfrf{t}}\overline{\cC}_{\mathbf{S}0,2}\E[|\theta_0|^4]+\widetilde{\cC}_{\mathbf{S}0,2} \right)^{1/2}  \\
&\leq \lambda^2 \left(e^{- \overline{a}  n/2} \overline{\cC}_{\mathbf{S}2}\E[|\theta_0|^4] +\widetilde{\cC}_{\mathbf{S}2} \right),
\end{align*}
where the third inequality holds due to Lemma \ref{lem:2ndpthmmtlip} and \ref{lem:oserroralglip}, and $\overline{\cC}_{\mathbf{S}2}$, $\widetilde{\cC}_{\mathbf{S}2}$ are given in \eqref{eq:graditoestlipconst}.

\item To obtain the third inequality \eqref{lem:Mestlipineq3}, we use Definition \ref{def:Mdeflip} and write the following
\begin{align}
&\E\left[2\lambda\beta^{-1} \left\langle  \int_{{\lfrf{t}}}^t  \int_{{\lfrf{s}}}^s \left(H (\widetilde{\zeta}^{\lambda, n}_{\lfrf{r}}) - H (\overline{\Theta}_{\lfrf{r}}^{\lambda})\right) \,\rmd B_r^{\lambda} \,\rmd s ,  \int_{\lfrf{t}}^t    \mathfrak{M}(\overline{\Theta}^{\lambda}_s, \overline{\Theta}^{\lambda}_{\lfrf{s}})   \,\rmd B_s^{\lambda}\right\rangle\right]\nonumber\\
\begin{split}\label{eq:Mestthdlipineqexpsn} 
& = 2\lambda\beta^{-1} \E\left[\left\langle  \int_{{\lfrf{t}}}^t  \int_{{\lfrf{s}}}^s \left(H (\widetilde{\zeta}^{\lambda, n}_{\lfrf{r}}) - H (\overline{\Theta}_{\lfrf{r}}^{\lambda})\right) \,\rmd B_r^{\lambda} \,\rmd s ,  \right.\right.\\
&\qquad \left.\left. \int_{\lfrf{t}}^t   \left\langle\nabla H(\overline{\Theta}^{\lambda}_{\lfrf{s}}),    
\int_{\lfrf{s}}^s - \lambda h (\overline{\Theta}^{\lambda}_{\lfrf{r}})   \,\rmd r
\right\rangle \,\rmd B_s^{\lambda}\right\rangle\right] \\
& \quad+ 2\lambda\beta^{-1} \E\left[\left\langle  \int_{{\lfrf{t}}}^t  \int_{{\lfrf{s}}}^s \left(H (\widetilde{\zeta}^{\lambda, n}_{\lfrf{r}}) - H (\overline{\Theta}_{\lfrf{r}}^{\lambda})\right) \,\rmd B_r^{\lambda} \,\rmd s ,   \right.\right.\\
&\qquad \left.\left.  \int_{\lfrf{t}}^t   \left\langle    \nabla H(\overline{\Theta}^{\lambda}_{\lfrf{s}}),    
\int_{\lfrf{s}}^s     \int_{\lfrf{r}}^r \lambda^2 H ( \overline{\Theta}^{\lambda}_{\lfrf{\nu}})h ( \overline{\Theta}^{\lambda}_{\lfrf{\nu}}) \,\rmd \nu   \,\rmd r
\right\rangle \,\rmd B_s^{\lambda}\right\rangle\right] \\
& \quad+ 2\lambda\beta^{-1} \E\left[\left\langle  \int_{{\lfrf{t}}}^t  \int_{{\lfrf{s}}}^s \left(H (\widetilde{\zeta}^{\lambda, n}_{\lfrf{r}}) - H (\overline{\Theta}_{\lfrf{r}}^{\lambda})\right) \,\rmd B_r^{\lambda} \,\rmd s ,   \right.\right.\\
&\qquad \left.\left.  \int_{\lfrf{t}}^t   \left\langle    \nabla H(\overline{\Theta}^{\lambda}_{\lfrf{s}}),    
\int_{\lfrf{s}}^s    \int_{\lfrf{r}}^r -\lambda^2 \beta^{-1}\Upsilon (\overline{\Theta}^{\lambda}_{\lfrf{\nu}}) \,\rmd \nu   \,\rmd r
\right\rangle \,\rmd B_s^{\lambda}\right\rangle\right] \\
& \quad+ 2\lambda\beta^{-1} \E\left[\left\langle  \int_{{\lfrf{t}}}^t  \int_{{\lfrf{s}}}^s \left(H (\widetilde{\zeta}^{\lambda, n}_{\lfrf{r}}) - H (\overline{\Theta}_{\lfrf{r}}^{\lambda})\right) \,\rmd B_r^{\lambda} \,\rmd s ,   \right.\right.\\
&\qquad \left.\left. \int_{\lfrf{t}}^t   \left\langle   \nabla H(\overline{\Theta}^{\lambda}_{\lfrf{s}}),    
\int_{\lfrf{s}}^s -\lambda\sqrt{2\lambda\beta^{-1}} \int_{{\lfrf{r}}}^r H (\overline{\Theta}_{\lfrf{\nu}}^{\lambda})\,\rmd B_{\nu}^{\lambda}    \,\rmd r
\right\rangle \,\rmd B_s^{\lambda}\right\rangle\right] \\
& \quad+ 2\lambda\beta^{-1} \E\left[\left\langle  \int_{{\lfrf{t}}}^t  \int_{{\lfrf{s}}}^s \left(H (\widetilde{\zeta}^{\lambda, n}_{\lfrf{r}}) - H (\overline{\Theta}_{\lfrf{r}}^{\lambda})\right) \,\rmd B_r^{\lambda} \,\rmd s ,   \right.\right.\\
&\qquad \left.\left. \int_{\lfrf{t}}^t   \left\langle\nabla H(\overline{\Theta}^{\lambda}_{\lfrf{s}}),    
\sqrt{2\lambda\beta^{-1}}\int_{\lfrf{s}}^s \, \rmd B^{\lambda}_r
\right\rangle \,\rmd B_s^{\lambda}\right\rangle\right]  .
\end{split}
\end{align}
Recall that $(\mathcal{F}^{\lambda}_t)_{t \geq 0}$ is the completed natural filtration of $(B^{\lambda}_t)_{t\geq0}$. Then, we note that, for any $i,j,k=1, \dots, d$, it holds that
\begin{align*}
&\E\left[ \left( \int_{{\lfrf{t}}}^t   \int_{{\lfrf{s}}}^s \left(H^{(i,j)} (\widetilde{\zeta}^{\lambda, n}_{\lfrf{r}}) - H^{(i,j)} (\overline{\Theta}_{\lfrf{r}}^{\lambda})\right) \,\rmd (B_r^{\lambda})^{(j)} \,\rmd s   \right)\right.\\
&\qquad \times \left. \left(  \int_{\lfrf{t}}^t   \left(  \partial_{\theta^{(k)}} H^{(i,j)}(\overline{\Theta}^{\lambda}_{\lfrf{s}})   
\int_{\lfrf{s}}^s \, \rmd (B^{\lambda}_r)^{(k)}\right)
\,\rmd (B_s^{\lambda})^{(j)}\right) \right]\\
& = \E\Bigg[ \left(H^{(i,j)} (\widetilde{\zeta}^{\lambda, n}_{\lfrf{t}}) - H^{(i,j)} (\overline{\Theta}_{\lfrf{t}}^{\lambda})\right)\partial_{\theta^{(k)}} H^{(i,j)}(\overline{\Theta}^{\lambda}_{\lfrf{t}})    \Bigg.\\
&\qquad \times \left. \E\left[  \left.\int_{{\lfrf{t}}}^t   \int_{{\lfrf{s}}}^s \,\rmd (B_r^{\lambda})^{(j)} \,\rmd s     \int_{\lfrf{t}}^t    
\int_{\lfrf{s}}^s \, \rmd (B^{\lambda}_r)^{(k)}  \,\rmd (B_s^{\lambda})^{(j)} 
\right|\mathcal{F}^{\lambda}_{\lfrf{t}}\right]\right]\\
& = \E\Bigg[ \left(H^{(i,j)} (\widetilde{\zeta}^{\lambda, n}_{\lfrf{t}}) - H^{(i,j)} (\overline{\Theta}_{\lfrf{t}}^{\lambda})\right)\partial_{\theta^{(k)}} H^{(i,j)}(\overline{\Theta}^{\lambda}_{\lfrf{t}})    \Bigg.\\
&\qquad \times \left. \E\left[  \left. \int_{{\lfrf{t}}}^t (t-s) \,\rmd (B_s^{\lambda})^{(j)}  \int_{\lfrf{t}}^t    
\int_{\lfrf{s}}^s \, \rmd (B^{\lambda}_r)^{(k)}  \,\rmd (B_s^{\lambda})^{(j)} 
\right|\mathcal{F}^{\lambda}_{\lfrf{t}}\right]\right]\\
& = \E\Bigg[ \left(H^{(i,j)} (\widetilde{\zeta}^{\lambda, n}_{\lfrf{t}}) - H^{(i,j)} (\overline{\Theta}_{\lfrf{t}}^{\lambda})\right)\partial_{\theta^{(k)}} H^{(i,j)}(\overline{\Theta}^{\lambda}_{\lfrf{t}})    \Bigg.\\
&\qquad \times \left.\left( t\int_{{\lfrf{t}}}^t   \E\left[  \left.  \int_{\lfrf{t}}^s \, \rmd (B^{\lambda}_r)^{(k)}  
\right|\mathcal{F}^{\lambda}_{\lfrf{t}}\right]\,\rmd s  -  \int_{{\lfrf{t}}}^t s  \E\left[  \left.  \int_{\lfrf{t}}^s \, \rmd (B^{\lambda}_r)^{(k)}  
\right|\mathcal{F}^{\lambda}_{\lfrf{t}}\right]\,\rmd s  \right)\right]\\
& = 0.
\end{align*}
This implies that the last term in \eqref{eq:Mestthdlipineqexpsn} is zero. Indeed, we have that
\begin{align*}&2\lambda\beta^{-1} \E\left[\left\langle  \int_{{\lfrf{t}}}^t  \int_{{\lfrf{s}}}^s \left(H (\widetilde{\zeta}^{\lambda, n}_{\lfrf{r}}) - H (\overline{\Theta}_{\lfrf{r}}^{\lambda})\right) \,\rmd B_r^{\lambda} \,\rmd s ,   \right.\right.\\
&\qquad \left.\left. \int_{\lfrf{t}}^t   \left\langle\nabla H(\overline{\Theta}^{\lambda}_{\lfrf{s}}),    
\sqrt{2\lambda\beta^{-1}}\int_{\lfrf{s}}^s \, \rmd B^{\lambda}_r
\right\rangle \,\rmd B_s^{\lambda}\right\rangle\right]\\
& = (2\lambda\beta^{-1})^{3/2} \E\left[\sum_{i=1}^d\left( \int_{{\lfrf{t}}}^t \sum_{j=1}^d  \int_{{\lfrf{s}}}^s \left(H^{(i,j)} (\widetilde{\zeta}^{\lambda, n}_{\lfrf{r}}) - H^{(i,j)} (\overline{\Theta}_{\lfrf{r}}^{\lambda})\right) \,\rmd (B_r^{\lambda})^{(j)} \,\rmd s   \right)\right.\\
&\qquad \times \left. \left( \sum_{j = 1}^d \int_{\lfrf{t}}^t   \left(\sum_{k=1}^d \partial_{\theta^{(k)}} H^{(i,j)}(\overline{\Theta}^{\lambda}_{\lfrf{s}})   
\int_{\lfrf{s}}^s \, \rmd (B^{\lambda}_r)^{(k)}\right)
 \,\rmd (B_s^{\lambda})^{(j)}\right) \right]=0.
\end{align*}
Then, by using Remark \ref{rmk:growthclip} and \eqref{eq:Mestthdlipineqexpsn} with the result above, we obtain that
\begin{align*}
&\E\left[2\lambda\beta^{-1} \left\langle  \int_{{\lfrf{t}}}^t  \int_{{\lfrf{s}}}^s \left(H (\widetilde{\zeta}^{\lambda, n}_{\lfrf{r}}) - H (\overline{\Theta}_{\lfrf{r}}^{\lambda})\right) \,\rmd B_r^{\lambda} \,\rmd s ,  \int_{\lfrf{t}}^t    \mathfrak{M}(\overline{\Theta}^{\lambda}_s, \overline{\Theta}^{\lambda}_{\lfrf{s}})   \,\rmd B_s^{\lambda}\right\rangle\right] \\ 
&\leq 4\lambda^2\beta^{-2} \E\left[\left| \int_{{\lfrf{t}}}^t  \int_{{\lfrf{s}}}^s \left(H (\widetilde{\zeta}^{\lambda, n}_{\lfrf{r}}) - H (\overline{\Theta}_{\lfrf{r}}^{\lambda})\right) \,\rmd B_r^{\lambda} \,\rmd s  \right|^2\right]\\
&\quad + \lambda^2\E\left[\left|\int_{\lfrf{t}}^t   \left\langle\nabla H(\overline{\Theta}^{\lambda}_{\lfrf{s}}),    
\int_{\lfrf{s}}^s -   h (\overline{\Theta}^{\lambda}_{\lfrf{r}})   \,\rmd r      \right\rangle \,\rmd B_s^{\lambda}  \right|^2\right]\\ 
&\quad + \lambda^4\E\left[\left|\int_{\lfrf{t}}^t   \left\langle    \nabla H(\overline{\Theta}^{\lambda}_{\lfrf{s}}),    
\int_{\lfrf{s}}^s     \int_{\lfrf{r}}^r  H ( \overline{\Theta}^{\lambda}_{\lfrf{\nu}})h ( \overline{\Theta}^{\lambda}_{\lfrf{\nu}}) \,\rmd \nu   \,\rmd r   \right\rangle \,\rmd B_s^{\lambda} \right|^2\right]\\ 
&\quad + \lambda^4\beta^{-2}\E\left[\left| \int_{\lfrf{t}}^t   \left\langle    \nabla H(\overline{\Theta}^{\lambda}_{\lfrf{s}}),    
\int_{\lfrf{s}}^s    \int_{\lfrf{r}}^r - \Upsilon (\overline{\Theta}^{\lambda}_{\lfrf{\nu}}) \,\rmd \nu   \,\rmd r     \right\rangle \,\rmd B_s^{\lambda} \right|^2\right]\\ 
&\quad +2\lambda^3\beta^{-1}\E\left[\left|  \int_{\lfrf{t}}^t   \left\langle   \nabla H(\overline{\Theta}^{\lambda}_{\lfrf{s}}),    
\int_{\lfrf{s}}^s -\int_{{\lfrf{r}}}^r H (\overline{\Theta}_{\lfrf{\nu}}^{\lambda})\,\rmd B_{\nu}^{\lambda}    \,\rmd r \right\rangle \,\rmd B_s^{\lambda} \right|^2\right]\\ 
&\leq 4\lambda^2\beta^{-2} \int_{{\lfrf{t}}}^t  \int_{{\lfrf{s}}}^s \E\left[|H (\widetilde{\zeta}^{\lambda, n}_{\lfrf{r}}) - H (\overline{\Theta}_{\lfrf{r}}^{\lambda})|_{\cF}^2\right]\,\rmd r \,\rmd s \\
&\quad + \lambda^2\int_{\lfrf{t}}^t \E\left[\sum_{i,j=1}^d\left|  \left\langle\nabla H^{(i,j)}(\overline{\Theta}^{\lambda}_{\lfrf{s}}),    h (\overline{\Theta}^{\lambda}_{\lfrf{s}})   \right\rangle  \right|^2\right]\, \rmd s \\ 
&\quad + \lambda^4\int_{\lfrf{t}}^t \E\left[\sum_{i,j=1}^d\left|  \left\langle\nabla H^{(i,j)}(\overline{\Theta}^{\lambda}_{\lfrf{s}}),    H ( \overline{\Theta}^{\lambda}_{\lfrf{s}})h ( \overline{\Theta}^{\lambda}_{\lfrf{s}})   \right\rangle  \right|^2\right]\, \rmd s \\ 
&\quad + \lambda^4\beta^{-2}\int_{\lfrf{t}}^t \E\left[\sum_{i,j=1}^d\left|  \left\langle\nabla H^{(i,j)}(\overline{\Theta}^{\lambda}_{\lfrf{s}}),   \Upsilon (\overline{\Theta}^{\lambda}_{\lfrf{s}})  \right\rangle  \right|^2\right]\, \rmd s \\ 
&\quad +2\lambda^3\beta^{-1} \int_{\lfrf{t}}^t \E\left[\sum_{i,j=1}^d\left|  \left\langle\nabla H^{(i,j)}(\overline{\Theta}^{\lambda}_{\lfrf{s}}),  H (\overline{\Theta}_{\lfrf{s}}^{\lambda}) \int_{\lfrf{s}}^s \int_{{\lfrf{r}}}^r \,\rmd B_{\nu}^{\lambda}    \,\rmd r  \right\rangle  \right|^2\right]\, \rmd s \\ 
&\leq 8\lambda^2\beta^{-2} d\overline{L}_2^2\int_{{\lfrf{t}}}^t  \E\left[|\widetilde{\zeta}^{\lambda, n}_{\lfrf{s}}|^2+|\overline{\Theta}_{\lfrf{s}}^{\lambda}|^2  \right] \,\rmd s \\
&\quad +  \lambda^2\int_{\lfrf{t}}^t \E\left[d^2\cKl_0^2(1+|\overline{\Theta}^{\lambda}_{\lfrf{s}}|^q)^2\cKl_1^2(1+|\overline{\Theta}^{\lambda}_{\lfrf{s}}|)^2  \right]\, \rmd s \\ 
&\quad +  \lambda^4\int_{\lfrf{t}}^t \E\left[d^2\cKl_0^2(1+|\overline{\Theta}^{\lambda}_{\lfrf{s}}|^q)^2\overline{L}_3^2\cKl_1^2(1+|\overline{\Theta}^{\lambda}_{\lfrf{s}}|)^2  \right]\, \rmd s \\ 
&\quad +  \lambda^4\beta^{-2}\int_{\lfrf{t}}^t \E\left[d^2\cKl_0^2(1+|\overline{\Theta}^{\lambda}_{\lfrf{s}}|^q)^2d^2\overline{L}_2^2 \right]\, \rmd s \\ 
&\quad + 2\lambda^3\beta^{-1}  \int_{\lfrf{t}}^t \E\left[d^2\cKl_0^2(1+|\overline{\Theta}^{\lambda}_{\lfrf{s}}|^q)^2\overline{L}_3^2 \left| \int_{\lfrf{s}}^s \int_{{\lfrf{r}}}^r \,\rmd B_{\nu}^{\lambda}    \,\rmd r  \right|^2   \right]\, \rmd s \\ 
&\leq 8\lambda^2\beta^{-2} d\overline{L}_2^2 \int_{{\lfrf{t}}}^t  \E\left[  V_4(\widetilde{\zeta}^{\lambda, n}_{\lfrf{s}})+1+|\overline{\Theta}_{\lfrf{s}}^{\lambda}|^4   \right] \,\rmd s \\
&\quad + 4 \lambda^2(  \cKl_1^2+\overline{L}_3^2\cKl_1^2+d^2\beta^{-2} \overline{L}_2^2)d^2\cKl_0^2\int_{\lfrf{t}}^t \E\left[1+|\overline{\Theta}^{\lambda}_{\lfrf{s}}|^4 \right]\, \rmd s \\ 
&\quad + 8\lambda^2 \beta^{-1}d^2\cKl_0^2 \overline{L}_3^2\int_{\lfrf{t}}^t \left(\E\left[1+|\overline{\Theta}^{\lambda}_{\lfrf{s}}|^4\right]\right)^{1/2}(3(d+4))\, \rmd s \\ 
&\leq 24\lambda^2\beta^{-2} d\overline{L}_2^2 \left( e^{-\lambda\overline{a}\lfrf{t}/2}\E[|\theta_0|^4] + \left( \ccl_2\left(1+1/\overline{a}\right)+1\right) +\mathrm{v}_4(\cMl_V(4)\right)\\
&\quad + 2^7 \lambda^2(  \cKl_1^2+\overline{L}_3^2\cKl_1^2+\beta^{-2} \overline{L}_2^2+\beta^{-1}\overline{L}_3^2)d^4\cKl_0^2 \\
&\qquad \times \left( e^{-\lambda\overline{a}\lfrf{t}/2}\E[|\theta_0|^4] +\ccl_2\left(1+1/\overline{a}\right)+1\right)\\
&\leq \lambda^2\left(e^{- \overline{a} n/4} 5\overline{\cC}_{\mathbf{S}2}\E[|\theta_0|^4] +5\widetilde{\cC}_{\mathbf{S}2} \right),
\end{align*}
where the fifth inequality holds due to Lemma \ref{lem:2ndpthmmtlip} and \ref{lem:zetaprocmelip}, and where $\overline{\cC}_{\mathbf{S}2}$, $\widetilde{\cC}_{\mathbf{S}2}$ are given in \eqref{eq:graditoestlipconst}.
\end{enumerate}
This completes the proof.
\end{proof}


\begin{proof}[\textbf{Proof of Lemma \ref{lem:w1converlipp1} }]
\label{lem:w1converlipp1proof}
By using the definitions of  $\overline{\Theta}^{\lambda}_t$ in \eqref{eq:aholahoproclip} and $\widetilde{\zeta}^{\lambda, n}_t $ in Definition \ref{def:auxzetalip}, and by applying It\^o's formula, we obtain, for any $n \in \N_0$, $t \in (nT, (n+1)T]$,
\begin{align}
&W_2^2(\mathcal{L}(\overline{\Theta}^{\lambda}_t),\mathcal{L}(\widetilde{\zeta}^{\lambda, n}_t)) \nonumber\\
&\leq \E\left[\left|\overline{\Theta}^{\lambda}_t - \widetilde{\zeta}^{\lambda, n}_t\right|^2\right] \nonumber\\
&=-2\lambda \E\left[\int_{nT}^t \left\langle\overline{\Theta}^{\lambda}_s - \widetilde{\zeta}^{\lambda, n}_s, h (\overline{\Theta}^{\lambda}_{\lfrf{s}}) -h(\widetilde{\zeta}^{\lambda, n}_s)  \right\rangle\, \rmd s\right]\nonumber\\
&\quad +2\lambda \E\left[\int_{nT}^t \left\langle\overline{\Theta}^{\lambda}_s - \widetilde{\zeta}^{\lambda, n}_s,  \lambda \int_{\lfrf{s}}^s\left(H ( \overline{\Theta}^{\lambda}_{\lfrf{r}})h ( \overline{\Theta}^{\lambda}_{\lfrf{r}})-\beta^{-1}\Upsilon (\overline{\Theta}^{\lambda}_{\lfrf{r}}) \right)\,\rmd r  \right\rangle\, \rmd s\right]\nonumber\\
&\quad -2\lambda \E\left[\int_{nT}^t \left\langle\overline{\Theta}^{\lambda}_s - \widetilde{\zeta}^{\lambda, n}_s,     \sqrt{2\lambda\beta^{-1}} \int_{{\lfrf{s}}}^s H (\overline{\Theta}_{\lfrf{r}}^{\lambda}) \,\rmd B_r^{\lambda}  \right\rangle\, \rmd s\right]\nonumber\\
\begin{split}\label{eq:L2convspltinglip}
& =-2\lambda \E\left[\int_{nT}^t \left\langle\overline{\Theta}^{\lambda}_s - \widetilde{\zeta}^{\lambda, n}_s, h(\overline{\Theta}^{\lambda}_s) - h(\widetilde{\zeta}^{\lambda, n}_s)  \right\rangle\, \rmd s\right] \\
& \quad -2\lambda \E\left[\int_{nT}^t \left\langle\overline{\Theta}^{\lambda}_s - \widetilde{\zeta}^{\lambda, n}_s, h (\overline{\Theta}^{\lambda}_{\lfrf{s}})  - h(\overline{\Theta}^{\lambda}_s)  \right\rangle\, \rmd s\right] \\
&\quad +2\lambda \E\left[\int_{nT}^t \left\langle\overline{\Theta}^{\lambda}_s - \widetilde{\zeta}^{\lambda, n}_s,  \lambda \int_{\lfrf{s}}^s\left(H ( \overline{\Theta}^{\lambda}_{\lfrf{r}})h ( \overline{\Theta}^{\lambda}_{\lfrf{r}})-\beta^{-1}\Upsilon (\overline{\Theta}^{\lambda}_{\lfrf{r}}) \right)\,\rmd r  \right\rangle\, \rmd s\right] \\
&\quad - 2\lambda \E\left[\int_{nT}^t \left\langle\overline{\Theta}^{\lambda}_s - \widetilde{\zeta}^{\lambda, n}_s, \sqrt{2\lambda\beta^{-1}} \int_{{\lfrf{s}}}^s H (\overline{\Theta}_{\lfrf{r}}^{\lambda})\,\rmd B_r^{\lambda}  \right\rangle\, \rmd s\right].
\end{split}
\end{align}
By applying It\^o's formula to $ h(\overline{\Theta}^{\lambda}_s)$, we obtain \eqref{eq:holaproclipito}. Substituting \eqref{eq:holaproclipito} into \eqref{eq:L2convspltinglip}, using Assumption \ref{asm:ALLlip} and Young's inequality yield
\begin{align}
&\E\left[\left|\overline{\Theta}^{\lambda}_t - \widetilde{\zeta}^{\lambda, n}_t\right|^2\right] \nonumber\\
&\leq 2\lambda \overline{L}_3\int_{nT}^t  \E\left[|\overline{\Theta}^{\lambda}_s - \widetilde{\zeta}^{\lambda, n}_s|^2 \right] \, \rmd s   \nonumber\\
& \quad +2\lambda \E\left[\int_{nT}^t \left\langle\overline{\Theta}^{\lambda}_s - \widetilde{\zeta}^{\lambda, n}_s, 
-\lambda  \int_{\lfrf{s}}^s \left(H(\overline{\Theta}^{\lambda}_r) - H ( \overline{\Theta}^{\lambda}_{\lfrf{r}}) \right)h (\overline{\Theta}^{\lambda}_{\lfrf{r}})\,\rmd r  \right\rangle\, \rmd s\right]   \nonumber\\
& \quad +2\lambda \E\left[\int_{nT}^t \left\langle\overline{\Theta}^{\lambda}_s - \widetilde{\zeta}^{\lambda, n}_s, 
\lambda^2\int_{\lfrf{s}}^s H(\overline{\Theta}^{\lambda}_r)\int_{\lfrf{r}}^r H ( \overline{\Theta}^{\lambda}_{\lfrf{\nu}})h ( \overline{\Theta}^{\lambda}_{\lfrf{\nu}}) \,\rmd \nu\,\rmd r   \right\rangle\, \rmd s\right]  \nonumber \\
& \quad +2\lambda \E\left[\int_{nT}^t \left\langle\overline{\Theta}^{\lambda}_s - \widetilde{\zeta}^{\lambda, n}_s, 
-\lambda^2\beta^{-1}\int_{\lfrf{s}}^s H(\overline{\Theta}^{\lambda}_r)\int_{\lfrf{r}}^r  \Upsilon (\overline{\Theta}^{\lambda}_{\lfrf{\nu}}) \,\rmd \nu\,\rmd r  \right\rangle\, \rmd s\right] \nonumber\\
& \quad +2\lambda \E\left[\int_{nT}^t \left\langle\overline{\Theta}^{\lambda}_s - \widetilde{\zeta}^{\lambda, n}_s, 
-\lambda\sqrt{2\lambda\beta^{-1}} \int_{\lfrf{s}}^s  H(\overline{\Theta}^{\lambda}_r) \int_{{\lfrf{r}}}^r H (\overline{\Theta}_{\lfrf{\nu}}^{\lambda})\,\rmd B_{\nu}^{\lambda}\,\rmd r    \right\rangle\, \rmd s\right]  \nonumber \\
& \quad +2\lambda \E\left[\int_{nT}^t \left\langle\overline{\Theta}^{\lambda}_s - \widetilde{\zeta}^{\lambda, n}_s, 
\sqrt{2\lambda\beta^{-1}}\int_{\lfrf{s}}^s \left( H(\overline{\Theta}^{\lambda}_r) - H (\overline{\Theta}_{\lfrf{r}}^{\lambda})\right) \, \rmd B^{\lambda}_r   \right\rangle\, \rmd s\right]  \nonumber \\
& \quad +2\lambda \E\left[\int_{nT}^t \left\langle\overline{\Theta}^{\lambda}_s - \widetilde{\zeta}^{\lambda, n}_s, 
\lambda\beta^{-1}\int_{\lfrf{s}}^s\left(\Upsilon (\overline{\Theta}^{\lambda}_r) - \Upsilon (\overline{\Theta}^{\lambda}_{\lfrf{r}}) \right)\rmd r   \right\rangle\, \rmd s\right]  \nonumber\\
\begin{split}\label{eq:L2convspltingitoYlip}
&\leq  \lambda (2\overline{L}_3+5)\int_{nT}^t  \E\left[|\overline{\Theta}^{\lambda}_s - \widetilde{\zeta}^{\lambda, n}_s|^2 \right] \, \rmd s\\
& \quad +\lambda\int_{nT}^t \E\left[\left|-\lambda  \int_{\lfrf{s}}^s \left(H(\overline{\Theta}^{\lambda}_r) - H ( \overline{\Theta}^{\lambda}_{\lfrf{r}}) \right)h (\overline{\Theta}^{\lambda}_{\lfrf{r}})\,\rmd r \right|^2 \right] \, \rmd s \\
& \quad +\lambda\int_{nT}^t \E\left[\left|\lambda^2\int_{\lfrf{s}}^s H(\overline{\Theta}^{\lambda}_r)\int_{\lfrf{r}}^r H ( \overline{\Theta}^{\lambda}_{\lfrf{\nu}})h ( \overline{\Theta}^{\lambda}_{\lfrf{\nu}}) \,\rmd \nu\,\rmd r   \right|^2 \right] \, \rmd s \\ 
& \quad +\lambda\int_{nT}^t \E\left[\left|-\lambda^2\beta^{-1}\int_{\lfrf{s}}^s H(\overline{\Theta}^{\lambda}_r)\int_{\lfrf{r}}^r  \Upsilon (\overline{\Theta}^{\lambda}_{\lfrf{\nu}}) \,\rmd \nu\,\rmd r  \right|^2 \right] \, \rmd s \\ 
& \quad +\lambda\int_{nT}^t \E\left[\left|-\lambda\sqrt{2\lambda\beta^{-1}} \int_{\lfrf{s}}^s  H(\overline{\Theta}^{\lambda}_r) \int_{{\lfrf{r}}}^r H (\overline{\Theta}_{\lfrf{\nu}}^{\lambda})\,\rmd B_{\nu}^{\lambda}\,\rmd r   \right|^2 \right] \, \rmd s \\ 
& \quad +\lambda\int_{nT}^t \E\left[\left|\lambda\beta^{-1}\int_{\lfrf{s}}^s\left(\Upsilon (\overline{\Theta}^{\lambda}_r) - \Upsilon (\overline{\Theta}^{\lambda}_{\lfrf{r}}) \right)\rmd r   \right|^2 \right] \, \rmd s \\  
& \quad +2\lambda \E\left[\int_{nT}^t \left\langle\overline{\Theta}^{\lambda}_s - \widetilde{\zeta}^{\lambda, n}_s, 
\sqrt{2\lambda\beta^{-1}}\int_{\lfrf{s}}^s \left( H(\overline{\Theta}^{\lambda}_r) - H (\overline{\Theta}_{\lfrf{r}}^{\lambda})\right) \, \rmd B^{\lambda}_r   \right\rangle\, \rmd s\right] .
\end{split} 
\end{align}
By using Lemma \ref{lem:graditoestlip} and \eqref{eq:L2convspltingitoYlip} becomes
\begin{align}
\begin{split}\label{eq:L2convspltinglipMterm}
\E\left[\left|\overline{\Theta}^{\lambda}_t - \widetilde{\zeta}^{\lambda, n}_t\right|^2\right]
&\leq  \lambda (2\overline{L}_3+5)\int_{nT}^t  \E\left[|\overline{\Theta}^{\lambda}_s - \widetilde{\zeta}^{\lambda, n}_s|^2 \right] \, \rmd s\\
&\quad + 5\lambda^{2+q} \left(e^{- \overline{a} n/2} \overline{\cC}_{\mathbf{S}2}\E[|\theta_0|^4] +\widetilde{\cC}_{\mathbf{S}2} \right)\\
&\quad + \Jl_1+\Jl_2,
\end{split}
\end{align}
where
\begin{align*}
\Jl_1^{\lambda}(t)&:= 2\lambda \E\left[\int_{nT}^t \left\langle\overline{\Theta}^{\lambda}_s - \widetilde{\zeta}^{\lambda, n}_s, 
\sqrt{2\lambda\beta^{-1}}\int_{\lfrf{s}}^s \left( H(\overline{\Theta}^{\lambda}_r) - H_{\lambda}(\overline{\Theta}_{\lfrf{r}}^{\lambda}) - \mathfrak{M}(\overline{\Theta}^{\lambda}_r, \overline{\Theta}^{\lambda}_{\lfrf{r}})     \right)   \, \rmd B^{\lambda}_r   \right\rangle\, \rmd s\right],\\
\Jl_2^{\lambda}(t)&:= 2\lambda \E\left[\int_{nT}^t \left\langle\overline{\Theta}^{\lambda}_s - \widetilde{\zeta}^{\lambda, n}_s, 
\sqrt{2\lambda\beta^{-1}}\int_{\lfrf{s}}^s \mathfrak{M}(\overline{\Theta}^{\lambda}_r, \overline{\Theta}^{\lambda}_{\lfrf{r}}) \, \rmd B^{\lambda}_r   \right\rangle\, \rmd s\right]
\end{align*}
with $\mathfrak{M}$ defined in Definition \ref{def:Mdeflip}. By using Young's inequality and Lemma \ref{lem:Mestlip}, we have that
\begin{align}\label{eq:L2convspltinglipMtermJ1}
\Jl_1^{\lambda}(t) \leq \lambda \int_{nT}^t  \E\left[|\overline{\Theta}^{\lambda}_s - \widetilde{\zeta}^{\lambda, n}_s|^2 \right] \, \rmd s +  \lambda^{2+q}\left(e^{- \overline{a} n/2}  \overline{\cC}_{\mathbf{S}2}\E[|\theta_0|^4] + \widetilde{\cC}_{\mathbf{S}2} \right).
\end{align}
To establish an upper bound for $\Jl_2^{\lambda}(t)$, we recall the definitions of $(\overline{\Theta}^{\lambda}_t)_{t\geq 0}$ and $(\widetilde{\zeta}^{\lambda, n}_t)_{t \geq 0}$ given in \eqref{eq:aholahoproclip} and Definition \ref{def:auxzetalip}, respectively, and consider the following splitting:
\begin{align*}
\Jl_2^{\lambda}(t)
&= 2\lambda \int_{nT}^t\E\left[ \left\langle \overline{\Theta}^{\lambda}_s - \overline{\Theta}^{\lambda}_{\lfrf{s}} -( \widetilde{\zeta}^{\lambda, n}_s- \widetilde{\zeta}^{\lambda, n}_{\lfrf{s}}), 
\sqrt{2\lambda\beta^{-1}}\int_{\lfrf{s}}^s \mathfrak{M}(\overline{\Theta}^{\lambda}_r, \overline{\Theta}^{\lambda}_{\lfrf{r}}) \, \rmd B^{\lambda}_r   \right\rangle\right] \, \rmd s\\
&= 2\lambda \int_{nT}^t\E\left[ \left\langle \int_{\lfrf{s}}^s \left(\lambda (h(\widetilde{\zeta}^{\lambda, n}_r) - h (\overline{\Theta}^{\lambda}_{\lfrf{r}})) +\lambda^2\int_{\lfrf{r}}^r\left(H ( \overline{\Theta}^{\lambda}_{\lfrf{\nu}})h ( \overline{\Theta}^{\lambda}_{\lfrf{\nu}})-\beta^{-1}\Upsilon (\overline{\Theta}^{\lambda}_{\lfrf{\nu}}) \right)\,\rmd \nu \right.\right. \right.\\
&\qquad \left.\left.\left.  - \lambda\sqrt{2\lambda\beta^{-1}} \int_{{\lfrf{r}}}^r H (\overline{\Theta}_{\lfrf{\nu}}^{\lambda})\,\rmd B_{\nu}^{\lambda}\right)\,\rmd r, \sqrt{2\lambda\beta^{-1}}\int_{\lfrf{s}}^s \mathfrak{M}(\overline{\Theta}^{\lambda}_r, \overline{\Theta}^{\lambda}_{\lfrf{r}}) \, \rmd B^{\lambda}_r   \right\rangle\right] \, \rmd s\\
&= 2\lambda \int_{nT}^t\E\left[ \left\langle \int_{\lfrf{s}}^s \left(\lambda (h(\overline{\Theta}^{\lambda}_r)- h(\overline{\Theta}^{\lambda}_{\lfrf{r}})) +\lambda^2\int_{\lfrf{r}}^r\left(H ( \overline{\Theta}^{\lambda}_{\lfrf{\nu}})h ( \overline{\Theta}^{\lambda}_{\lfrf{\nu}})-\beta^{-1}\Upsilon (\overline{\Theta}^{\lambda}_{\lfrf{\nu}}) \right)\,\rmd \nu \right.\right. \right.\\
&\qquad \left.\left.\left.  - \lambda\sqrt{2\lambda\beta^{-1}} \int_{{\lfrf{r}}}^r H (\overline{\Theta}_{\lfrf{\nu}}^{\lambda})\,\rmd B_{\nu}^{\lambda}\right)\,\rmd r, \sqrt{2\lambda\beta^{-1}}\int_{\lfrf{s}}^s \mathfrak{M}(\overline{\Theta}^{\lambda}_r, \overline{\Theta}^{\lambda}_{\lfrf{r}}) \, \rmd B^{\lambda}_r   \right\rangle\right] \, \rmd s +\Jl_{2,1}^{\lambda}(t),
\end{align*}
where the first equality holds due to the following:
\begin{align}\label{eq:gridptmeasmartingalelip}
\begin{split}
&\E\left[ \left\langle  \overline{\Theta}^{\lambda}_{\lfrf{s}} +\widetilde{\zeta}^{\lambda, n}_{\lfrf{s}}, 
\sqrt{2\lambda\beta^{-1}}\int_{\lfrf{s}}^s \mathfrak{M}(\overline{\Theta}^{\lambda}_r, \overline{\Theta}^{\lambda}_{\lfrf{r}}) \, \rmd B^{\lambda}_r   \right\rangle\right] \\
&=\E\left[ \left\langle  \overline{\Theta}^{\lambda}_{\lfrf{s}} +\widetilde{\zeta}^{\lambda, n}_{\lfrf{s}}, \E\left[\left.
\sqrt{2\lambda\beta^{-1}}\int_{\lfrf{s}}^s \mathfrak{M}(\overline{\Theta}^{\lambda}_r, \overline{\Theta}^{\lambda}_{\lfrf{r}}) \, \rmd B^{\lambda}_r\right|\mathcal{F}^{\lambda}_{\lfrf{s}}\right]   \right\rangle\right]  = 0,
\end{split}
\end{align}
and where
\[
\Jl_{2,1}^{\lambda}(t) =  2\lambda \int_{nT}^t\E\left[ \left\langle \int_{\lfrf{s}}^s \lambda (h(\widetilde{\zeta}^{\lambda, n}_r) - h(\overline{\Theta}^{\lambda}_r)) \,\rmd r,
\sqrt{2\lambda\beta^{-1}}\int_{\lfrf{s}}^s \mathfrak{M}(\overline{\Theta}^{\lambda}_r, \overline{\Theta}^{\lambda}_{\lfrf{r}}) \, \rmd B^{\lambda}_r   \right\rangle\right] \, \rmd s.
\]
By applying Cauchy-Schwarz inequality, Corollary \ref{cor:graditoublip} and Lemma \ref{lem:Mestlip}, we further obtain that
\begin{align}\label{eq:L2convspltinglipMtermJ2}
\Jl_2^{\lambda}(t)
&\leq 2\lambda \int_{nT}^t\left(\E\left[ \left| \int_{\lfrf{s}}^s \lambda \left( h(\overline{\Theta}^{\lambda}_r)- h(\overline{\Theta}^{\lambda}_{\lfrf{r}}) +\lambda \int_{\lfrf{r}}^r\left(H ( \overline{\Theta}^{\lambda}_{\lfrf{\nu}})h ( \overline{\Theta}^{\lambda}_{\lfrf{\nu}})-\beta^{-1}\Upsilon (\overline{\Theta}^{\lambda}_{\lfrf{\nu}}) \right)\,\rmd \nu \right.\right. \right.\right. \nonumber\\
&\qquad \left.\left.\left.\left.  -  \sqrt{2\lambda\beta^{-1}} \int_{{\lfrf{r}}}^r H (\overline{\Theta}_{\lfrf{\nu}}^{\lambda})\,\rmd B_{\nu}^{\lambda}\right)\,\rmd r  \right|^2\right]\right)^{1/2} \nonumber \\
&\qquad \times \left(\E\left[\left| \sqrt{2\lambda\beta^{-1}}\int_{\lfrf{s}}^s \mathfrak{M}(\overline{\Theta}^{\lambda}_r, \overline{\Theta}^{\lambda}_{\lfrf{r}}) \, \rmd B^{\lambda}_r\right|^2\right]\right)^{1/2} \, \rmd s +\Jl_{2,1}^{\lambda}(t)\nonumber\\
&\leq  2\lambda \int_{nT}^t\left(\lambda^4 \left(e^{- \overline{a}n/2}36 \overline{\cC}_{\mathbf{S}2}\E[|\theta_0|^4] +36\widetilde{\cC}_{\mathbf{S}2} \right)\right)^{1/2} \nonumber\\
&\qquad \times \left(\lambda^2\left(e^{- \overline{a} n/2}  \overline{\cC}_{\mathbf{S}2}\E[|\theta_0|^4] + \widetilde{\cC}_{\mathbf{S}2} \right)\right)^{1/2} \, \rmd s+\Jl_{2,1}^{\lambda}(t) \nonumber\\
&\leq 12\lambda^3\left(e^{- \overline{a} n/2}  \overline{\cC}_{\mathbf{S}2}\E[|\theta_0|^4] + \widetilde{\cC}_{\mathbf{S}2} \right)+\Jl_{2,1}^{\lambda}(t).
\end{align}
At this stage, the task reduces to upper bound $\Jl_{2,1}^{\lambda}(t)$. To achieve this, we apply Cauchy-Schwarz inequality, Corollary \ref{cor:graditoublip} and Lemma \ref{lem:Mestlip} to obtain
\begin{align}\label{eq:L2convspltinglipMtermJ21}
\Jl_{2,1}^{\lambda}(t)
& =  2\lambda \int_{nT}^t\E\left[ \left\langle \int_{\lfrf{s}}^s \lambda \left(h(\widetilde{\zeta}^{\lambda, n}_r) - h ( \widetilde{\zeta}^{\lambda, n}_{\lfrf{r}})  - \sqrt{2\lambda\beta^{-1}} \int_{{\lfrf{r}}}^r H (\widetilde{\zeta}^{\lambda, n}_{\lfrf{\nu}})\,\rmd B_{\nu}^{\lambda}  
  - \Big(h(\overline{\Theta}^{\lambda}_r) \Big.\right.\right.\right. \nonumber\\
&\qquad \left.\left.\left.\Big. - h (\overline{\Theta}^{\lambda}_{\lfrf{r}}) - \sqrt{2\lambda\beta^{-1}} \int_{{\lfrf{r}}}^r H (\overline{\Theta}_{\lfrf{\nu}}^{\lambda})\,\rmd B_{\nu}^{\lambda} \Big)\right) \,\rmd r, 
\sqrt{2\lambda\beta^{-1}}\int_{\lfrf{s}}^s \mathfrak{M}(\overline{\Theta}^{\lambda}_r, \overline{\Theta}^{\lambda}_{\lfrf{r}}) \, \rmd B^{\lambda}_r   \right\rangle\right] \, \rmd s \nonumber\\
&\quad +2\lambda \int_{nT}^t\E\left[ \left\langle \int_{\lfrf{s}}^s \lambda\sqrt{2\lambda\beta^{-1}} \int_{{\lfrf{r}}}^r (H (\widetilde{\zeta}^{\lambda, n}_{\lfrf{\nu}}) -  H (\overline{\Theta}_{\lfrf{\nu}}^{\lambda}))\,\rmd B_{\nu}^{\lambda}  \,\rmd r,\right.\right. \nonumber\\
&\qquad \left.\left. \sqrt{2\lambda\beta^{-1}}\int_{\lfrf{s}}^s \mathfrak{M}(\overline{\Theta}^{\lambda}_r, \overline{\Theta}^{\lambda}_{\lfrf{r}}) \, \rmd B^{\lambda}_r   \right\rangle\right] \, \rmd s\nonumber\\
&\leq 2\lambda \int_{nT}^t\left(\E\left[ \left| \int_{\lfrf{s}}^s \lambda \left(h(\widetilde{\zeta}^{\lambda, n}_r) - h ( \widetilde{\zeta}^{\lambda, n}_{\lfrf{r}})  - \sqrt{2\lambda\beta^{-1}} \int_{{\lfrf{r}}}^r H (\widetilde{\zeta}^{\lambda, n}_{\lfrf{\nu}})\,\rmd B_{\nu}^{\lambda}  
  - \Big(h(\overline{\Theta}^{\lambda}_r) \Big.\right.\right.\right.\right.\nonumber\\
&\qquad \left.\left.\left.\left.\Big. - h (\overline{\Theta}^{\lambda}_{\lfrf{r}}) - \sqrt{2\lambda\beta^{-1}} \int_{{\lfrf{r}}}^r H (\overline{\Theta}_{\lfrf{\nu}}^{\lambda})\,\rmd B_{\nu}^{\lambda} \Big)\right) \,\rmd r
   \right|^2\right]\right)^{1/2}  \nonumber\\
&\qquad \times  \left(\E\left[\left|\sqrt{2\lambda\beta^{-1}}\int_{\lfrf{s}}^s \mathfrak{M}(\overline{\Theta}^{\lambda}_r, \overline{\Theta}^{\lambda}_{\lfrf{r}}) \, \rmd B^{\lambda}_r\right|^2\right]\right)^{1/2}\, \rmd s\nonumber\\
&\quad +2\lambda^3 \left(e^{- \overline{a} n/4}5  \overline{\cC}_{\mathbf{S}2}\E[|\theta_0|^4] +5 \widetilde{\cC}_{\mathbf{S}2} \right)\nonumber\\
&\leq 2\lambda \int_{nT}^t\left( \lambda^4\left(e^{- \overline{a}n/4}72 \overline{\cC}_{\mathbf{S}2}\E[|\theta_0|^4] +72\widetilde{\cC}_{\mathbf{S}2} \right)\right)^{1/2}  \left(  \lambda^2\left(e^{- \overline{a} n/2}  \overline{\cC}_{\mathbf{S}2}\E[|\theta_0|^4] + \widetilde{\cC}_{\mathbf{S}2} \right)\right)^{1/2}\, \rmd s\nonumber\\
&\quad +10\lambda^3\left(e^{- \overline{a} n/4}  \overline{\cC}_{\mathbf{S}2}\E[|\theta_0|^4] + \widetilde{\cC}_{\mathbf{S}2} \right)\nonumber\\
&\leq 34\lambda^3\left(e^{- \overline{a} n/4}  \overline{\cC}_{\mathbf{S}2}\E[|\theta_0|^4] + \widetilde{\cC}_{\mathbf{S}2} \right).
\end{align}
Substituting \eqref{eq:L2convspltinglipMtermJ21} into \eqref{eq:L2convspltinglipMtermJ2} yields
\begin{equation}\label{eq:L2convspltinglipMtermJ2ub}
\Jl_2^{\lambda}(t) \leq 46\lambda^3\left(e^{- \overline{a} n/4}  \overline{\cC}_{\mathbf{S}2}\E[|\theta_0|^4] + \widetilde{\cC}_{\mathbf{S}2} \right).
\end{equation}
By applying \eqref{eq:L2convspltinglipMtermJ1} and \eqref{eq:L2convspltinglipMtermJ2ub} to \eqref{eq:L2convspltinglipMterm}, we obtain that
\begin{align*}
\E\left[\left|\overline{\Theta}^{\lambda}_t - \widetilde{\zeta}^{\lambda, n}_t\right|^2\right]
&\leq  \lambda (2\overline{L}_3+6)\int_{nT}^t  \E\left[|\overline{\Theta}^{\lambda}_s - \widetilde{\zeta}^{\lambda, n}_s|^2 \right] \, \rmd s +   52\lambda^{2+q}\left(e^{- \overline{a} n/4}  \overline{\cC}_{\mathbf{S}2}\E[|\theta_0|^4] + \widetilde{\cC}_{\mathbf{S}2} \right),
\end{align*}
which, by applying Gr\"{o}nwall's lemma, yields
\[
W_2^2(\mathcal{L}(\overline{\Theta}^{\lambda}_t),\mathcal{L}(\widetilde{\zeta}^{\lambda, n}_t))
\leq \E\left[\left|\overline{\Theta}^{\lambda}_t - \widetilde{\zeta}^{\lambda, n}_t\right|^2\right]
\leq \lambda^{2+q}\left(e^{- \overline{a} n/4}  \cC_{\Lin,0}\E[|\theta_0|^4] + \cC_{\Lin,1} \right),
\]
where
\begin{equation}\label{eq:w1converlipp1const}
 \cC_{\Lin,0}:=52e^{2\overline{L}_3+6} \overline{\cC}_{\mathbf{S}2}, \quad   \cC_{\Lin,1}: = 52e^{2\overline{L}_3+6}\widetilde{\cC}_{\mathbf{S}2}
\end{equation}
with $\overline{\cC}_{\mathbf{S}2}$, $\widetilde{\cC}_{\mathbf{S}2}$ given in \eqref{eq:graditoestlipconst}.
\end{proof}



\end{document}